\documentclass[a4paper,12pt,english,reqno]{amsart}

\usepackage{color,soul}
\usepackage{fullpage}
\usepackage{listings}
\usepackage{hhline}
\usepackage{amssymb}
\usepackage [autostyle, english = american]{csquotes}
\usepackage{url}
\PassOptionsToPackage{unicode}{hyperref}
\usepackage{hyperref}
\MakeOuterQuote{"}
\usepackage{longtable}

\newtheorem{theorem}{Theorem}

\newtheorem{remark}{Remark}

\newtheorem{exam}{Example}

\newcommand{\Z}{{\mathbb Z}}

\newcommand{\Q}{{\mathbb Q}}

\title{Integers representable as a difference of two rational fourth powers}
\author[1]{Ashleigh Ratcliffe}
\address{School of Computing and Mathematical Sciences, University of Leicester, UK}
\email{amw64@leicester.ac.uk}
\author[2]{Nguyen Xuan Tho}
\address{Faculty of Mathematics and Informatics, Hanoi University of Science and Technology \\
Hanoi, Vietnam}
     \email{tho.nguyenxuan1@hust.edu.vn}
\begin{document}

\begin{abstract}  In Section 6.6 of the book {\it Number Theory, Volume I: Tools and Diophantine Equations,  Graduate Texts in Mathematics, Volume 239, Springer (2007)}, Cohen investigated the solubility of the equation $n=x^4+y^4$ in rational numbers $x,y$ for all positive integers $n \leq 10000$. Motivated by this, we study the equation $n=x^4-y^4$ and obtain the complete list of positive integers $n\leq 10000$ that can be represented in this form for some nonzero rational numbers $x$ and $y$. 
\end{abstract}
\maketitle
\section{Introduction} 
The only case of Fermat's Last Theorem proved by Fermat himself  is the case $n=4$, which states that the equation $x^4+y^4=z^4$ has no solutions in nonzero integers $x,y,z$. By rewriting this equation as $(x/z)^4+(y/z)^4=1$ or $(z/y)^4-(x/y)^4=1$, we conclude that $1$ cannot be expressed as either a sum or a difference of two nonzero rational fourth powers. It is therefore natural to ask which integers $n$ admit such representations.

The question about the sum $x^4+y^4$ has been well studied. Demjanenko \cite{MR0228426} investigated the solubility of 
\begin{equation}\label{eq:x4py4n}
x^4+y^4=n
\end{equation}
in nonzero rational numbers $x,y$ for positive integers $n$ and proved that if the curve $x^4+y^2=n$ has the rank at most $1$, then equation \eqref{eq:x4py4n} has no solutions. Serre \cite{Serre} investigated equation \eqref{eq:x4py4n} for $n\leq 100$ and gave the challenge of proving that all rational solutions to the equation $x^4+y^4=17$ are $(x,y)=(\pm 2,\pm 1),(\pm 1,\pm 2)$. This challenge was later resolved by Flynn and Wetherell \cite{Flynn}. Bremner \cite{Bremner1}  and  Silverman \cite{Silverman1} extended Demjanenko's work in different directions. Motivated by the work of Demjanenko \cite{MR0228426} and Serre \cite{Serre}, Cohen \cite[Section 6.6]{MR2312337} investigated equation \eqref{eq:x4py4n} and determined its solubility in rational numbers $x,y$ for all positive integers $n\leq 10000$, except for $n=7537$ and $n=8882$. Bremner and Tho \cite{MR4458887} completed these remaining cases, thereby finishing the list of all positive integers $n\leq 10000$ representable as the sum of two rational fourth powers.

Cohen \cite[Chapter 6, Exercise 53]{MR2312337} asked to determine all positive integers $n \leq 100$ for which the equation
\begin{equation}\label{eq:x4my4n}
	x^4-y^4=n
\end{equation}
is solvable in rational numbers $x,y$. This was done by Grechuk \cite[Section 6.3.4]{mainbook}, where the problem was in fact solved for all $n \leq 218$, and the case $n=219$ was left open. This case and the case $n=8575$ were resolved by Tho in \cite{THO2026125764}. Here, we combine the methods of Grechuk \cite[Section 6.3.4]{mainbook} and Tho \cite{THO2026125764} to obtain the complete list of positive integers $n \leq 10000$ which can be represented as a difference of two nonzero rational fourth powers\footnote{One may add perfect fourth powers to our list to obtain the list of all positive integers $n \leq 10000$ which can be represented as a difference of any (not necessarily nonzero) two rational fourth powers.}. This is on par with  the survey of Cohen \cite{MR2312337} and the work of Bremner and Tho \cite{MR4458887} for the sum of two rational fourth powers.  For the related problem of sum of two sixth powers, see the work of  Newton and Rouse \cite{NR}. We now state the main result of this paper. 
\begin{theorem}\label{th:main}
The only integers $1\leq n \leq 10000$ which can be represented as a difference of two nonzero rational fourth powers are precisely the numbers listed in Table \ref{tb:rationalsols}.
\end{theorem}

%In \cite{mainbook} this problem was solved for $0<n<219$, then in a recent paper of \cite{THO2026125764}, the equation \eqref{eq:x4my4mnz4} with $n=219$ is now solved. 
In Section \ref{sec:methods}, we explain and implement the methods used  in \cite{mainbook}, which solve \eqref{eq:x4my4n} for most values of $n$ in the range $1 \leq n \leq 10000$. In Section \ref{sec:factor}, we explain and implement the method from  \cite{THO2026125764} to prove the insolubility of the remaining cases. For each value of $n$, the methods in this paper either prove that there are no rational solutions to \eqref{eq:x4my4n}, or produce a rational solution. The values of $n\leq 10000$ for which equation \eqref{eq:x4my4n} has a rational solution are listed in Table \ref{tb:rationalsols}. 
% along with the representation of the integer as a difference of two rational fourth powers.
% Searching for rational solutions to these equations can be done easily in Magma.

An integer $n>0$ is expressible as the difference of two nonzero rational fourth powers if and only if the equation
	\begin{equation}\label{eq:x4my4mnz4}
		x^4-y^4=nz^4
	\end{equation} 
 has an integer solution $(x,y,z)$ with $xyz \neq 0$. 
 
 If $n=k^4 n'$ for some integer $k\geq 1$ and $n'$ fourth-power free, then $(x,y)$ is a rational solution to \eqref{eq:x4my4n} if and only if $(x/k,y/k)$ is a rational solution to $x^4-y^4=n'$. Hence, it suffices to determine the solubility of \eqref{eq:x4my4n} and \eqref{eq:x4my4mnz4} for fourth power free $n$. 
 
%If $n=k^4 n'$ for some integer $k\geq 1$ and $n'$ fourth-power free, then $(x,y,z)$ is an integer solution to \eqref{eq:x4my4mnz4} if and only if $(x/k,y/k,z)$ is an integer solution to $x^4-y^4=n'z^4$. Hence, it is sufficient to determine the solvability of \eqref{eq:x4my4mnz4} for fourth-power-free $n$. 

%If $n$ is a perfect fourth power, then $n=k^4$ for some integer $k>0$, hence equation \eqref{eq:x4my4mnz4} has a solution $(x,y,z)=(k,0,1)$, and any other solution will be a solution $(x,y,kz)$ to \eqref{eq:x4my4mnz4} with $n=n'$ where $n'$ is fourth-power free. 

\section{Elliptic Curves, Mordell-Weil sieve, and Pythagorean triples }\label{sec:methods}

\subsection{Easy reductions}\label{sec:easy}
Finding small integer solutions to \eqref{eq:x4my4mnz4} is easy with computer algebra systems such as Magma \cite{MR1484478}. For example, given $n$, the code:
\begin{lstlisting}
P<x,y,z> := ProjectiveSpace(Rationals(),2); 
C := Curve(P,x^4-y^4-n z^4); 
RationalPoints(C : Bound := 10000);
\end{lstlisting}
returns all solutions to \eqref{eq:x4my4mnz4} with height bounded by $10000$. This is how solutions in Table \ref{tb:rationalsols} were found. The difficult part is to prove that for the remaining values of $n$ equation \eqref{eq:x4my4mnz4} has no integer solution with $xyz \neq 0$. Several known methods can be implemented to prove nonexistence of solutions.
First, for any nontrivial integer solution $(x,y,z)$ to \eqref{eq:x4my4mnz4}:
\begin{itemize} 
	\item[] $(X,Y)=\left(\dfrac{y^2}{z^2},\dfrac{x^2y}{z^3}\right)$ is a rational solution to the equation $Y^2=X^3+nX$, 
	\item[] $(X,Y)=\left(\dfrac{x^2}{z^2},\dfrac{xy^2}{z^3}\right)$ is a rational solution to the equation $Y^2=X^3-nX$, 
	\item[] $(X,Y)=\left(\dfrac{nx^2}{y^2},\dfrac{n^2xz^2}{y^3}\right)$ is a rational solution to the equation $Y^2=X^3-n^2X$.
\end{itemize}
Hence, if we can prove that any of the curves
\begin{equation}\label{eq:easycurves}
Y^2=X^3-n^2X, \quad Y^2=X^3-nX, \quad Y^2=X^3+nX
\end{equation}
has no rational points corresponding to $xyz \neq 0$, then equation \eqref{eq:x4my4mnz4} has no nontrivial integer solutions. If any of the curves \eqref{eq:easycurves} have rank $0$, then it has finitely many rational points and these are its torsion points. We can use Magma \cite{MR1484478} to find all torsion points of a given elliptic curve, and then we can easily check whether these points correspond to integer solutions $(x,y,z)$ with $xyz \neq 0$ of \eqref{eq:x4my4mnz4}. This approach allows us to prove the insolubility of \eqref{eq:x4my4mnz4} for all integers $1\leq n \leq 10000$ for which one of the curves \eqref{eq:easycurves} has rank $0$. From now on, we consider only values of $n$ for which all the curves \eqref{eq:easycurves} have positive rank. For these values of $n$, we use the following theorem.  
	
\begin{theorem}\label{T2}\cite[Section 6.3.4]{mainbook} 
We can reduce the problem of solving \eqref{eq:x4my4mnz4} in integers to solving equations of the form
\begin{equation}\label{eq:abcuvw}
a^2u^8+b^2v^8=c w^4
\end{equation}
for pairwise coprime positive integers $a,b,c$ with $abc\in\{n/8,2n\}$ and integer variables $u,v,w$ such that
\begin{equation}\label{eq:cond}
\gcd(au,bv)= \gcd(au,cw)= \gcd(bv,cw)= 1.
\end{equation}
\end{theorem}

\begin{proof} 
To be self-contained, we include the proof. Assume that an integer solution $(x,y,z)$ with $xyz \neq 0$ to \eqref{eq:x4my4mnz4} exists. Without loss of generality, we can assume that $x,y,z >0$, since if $(x,y,z)$ is a solution, then so are $(\pm x,\pm y, \pm z)$. We may further assume that $\gcd(x,y,z)=1$, as if $d=\gcd(x,y,z)>1$, then $(x/d,y/d,z/d)$ is also a solution with $\gcd(x/d,y/d,z/d)=1$.  We can further assume that $x,y,z$ are pairwise coprime. If $x$ and $y$ share a common factor $p$, then the left-hand side of \eqref{eq:x4my4mnz4} is divisible by $p^4$, and therefore the right-hand side must also be divisible by $p^4$, and $n$ is not divisible by a fourth power, so $z$ must be. A similar argument shows that $\gcd(x,z)=\gcd(y,z)=1$. 
Therefore, either $x$ and $y$ are both odd, or $x$ and $y$ have different parities. 

We first consider the case that $x$ and $y$ are both odd. We use the change of variables $t=(x+y)/2$, $s=(x-y)/2$, where $t$ and $s$ have different parity and are coprime. After substituting into \eqref{eq:x4my4mnz4}, we obtain
$$
8st(s^2+t^2)=nz^4.
$$
We must have either that $n$ is divisible by $8$, in which case let $n'=n/8$ and $z=r$, otherwise, $z$ is even so can be written as $z=2r$ for integer $r$ and let $n'=2n$. In either case, we obtain 
\begin{equation}\label{eq:red}
st(s^2+t^2)=n'r^4.
\end{equation}

Next we consider the case that $x$ and $y$ have different parity. We use the change of variables $t=x+y$ and $s=x-y$, where $s$ and $t$ are both odd and coprime (because $x$ and $y$ are coprime). After substituting into \eqref{eq:x4my4mnz4}, we obtain
$$
st(s^2+t^2)=2nz^4,
$$
letting $n'=2n$ and $z=r$, we obtain \eqref{eq:red}.
%We must have either that $n$ is divisible by $8$, in which case let $n'=n/8$ and $z=r$, otherwise, $z$ is even so can be written as $z=2r$ for integer $r$, and we obtain \eqref{eq:red} where $n'=2n$. 

Because $s$ and $t$ are coprime, $s^2+t^2$ is also coprime to $s$ and $t$. Solutions to \eqref{eq:red} must satisfy
\begin{equation}\label{stabc}
s=au^4, \quad t=bv^4, \quad s^2+t^2=c w^4,
\end{equation}
where $\gcd(au,bv)=\gcd(au,cw)=\gcd(bv,cw)=1$ and $abc=n'$.
After substituting \eqref{stabc} into \eqref{eq:red}, we obtain 
$$
a^2u^8+b^2v^8=c w^4.
$$
\end{proof}
If $c$ is a perfect fourth power and $a$ or $b$ is a perfect square, the equation \eqref{eq:abcuvw} is easy to solve as we can reduce the equation to 
\begin{equation}\label{eq:mordell}
X^4-Y^4=Z^2,
\end{equation}
which only has integer solutions with $XYZ=0$ (see Mordell \cite[Theorem~2, page 17]{Mordell}). 

\subsection{Local obstructions and Mordell-Weil sieve}\label{sec:localsieve}
In many cases, the equation \eqref{eq:abcuvw} has local obstructions. That is, for some prime $p$ and positive integer $k$, all solutions to the equation modulo $p^k$ have at least two of $u,v,w$ divisible by $p$, which contradicts the pairwise coprimality of $u,v,w$. 

For other cases, we can implement a similar method as in Section \ref{sec:easy}, reducing the equations to elliptic curves. 
% Some of these equations can be solved by considering the corresponding homogeneous equations 
% \begin{equation}\label{eq:4abcuvw}
% a^2U^4+b^2V^4=c w^4=0
% \end{equation}
% where $(U,V)=(u^2,v^2)$. We can reduce this equation to the following elliptic curves.
The problem of solving \eqref{eq:abcuvw} can be reduced to finding rational points on the curves
\begin{equation}\label{ell:hardcurves}
(i) \,\, Y^2=X^3+a^2b^2c^2X, \quad (ii) \,\, Y^2=X^3-a^2b^4cX, \quad (iii) \,\, Y^2=X^3-a^4b^2cX.
\end{equation}
We have that for any nontrivial integer solution $(u,v,w)$ to \eqref{eq:abcuvw}:
\begin{itemize} 
	\item[] $(X,Y)=\left(\dfrac{b^2cv^4}{u^4},\dfrac{b^2c^2v^2w^2}{u^6}\right)$ is a rational solution to equation \eqref{ell:hardcurves}(i),%$Y^2=X^3+a^2b^2c^2X$, 
	\item[] $(X,Y)=\left(\dfrac{b^2cw^2}{u^4},\dfrac{b^4cv^4w}{u^6}\right)$ is a rational solution to equation \eqref{ell:hardcurves}(ii),% $Y^2=X^3-a^2b^4cX$, 
	\item[] $(X,Y)=\left(- \dfrac{a^2b^2v^4}{w^2},\dfrac{a^4b^2u^4v^2}{w^3}\right)$ is a rational solution to equation \eqref{ell:hardcurves}(iii).% $Y^2=X^3-a^4b^2cX$.
\end{itemize}
If any of the curves in \eqref{ell:hardcurves} has rank $0$, then the curve has finitely many rational points, we can use Magma \cite{MR1484478} to find them and to check whether these points correspond to integer solutions $(x,y,z)$ with $xyz \neq 0$ to \eqref{eq:x4my4mnz4}. However, in cases when all of these curves have positive rank, we must use different methods to solve \eqref{eq:abcuvw}.

The next method we use is the Mordell-Weil sieve. 
%We now present details of the Mordell-Weil sieve, first by describing the general method and then applying this to an example. 
The elliptic curves in \eqref{ell:hardcurves} are equivalent to the curves 
\begin{equation}\label{ell:hardcurveshom}
(i) \,\, \mathbf{Y}^2\mathbf{Z}=\mathbf{X}^3+a^2b^2c^2\mathbf{X}\mathbf{Z}^2, \quad (ii) \,\, \mathbf{Y}^2\mathbf{Z}=\mathbf{X}^3-a^2b^4c\mathbf{X}\mathbf{Z}^2, \quad (iii) \,\, \mathbf{Y}^2\mathbf{Z}=\mathbf{X}^3-a^4b^2c\mathbf{X}\mathbf{Z}^2
\end{equation}
which are the curves in homogeneous form. To transform curves \eqref{ell:hardcurves} to curves \eqref{ell:hardcurveshom} we use the change of variables $(X,Y)\to (\mathbf{X/Z},\mathbf{Y/Z})$. Any rational point $(X,Y)$ on \eqref{ell:hardcurves} can be expressed as an integer point $(\mathbf{X},\mathbf{Y},\mathbf{Z})$ on \eqref{ell:hardcurveshom}. For an elliptic curve in \eqref{ell:hardcurveshom} with positive rank, its rational points are generated by
a finite list of generators $(\mathbf{X}, \mathbf{Y},\mathbf{Z}) = (X_i : Y_i : Z_i)$. 
We first present Theorem 6.13 of \cite{mainbook}. Let $\psi$ be the ``reduction modulo $p$
map'', sending any $(X : Y : Z)$ to $(X \text{\, mod\,\,} p : Y \text{\, mod\,\,} p : Z \text{\, mod\,\,} p)$.
\begin{theorem}\label{th:mwsieve}
Let $A,B$ be integers such that $D= 4A^3 + 27B^2 \neq 0$. Let $P_1$ and $P_2$ be rational points on the elliptic curve 
\begin{equation}\label{eq:projshortwf}
Y^2Z=X^3+AXZ^2+BZ^3.
\end{equation}
Let $p$ be a prime not dividing $D$. Then
$$
\psi(P_1 + P_2)= \psi(P_1) + \psi(P_2),
$$
where the “$+$” on the left-hand side is the usual addition of points on elliptic curves, while the “$+$” on the right-hand side is defined by the same formulas but with arithmetic operations modulo $p$.
\end{theorem}
The proof of Theorem \ref{th:mwsieve} can be found in \cite[Section A.5]{silverman1992rational}. All rational points on \eqref{eq:projshortwf} are given by
$$
P=\sum_{i=1}^m k_i P_i,
$$
where $\{P_1,\dots,P_m\}$ is a list of generators. By Theorem \ref{th:mwsieve},
$$
\psi(P)=\sum_{i=1}^m k_i \psi(P_i).
$$
The sequence $(n\psi(P_i))_{n\in \Z}$ must be periodic with some period $T_i$. Hence, we can assume that $0 \leq k_i <T_i$ for all $i$. This allows us to compute all possible values of $\psi(P)$. Therefore, although the multiplication of the points $P_i$ generates infinitely many integer triples $(X_i : Y_i : Z_i )$, by taking these triples modulo some prime $p$, which does not
divide $2abc$, we obtain a list of generators of a certain finite abelian group.

%Therefore, while an elliptic curve with positive rank has infinitely many rational points, the set of rational points modulo $p$ forms a cyclic group. 

\textbf{Computational implementation.}
Fix a prime $p$ not dividing $2abc$. We can solve equation \eqref{eq:abcuvw} modulo $p$. For each solution $(u,v,w)$, we can perform the change of variables $(u, v, w) \to (X, Y, Z)$ (to reduce the equation to an elliptic curve of the form \eqref{ell:hardcurveshom}) and obtain a list of possible integer points $(X,Y,Z)$ on the curve. We can compute the generators of the curve \eqref{ell:hardcurveshom} and obtain the group of solutions $(X,Y,Z)$ modulo $p$. If the intersection of both lists obtained is empty then there are no solutions to \eqref{eq:abcuvw}. For some equations, we may need to execute this procedure on more than one of the curves \eqref{ell:hardcurves} to obtain an empty intersection.

\begin{exam}
Let us consider $n=31$. In this case, all triples $(a,b,c)$ are in the set $$
\{(1,1,62),(2,1,31),(31,1,2),(62,1,1),(2,31,1) \}.
$$
For triple $(1,1,62)$, equation \eqref{eq:abcuvw} has no solutions in coprime integers as modulo $8$ analysis implies that at least two of the variables are divisible by $2$. For triple $(2,1,31)$, equation \eqref{eq:abcuvw} has no solutions in coprime integers as modulo $4$ analysis implies that at least two of the variables are divisible by $2$.  For triple $(62,1,1)$, equation \eqref{ell:hardcurves} (iii) has rank $0$. For triple $(2,31,1)$, equation \eqref{ell:hardcurves} (ii) has rank $0$. It remains to consider triple $(31,1,2)$, that is, we need to solve equation 
\begin{equation}\label{eq:trip3112}
961u^8+v^8=2 w^4
\end{equation}
in pairwise coprime variables $u,v,w$. As explained above, we can reduce the equation to the following elliptic curve
\begin{equation}\label{ex:elltrip3112}
Y^2Z=X^3+3844XZ^2
\end{equation}
with the change of variables $(X,Y,Z)=\left(2u^2v^4,4v^2w^2,u^6\right)$.

The group of rational points on curve \eqref{ex:elltrip3112} is generated by the rational points $$\{(98:-1148:1),(0:0:1)\}.$$ %The group of rational points on curve (ii) is generated by points $\{(50,170,1),(-62,-961,8),(0,0,1)\}$. The group of rational points on curve (iii) is generated by points $\{(1922,-59582,1),(0,0,1)\}$. 

The solutions of \eqref{eq:trip3112} modulo $5$ are 
$$
\begin{aligned}
(u,v,w)\in \{ & (0,0,0),(n_1,n_2,n_3): n_i \in \{1,2,3,4\}\},
\end{aligned}
$$
after transforming the solutions to variables $X,Y,Z$ modulo $5$, we obtain that the rational points on \eqref{ex:elltrip3112} modulo $5$ are
$$
\begin{aligned}
(X,Y,Z)\in \{ & (0,0,0),(2,1,1),(2,4,1)\}.
\end{aligned}
$$

Multiplying the generators, we obtain that the list of all possible points on \eqref{ex:elltrip3112} modulo $5$ is
$$
(X,Y,Z) \in \{(0,1,0),(0,0,1),(3,2,1),(3,3,1) \}.
$$
The intersection of these lists is empty. Therefore, there are no integer solutions of \eqref{eq:trip3112} in pairwise coprime variables.
\end{exam}

\begin{table}
\begin{center}
\begin{tabular}{|c|c||c|c|c|c|}
	\hline
	$n$  & Equation & $n$  & Equation  \\\hline \hline
    $219$ & $4u^8 + 9v^8 - 73w^4=0$ &
    $463$ & $214369 u^8 + v^8 - 2 w^4=0$ \\\hline
   $514$ & $16u^8+v^8-257w^4=0$ &
   $885$ & $87025u^8+36v^8-w^4=0$ \\\hline
$1269$ & $2916 u^8 + 2209 v^8 -  w^4=0$  &
$1305$ & $81 u^8 + 4 v^8 - 145  w^4=0$ \\\hline
$1305$ & $324 u^8 + v^8 - 145  w^4=0$  &
$1493$ & $4 u^8 + v^8 - 1493  w^4=0$ \\\hline
 1561 & $49729 u^8 + 196 v^8 -  w^4=0$ &
 1610 & $25921 u^8 + 400 v^8 -  w^4=0$ \\\hline
1640 & $25 u^8 + v^8 - 41  w^4=0$ &
 1731 & $36 u^8 + v^8 - 577  w^4=0$ \\\hline 
 2035 & $121 u^8 + 4 v^8 - 185  w^4=0$ &
 2130 & $5041 u^8 + 3600 v^8 -  w^4=0$ \\\hline
 2130 & $80656 u^8 + 225 v^8 -  w^4=0$ &
2213 & $4 u^8+v^8 -2213  w^4=0$  \\\hline
 2719 & $7392961 u^8 + v^8 - 2  w^4=0$ &
 2329 & $18769 u^8 + 1156 v^8 -  w^4=0$  \\\hline
 2810 & $ 25 u^8 + 16 v^8 - 281  w^4=0$ &
 3185 & $2401 u^8 + 4 v^8 - 65  w^4=0$ \\\hline
 3185 & $9604 u^8 + v^8 - 65  w^4=0$ &
 3265 & $4 u^8 + v^8 - 3265  w^4=0$ \\\hline
 3503 & $51076 u^8 + 961 v^8 -  w^4=0$ & 
 3540 & $87025 u^8 + 576 v^8 -  w^4=0$  \\\hline
3690 & $32400 u^8 + 1681 v^8 -  w^4=0$ &
3690 & $136161 u^8 + 400 v^8 -  w^4=0$ \\\hline
3906 & $15376 u^8 + 3969 v^8 -  w^4=0$ &
3906 & $63504 u^8 + 961 v^8 -  w^4=0$  \\\hline
3906 & $77841 u^8 + 784 v^8 -  w^4=0$&
3906 & $1245456 u^8 + 49 v^8 -  w^4=0$ \\\hline
 4100 & $64 u^8 + v^8 - 1025  w^4=0 $ & 
4165 & $2401 u^8 + 4 v^8 - 85  w^4=0$  \\\hline
4165 & $9604 u^8 + v^8 - 85  w^4=0$ &
 4317 & $2070721 u^8+36 v^8- w^4=0$ \\\hline
 4357 & $4 u^8 + v^8 - 4357  w^4=0$ &
 4559 &$ 37636 u^8 + 2209 v^8 -  w^4=0$  \\\hline
 4669& $ 1779556u^8+49v^8- w^4=0 $ &
 4901 & $4u^8+v^8-4901 w^4=0$ \\\hline
 4961 & $14641u^8+v^8-82 w^4=0$   &
 4991 & $188356 u^8 + 529 v^8 -  w^4=0$  \\\hline
 5076 & $46656 u^8 + 2209 v^8 -  w^4=0$ &
 5775& $1764 u^8 + 121 v^8 - 25  w^4=0$  \\\hline
 5807& $33721249 u^8+v^8 -2  w^4=0 $ &
 5911 & $66049 u^8 + 2116 v^8 -  w^4=0$ \\\hline
 5911 & $264196 u^8 + 529 v^8 -  w^4=0$ &
 5983 & $148996 u^8 + 961 v^8 -  w^4=0$  \\\hline
 6001 & $289 u^8+4 v^8-353  w^4=0$ &
6117& $4157521 u^8+36 v^8- w^4=0$ \\\hline
6244 & $49729 u^8 + 3136 v^8 -  w^4=0$ &
 6355 & $168100 u^8 + 961 v^8 -  w^4=0$ \\\hline
 6440 & $25921 u^8 + 25 v^8 -  w^4=0$ &
 6625 & $4 u^8 + v^8 - 6625  w^4=0$ \\\hline
 6625 & $15625 u^8 + v^8 - 106  w^4=0$ &
 6629 & $3587236 u^8+49 v^8 - w^4=0 $ \\\hline
 6644 & $22801 u^8 + 7744 v^8 -  w^4=0$ &
 6821 & $128881 u^8+ 1444 v^8- w^4=0 $ \\\hline
 6984 & $81 u^8 + v^8 - 97  w^4=0$ &
 7101 & $69169 u^8 + 2916 v^8 -  w^4=0$ \\\hline
 7769 & $208849 u^8 + 1156 v^8 -  w^4=0$  &
 7861 & $5044516 u^8+49 v^8- w^4=0$ \\\hline
 8001 & $64516 u^8 + 3969 v^8 -  w^4=0$  &
 8001 & $5225796 u^8 + 49 v^8 -  w^4=0$ \\\hline
 8005 & $100 u^8 + v^8 - 1601  w^4=0$  &
8075 & $361 u^8 + 4 v^8 - 425  w^4=0$ \\\hline
 8133 & $7349521 u^8 + 36 v^8 -  w^4=0$  &
 8205 & $7480225 u^8 + 36 v^8 -  w^4=0$ \\\hline
 8235 & $2916u^8+25v^8-61 w^4=0$ &
 8520 & $5041 u^8 + 225 v^8 -  w^4=0$   \\\hline
8636 & $18496 u^8 + 16129 v^8 -  w^4=0$ &
 8705 & $4u^8+v^8-8705 w^4=0$ \\\hline
 8931 & $36u^8+v^8-2977 w^4=0$  &
 8965 & $106276 u^8 + 3025 v^8 -  w^4=0$ \\\hline
 9015 & $36u^8+25v^8-601 w^4=0$  & 
 9357 & $9728161 u^8+36 v^8 - w^4=0$  \\\hline
 
  9860 & $21025 u^8 + 18496 v^8 -  w^4=0$ &
 9991 & $37636u^8+10609v^8- w^4=0$ \\\hline
\end{tabular}
\caption{Equations not solvable by the methods in Sections \ref{sec:easy} or \ref{sec:localsieve}. \label{tb:outputted}}
\end{center}
\end{table}
\subsection{Pythagorean Triples}\label{sec:pythag}
In many cases, the methods of Section \ref{sec:easy} and \ref{sec:localsieve} are enough to solve \eqref{eq:x4my4mnz4} completely. We next give an elementary method which can be implemented if $c$ is a perfect square.

It is well known that all integer solutions to the equation
\begin{equation}\label{eq:pythagoras}
x^2+y^2=z^2
\end{equation} 
can be described as a finite union of polynomial parametrisations and we call solutions $(x,y,z)$ Pythagorean triples.
Moreover, any solution $(x,y,z)$ to \eqref{eq:pythagoras} with $x,y,z$ pairwise coprime is called a primitive Pythagorean triple and it is well known \cite{mainbook} that these solutions (up to the exchange of $x$ and $y$) are completely described by
\begin{equation}\label{sol:pythag}
(x,y,z)=(2mn,m^2-n^2,m^2+n^2), \quad m,n \in \mathbb{Z}, \quad 2|mn \text{ and } \gcd(m,n)=1.
\end{equation}

By \eqref{eq:abcuvw}, $a$ and $b$ are either opposite parity or they are both odd. If $a$ and $b$ have opposite parity, we label the equation so that $a$ is even.
Given an equation \eqref{eq:abcuvw} with $c$ a perfect square, that is $c=k^2$ for some positive integer $k$. We can then reduce \eqref{eq:abcuvw} to \eqref{eq:pythagoras} using the change of variables $x=au^4$, $y=bv^4$ and $z=kw^2$. We then obtain the system of equations 
\begin{equation}\label{sys:pyabopp}
2mn=au^4, \quad m^2-n^2=bv^4, \quad m^2+n^2=kw^2.
\end{equation}
If $a$ and $b$ are both odd, then only one of $u$ or $v$ may be even, so we must also solve the system
\begin{equation}\label{sys:pyabodd}
m^2-n^2=au^4, \quad 2mn=bv^4, \quad m^2+n^2=kw^2.
\end{equation}
As $a,b,k$ are positive, we may assume that $m>n>0$. Then equation $2mn=au^4$ has two-monomials and for which we can describe all solutions with $m,n$ coprime, and these will be of the form $m=dU^4$, $n=eV^4$ where $d,e$ are coprime integers and $U,V$ are coprime variables. We then substitute these into $m^2-n^2=bv^4$ and obtain\footnote{The equation of the form \eqref{eq:depythag} for \eqref{sys:pyabodd} is obtained similarly.}
\begin{equation}\label{eq:depythag}
d^2U^8-e^2V^8=bv^4,
\end{equation}
and we inherit the condition
\begin{equation}\label{cond:depythag}
\gcd(dU, eV) = \gcd(dU, bv) = \gcd(eV, bv) = 1
\end{equation}
which is a similar equation to \eqref{eq:abcuvw} and we can use the same methods described earlier. Note that in this case, the curves \eqref{ell:hardcurves} are now
$$
(i) \,\, Y^2=X^3-d^2e^2b^2X, \quad (ii) \,\, Y^2=X^3-d^2e^4bX, \quad (iii) \,\, Y^2=X^3+d^4e^2bX.
$$

\begin{exam}
To illustrate this method, we consider \eqref{eq:x4my4mnz4} with $n=885$. The methods of Sections \ref{sec:easy} and \ref{sec:localsieve} cannot solve the equation
\begin{equation}\label{eq:n885}
    36u^8+87025v^8=w^4,
\end{equation}
which is \eqref{eq:abcuvw} with $(a,b,c)=(6,295,1)$.
We can then reduce \eqref{eq:n885} to \eqref{eq:pythagoras} using the change of variables $x=6u^4$, $y=295v^4$ and $z=w^2$. By \eqref{sol:pythag}, we then obtain the system of equations 
$$
2mn=6u^4, \quad m^2-n^2=295v^4, \quad m^2+n^2=w^2.
$$
We have $mn=3u^4$, whose solutions in coprime variables are $(m,n,u)=(3U^4,V^4,UV)$ or $(U^4,3V^4,UV)$ for coprime integers $U,V$. In the first case, substituting $(m,n)=(3U^4,V^4)$ into $m^2-n^2=295v^4$ we obtain 
$$
9U^8-V^8=295v^4
$$
and modulo $3$ analysis shows that $V$ and $v$ are both divisible by $3$, however then both $m$ and $n$ are divisible by $3$, a contradiction.
In the second case, substituting $(m,n)=(U^4,3V^4)$ into $m^2-n^2=295v^4$ we obtain 
$$
U^8-9V^8=295v^4
$$
and modulo $5$ analysis shows that this equation has no solutions with coprime $U,V$.
Therefore \eqref{eq:n885} has no integer solutions satisfying \eqref{eq:cond} and this completes the proof that \eqref{eq:x4my4mnz4} with $n=885$ has no integer solutions with $xyz \neq 0$.
\end{exam}
\subsection{Implementation for \texorpdfstring{$1\leq n \leq 10000$}{}}\label{sec:implement}

We implemented the methods of Sections \ref{sec:easy} and \ref{sec:localsieve} in Magma \cite{MR1484478}. The code first checks whether the equation has a nonzero integer solution. It next calculates the rank of the curves \eqref{eq:easycurves} and excludes equations which have rank $0$, it finds all triples $(a,b,c)$ satisfying $abc=n'$, checks for local solubility, and then proceeds to calculate rank of curves \eqref{ell:hardcurves}. If the rank of all the curves are positive, it tries to find a suitable prime to obtain a contradiction in the group of rational points. If it cannot find a suitable prime $p<229$ it outputs the unsolvable case, which needs to be investigated by hand as to whether it can be solved using the method of Section \ref{sec:factor}.  Continuing the investigation up to $n=10000$, the code outputs the equations in Table \ref{tb:outputted} as not solvable by methods of Sections \ref{sec:easy}, \ref{sec:localsieve}. As we can see from Table \ref{tb:factorisation}, the methods of Sections \ref{sec:easy}, \ref{sec:localsieve}, and \ref{sec:pythag} can determine the solubility of all equations \eqref{eq:x4my4mnz4} with $1\leq n \leq 218$, and the only exceptions in the range $1\leq n \leq 1000$ are $n=219,463$, and $514$. Overall, there are $33$ (non-equivalent) values of $n$ in the range $1\leq n \leq 10000$ for which these methods are not sufficient. Equation \eqref{eq:x4my4mnz4} with $n=219$ is solved with details in \cite{THO2026125764}. We will see in Section \ref{sec:factor} that all other remaining cases in the range $1\leq n \leq 10000$ are also solvable by the method in \cite{THO2026125764}.

\subsection{How to use the Magma program} \label{magma}
Load\footnote{The Magma files for this implementation are available on Github \url{https://github.com/ashleighratcliffe/DifferenceOfPowers}.} file ``diffpowerex.m'' (which also loads files ``fullsieve.m'' and ``sieve.m''), the program is executed using the command 
$$
\texttt{Diff4powers(k,m)}, 
$$
where $k$ and $m$ are integers which are the coefficients of the equation
 \begin{equation}
kx^4 - ky^4= mz^4.
\end{equation}
The outputs can be ``factor 4th power'', ``excluded at easy curves'', ``Exclude by Mordell's Theorem'', $[\,\, ]$, or a set of equations. The first means the integer is equivalent to solving one with smaller $n$\footnote{Because we did this investigation in order of $n$, the equation with smaller $n$ have already been solved.}. The second means that at least one of the curves \eqref{eq:easycurves} has rank $0$. The third means that some equations \eqref{eq:abcuvw} are reducible to solving \eqref{eq:mordell}. The fourth means that all triples $(a,b,c)$ have corresponding equations \eqref{eq:abcuvw} have either local obstructions or at least one of the corresponding curves \eqref{ell:hardcurves} have rank $0$, or the equation is solvable using the Mordell-Weil sieve method explained above. If the output is a set of equations, these equations must be solved using the method in Section \ref{sec:pythag} or \ref{sec:factor}.

\begin{remark}
In Table \ref{tb:outputted}, for any equations where $c$ is a perfect square, we can attempt to implement the method in Section \ref{sec:pythag}. This method works in all but 6 cases. The details to solve these equations are presented in Table \ref{tb:pythag}.  This significantly reduces the number of equations left to solve, and this list is presented in Table \ref{tb:factorisation}, and Section \ref{sec:factor} solves these equations explicitly. 
\end{remark}

After solving the equations in Table \ref{tb:pythag}, we obtain that the equations listed in Table \ref{tb:factorisation} are not solvable by the methods of Sections \ref{sec:easy}, \ref{sec:localsieve} or \ref{sec:pythag}. Solving the equations in Table \ref{tb:factorisation} is the subject of Section \ref{sec:factor}.
\begin{table}
\begin{center}
\begin{tabular}{|c|c||c|c|c|c|}
	\hline
	$n$  & Equation & $n$  & Equation  \\\hline \hline
    $219$ & $4u^8 + 9v^8 - 73w^4=0$ &
    $463$ & $214369 u^8 + v^8 - 2 w^4=0$ \\\hline
   $514$ & $16u^8+v^8-257w^4=0$ &
  % $885$ & $87025u^8+36v^8-w^4=0$ \\\hline
%$1269$ & $2916 u^8 + 2209 v^8 -  w^4=0$  &
$1305$ & $81 u^8 + 4 v^8 - 145  w^4=0$ \\\hline
$1305$ & $324 u^8 + v^8 - 145  w^4=0$  &
$1493$ & $4 u^8 + v^8 - 1493  w^4=0$ \\\hline
 %1561 & $49729 u^8 + 196 v^8 -  w^4=0$ &
 1610 & $25921 u^8 + 400 v^8 -  w^4=0$ &
1640 & $25 u^8 + v^8 - 41  w^4=0$ \\\hline
 1731 & $36 u^8 + v^8 - 577  w^4=0$ &
 2035 & $121 u^8 + 4 v^8 - 185  w^4=0$ \\\hline
% 2130 & $5041 u^8 + 3600 v^8 -  w^4=0$ \\\hline
 %2130 & $80656 u^8 + 225 v^8 -  w^4=0$ &
2213 & $4 u^8+v^8 -2213  w^4=0$  &
 2719 & $7392961 u^8 + v^8 - 2  w^4=0$ \\\hline
% 2329 & $18769 u^8 + 1156 v^8 -  w^4=0$  \\\hline
 2810 & $ 25 u^8 + 16 v^8 - 281  w^4=0$ &
 3185 & $2401 u^8 + 4 v^8 - 65  w^4=0$ \\\hline
 3185 & $9604 u^8 + v^8 - 65  w^4=0$ &
 3265 & $4 u^8 + v^8 - 3265  w^4=0$ \\\hline
% 3503 & $51076 u^8 + 961 v^8 -  w^4=0$ & 
% 3540 & $87025 u^8 + 576 v^8 -  w^4=0$  \\\hline
%3690 & $32400 u^8 + 1681 v^8 -  w^4=0$ &
%3690 & $136161 u^8 + 400 v^8 -  w^4=0$ \\\hline
3906 & $15376 u^8 + 3969 v^8 -  w^4=0$ &
3906 & $63504 u^8 + 961 v^8 -  w^4=0$  \\\hline
%3906 & $77841 u^8 + 784 v^8 -  w^4=0$&
%3906 & $1245456 u^8 + 49 v^8 -  w^4=0$ \\\hline
 4100 & $64 u^8 + v^8 - 1025  w^4=0 $ & 
4165 & $2401 u^8 + 4 v^8 - 85  w^4=0$  \\\hline
4165 & $9604 u^8 + v^8 - 85  w^4=0$ &
% 4317 & $2070721 u^8+36 v^8- w^4=0$ \\\hline
 4357 & $4 u^8 + v^8 - 4357  w^4=0$ \\\hline
% 4559 &$ 37636 u^8 + 2209 v^8 -  w^4=0$  \\\hline
 %4669& $ 1779556u^8+49v^8- w^4=0 $ &
 4901 & $4u^8+v^8-4901 w^4=0$ &
 4961 & $14641u^8+v^8-82 w^4=0$   \\\hline
% 4991 & $188356 u^8 + 529 v^8 -  w^4=0$  \\\hline
% 5076 & $46656 u^8 + 2209 v^8 -  w^4=0$ &
% 5775& $1764 u^8 + 121 v^8 - 25  w^4=0$  \\\hline
 5807& $33721249 u^8+v^8 -2  w^4=0 $ &
 5911 & $66049 u^8 + 2116 v^8 -  w^4=0$ \\\hline
% 5911 & $264196 u^8 + 529 v^8 -  w^4=0$ &
% 5983 & $148996 u^8 + 961 v^8 -  w^4=0$  \\\hline
 6001 & $289 u^8+4 v^8-353  w^4=0$ &
%6117& $4157521 u^8+36 v^8- w^4=0$ \\\hline
%6244 & $49729 u^8 + 3136 v^8 -  w^4=0$ &
% 6355 & $168100 u^8 + 961 v^8 -  w^4=0$ \\\hline
 6440 & $25921 u^8 + 25 v^8 -  w^4=0$ \\\hline
 6625 & $4 u^8 + v^8 - 6625  w^4=0$ &
 6625 & $15625 u^8 + v^8 - 106  w^4=0$ \\\hline
% 6629 & $3587236 u^8+49 v^8 - w^4=0 $ \\\hline
% 6644 & $22801 u^8 + 7744 v^8 -  w^4=0$ &
% 6821 & $128881 u^8+ 1444 v^8- w^4=0 $ \\\hline
 6984 & $81 u^8 + v^8 - 97  w^4=0$ &
% 7101 & $69169 u^8 + 2916 v^8 -  w^4=0$ \\\hline
% 7769 & $208849 u^8 + 1156 v^8 -  w^4=0$  &
% 7861 & $5044516 u^8+49 v^8- w^4=0$ \\\hline
 8001 & $64516 u^8 + 3969 v^8 -  w^4=0$  \\\hline
% 8001 & $5225796 u^8 + 49 v^8 -  w^4=0$ \\\hline
 8005 & $100 u^8 + v^8 - 1601  w^4=0$  &
8075 & $361 u^8 + 4 v^8 - 425  w^4=0$ \\\hline
% 8133 & $7349521 u^8 + 36 v^8 -  w^4=0$  &
 %8205 & $7480225 u^8 + 36 v^8 -  w^4=0$ \\\hline
 8235 & $2916u^8+25v^8-61 w^4=0$ &
% 8520 & $5041 u^8 + 225 v^8 -  w^4=0$   \\\hline
%8636 & $18496 u^8 + 16129 v^8 -  w^4=0$ &
 8705 & $4u^8+v^8-8705 w^4=0$ \\\hline
 8931 & $36u^8+v^8-2977 w^4=0$  &
% 8965 & $106276 u^8 + 3025 v^8 -  w^4=0$ \\\hline
 9015 & $36u^8+25v^8-601 w^4=0$ \\\hline
% 9357 & $9728161 u^8+36 v^8 - w^4=0$  \\\hline
 
%  9860 & $21025 u^8 + 18496 v^8 -  w^4=0$ &
% 9991 & $37636u^8+10609v^8- w^4=0$ \\\hline
\end{tabular}
\caption{Equations not solvable by the methods in Sections \ref{sec:easy}, \ref{sec:localsieve} or \ref{sec:pythag}. \label{tb:factorisation}}
\end{center}
\end{table}

\section{Factorization in \texorpdfstring{$\Z[i]$}{}}\label{sec:factor}  

In this section, we use the method from \cite{THO2026125764} to show that none of the equations listed in Table\footnote{In this section we choose to solve all equations in Table \ref{tb:outputted} to show details for solving all equations not excluded by the diffpowers program.} \ref{tb:outputted} has integer solutions $(u,v,w)$ satisfying \eqref{eq:cond}. This will complete the proof of Theorem \ref{th:main}.
The main idea of this method is as follows. First, we write equation \eqref{eq:abcuvw} as
\[(au^4+bv^4i)(au^4-bv^4i)=cw^4.\]
Then there exist integers $s,t$, and a Gaussian integer $\alpha|c$ such that  \begin{equation}\label{abst}au^4+bv^4i=\alpha (s+ti)^4.\end{equation}
Equating the real and imaginary parts on both sides of \eqref{abst}, we obtain a system
\begin{equation}\label{abFG}
    \begin{cases}
        u^4=F_{\alpha}(s,t),\\
        v^4=G_{\alpha}(s,t),
    \end{cases}
\end{equation}
where $F_{\alpha}(s,t)$ and $G_{\alpha}(s,t)$ are degree four homogeneous polynomials in $\Z[s,t]$. System \eqref{abFG} defines a scheme in the  projective space $\mathbb{P}^3(\Q)$. Using Magma \cite{MR1484478}, one can check that this scheme is locally insoluble at a suitable prime.

We now implement this method to solve all the remaining equations in Table \ref{tb:factorisation}. See Appendix \ref{AppendixB}.

\section{Conclusion}
In summary, the methods of this paper allow us to solve all equations \eqref{eq:x4my4mnz4} with $1\leq n\leq 10000$ in integers with $xyz \neq 0$. Table \ref{tb:rationalsols} lists all integers $1\leq n\leq 10000$ that can be expressed as a difference of two rational fourth powers, and such a representation. The lists of representable integers are available on the On-line Encyclopedia of Integer Sequences (OEIS) \cite{oeis}, specifically, the positive integers expressible as a difference of two nonzero rational fourth powers is \href{https://oeis.org/search?q=A393616&language=english&go=Search}{A393616} and the positive integers expressible as a difference of two rational fourth powers is \href{https://oeis.org/search?q=A393617&language=english&go=Search}{A393617}. For all the equations which are outputted by the Magma program described in Section \ref{magma}, the necessary information and computations to prove the non-existence of integer solutions to these equations are given in the appendices of the extended version of this paper at \url{http://arxiv.org/abs/2604.15832}.

\section{Acknowledgments} 
Nguyen Xuan Tho is supported by the Vietnam National Foundation for Science and Technology Development (NAFOSTED), under grant number 101.04-2023.21.

\begin{table}
\begin{center}
\begin{tabular}{|c|c||c|c||c|c|}
	\hline
	$n$  & Representation & $n$  & Representation & $n$  & Representation \\\hline \hline
5 & $(3/2)^4 -(1/2)^4$ & 15 & $2^4-  1^4$ & 34 & $(5/2)^4-(3/2)^4$ \\\hline
%16 (2 : 0 : 1) \\\hline
39 &$(5/2)^4 -(1/2)^4$ & 65& $3^4-  2^4$ & 80 &$3^4  -1^4$ \\\hline
%81 (-3 : 0 : 1)
84 &$(31/10)^4-(17/10)^4$ & 111 & $(7/2)^4-(5/2)^4$ & 145& $(7/2)^4-(3/2)^4$ \\\hline
150 &$(7/2)^4-(1/2)^4$ & 175& $4^4-3^4$ & 239 & $(120/13)^4-(119/13)^4$ \\\hline
240 & $4^4-2^4$ &
 255 &$4^4-1^4$ &
%256 (4 : 0 : 1)
260 &$(9/2)^4-(7/2)^4$  \\\hline
369 &$5^4 - 4^4$ & 
 371& $(9/2)^4-(5/2)^4$ & 
 405 &$(9/2)^4-(3/2)^4$ \\\hline
410 &$(9/2)^4-(1/2)^4$ &
505& $ (11/2)^4-(9/2)^4$ &
527 &$(24/5)^4-(7/5)^4$ \\\hline
544& $5^4-  3^4$ & 609 & $5^4-2^4$ &
624& $5^4  -1^4$ \\\hline
%625 (-5 : 0 : 1)
671& $6^4 - 5^4$ &
765 &$(11/2)^4  -(7/2)^4$ &
870& $(13/2)^4  -(11/2)^4$ \\\hline
876 &$(11/2)^4-( 5/2)^4$ &
910 &$(11/2)^4 -( 3/2)^4$ & 915 &$(11/2)^4 -(1/2)^4$ \\\hline
1025& $(33/4)^4-(31/4)^4$ &
1040& $6^4 -4^4$ &
1105 &$ 7^4 -6^4$ \\\hline
1185 &$(41/5)^4  -(38/5)^4$ &
1215 &$6^4  -3^4$ &
1280 &$6^4- 2^4$ \\\hline
1295 & $6^4- 1^4$ & %\\\hline
%1296 (6 : 0 : 1)
1344 &$ (31/5)^4-(17/5)^4$  &
1375& $(13/2)^4-( 9/2)^4$ \\\hline % \\\hline
1379 &$(15/2)^4-(13/2)^4$ &
1635& $(13/2)^4 -(7/2)^4$ &
1695 &$8^4- 7^4$  \\\hline
1746& $(13/2)^4-( 5/2)^4$ &
1776 &$7^4-5^4$ &
1780 &$(13/2)^4-(3/2)^4$  \\\hline
1785& $(13/2)^4  -(1/2)^4$ &  % \\\hline
2056 & $(17/2)^4-(15/2)^4$  &
2145& $7^4-  4^4$  \\\hline
2249& $(15/2)^4  -(11/2)^4$ & % \\\hline
2320 &$7^4-3^4$ &
2385 & $7^4  -2^4 $   \\\hline
2400 &$7^4 -1^4$ &
%2401 (-7 : 0 : 1)
2465 &$9^4-  8^4$  & % \\\hline
2754& $(15/2)^4-(9/2)^4$ \\\hline
2800& $8^4 -6^4$& 
2925 & $(19/2)^4-(17/2)^4$ & % \\\hline
3014& $(15/2)^4 -(7/2)^4$  \\\hline
3099 & $(35/4)^4-( 29/4)^4$ & %\\\hline
3125& $(15/2)^4-(5/2)^4$ &
3159& $(15/2)^4-(3/2)^4$ \\\hline
3164 & $(15/2)^4  -(1/2 )^4$ & % \\\hline
3281& $(41/3)^4-(40/3)^4$ &
3435 &$(17/2)^4-(13/2 )^4$  \\\hline
3439 & $ 10^4- 9^4$ & %\\\hline
3471& $8^4  -5^4$ &
3502 & $(107/10)^4-(99/10 )^4$  \\\hline
3824 &$(240/13)^4-(238/13)^4$ &
3840 &$8^4-4^4$ &
3895 & $(79/10)^4 -( 3/10 )^4$  \\\hline
4010& $(21/2)^4-(19/2)^4$  &
4015& $8^4  -3^4$ & % \\\hline
4080 & $8^4-2^4$ \\\hline
4095 &$ 8^4- 1^4$ & % \\\hline
%4096 (8 : 0 : 1) \\\hline
4160 &$9^4- 7^4$ &
4305 &$(17/2)^4 -(11/2)^4$ \\\hline
4641 &$11^4 - 10^4$ & % \\\hline
4810& $(17/2)^4 -(9/2 )^4$ & % \\\hline
4823& $(183/10)^4-( 181/10 )^4$ \\\hline
4981 &$(19/2)^4-( 15/2 )^4$  & %\\\hline
5070& $(17/2)^4 - (7/2)^4$ & % \\\hline
5181& $(17/2)^4 -(5/2)^4$ \\\hline
5215& $(17/2)^4-( 3/2)^4$ & %\\\hline
5220 & $(17/2)^4- ( 1/2 )^4$ & %\\\hline
5245& $(37/4)^4  -(27/4 )^4$ \\\hline
5265 & $9^4  -6^4$& 
5335& $(23/2)^4  -(21/2)^4$ & %\\\hline
5904 &$10^4 -8^4$ \\\hline
5936 &$9^4-5^4$&
6095 & $12^4  -11^4$ & %\\\hline
6305 & $9^4 -4^4$ \\\hline
6360 & $(19/2)^4- (13/2 )^4$ & % \\\hline
6480& $9^4-3^4$ &
6545& $9^4- 2^4$ \\\hline
6560& $9^4-1^4$ &
%6561 (9 : 0 : 1)
6565 &$(83/6)^4-( 79/6 )^4$ &
6649& $(271/30)^4-( 53/30 )^4$   \\\hline
6804 &$(93/10)^4-(51/10)^4$&
6924& $(25/2)^4-( 23/2 )^4$ & % \\\hline
6935& $(21/2)^4- (17/2 )^4$ \\\hline
7140& $(239/26)^4-(1/26 )^4$ & % \\\hline
7230 & $(19/2)^4 -(11/2 )^4$ & % \\\hline
7511& $(39/4)^4-(25/4)^4$ \\\hline
7585& $(113/12)^4  -(49/12 )^4$ &% \\\hline
7599 & $10^4- 7^4$&%\\\hline
7735 & $(19/2)^4-(9/2)^4$ \\\hline
7825& $13^4  -12^4$ &%\\\hline
7995 &$(19/2)^4- (7/2)^4$ &%\\\hline
8080& $11^4- 9^4$ \\\hline
8106 & $(19/2)^4 -(5/2 )^4$ &
8140& $(19/2)^4  -(3/2 )^4$ &%\\\hline
8145& $(19/2)^4  -(1/2 )^4$ \\\hline
8194& $(65/4)^4  -(63/4 )^4$ &
8432& $(48/5)^4-(14/5)^4$ &
8704& $10^4 -6^4$ \\\hline
8801& $(27/2)^4  -(25/2)^4$&% \\\hline
8991 &$(21/2)^4 -( 15/2)^4$ &
9345 &$(23/2)^4 -(19/2 )^4$ \\\hline
9375& $10^4 - 5^4$ &
9744 &$10^4 - 4^4$&
9855& $14^4- 13^4$ \\\hline
9919& $10^4  -3^4$ &%\\\hline
9945 & $(41/4)^4 -( 23/4 )^4$ &% \\\hline
9984 &$10^4  -2^4$ \\\hline
9999& $10^4 - 1 ^4$ &&&& \\\hline
%10000 (-10 : 0 : 1)

\end{tabular}
\caption{Positive integers $n\leq 10000$ representable as a difference of two non-zero rational fourth powers, and their representations. \label{tb:rationalsols}}
\end{center}
\end{table}

\appendix
\section{Details for primitive Pythagorean triple method}
The details in Table \ref{tb:pythag} provide for each $n$, the equations of the form \eqref{eq:abcuvw} displayed in Table \ref{tb:outputted} which can be solved with the method in Section \ref{sec:pythag}. For each equation of the form \eqref{eq:abcuvw}, we list all the coprime pairs $(d,e)$ such that $2de=a$ (and $2de=b$ in the case \eqref{sys:pyabodd}) for which we need to solve \eqref{eq:depythag} (or respectively $d^2U^8-e^2V^8=au^4$), then in the final column we provide an integer $n$ for which we obtain a contradiction with the conditions of \eqref{cond:depythag} by solving the equation modulo $n$.

\begin{center}
\begin{longtable}{|c|c|c|c|}
\caption{Necessary details to solve equations in Table \ref{tb:outputted} which are solvable by Pythagorean triple method. \label{tb:pythag}} \\

\hline \multicolumn{1}{|c|}{\textbf{$n$}} & \multicolumn{1}{c|}{\textbf{Equation \eqref{eq:abcuvw} }} & \multicolumn{1}{c|}{\textbf{$(d,e)$ in \eqref{eq:depythag}}} &  \multicolumn{1}{|c|}{\textbf{No solutions modulo}} \\ \hline 
\endfirsthead

\multicolumn{4}{c}%
{{\bfseries \tablename\ \thetable{} -- continued from previous page}} \\
\hline  \multicolumn{1}{|c|}{\textbf{$n$}} & \multicolumn{1}{c|}{\textbf{Equation \eqref{eq:abcuvw} }} & \multicolumn{1}{c|}{\textbf{$(d,e)$ in \eqref{eq:depythag}}} &  \multicolumn{1}{|c|}{\textbf{No solutions modulo}} \\ \hline 
\endhead

\hline \multicolumn{4}{|r|}{{Continued on next page}} \\ \hline
\endfoot

\hline 
\endlastfoot
%$n$ & Equation \eqref{eq:abcuvw} & $(d,e)$ in \eqref{eq:depythag} & No solutions modulo \\\hline
885 & $36u^8+87025v^8 =w^4$ & $(1,3)$ & 5 \\
& & $(3,1)$ & 3 \\\hline
1269 & $2916u^8 + 2209v^8= w^4$ & $(1,27)$ & 3 \\
&& $(27,1)$ & 5 \\\hline
1561 & $196 u^8 +  49729v^8= w^4$ & $(1,7)$ & 5 \\
&& $(7,1)$ & 13 \\\hline
2130 & $3600 u^8 + 5041 v^8 = w^4$ & $(1,30)$ & 3 \\ 
&& $(2,15)$ & 3 \\
&& $(3,10)$ & 4 \\
&& $(5,6)$ & 3 \\
&& $(6,5)$ & 17 \\
&& $(10,3)$ & 3 \\
&& $(15,2)$ & 4 \\ 
&& $(30,1)$ & 5 \\\cline{2-4}
& $80656 u^8 + 225 v^8=w^4$ & $(1,142)$ & 4 \\
&& $(2,71)$ & 5 \\
&& $(71,2)$ & 4 \\ 
&& $(142,1)$ & 5 \\\hline
2329 & $1156 u^8 + 18769 v^8=w^4$ & $(1,17)$ & 16 \\
&& $(17,1)$ & 5 \\\hline
3503 & $51076 u^8 + 961 v^8=w^4$ & $(1,113)$ & 113 \\
&& $(113,1)$ & 5 \\\hline
3540 & $576  u^8 +87025 v^8=w^4$ & $(1,12)$ & 4 \\
&& $(3,4)$ & 3 \\
&& $(4,3)$ & 5 \\
&& $(12,1)$ & 3 \\\hline
3690 & $32400 u^8 + 1681 v^8= w^4$ & $(1,90)$ & 3 \\
&& $(2,45)$ & 3 \\
&& $(5,18)$ & 3 \\
&& $(9,10)$ & 16 \\
&& $(10,9)$ & 3 \\
&& $(18,5)$ & 4 \\
&& $(45,2)$ & 41 \\
&& $(90,1)$ & 4 \\\cline{2-4}
& $400  u^8 + 136161v^8=  w^4$ & $(1,10)$ & 5 \\
&& $(2,5)$ & 4 \\
&& $(5,2)$ & 5 \\
&& $(10,1)$ & 4 \\\hline
3906 & $ 784 u^8 +77841 v^8=w^4$ & $(1,14)$ & 4 \\
&& $(2,7)$ & 7 \\
&& $(7,2)$ & 4 \\
&& $(14,1)$ & 16 \\\cline{2-4}
& $1245456 u^8 + 49 v^8 =w^4$ & $(1,558)$ & 4 \\
&& $(2,279)$ & 5 \\
&& $(9,62)$ & 3 \\
&& $(18,31)$ & 3 \\
&& $(31,18)$ & 4 \\
&& $(62,9)$ & 5 \\
&& $(279,2)$ & 3 \\
&& $(558,1)$ & 3 \\\hline
4317 & $36 u^8+ 2070721v^8= w^4$ & $(1,3)$ & 3 \\
&& $(3,1)$ & 73 \\\hline
4559 & $37636 u^8 + 2209 v^8= w^4$ & $(1,97)$ & 13 \\
&& $(97,1)$ & 5 \\\hline
4669 & $1779556u^8+49v^8= w^4$ & $(1,667)$ & 23 \\
&& $(23,29)$ & 5 \\
&& $(29,23)$ & 16 \\
&& $(667,1)$ & 5 \\\hline
4991 & $188356 u^8 + 529 v^8=  w^4$ & $(1,217)$ & 5 \\
&& $(7,31)$ & 7 \\
&& $(31,7)$ & 5 \\
&& $(217,1)$ & 7 \\\hline
$5076$ & $46656 u^8 + 2209 v^8=  w^4$ & $(1,108)$ & 3 \\
&& $(4,27)$ & 3 \\
&& $(27,4)$ & 4 \\
&& $(108,1)$ & 5 \\\hline
$5775$ & $1764 u^8 + 121 v^8=25  w^4$ & $(1,21)$ & 3 \\
&& $(3,7)$ & 8 \\
&& $(7,3)$ & 3 \\
&& $(21,1)$ & 7 \\\hline
5911 & $264196 u^8 + 529 v^8 =w^4$ & $(1,257)$ & 5 \\
&& $(257,1)$ & 16 \\\hline
5983 & $148996 u^8 + 961 v^8= w^4$ & $(1,193)$ & 193 \\
&& $(193,1)$ & 5 \\\hline
6117 & $36 u^8+4157521 v^8=w^4$ & $(1,3)$ & 3 \\
&& $(3,1)$ & 16 \\\hline
6244 & $3136  u^8 +49729 v^8=  w^4$ & $(1,28)$ & 4 \\
&& $(4,7)$ & 5 \\
&& $(7,4)$ & 4 \\
&& $(28,1)$ & 13 \\\hline
6355 & $168100 u^8 + 961 v^8=w^4$ & $(1,205)$ & 16 \\
&& $(5,41)$ & 5 \\
&& $(41,5)$ & 16 \\
&& $(205,1)$ & 5 \\\hline
6629 & $3587236 u^8+49 v^8=w^4$ & $(1,947)$ & 13 \\
&& $(947,1)$ & 5 \\\hline
6644 & $ 7744u^8 + 22801 v^8 =w^4$ & $(1,44)$ & 4 \\
&& $(4,11)$ & 11 \\
&& $(11,4)$ & 4 \\
&& $(44,1)$ & 16 \\\hline
6821 & $1444 u^8+ 128881 v^8= w^4$ & $(1,19)$ & 13 \\
&& $(19,1)$ & 16 \\\hline
7101 & $ 2916u^8 + 69169 v^8= w^4$ & $(1,27)$ & 3 \\
&& $(27,1)$ & 16 \\\hline
7769 & $1156 u^8 + 208849 v^8= w^4$ & $(1,17)$ & 16 \\
&& $(17,1)$ & 5 \\\hline
7861 & $5044516 u^8+49 v^8= w^4$ & $(1,1123)$ & 13 \\
&& $(1123,1)$ & 5 \\\hline
8001 & $5225796 u^8 + 49 v^8= w^4$ & $(1,1143)$ & 16 \\
&& $(9,127)$ & 3 \\
&& $(127,9)$ & 5 \\
&& $(1143,1)$ & 3 \\\hline
8133 & $36u^8 +7349521   v^8=w^4$ & $(1,3)$ & 3 \\
&& $(3,1)$ & 5 \\\hline
8205 & $36u^8+7480225v^8=w^4$ & $(1,3)$ & 3 \\
&& $(3,1)$ & 5 \\\hline
8520 & $5041 u^8 + 225 v^8=  w^4$ & $(1,568)$ & 4 \\
&& $(8,71)$ & 5 \\
&& $(71,8)$ & 4 \\
&& $(568,1)$ & 5 \\
&& $(1,120)$ & 3 \\
&& $(3,40)$ & 4 \\
&& $(5,24)$ & 3 \\
&& $(8,15)$ & 3 \\
&& $(15,8)$ & 4 \\
&& $(24,5)$ & 17 \\
&& $(40,3)$ & 3 \\
&& $(120,1)$ & 5  \\\hline
8636& $18496 u^8 + 16129 v^8 = w^4$ & $(1,68)$ & 4 \\
&& $(4,17)$ & 17 \\
&& $(17,4)$ & 4 \\
&& $(68,1)$ & 5 \\\hline
8965 & $106276 u^8 + 3025 v^8 = w^4$ & $(1,163)$ & 5 \\
&& $(163,1)$ & 5 \\\hline
9357 & $36 u^8+ 9728161 v^8 = w^4$ & $(1,3)$ & 3 \\
&& $(3,1)$ & 29 \\\hline
9860 & $18496 u^8 +  21025 v^8 = w^4$ & $(1,68)$ & 5 \\
&& $(4,17)$ & 4 \\
&& $(17,4)$ & 5 \\
&& $(68,1)$ & 4 \\\hline
9991 & $37636u^8+10609v^8 = w^4$ & $(1,97)$ & 5 \\
&& $(97,1)$ & 16 \\
%\end{tabular}
%\caption{Details to solve equations ... by Pythagorean triple method. \label{tb:pythag}}
%\end{center}
%\end{table}
\end{longtable}
\end{center}

\section{Details for factorization method}\label{AppendixB}
\subsection{The case of $n=219$.} %\[n=219\]  
This was solved by Tho \cite{THO2026125764}. 

\subsection{The case of $n=463$.}%\[n=463\] 
 According to Table \ref{tb:factorisation}, in the case $n=463$ it remains to
show that equation
\begin{equation}\label{n=463}
214369 u^8 + v^8=2 w^4.
\end{equation}
has no integer solutions satisfying \eqref{eq:cond}, which in this case reduces to
\[\gcd(463u,v)=\gcd(u,2w)=\gcd(v,2w)=1.\]
Hence, $2\nmid uv$.  
Write \eqref{n=463} as 
\begin{equation}\label{463.1} (463u^4+v^4i)(463u^4-v^4i)=(1+i)(1-i)w^4.
\end{equation}

Let $d\in \Z[i]$ such that $d|\gcd(463u^4+v^4i,463u^4-v^4i)$. Then $d|926 u^4$,$d|2v^4$, $d|2w^4$. Hence, $d|2$. Note that $u$ and $v$ are odd. 
This implies that there exist integers $s,t$ such that 
\[463u^4+v^4i=i^{\epsilon}(1+i)(s+ti)^4,\]
with $\epsilon \in \{0,1,2,3\}$.

{\bf Case 1:} $463u^4+v^4i=(1+i)(s+ti)^4$. Equating the real and imaginary parts on both sides of this equation gives
\begin{equation}\label{463-1}
\begin{cases}
s^4 - 4s^3t - 6s^2t^2 +4st^3 + t^4-463u^4=0,\\
s^4 + 4s^3t - 6s^2t^2 - 4st^3 + t^4-v^4=0.
\end{cases}
\end{equation}

%_<x>:=PolynomialRing(Rationals());
%k<i>:=NumberField(x^2+1);
%K<s,t>:=PolynomialRing(k,2);
%e:=0;
%F:=i^e*(1+i)*(s+t*i)^4;
%F;
%L<s,t>:=PolynomialRing(Rationals(),2);
%_<i>:=PolynomialRing(L);
%F:=(i + 1)*s^4 + (4*i - 4)*s^3*t + (-6*i - 6)*s^2*t^2 + (-4*i + 4)*s*t^3 + (i +
%    1)*t^4;
%F;
%A:= s^4 - 4*s^3*t - 6*s^2*t^2 + 4*s*t^3 + t^4;
%B:=s^4 + 4*s^3*t - 6*s^2*t^2 - 4*s*t^3 + t^4;
%F-A-i*B;
\begin{lstlisting}
P<s,t,u,v>:=ProjectiveSpace(Rationals(),3);
A:=s^4 - 4*s^3*t - 6*s^2*t^2 +4*s*t^3 + t^4-463*u^4;
B:=s^4 + 4*s^3*t - 6*s^2*t^2 - 4*s*t^3 + t^4-v^4;
S:=Scheme(P,[A,B]);
IsLocallySolvable(S,2);
\end{lstlisting}

The scheme defined by system \eqref{463-1} is locally insoluble at $2$.

{\bf Case 2:} $463u^4+v^4i=i(1+i)(s+ti)^4$. Equating the real and imaginary parts on both sides of this equation gives
\begin{equation}\label{463-2}
\begin{cases}
- s^4 - 4s^3t + 6s^2t^2 + 4st^3 - t^4-463u^4=0,\\
s^4 - 4s^3t - 6s^2t^2 + 4st^3 + t^4-v^4=0.
\end{cases}
\end{equation}

%_<x>:=PolynomialRing(Rationals());
%k<i>:=NumberField(x^2+1);
%K<s,t>:=PolynomialRing(k,2);
%e:=1;
%F:=i^e*(1+i)*(s+t*i)^4;
%F;
%L<s,t>:=PolynomialRing(Rationals(),2);
%_<i>:=PolynomialRing(L);
%F:=(i - 1)*s^4 + (-4*i - 4)*s^3*t + (-6*i + 6)*s^2*t^2 + (4*i + 4)*s*t^3 + (i -
%   1)*t^4;
%F;
%A:= - s^4 - 4*s^3*t + 6*s^2*t^2 +
%   4*s*t^3 - t^4;
%B:=s^4 - 4*s^3*t - 6*s^2*t^2 + 4*s*t^3 + t^4;
%F-A-i*B;
\begin{lstlisting}
P<s,t,u,v>:=ProjectiveSpace(Rationals(),3);
A:=- s^4 - 4*s^3*t + 6*s^2*t^2 + 4*s*t^3 - t^4-463*u^4;
B:=s^4 - 4*s^3*t - 6*s^2*t^2 + 4*s*t^3 + t^4-v^4;
S:=Scheme(P,[A,B]);
IsLocallySolvable(S,3);
\end{lstlisting}

The scheme defined by system \eqref{463-2} is locally insoluble at $3$.

{\bf Case 3:} $463u^4+v^4i=i^2(1+i)(s+ti)^4$. Equating the real and imaginary parts on both sides of this equation 
\begin{equation}\label{463-3}
\begin{cases}
- s^4 + 4s^3t + 6s^2t^2 - 4st^3 - t^4-463u^4=0,\\
-s^4 - 4s^3t + 6s^2t^2 + 4st^3 - t^4-v^4=0.
\end{cases}
\end{equation}
%_<x>:=PolynomialRing(Rationals());
%k<i>:=NumberField(x^2+1);
%K<s,t>:=PolynomialRing(k,2);
%e:=2;
%F:=i^e*(1+i)*(s+t*i)^4;
%F;
%L<s,t>:=PolynomialRing(Rationals(),2);
%_<i>:=PolynomialRing(L);
%F:=(-i - 1)*s^4 + (-4*i + 4)*s^3*t + (6*i + 6)*s^2*t^2 + (4*i - 4)*s*t^3 + (-i -
%    1)*t^4;
%F;
%A:=- s^4 + 4*s^3*t + 6*s^2*t^2 - 4*s*t^3 - t^4;
%B:=-s^4 - 4*s^3*t + 6*s^2*t^2 + 4*s*t^3 - t^4;
%F-A-i*B;

\begin{lstlisting}
P<s,t,u,v>:=ProjectiveSpace(Rationals(),3);
A:=- s^4 + 4*s^3*t + 6*s^2*t^2 - 4*s*t^3 - t^4-463*u^4;
B:=-s^4 - 4*s^3*t + 6*s^2*t^2 + 4*s*t^3 - t^4-v^4;
S:=Scheme(P,[A,B]);
IsLocallySolvable(S,2);
\end{lstlisting}

The scheme defined by system \eqref{463-3} is locally insoluble at $2$.

{\bf Case 4:} $463u^4+v^4i=i^3(1+i)(s+ti)^4$. Equating the real and imaginary parts on both sides of this equation gives
\begin{equation}\label{463-4}
\begin{cases}
s^4 + 4s^3t - 6s^2t^2 -  4st^3 + t^4-463u^4=0,\\
-s^4 + 4s^3t + 6s^2t^2 - 4st^3 - t^4-v^4=0.
\end{cases}
\end{equation}

%_<x>:=PolynomialRing(Rationals());
%k<i>:=NumberField(x^2+1);
%K<s,t>:=PolynomialRing(k,2);
%e:=3;
%F:=i^e*(1+i)*(s+t*i)^4;
%F;
%L<s,t>:=PolynomialRing(Rationals(),2);
%_<i>:=PolynomialRing(L);
%F:=(-i + 1)*s^4 + (4*i + 4)*s^3*t + (6*i - 6)*s^2*t^2 + (-4*i - 4)*s*t^3 + (-i +
%   1)*t^4;
%F;
%A:=s^4 + 4*s^3*t - 6*s^2*t^2 -  4*s*t^3 + t^4;
%B:=-s^4 + 4*s^3*t + 6*s^2*t^2 - 4*s*t^3 - t^4;
%F-A-i*B;
%
\begin{lstlisting}
P<s,t,u,v>:=ProjectiveSpace(Rationals(),3);
A:=s^4 + 4*s^3*t - 6*s^2*t^2 -  4*s*t^3 + t^4-463*u^4;
B:=-s^4 + 4*s^3*t + 6*s^2*t^2 - 4*s*t^3 - t^4-v^4;
S:=Scheme(P,[A,B]);
IsLocallySolvable(S,2);
\end{lstlisting}

The scheme defined by system \eqref{463-4} is locally insoluble at $2$.

\subsection{The case of $n=514$.}%\[n=514.\] 
According to Table \ref{tb:factorisation}, in the case $n=514$ it remains to
show that equation
\begin{equation}\label{n=514}
16u^8+v^8=257w^4
\end{equation}
has no integer solutions satisfying \eqref{eq:cond}, which in this case reduces to \[\gcd(2u,v)=\gcd(2u,257w)=\gcd(v,257w)=1.\]
Hence, $257\nmid uv$.  
Write \eqref{n=514} as 
\begin{equation}\label{514.1} (4u^4+v^4i)(4u^4-v^4i)=(1+16i)(1-16i)w^4.
\end{equation}

Let $d\in \Z[i]$ such that $d|\gcd(4u^4+v^4i,4u^4-v^4i)$. Then $d| 8u^4$,$d|2v^4$, $d|257w^4$. Hence, $d|1$. Thus $4u^4+v^4i$ and $4u^4-v^4i$ are coprime in $\Z[i]$. Note that $u$ and $v$ odd. 
This implies that there exist integers $s,t$ such that 
\[4u^4+v^4i=i^{\epsilon}(1\pm 16i)(s+ti)^4,\]
with $\epsilon\in \{0,1,2,3\}$.

{\bf Case 1:} $4u^4+v^4i=(1+16i)(s+ti)^4$. Equating the real and imaginary parts on both sides of this equation gives
\begin{equation}\label{514-1}
\begin{cases}
s^4 - 64s^3t -6s^2t^2 + 64st^3 + t^4-4u^4=0,\\
16s^4 + 4s^3t - 96s^2t^2 - 4st^3 + 16t^4-v^4=0.
\end{cases}
\end{equation}

%_<x>:=PolynomialRing(Rationals());
%k<i>:=NumberField(x^2+1);
%K<s,t>:=PolynomialRing(k,2);
%e:=0;
%F:=i^e*(1+16*i)*(s+t*i)^4;
%F;
%L<s,t>:=PolynomialRing(Rationals(),2);
%_<i>:=PolynomialRing(L);
%F:=(16*i + 1)*s^4 + (4*i - 64)*s^3*t + (-96*i - 6)*s^2*t^2 + (-4*i + 64)*s*t^3 +
%   (16*i + 1)*t^4;
%F;
%A:=s^4 - 64*s^3*t -6*s^2*t^2 + 64*s*t^3 + t^4;
%B:=(16*s^4 + 4*s^3*t - 96*s^2*t^2 - 4*s*t^3 + 16*t^4);
\begin{lstlisting}
P<s,t,u,v>:=ProjectiveSpace(Rationals(),3);
A:=s^4 - 64*s^3*t -6*s^2*t^2 + 64*s*t^3 + t^4-4*u^4;
B:=16*s^4 + 4*s^3*t - 96*s^2*t^2 - 4*s*t^3 + 16*t^4-v^4;
S:=Scheme(P,[A,B]);
IsLocallySolvable(S,2);
\end{lstlisting}
The scheme defined by system \eqref{514-1} is locally insoluble at $2$.

{\bf Case 2:} $4u^4+v^4i=i(1+16i)(s+ti)^4$. Equating the real and imaginary parts on both sides of this equation gives
\begin{equation}\label{514-2}
\begin{cases}
-16s^4 - 4s^3t + 96s^2t^2 + 4st^3 - 16t^4 - 4u^4=0,\\
s^4 - 64s^3t - 6s^2t^2 + 64st^3 + t^4 - v^4=0.
\end{cases}
\end{equation}
%_<x>:=PolynomialRing(Rationals());
%k<i>:=NumberField(x^2+1);
%K<s,t>:=PolynomialRing(k,2);
%e:=1;
%F:=i^e*(1+16*i)*(s+t*i)^4;
%F;
%L<s,t>:=PolynomialRing(Rationals(),2);
%_<i>:=PolynomialRing(L);
%F:=(i - 16)*s^4 + (-64*i - 4)*s^3*t + (-6*i + 96)*s^2*t^2 + (64*i + 4)*s*t^3 + (i -
%   16)*t^4;
%F;
%A:=- 16*s^4 - 4*s^3*t + 96*s^2*t^2
%   + 4*s*t^3 - 16*t^4;
%B:=(s^4 - 64*s^3*t - 6*s^2*t^2 + 64*s*t^3 + t^4);
%F-A-i*B;

\begin{lstlisting}
P<s,t,u,v>:=ProjectiveSpace(Rationals(),3);
A:=- 16*s^4 - 4*s^3*t + 96*s^2*t^2 + 4*s*t^3 - 16*t^4-4*u^4;
B:=s^4 - 64*s^3*t - 6*s^2*t^2 + 64*s*t^3 + t^4-v^4;
S:=Scheme(P,[A,B]);
IsLocallySolvable(S,3);
\end{lstlisting}

The scheme defined by system \eqref{514-2} is locally insoluble at $3$.

{\bf Case 3:} $4u^4+v^4i=i^2(1+16i)(s+ti)^4$. Equating the real and imaginary parts on both sides of this equation gives
\begin{equation}\label{514-3}
\begin{cases}
- s^4 + 64s^3t + 6s^2t^2 - 64st^3 - t^4-4u^4=0,\\
-16s^4 - 4s^3t + 96s^2t^2 + 4st^3 - 16t^4-v^4=0;
\end{cases}
\end{equation}

%_<x>:=PolynomialRing(Rationals());
%k<i>:=NumberField(x^2+1);
%K<s,t>:=PolynomialRing(k,2);
%e:=2;
%F:=i^e*(1+16*i)*(s+t*i)^4;
%F;
%L<s,t>:=PolynomialRing(Rationals(),2);
%_<i>:=PolynomialRing(L);
%F:=(-16*i - 1)*s^4 + (-4*i + 64)*s^3*t + (96*i + 6)*s^2*t^2 + (4*i - 64)*s*t^3 +
%   (-16*i - 1)*t^4;
%F;
%A:=- s^4 + 64*s^3*t + 6*s^2*t^2 - 64*s*t^3 - t^4;
%B:=(-16*s^4 - 4*s^3*t + 96*s^2*t^2 + 4*s*t^3 - 16*t^4);
%F-A-i*B;
\begin{lstlisting}
P<s,t,u,v>:=ProjectiveSpace(Rationals(),3);
A:=- s^4 + 64*s^3*t + 6*s^2*t^2 - 64*s*t^3 - t^4-4*u^4;
B:=-16*s^4 - 4*s^3*t + 96*s^2*t^2 + 4*s*t^3 - 16*t^4-v^4;
S:=Scheme(P,[A,B]);
IsLocallySolvable(S,5);
\end{lstlisting}

The scheme defined by system \eqref{514-3} is locally insoluble at $5$.

{\bf Case 4:} $4u^4+v^4i=i^3(1+16i)(s+ti)^4$. Equating the real and imaginary parts on both sides of this equation gives
\begin{equation}\label{514-4}
\begin{cases}
16s^4 + 4s^3t - 96s^2t^2- 4st^3 + 16t^4-4u^4=0,\\
-s^4 + 64s^3t + 6s^2t^2 - 64st^3 - t^4-v^4=0.
\end{cases}
\end{equation}

%_<x>:=PolynomialRing(Rationals());
%k<i>:=NumberField(x^2+1);
%K<s,t>:=PolynomialRing(k,2);
%e:=3;
%F:=i^e*(1+16*i)*(s+t*i)^4;
%F;
%L<s,t>:=PolynomialRing(Rationals(),2);
%_<i>:=PolynomialRing(L);
%F:=(-i + 16)*s^4 + (64*i + 4)*s^3*t + (6*i - 96)*s^2*t^2 + (-64*i - 4)*s*t^3 + (-i
%    + 16)*t^4;
%F;
%A:=16*s^4 + 4*s^3*t - 96*s^2*t^2- 4*s*t^3 + 16*t^4;
%B:=-s^4 + 64*s^3*t + 6*s^2*t^2 - 64*s*t^3 - t^4;
%F-A-i*B;

\begin{lstlisting}
P<s,t,u,v>:=ProjectiveSpace(Rationals(),3);
A:=16*s^4 + 4*s^3*t - 96*s^2*t^2- 4*s*t^3 + 16*t^4-4*u^4;
B:=-s^4 + 64*s^3*t + 6*s^2*t^2 - 64*s*t^3 - t^4-v^4;
S:=Scheme(P,[A,B]);
IsLocallySolvable(S,2);
\end{lstlisting}

The scheme defined by system \eqref{514-4} is locally insoluble at $2$.

{\bf Case 5:} $4u^4+v^4i=(1-16i)(s+ti)^4$. Equating the real and imaginary parts on both sides of this equation gives
\begin{equation}\label{514-5}
\begin{cases}
s^4 + 64s^3t -  6s^2t^2 - 64st^3 + t^4-4u^4=0,\\
-16s^4 + 4s^3t + 96s^2t^2 - 4st^3 - 16t^4-v^4=0.
\end{cases}
\end{equation}
The scheme defined by system \eqref{514-5} is locally insoluble at $2$.

%_<x>:=PolynomialRing(Rationals());
%k<i>:=NumberField(x^2+1);
%K<s,t>:=PolynomialRing(k,2);
%e:=0;
%F:=i^e*(1-16*i)*(s+t*i)^4;
%F;
%L<s,t>:=PolynomialRing(Rationals(),2);
%_<i>:=PolynomialRing(L);
%F:=(-16*i + 1)*s^4 + (4*i + 64)*s^3*t + (96*i - 6)*s^2*t^2 + (-4*i - 64)*s*t^3 +
%   (-16*i + 1)*t^4;
%F;
%A:=s^4 + 64*s^3*t -  6*s^2*t^2 - 64*s*t^3 + t^4;
%B:=-16*s^4 + 4*s^3*t + 96*s^2*t^2 - 4*s*t^3 - 16*t^4;
%F-A-i*B;
\begin{lstlisting}
P<s,t,u,v>:=ProjectiveSpace(Rationals(),3);
A:=s^4 + 64*s^3*t -  6*s^2*t^2 - 64*s*t^3 + t^4-4*u^4;
B:=-16*s^4 + 4*s^3*t + 96*s^2*t^2 - 4*s*t^3 - 16*t^4-v^4;
S:=Scheme(P,[A,B]);
IsLocallySolvable(S,2);
\end{lstlisting}

{\bf Case 6:} $4u^4+v^4i=i(1-16i)(s+ti)^4$. Equating the real and imaginary parts on both sides of this equation gives
\begin{equation}\label{514-6}
\begin{cases}
16s^4 - 4s^3t - 96s^2t^2 + 4st^3 + 16t^4-4u^4=0,\\
s^4 + 64s^3t - 6s^2t^2 - 64st^3 + t^4-v^4=0;
\end{cases}
\end{equation}
The scheme defined by system \eqref{514-6} is locally insoluble at $5$.

%_<x>:=PolynomialRing(Rationals());
%k<i>:=NumberField(x^2+1);
%K<s,t>:=PolynomialRing(k,2);
%e:=1;
%F:=i^e*(1-16*i)*(s+t*i)^4;
%F;
%L<s,t>:=PolynomialRing(Rationals(),2);
%_<i>:=PolynomialRing(L);
%F:=(i + 16)*s^4 + (64*i - 4)*s^3*t + (-6*i - 96)*s^2*t^2 + (-64*i + 4)*s*t^3 + (i +
%    16)*t^4;
%F;
%A:=16*s^4 - 4*s^3*t - 96*s^2*t^2 + 4*s*t^3 + 16*t^4;%B:=s^4 + 64*s^3*t - 6*s^2*t^2 - 64*s*t^3 + t^4;
%F-A-i*B;
%
\begin{lstlisting}
P<s,t,u,v>:=ProjectiveSpace(Rationals(),3);
A:=16*s^4 - 4*s^3*t - 96*s^2*t^2 + 4*s*t^3 + 16*t^4-4*u^4;
B:=s^4 + 64*s^3*t - 6*s^2*t^2 - 64*s*t^3 + t^4-v^4;
S:=Scheme(P,[A,B]);
IsLocallySolvable(S,5);
\end{lstlisting}

{\bf Case 7:} $4u^4+v^4i=i^2(1-16i)(s+ti)^4$. Equating the real and imaginary parts on both sides of this equation gives
\begin{equation}\label{514-7}
\begin{cases}
- s^4 - 64s^3t +6s^2t^2 + 64st^3 - t^4-4u^4=0,\\
16s^4 - 4s^3t - 96s^2t^2 + 4st^3 + 16t^4-v^4=0.
\end{cases}
\end{equation}
The scheme defined by system \eqref{514-7} is locally insoluble at $3$.

%_<x>:=PolynomialRing(Rationals());
%k<i>:=NumberField(x^2+1);
%K<s,t>:=PolynomialRing(k,2);
%e:=2;
%F:=i^e*(1-16*i)*(s+t*i)^4;
%F;
%L<s,t>:=PolynomialRing(Rationals(),2);
%_<i>:=PolynomialRing(L);
%F:=(16*i - 1)*s^4 + (-4*i - 64)*s^3*t + (-96*i + 6)*s^2*t^2 + (4*i + 64)*s*t^3 +
%    (16*i - 1)*t^4;
%F;
%A:=- s^4 - 64*s^3*t +6*s^2*t^2 + 64*s*t^3 - t^4;
%B:=16*s^4 - 4*s^3*t - 96*s^2*t^2 + 4*s*t^3 + 16*t^4;
%F-A-i*B;
\begin{lstlisting}
P<s,t,u,v>:=ProjectiveSpace(Rationals(),3);
A:=- s^4 - 64*s^3*t +6*s^2*t^2 + 64*s*t^3 - t^4-4*u^4;
B:=16*s^4 - 4*s^3*t - 96*s^2*t^2 + 4*s*t^3 + 16*t^4-v^4;
S:=Scheme(P,[A,B]);
IsLocallySolvable(S,3);
\end{lstlisting}

{\bf Case 8:} $4u^4+v^4i=i^3(1-16i)(s+ti)^4$. Equating the real and imaginary parts on both sides of this equation gives
\begin{equation}\label{514-8}
\begin{cases}
- 16s^4 + 4s^3t + 96s^2t^2 - 4st^3 - 16t^4-4u^4=0,\\
-s^4 - 64s^3t + 6s^2t^2 + 64st^3 - t^4-v^4=0,
\end{cases}
\end{equation}

The scheme defined by system \eqref{514-8} is locally insoluble at $2$.

%_<x>:=PolynomialRing(Rationals());
%k<i>:=NumberField(x^2+1);
%K<s,t>:=PolynomialRing(k,2);
%e:=3;
%F:=i^e*(1-16*i)*(s+t*i)^4;
%F;
%L<s,t>:=PolynomialRing(Rationals(),2);
%_<i>:=PolynomialRing(L);
%F:=(-i - 16)*s^4 + (-64*i + 4)*s^3*t + (6*i + 96)*s^2*t^2 + (64*i - 4)*s*t^3 + (-i
%   - 16)*t^4;
%F;
%A:= - 16*s^4 + 4*s^3*t + 96*s^2*t^2 - 4*s*t^3 - 16*t^4;
%B:=-s^4 - 64*s^3*t + 6*s^2*t^2 + 64*s*t^3 - t^4;
%F-A-i*B;

\begin{lstlisting}
P<s,t,u,v>:=ProjectiveSpace(Rationals(),3);
A:= - 16*s^4 + 4*s^3*t + 96*s^2*t^2 - 4*s*t^3 - 16*t^4-4*u^4;
B:=-s^4 - 64*s^3*t + 6*s^2*t^2 + 64*s*t^3 - t^4-v^4;
S:=Scheme(P,[A,B]);
IsLocallySolvable(S,2);
\end{lstlisting}

\subsection{The case of $n=1305$.}%\[n=1305\quad (I)\] 
According to Table \ref{tb:factorisation}, in the case $n=1305$ it remains to
show that equation
\begin{equation}\label{1305}
81u^8+4v^8=145w^4
\end{equation}
has no integer solutions satisfying \eqref{eq:cond}, which in this case reduces to \[\gcd(3u,2v)=\gcd(2u,145w)=\gcd(2v,145w)=1,\] 
and to show that equation
\begin{equation}\label{1305.2}
324 u^8 + v^8=145w^4
\end{equation}
has no integer solutions satisfying \eqref{eq:cond}, which in this case reduces to \[\gcd(18u,v)=\gcd(18u,145w)=\gcd(v,145w)=1.\]

\subsubsection{The case of \eqref{1305}.}
We first consider \eqref{1305}, and we can write this as
\begin{equation}\label{1305.1} (9u^4+2v^4i)(9u^4-2iv^4)=(1+2i)(1-2i)(5+2i)(5-2i)w^4.
\end{equation}
Hence, $\gcd(9u^4+2v^4i,9u^4-2v^4i)=1$ and \[9u^4+2v^4i\equiv 2i-1\pmod{5}.\]
Thus $1-2i|9u^4+2v^4i$. This implies that there exist integers $s,t$ such that 
\[9u^4+2v^4i=i^{\epsilon}(1-2i)(5\pm 2i)(s+ti)^4,\]
with $\epsilon \in \{0,1,2,3\}$.

{\bf Case 1:} $9u^4+2v^4i=(1-2i)(5+2i)(s+ti)^4$. Equating the real and imaginary parts on both sides of this equation gives
\begin{equation}\label{1305-1}
\begin{cases}
9s^4 + 32s^3t -54s^2t^2 - 32st^3 + 9t^4-9u^4=0,\\
-8s^4 + 36s^3t + 48s^2t^2 - 36st^3 - 8t^4-2v^4=0.
\end{cases}
\end{equation}
The scheme defined by system \eqref{1305-1} is locally insoluble at $2$.

%\begin{lstlisting}
%_<x>:=PolynomialRing(Rationals());
%k<i>:=NumberField(x^2+1);
%K<s,t>:=PolynomialRing(k,2);
%e:=0;
%F:=i^e*(1-2*i)*(5+2*i)*(s+t*i)^4;
%F;
%L<s,t>:=PolynomialRing(Rationals(),2);
%_<i>:=PolynomialRing(L);
%F:=(-8*i + 9)*s^4 + (36*i + 32)*s^3*t + (48*i - 54)*s^2*t^2 + (-36*i - 32)*s*t^3 +
%    (-8*i + 9)*t^4;
%F;
%A:=9*s^4 + 32*s^3*t -54*s^2*t^2 - 32*s*t^3 + 9*t^4;
%B:=-8*s^4 + 36*s^3*t + 48*s^2*t^2 - 36*s*t^3 - 8*t^4;
%F-A-i*B;

\begin{lstlisting}
P<s,t,u,v>:=ProjectiveSpace(Rationals(),3);
A:=9*s^4 + 32*s^3*t -54*s^2*t^2 - 32*s*t^3 + 9*t^4-9*u^4;
B:=-8*s^4 + 36*s^3*t + 48*s^2*t^2 - 36*s*t^3 - 8*t^4-2*v^4;
S:=Scheme(P,[A,B]);
IsLocallySolvable(S,2);
\end{lstlisting}
{\bf Case 2:} $9u^4+2v^4i=i(1-2i)(5+2i)(s+ti)^4$. Equating the real and imaginary parts on both sides of this equation gives
\begin{equation}\label{1305-2}
\begin{cases}
8s^4 - 36s^3t - 48s^2t^2 + 36st^3 + 8t^4-9u^4=0,\\
9s^4 + 32s^3t - 54s^2t^2 - 32st^3 + 9t^4-2v^4=0.
\end{cases}
\end{equation}
The scheme defined by system \eqref{1305-2} is locally insoluble at $2$.

%_<x>:=PolynomialRing(Rationals());
%k<i>:=NumberField(x^2+1);
%K<s,t>:=PolynomialRing(k,2);
%e:=1;
%F:=i^e*(1-2*i)*(5+2*i)*(s+t*i)^4;
%F;
%L<s,t>:=PolynomialRing(Rationals(),2);
%_<i>:=PolynomialRing(L);
%F:=(9*i + 8)*s^4 + (32*i - 36)*s^3*t + (-54*i - 48)*s^2*t^2 + (-32*i + 36)*s*t^3 +
%(9*i + 8)*t^4;
%F;
%A:= 8*s^4 - 36*s^3*t - 48*s^2*t^2 + 36*s*t^3 + 8*t^4;
%B:=9*s^4 + 32*s^3*t - 54*s^2*t^2 - 32*s*t^3 + 9*t^4;
%F-A-i*B;
\begin{lstlisting}
P<s,t,u,v>:=ProjectiveSpace(Rationals(),3);
A:= 8*s^4 - 36*s^3*t - 48*s^2*t^2 + 36*s*t^3 + 8*t^4-9*u^4;
B:=9*s^4 + 32*s^3*t - 54*s^2*t^2 - 32*s*t^3 + 9*t^4-2*v^4;
S:=Scheme(P,[A,B]);
IsLocallySolvable(S,2);
\end{lstlisting}

{\bf Case 3:} $9u^4+2v^4i=i^2(1-2i)(5+2i)(s+ti)^4$. Equating the real and imaginary parts on both sides of this equation gives
\begin{equation}\label{1305-3}
\begin{cases}
-9s^4 - 32s^3t + 54s^2t^2 + 32st^3 - 9t^4 - 9u^4=0,\\
8s^4 - 36s^3t - 48s^2t^2 + 36st^3 + 8t^4 - 2v^4=0.
\end{cases}
\end{equation}
The scheme defined by system \eqref{1305-3} is locally insoluble at $2$.

%_<x>:=PolynomialRing(Rationals());
%k<i>:=NumberField(x^2+1);
%K<s,t>:=PolynomialRing(k,2);
%e:=2;
%F:=i^e*(1-2*i)*(5+2*i)*(s+t*i)^4;
%F;
%L<s,t>:=PolynomialRing(Rationals(),2);
%_<i>:=PolynomialRing(L);
%F:=(8*i - 9)*s^4 + (-36*i - 32)*s^3*t + (-48*i + 54)*s^2*t^2 + (36*i + 32)*s*t^3 +
%   (8*i - 9)*t^4;
%F;
%A:=- 9*s^4 - 32*s^3*t +54*s^2*t^2 + 32*s*t^3 - 9*t^4;
%B:=8*s^4 - 36*s^3*t - 48*s^2*t^2 + 36*s*t^3 + 8*t^4;
%F-A-i*B;

\begin{lstlisting}
P<s,t,u,v>:=ProjectiveSpace(Rationals(),3);
A:=- 9*s^4 - 32*s^3*t +54*s^2*t^2 + 32*s*t^3 - 9*t^4-9*u^4;
B:=8*s^4 - 36*s^3*t - 48*s^2*t^2 + 36*s*t^3 + 8*t^4-2*v^4;
S:=Scheme(P,[A,B]);
IsLocallySolvable(S,2);
\end{lstlisting}

{\bf Case 4:} $9u^4+2v^4i=i^3(1-2i)(5+2i)(s+ti)^4$. Equating the real and imaginary parts on both sides of this equation gives
\begin{equation}\label{1305-4}
\begin{cases}
-8s^4 + 36s^3t + 48s^2t^2 - 36st^3 - 8t^4 - 9u^4=0,\\
-9s^4 - 32s^3t + 54s^2t^2 + 32st^3 - 9t^4 - 2v^4=0.
\end{cases}
\end{equation}
The scheme defined by system \eqref{1305-4} is locally insoluble at $2$.

%\begin{lstlisting}
%_<x>:=PolynomialRing(Rationals());
%k<i>:=NumberField(x^2+1);
%K<s,t>:=PolynomialRing(k,2);
%e:=3;
%F:=i^e*(1-2*i)*(5+2*i)*(s+t*i)^4;
%F;
%L<s,t>:=PolynomialRing(Rationals(),2);
%_<i>:=PolynomialRing(L);
%F:=(-9*i - 8)*s^4 + (-32*i + 36)*s^3*t + (54*i + 48)*s^2*t^2 + (32*i - 36)*s*t^3 +
%   (-9*i - 8)*t^4;
%F;
%A:=- 8*s^4 + 36*s^3*t +48*s^2*t^2 - 36*s*t^3 - 8*t^4;
%B:=-9*s^4 - 32*s^3*t + 54*s^2*t^2 + 32*s*t^3 - 9*t^4;
%F-A-i*B;

\begin{lstlisting}
P<s,t,u,v>:=ProjectiveSpace(Rationals(),3);
A:= - 8*s^4 + 36*s^3*t +48*s^2*t^2 - 36*s*t^3 - 8*t^4-9*u^4;
B:=-9*s^4 - 32*s^3*t + 54*s^2*t^2 + 32*s*t^3 - 9*t^4-2*v^4;
S:=Scheme(P,[A,B]);
IsLocallySolvable(S,2);
\end{lstlisting}

{\bf Case 5:} $9u^4+2v^4i=(1-2i)(5-2i)(s+ti)^4$. Equating the real and imaginary parts on both sides of this equation gives
\begin{equation}\label{1305-5}
\begin{cases}
s^4 + 48s^3t - 6s^2t^2 - 48st^3 + t^4 - 9u^4=0,\\
-12s^4 + 4s^3t + 72s^2t^2 - 4st^3 - 12t^4 - 2v^4=0.
\end{cases}
\end{equation}
The scheme defined by system \eqref{1305-5} is locally insoluble at $2$.

%\begin{lstlisting}
%_<x>:=PolynomialRing(Rationals());
%k<i>:=NumberField(x^2+1);
%K<s,t>:=PolynomialRing(k,2);
%e:=0;
%F:=i^e*(1-2*i)*(5-2*i)*(s+t*i)^4;
%F;
%L<s,t>:=PolynomialRing(Rationals(),2);
%_<i>:=PolynomialRing(L);
%F:=(-12*i + 1)*s^4 + (4*i + 48)*s^3*t + (72*i - 6)*s^2*t^2 + (-4*i - 48)*s*t^3 +
%(-12*i + 1)*t^4;
%F;
%A:=s^4 + 48*s^3*t - 6*s^2*t^2 - 48*s*t^3 + t^4;
%B:=-12*s^4 + 4*s^3*t + 72*s^2*t^2 - 4*s*t^3 - 12*t^4;
%F-A-i*B;

\begin{lstlisting}
P<s,t,u,v>:=ProjectiveSpace(Rationals(),3);
A:=s^4 + 48*s^3*t - 6*s^2*t^2 - 48*s*t^3 + t^4-9*u^4;
B:=-12*s^4 + 4*s^3*t + 72*s^2*t^2 - 4*s*t^3 - 12*t^4-2*v^4;
S:=Scheme(P,[A,B]);
IsLocallySolvable(S,2);
\end{lstlisting}
{\bf Case 6:} $9u^4+2v^4i=i(1-2i)(5-2i)(s+ti)^4$. Equating the real and imaginary parts on both sides of this equation gives
\begin{equation}\label{1305-6}
\begin{cases}
12s^4 - 4s^3t - 72s^2t^2 + 4st^3 + 12t^4 - 9u^4=0,\\
s^4 + 48s^3t - 6s^2t^2 - 48st^3 + t^4 - 2v^4=0.
\end{cases}
\end{equation}
The scheme defined by system \eqref{1305-6} is locally insoluble at $2$.

%\begin{lstlisting}
%_<x>:=PolynomialRing(Rationals());
%k<i>:=NumberField(x^2+1);
%K<s,t>:=PolynomialRing(k,2);
%e:=1;
%F:=i^e*(1-2*i)*(5-2*i)*(s+t*i)^4;
%F;
%L<s,t>:=PolynomialRing(Rationals(),2);
%_<i>:=PolynomialRing(L);
%F:=(i + 12)*s^4 + (48*i - 4)*s^3*t + (-6*i - 72)*s^2*t^2 + (-48*i + 4)*s*t^3 + (i +
%   12)*t^4;
%F;
%A:=12*s^4 - 4*s^3*t - 72*s^2*t^2+ 4*s*t^3 + 12*t^4;
%B:=s^4 + 48*s^3*t - 6*s^2*t^2 - 48*s*t^3 + t^4;
%F-A-i*B;

\begin{lstlisting}
P<s,t,u,v>:=ProjectiveSpace(Rationals(),3);
A:=12*s^4 - 4*s^3*t - 72*s^2*t^2+ 4*s*t^3 + 12*t^4-9*u^4;
B:=s^4 + 48*s^3*t - 6*s^2*t^2 - 48*s*t^3 + t^4-2*v^4;
S:=Scheme(P,[A,B]);
IsLocallySolvable(S,2);
\end{lstlisting}
{\bf Case 7:} $9u^4+2v^4i=i^2(1-2i)(5-2i)(s+ti)^4$. Equating the real and imaginary parts on both sides of this equation gives
\begin{equation}\label{1305-7}
\begin{cases}
-s^4 - 48s^3t + 6s^2t^2 + 48st^3 - t^4 - 9u^4=0,\\
12s^4 - 4s^3t - 72s^2t^2 + 4st^3 + 12t^4 - 2v^4=0.
\end{cases}
\end{equation}
The scheme defined by system \eqref{1305-7} is locally insoluble at $2$.

%\begin{lstlisting}
%_<x>:=PolynomialRing(Rationals());
%k<i>:=NumberField(x^2+1);
%K<s,t>:=PolynomialRing(k,2);
%e:=2;
%F:=i^e*(1-2*i)*(5-2*i)*(s+t*i)^4;
%F;
%L<s,t>:=PolynomialRing(Rationals(),2);
%_<i>:=PolynomialRing(L);
%F:=(12*i - 1)*s^4 + (-4*i - 48)*s^3*t + (-72*i + 6)*s^2*t^2 + (4*i + 48)*s*t^3 +
%    (12*i - 1)*t^4;
%F;
%A:=- s^4 - 48*s^3*t +6*s^2*t^2 + 48*s*t^3 - t^4;
%B:=12*s^4 - 4*s^3*t - 72*s^2*t^2 + 4*s*t^3 + 12*t^4;
%F-A-i*B;

\begin{lstlisting}
P<s,t,u,v>:=ProjectiveSpace(Rationals(),3);
A:=- s^4 - 48*s^3*t +6*s^2*t^2 + 48*s*t^3 - t^4-9*u^4;
B:=12*s^4 - 4*s^3*t - 72*s^2*t^2 + 4*s*t^3 + 12*t^4-2*v^4;
S:=Scheme(P,[A,B]);
IsLocallySolvable(S,2);
\end{lstlisting}
{\bf Case 8:} $9u^4+2v^4i=i^3(1-2i)(5-2i)(s+ti)^4$. Equating the real and imaginary parts on both sides of this equation gives
\begin{equation}\label{1305-8}
\begin{cases}
-12s^4 + 4s^3t + 72s^2t^2 - 4st^3 - 12t^4 - 9u^4=0,\\
-s^4 - 48s^3t + 6s^2t^2 + 48st^3 - t^4 - 2v^4=0,
\end{cases}
\end{equation}
The scheme defined by system \eqref{1305-8} is locally insoluble at $2$.

%\begin{lstlisting}
%_<x>:=PolynomialRing(Rationals());
%k<i>:=NumberField(x^2+1);
%K<s,t>:=PolynomialRing(k,2);
%e:=3;
%F:=i^e*(1-2*i)*(5-2*i)*(s+t*i)^4;
%F;
%L<s,t>:=PolynomialRing(Rationals(),2);
%_<i>:=PolynomialRing(L);
%F:=(-i - 12)*s^4 + (-48*i + 4)*s^3*t + (6*i + 72)*s^2*t^2 + (48*i - 4)*s*t^3 + (-i
%    - 12)*t^4;
%F;
%A:=- 12*s^4 + 4*s^3*t + 72*s^2*t^2- 4*s*t^3 - 12*t^4;
%B:=-s^4 - 48*s^3*t + 6*s^2*t^2 + 48*s*t^3 - t^4;
%F-A-i*B;

\begin{lstlisting}
P<s,t,u,v>:=ProjectiveSpace(Rationals(),3);
A:=- 12*s^4 + 4*s^3*t + 72*s^2*t^2- 4*s*t^3 - 12*t^4-9*u^4;
B:=-s^4 - 48*s^3*t + 6*s^2*t^2 + 48*s*t^3 - t^4-2*v^4;
S:=Scheme(P,[A,B]);
IsLocallySolvable(S,2);
\end{lstlisting}

\subsubsection{{\bf{The case of \eqref{1305.2}}.}}%\[n=1305\quad (II)\] 
We now consider \eqref{1305.2}, which we can write as
\begin{equation} (18u^4+v^4i)(18u^4-iv^4)=(1+2i)(1-2i)(5+2i)(5-2i)w^4.
\end{equation}
Hence, \[\begin{split}
    18u^4+iv^4&\equiv -2+i\pmod{5}\\
    &\equiv 0\pmod{1+2i}\\
    &\not\equiv 0\pmod{1-2i}.
\end{split}\] 
This implies that there exist integers $s,t$ such that 
\[18u^4+v^4i=i^{\epsilon}(1+2i)(5\pm 2i)(s+ti)^4,\]
with $\epsilon \in \{0,1,2,3\}$.

{\bf Case 1:} $18u^4+v^4i=(1+2i)(5+2i)(s+ti)^4$. Equating the real and imaginary parts on both sides of this equation gives
\begin{equation}\label{1305.2-1}
\begin{cases}
s^4 - 48s^3t - 6s^2t^2 + 48st^3 + t^4 - 18u^4=0,\\
12s^4 + 4s^3t - 72s^2t^2 - 4st^3 + 12t^4 - v^4=0.
\end{cases}
\end{equation}
The scheme defined by system \eqref{1305.2-1} is locally insoluble at $2$.

%\begin{lstlisting}
%_<x>:=PolynomialRing(Rationals());
%k<i>:=NumberField(x^2+1);
%K<s,t>:=PolynomialRing(k,2);
%e:=0;
%F:=i^e*(1+2*i)*(5+2*i)*(s+t*i)^4;
%F;
%L<s,t>:=PolynomialRing(Rationals(),2);
%_<i>:=PolynomialRing(L);
%F:=(12*i + 1)*s^4 + (4*i - 48)*s^3*t + (-72*i - 6)*s^2*t^2 + (-4*i + 48)*s*t^3 +
%    (12*i + 1)*t^4;
%F;
%A:= s^4 - 48*s^3*t -6*s^2*t^2 + 48*s*t^3 + t^4;
%B:=12*s^4 + 4*s^3*t - 72*s^2*t^2 - 4*s*t^3 + 12*t^4;
%F-A-i*B;
\begin{lstlisting}
P<s,t,u,v>:=ProjectiveSpace(Rationals(),3);
A:= s^4 - 48*s^3*t -6*s^2*t^2 + 48*s*t^3 + t^4-18*u^4;
B:=12*s^4 + 4*s^3*t - 72*s^2*t^2 - 4*s*t^3 + 12*t^4-v^4;
S:=Scheme(P,[A,B]);
IsLocallySolvable(S,2);
\end{lstlisting}
{\bf Case 2:} $18u^4+v^4i=i(1+2i)(5+2i)(s+ti)^4$. Equating the real and imaginary parts on both sides of this equation gives
\begin{equation}\label{1305.2-2}
\begin{cases}
-12s^4 - 4s^3t + 72s^2t^2 + 4st^3 - 12t^4 - 18u^4=0,\\
s^4 - 48s^3t - 6s^2t^2 + 48st^3 + t^4 - v^4=0.
\end{cases}
\end{equation}
The scheme defined by system \eqref{1305.2-2} is locally insoluble at $2$.

%\begin{lstlisting}
%_<x>:=PolynomialRing(Rationals());
%k<i>:=NumberField(x^2+1);
%K<s,t>:=PolynomialRing(k,2);
%e:=1;
%F:=i^e*(1+2*i)*(5+2*i)*(s+t*i)^4;
%F;
%L<s,t>:=PolynomialRing(Rationals(),2);
%_<i>:=PolynomialRing(L);
%F:=(i - 12)*s^4 + (-48*i - 4)*s^3*t + (-6*i + 72)*s^2*t^2 + (48*i + 4)*s*t^3 + (i -
%    12)*t^4;
%F;
%A:=- 12*s^4 - 4*s^3*t + 72*s^2*t^2  + 4*s*t^3 - 12*t^4;
%B:=s^4 - 48*s^3*t - 6*s^2*t^2 + 48*s*t^3 + t^4;
%F-A-i*B;
\begin{lstlisting}
P<s,t,u,v>:=ProjectiveSpace(Rationals(),3);
A:=- 12*s^4 - 4*s^3*t + 72*s^2*t^2  + 4*s*t^3 - 12*t^4-18*u^4;
B:=s^4 - 48*s^3*t - 6*s^2*t^2 + 48*s*t^3 + t^4-v^4;
S:=Scheme(P,[A,B]);
IsLocallySolvable(S,2);
\end{lstlisting}

{\bf Case 3:} $18u^4+v^4i=i^2(1+2i)(5+2i)(s+ti)^4$. Equating the real and imaginary parts on both sides of this equation gives
\begin{equation}\label{1305.2-3}
\begin{cases}
-s^4 + 48s^3t + 6s^2t^2 - 48st^3 - t^4 - 18u^4=0,\\
-12s^4 - 4s^3t + 72s^2t^2 + 4st^3 - 12t^4 - v^4=0.
\end{cases}
\end{equation}
The scheme defined by system \eqref{1305.2-3} is locally insoluble at $2$.

%\begin{lstlisting}
%_<x>:=PolynomialRing(Rationals());
%k<i>:=NumberField(x^2+1);
%K<s,t>:=PolynomialRing(k,2);
%e:=2;
%F:=i^e*(1+2*i)*(5+2*i)*(s+t*i)^4;
%F;
%L<s,t>:=PolynomialRing(Rationals(),2);
%_<i>:=PolynomialRing(L);
%F:=(-12*i - 1)*s^4 + (-4*i + 48)*s^3*t + (72*i + 6)*s^2*t^2 + (4*i - 48)*s*t^3 +
%    (-12*i - 1)*t^4;
%F;
%A:= - s^4 + 48*s^3*t + 6*s^2*t^2 - 48*s*t^3 - t^4;
%B:=-12*s^4 - 4*s^3*t + 72*s^2*t^2 + 4*s*t^3 - 12*t^4;
%F-A-i*B;
\begin{lstlisting}
P<s,t,u,v>:=ProjectiveSpace(Rationals(),3);
A:= - s^4 + 48*s^3*t + 6*s^2*t^2 - 48*s*t^3 - t^4-18*u^4;
B:=-12*s^4 - 4*s^3*t + 72*s^2*t^2 + 4*s*t^3 - 12*t^4-v^4;
S:=Scheme(P,[A,B]);
IsLocallySolvable(S,2);
\end{lstlisting}

{\bf Case 4:} $18u^4+v^4i=i^3(1+2i)(5+2i)(s+ti)^4$. Equating the real and imaginary parts on both sides of this equation gives
\begin{equation}\label{1305.2-4}
\begin{cases}
12s^4 + 4s^3t - 72s^2t^2 - 4st^3 + 12t^4 - 18u^4=0,\\
-s^4 + 48s^3t + 6s^2t^2 - 48st^3 - t^4 - v^4=0.
\end{cases}
\end{equation}
The scheme defined by system \eqref{1305.2-4} is locally insoluble at $2$.
%\begin{lstlisting}
%_<x>:=PolynomialRing(Rationals());
%k<i>:=NumberField(x^2+1);
%K<s,t>:=PolynomialRing(k,2);
%e:=3;
%F:=i^e*(1+2*i)*(5+2*i)*(s+t*i)^4;
%F;
%L<s,t>:=PolynomialRing(Rationals(),2);
%_<i>:=PolynomialRing(L);
%F:=(-i + 12)*s^4 + (48*i + 4)*s^3*t + (6*i - 72)*s^2*t^2 + (-48*i - 4)*s*t^3 + (-i
%    + 12)*t^4;
%F;
%A:=12*s^4 + 4*s^3*t - 72*s^2*t^2   - 4*s*t^3 + 12*t^4;
%B:=-s^4 + 48*s^3*t + 6*s^2*t^2 - 48*s*t^3 - t^4;
%F-A-i*B;

\begin{lstlisting}
P<s,t,u,v>:=ProjectiveSpace(Rationals(),3);
A:=12*s^4 + 4*s^3*t - 72*s^2*t^2   - 4*s*t^3 + 12*t^4-18*u^4;
B:=-s^4 + 48*s^3*t + 6*s^2*t^2 - 48*s*t^3 - t^4-v^4;
S:=Scheme(P,[A,B]);
IsLocallySolvable(S,2);
\end{lstlisting}

{\bf Case 5:} $18u^4+v^4i=(1+2i)(5-2i)(s+ti)^4$. Equating the real and imaginary parts on both sides of this equation gives
\begin{equation}\label{1305.2-5}
\begin{cases}
9s^4 - 32s^3t - 54s^2t^2 + 32st^3 + 9t^4 - 18u^4=0,\\
8s^4 + 36s^3t - 48s^2t^2 - 36st^3 + 8t^4 - v^4=0.
\end{cases}
\end{equation}
The scheme defined by system \eqref{1305.2-5} is locally insoluble at $2$.
%\begin{lstlisting}
%_<x>:=PolynomialRing(Rationals());
%k<i>:=NumberField(x^2+1);
%K<s,t>:=PolynomialRing(k,2);
%e:=0;
%F:=i^e*(1+2*i)*(5-2*i)*(s+t*i)^4;
%F;
%L<s,t>:=PolynomialRing(Rationals(),2);
%_<i>:=PolynomialRing(L);
%F:=(8*i + 9)*s^4 + (36*i - 32)*s^3*t + (-48*i - 54)*s^2*t^2 + (-36*i + 32)*s*t^3 +
%    (8*i + 9)*t^4;
%F;
%A:=9*s^4 - 32*s^3*t -  54*s^2*t^2 + 32*s*t^3 + 9*t^4;
%B:=8*s^4 + 36*s^3*t - 48*s^2*t^2 - 36*s*t^3 + 8*t^4;
%F-A-i*B;

\begin{lstlisting}
P<s,t,u,v>:=ProjectiveSpace(Rationals(),3);
A:=9*s^4 - 32*s^3*t -  54*s^2*t^2 + 32*s*t^3 + 9*t^4-18*u^4;
B:=8*s^4 + 36*s^3*t - 48*s^2*t^2 - 36*s*t^3 + 8*t^4-v^4;
S:=Scheme(P,[A,B]);
IsLocallySolvable(S,2);
\end{lstlisting}
{\bf Case 6:} $18u^4+v^4i=i(1+2i)(5-2i)(s+ti)^4$. Equating the real and imaginary parts on both sides of this equation gives
\begin{equation}\label{1305.2-6}
\begin{cases}
-8s^4 - 36s^3t + 48s^2t^2 + 36st^3 - 8t^4 - 18u^4=0,\\
9s^4 - 32s^3t - 54s^2t^2 + 32st^3 + 9t^4 - v^4=0.
\end{cases}
\end{equation}
The scheme defined by system \eqref{1305.2-6} is locally insoluble at $3$.
%\begin{lstlisting}
%_<x>:=PolynomialRing(Rationals());
%k<i>:=NumberField(x^2+1);
%K<s,t>:=PolynomialRing(k,2);
%e:=1;
%F:=i^e*(1+2*i)*(5-2*i)*(s+t*i)^4;
%F;
%L<s,t>:=PolynomialRing(Rationals(),2);
%_<i>:=PolynomialRing(L);
%F:=(9*i - 8)*s^4 + (-32*i - 36)*s^3*t + (-54*i + 48)*s^2*t^2 + (32*i + 36)*s*t^3 +
%    (9*i - 8)*t^4;
%F;
%A:= - 8*s^4 - 36*s^3*t +48*s^2*t^2 + 36*s*t^3 - 8*t^4;
%B:=9*s^4 - 32*s^3*t - 54*s^2*t^2 + 32*s*t^3 + 9*t^4;
%F-A-i*B;

\begin{lstlisting}
P<s,t,u,v>:=ProjectiveSpace(Rationals(),3);
A:= - 8*s^4 - 36*s^3*t +48*s^2*t^2 + 36*s*t^3 - 8*t^4-18*u^4;
B:=9*s^4 - 32*s^3*t - 54*s^2*t^2 + 32*s*t^3 + 9*t^4-v^4;
S:=Scheme(P,[A,B]);
IsLocallySolvable(S,3);
\end{lstlisting}
{\bf Case 7:} $18u^4+v^4i=i^2(1+2i)(5-2i)(s+ti)^4$. Equating the real and imaginary parts on both sides of this equation gives
\begin{equation}\label{1305.2-7}
\begin{cases}
-9s^4 + 32s^3t + 54s^2t^2 - 32st^3 - 9t^4 - 18u^4=0,\\
-8s^4 - 36s^3t + 48s^2t^2 + 36st^3 - 8t^4 - v^4=0.
\end{cases}
\end{equation}
The scheme defined by system \eqref{1305.2-7} is locally insoluble at $2$.
%\begin{lstlisting}
%_<x>:=PolynomialRing(Rationals());
%k<i>:=NumberField(x^2+1);
%K<s,t>:=PolynomialRing(k,2);
%e:=2;
%F:=i^e*(1+2*i)*(5-2*i)*(s+t*i)^4;
%F;
%L<s,t>:=PolynomialRing(Rationals(),2);
%_<i>:=PolynomialRing(L);
%F:=(-8*i - 9)*s^4 + (-36*i + 32)*s^3*t + (48*i + 54)*s^2*t^2 + (36*i - 32)*s*t^3 +
%    (-8*i - 9)*t^4;
%F;
%A:=- 9*s^4 + 32*s^3*t + 54*s^2*t^2 - 32*s*t^3 - 9*t^4;
%B:=-8*s^4 - 36*s^3*t + 48*s^2*t^2 + 36*s*t^3 - 8*t^4;
%F-A-i*B;

\begin{lstlisting}
P<s,t,u,v>:=ProjectiveSpace(Rationals(),3);
A:=- 9*s^4 + 32*s^3*t + 54*s^2*t^2 - 32*s*t^3 - 9*t^4-18*u^4;
B:=-8*s^4 - 36*s^3*t + 48*s^2*t^2 + 36*s*t^3 - 8*t^4-v^4;
S:=Scheme(P,[A,B]);
IsLocallySolvable(S,2);
\end{lstlisting}
{\bf Case 8:} $18u^4+v^4i=i^3(1+2i)(5-2i)(s+ti)^4$. Equating the real and imaginary parts on both sides of this equation gives
\begin{equation}\label{1305.2-8}
\begin{cases}
8s^4 + 36s^3t - 48s^2t^2 - 36st^3 + 8t^4 - 18u^4=0,\\
-9s^4 + 32s^3t + 54s^2t^2 - 32s*t^3 - 9t^4 - v^4=0.
\end{cases}
\end{equation}
The scheme defined by system \eqref{1305.2-8} is locally insoluble at $2$.

%\begin{lstlisting}
%_<x>:=PolynomialRing(Rationals());
%k<i>:=NumberField(x^2+1);
%K<s,t>:=PolynomialRing(k,2);
%e:=3;
%F:=i^e*(1+2*i)*(5-2*i)*(s+t*i)^4;
%F;
%L<s,t>:=PolynomialRing(Rationals(),2);
%_<i>:=PolynomialRing(L);
%F:=(-9*i + 8)*s^4 + (32*i + 36)*s^3*t + (54*i - 48)*s^2*t^2 + (-32*i - 36)*s*t^3 +
%    (-9*i + 8)*t^4;
%F;
%A:= 8*s^4 + 36*s^3*t -48*s^2*t^2 - 36*s*t^3 + 8*t^4;
%B:=-9*s^4 + 32*s^3*t + 54*s^2*t^2 - 32*s*t^3 - 9*t^4;
%F-A-i*B;

\begin{lstlisting}
P<s,t,u,v>:=ProjectiveSpace(Rationals(),3);
A:= 8*s^4 + 36*s^3*t -48*s^2*t^2 - 36*s*t^3 + 8*t^4-18*u^4;
B:=-9*s^4 + 32*s^3*t + 54*s^2*t^2 - 32*s*t^3 - 9*t^4-v^4;
S:=Scheme(P,[A,B]);
IsLocallySolvable(S,2);
\end{lstlisting}

\subsection{The case of $n=1493$.}%\[n=1493\] 
According to Table \ref{tb:factorisation}, in the case $n=1493$ it remains to
show that equation
\begin{equation}\label{1493}
4u^8+v^8=1493w^4
\end{equation}
has no integer solutions satisfying \eqref{eq:cond}, which in this case reduces to \[\gcd(2u,v)=\gcd(2u,1493w)=\gcd(v,1493w)=1.\]  Write \eqref{1493} as
\begin{equation}\label{1493.1} (2u^4+v^4i)(2u^4-v^4i)=(38+7i)(38-7i)w^4.
\end{equation}
This implies that there exist integers $s,t$ such that 
\[2u^4+v^4i=i^{\epsilon}(38\pm 7i)(s+ti)^4,\]
with $\epsilon \in \{0,1,2,3\}$.

{\bf Case 1:} $2u^4+v^4i=(38+7i)(s+ti)^4$. Equating the real and imaginary parts on both sides of this equation gives
\begin{equation}\label{1493-1}
\begin{cases}
38s^4 - 28s^3t - 228s^2t^2 + 28st^3 + 38t^4 - 2u^4=0,\\
7s^4 + 152s^3t - 42s^2t^2 - 152st^3 + 7t^4 - v^4=0.
\end{cases}
\end{equation}
The scheme defined by system \eqref{1493-1} is locally insoluble at $2$.

%\begin{lstlisting}
%_<x>:=PolynomialRing(Rationals());
%k<i>:=NumberField(x^2+1);
%K<s,t>:=PolynomialRing(k,2);
%e:=0;
%F:=i^e*(38+7*i)*(s+t*i)^4;
%F;
%L<s,t>:=PolynomialRing(Rationals(),2);
%_<i>:=PolynomialRing(L);
%F:=(7*i + 38)*s^4 + (152*i - 28)*s^3*t + (-42*i - 228)*s^2*t^2 + (-152*i +
%   28)*s*t^3 + (7*i + 38)*t^4;
%F;
%A:=38*s^4 - 28*s^3*t -228*s^2*t^2 + 28*s*t^3 + 38*t^4;
%B:=7*s^4 + 152*s^3*t - 42*s^2*t^2 - 152*s*t^3 + 7*t^4;
%F-A-i*B;

\begin{lstlisting}
P<s,t,u,v>:=ProjectiveSpace(Rationals(),3);
A:=38*s^4 - 28*s^3*t -228*s^2*t^2 + 28*s*t^3 + 38*t^4-2*u^4;
B:=7*s^4 + 152*s^3*t - 42*s^2*t^2 - 152*s*t^3 + 7*t^4-v^4;
S:=Scheme(P,[A,B]);
IsLocallySolvable(S,2);
\end{lstlisting}

{\bf Case 2:} $2u^4+v^4i=i(38+7i)(s+ti)^4$. Equating the real and imaginary parts on both sides of this equation gives
\begin{equation}\label{1493-2}
\begin{cases}
-7s^4 - 152s^3t + 42s^2t^2 + 152st^3 - 7t^4 - 2u^4=0,\\
38s^4 - 28s^3t - 228s^2t^2 + 28st^3 + 38t^4 - v^4=0.
\end{cases}
\end{equation}
The scheme defined by system \eqref{1493-2} is locally insoluble at $2$.
%\begin{lstlisting}
%_<x>:=PolynomialRing(Rationals());
%k<i>:=NumberField(x^2+1);
%K<s,t>:=PolynomialRing(k,2);
%e:=1;
%F:=i^e*(38+7*i)*(s+t*i)^4;
%F;
%L<s,t>:=PolynomialRing(Rationals(),2);
%_<i>:=PolynomialRing(L);
%F:=(38*i - 7)*s^4 + (-28*i - 152)*s^3*t + (-228*i + 42)*s^2*t^2 + (28*i +
%   152)*s*t^3 + (38*i - 7)*t^4;
%F;
%A:=- 7*s^4 - 152*s^3*t + 42*s^2*t^2 + 152*s*t^3 - 7*t^4;
%B:=38*s^4 - 28*s^3*t - 228*s^2*t^2 + 28*s*t^3 + 38*t^4;
%F-A-i*B;

\begin{lstlisting}
P<s,t,u,v>:=ProjectiveSpace(Rationals(),3);
A:=- 7*s^4 - 152*s^3*t + 42*s^2*t^2 + 152*s*t^3 - 7*t^4-2*u^4;
B:=38*s^4 - 28*s^3*t - 228*s^2*t^2 + 28*s*t^3 + 38*t^4-v^4;
S:=Scheme(P,[A,B]);
IsLocallySolvable(S,2);
\end{lstlisting}

{\bf Case 3:} $2u^4+v^4i=i^2(38+7i)(s+ti)^4$. Equating the real and imaginary parts on both sides of this equation gives
\begin{equation}\label{1493-3}
\begin{cases}
- 38s^4 + 28s^3t +228s^2t^2 - 28st^3 - 38t^4-2u^4=0,\\
-7s^4 - 152s^3t + 42s^2t^2 + 152st^3 - 7t^4-v^4=0;
\end{cases}
\end{equation}
The scheme defined by system \eqref{1493-3} is locally insoluble at $5$.

%\begin{lstlisting}
%_<x>:=PolynomialRing(Rationals());
%k<i>:=NumberField(x^2+1);
%K<s,t>:=PolynomialRing(k,2);
%e:=2;
%F:=i^e*(38+7*i)*(s+t*i)^4;
%F;
%L<s,t>:=PolynomialRing(Rationals(),2);
%_<i>:=PolynomialRing(L);
%F:=(-7*i - 38)*s^4 + (-152*i + 28)*s^3*t + (42*i + 228)*s^2*t^2 + (152*i -
%    28)*s*t^3 + (-7*i - 38)*t^4;
%F;
%A:=- 38*s^4 + 28*s^3*t +228*s^2*t^2 - 28*s*t^3 - 38*t^4;
%B:=-7*s^4 - 152*s^3*t + 42*s^2*t^2 + 152*s*t^3 - 7*t^4;
%F-A-i*B;

\begin{lstlisting}
P<s,t,u,v>:=ProjectiveSpace(Rationals(),3);
A:=- 38*s^4 + 28*s^3*t +228*s^2*t^2 - 28*s*t^3 - 38*t^4-2*u^4;
B:=-7*s^4 - 152*s^3*t + 42*s^2*t^2 + 152*s*t^3 - 7*t^4-v^4;
S:=Scheme(P,[A,B]);
IsLocallySolvable(S,5);
\end{lstlisting}

{\bf Case 4:} $2u^4+v^4i=i^3(38+7i)(s+ti)^4$. Equating the real and imaginary parts on both sides of this equation gives
\begin{equation}\label{1493-4}
\begin{cases}
7s^4 + 152s^3t - 42s^2t^2 - 152st^3 + 7t^4 - 2u^4=0,\\
-38s^4 + 28s^3t + 228s^2t^2 - 28st^3 - 38t^4 - v^4=0.
\end{cases}
\end{equation}
The scheme defined by system \eqref{1493-4} is locally insoluble at $2$.

\begin{lstlisting}
P<s,t,u,v>:=ProjectiveSpace(Rationals(),3);
A:=7*s^4 + 152*s^3*t - 42*s^2*t^2 - 152*s*t^3 + 7*t^4 - 2*u^4;
B:=-38*s^4 + 28*s^3*t + 228*s^2*t^2 - 28*s*t^3 - 38*t^4 - v^4;
S:=Scheme(P,[A,B]);
IsLocallySolvable(S,2);
\end{lstlisting}

{\bf Case 5:} $2u^4+v^4i=(38-7i)(s+ti)^4$. Equating the real and imaginary parts on both sides of this equation gives
\begin{equation}\label{1493-5}
\begin{cases}
38s^4 + 28s^3t - 228s^2t^2 - 28st^3 + 38t^4 - 2u^4=0,\\
-7s^4 + 152s^3t + 42s^2t^2 - 152st^3 - 7t^4 - v^4=0.
\end{cases}
\end{equation}
The scheme defined by system \eqref{1493-5} is locally insoluble at $2$.

%\begin{lstlisting}
%_<x>:=PolynomialRing(Rationals());
%k<i>:=NumberField(x^2+1);
%K<s,t>:=PolynomialRing(k,2);
%e:=0;
%F:=i^e*(38-7*i)*(s+t*i)^4;
%F;
%L<s,t>:=PolynomialRing(Rationals(),2);
%_<i>:=PolynomialRing(L);
%F:=(-7*i + 38)*s^4 + (152*i + 28)*s^3*t + (42*i - 228)*s^2*t^2 + (-152*i -
%    28)*s*t^3 + (-7*i + 38)*t^4;
%F;
%A:=38*s^4 + 28*s^3*t -228*s^2*t^2 - 28*s*t^3 + 38*t^4;
%B:=-7*s^4 + 152*s^3*t + 42*s^2*t^2 - 152*s*t^3 - 7*t^4;
%F-A-i*B;

\begin{lstlisting}
P<s,t,u,v>:=ProjectiveSpace(Rationals(),3);
A:=38*s^4 + 28*s^3*t -228*s^2*t^2 - 28*s*t^3 + 38*t^4-2*u^4;
B:=-7*s^4 + 152*s^3*t + 42*s^2*t^2 - 152*s*t^3 - 7*t^4-v^4;
S:=Scheme(P,[A,B]);
IsLocallySolvable(S,2);
\end{lstlisting}

{\bf Case 6:} $2u^4+v^4i=i(38-7i)(s+ti)^4$. Equating the real and imaginary parts on both sides of this equation gives
\begin{equation}\label{1493-6}
\begin{cases}
7s^4 - 152s^3t - 42s^2t^2 + 152st^3 + 7t^4 - 2u^4=0,\\
38s^4 + 28s^3t - 228s^2t^2 - 28st^3 + 38t^4 - v^4=0.
\end{cases}
\end{equation}
The scheme defined by system \eqref{1493-6} is locally insoluble at $2$.
%
%\begin{lstlisting}
%_<x>:=PolynomialRing(Rationals());
%k<i>:=NumberField(x^2+1);
%K<s,t>:=PolynomialRing(k,2);
%e:=1;
%F:=i^e*(38-7*i)*(s+t*i)^4;
%F;
%L<s,t>:=PolynomialRing(Rationals(),2);
%_<i>:=PolynomialRing(L);
%F:=(38*i + 7)*s^4 + (28*i - 152)*s^3*t + (-228*i - 42)*s^2*t^2 + (-28*i +
%   152)*s*t^3 + (38*i + 7)*t^4;
%F;
%A:=7*s^4 - 152*s^3*t -42*s^2*t^2 + 152*s*t^3 + 7*t^4;
%B:=38*s^4 + 28*s^3*t - 228*s^2*t^2 - 28*s*t^3 + 38*t^4;
%F-A-i*B;

\begin{lstlisting}
P<s,t,u,v>:=ProjectiveSpace(Rationals(),3);
A:=7*s^4 - 152*s^3*t -42*s^2*t^2 + 152*s*t^3 + 7*t^4-2*u^4;
B:=38*s^4 + 28*s^3*t - 228*s^2*t^2 - 28*s*t^3 + 38*t^4-v^4;
S:=Scheme(P,[A,B]);
IsLocallySolvable(S,2);
\end{lstlisting}

{\bf Case 7:} $2u^4+v^4i=i^2(38-7i)(s+ti)^4$. Equating the real and imaginary parts on both sides of this equation gives
\begin{equation}\label{1493-7}
\begin{cases}
-38s^4 - 28s^3t + 228s^2t^2 + 28st^3 - 38t^4 - 2u^4=0,\\
7s^4 - 152s^3t - 42s^2t^2 + 152st^3 + 7t^4 - v^4=0.
\end{cases}
\end{equation}
The scheme defined by system \eqref{1493-7} is locally insoluble at $2$.

%\begin{lstlisting}
%_<x>:=PolynomialRing(Rationals());
%k<i>:=NumberField(x^2+1);
%K<s,t>:=PolynomialRing(k,2);
%e:=2;
%F:=i^e*(38-7*i)*(s+t*i)^4;
%F;
%L<s,t>:=PolynomialRing(Rationals(),2);
%_<i>:=PolynomialRing(L);
%F:=(7*i - 38)*s^4 + (-152*i - 28)*s^3*t + (-42*i + 228)*s^2*t^2 + (152*i +
%    28)*s*t^3 + (7*i - 38)*t^4;
%F;
%A:=- 38*s^4 - 28*s^3*t + 228*s^2*t^2 + 28*s*t^3 - 38*t^4;
%B:=7*s^4 - 152*s^3*t - 42*s^2*t^2 + 152*s*t^3 + 7*t^4;
%F-A-i*B;

\begin{lstlisting}
P<s,t,u,v>:=ProjectiveSpace(Rationals(),3);
A:=- 38*s^4 - 28*s^3*t + 228*s^2*t^2 + 28*s*t^3 - 38*t^4-2*u^4;
B:=7*s^4 - 152*s^3*t - 42*s^2*t^2 + 152*s*t^3 + 7*t^4-v^4;
S:=Scheme(P,[A,B]);
IsLocallySolvable(S,2);
\end{lstlisting}

{\bf Case 8:} $2u^4+v^4i=i^3(38-7i)(s+ti)^4$. Equating the real and imaginary parts on both sides of this equation gives
\begin{equation}\label{1493-8}
\begin{cases}
-7s^4 + 152s^3t + 42s^2t^2 - 152st^3 - 7t^4 - 2u^4=0,\\
-38s^4 - 28s^3t + 228s^2t^2 + 28st^3 - 38t^4 - v^4=0.
\end{cases}
\end{equation}
The scheme defined by system \eqref{1493-8} is locally insoluble at $2$.

%\begin{lstlisting}
%_<x>:=PolynomialRing(Rationals());
%k<i>:=NumberField(x^2+1);
%K<s,t>:=PolynomialRing(k,2);
%e:=3;
%F:=i^e*(38-7*i)*(s+t*i)^4;
%F;
%L<s,t>:=PolynomialRing(Rationals(),2);
%_<i>:=PolynomialRing(L);
%F:=(-38*i - 7)*s^4 + (-28*i + 152)*s^3*t + (228*i + 42)*s^2*t^2 + (28*i -
%    152)*s*t^3 + (-38*i - 7)*t^4;
%F;
%A:=- 7*s^4 + 152*s^3*t +42*s^2*t^2 - 152*s*t^3 - 7*t^4;
%B:=-38*s^4 - 28*s^3*t + 228*s^2*t^2 + 28*s*t^3 - 38*t^4;
%F-A-i*B;

\begin{lstlisting}
P<s,t,u,v>:=ProjectiveSpace(Rationals(),3);
A:=- 7*s^4 + 152*s^3*t +42*s^2*t^2 - 152*s*t^3 - 7*t^4-2*u^4;
B:=-38*s^4 - 28*s^3*t + 228*s^2*t^2 + 28*s*t^3 - 38*t^4-v^4;
S:=Scheme(P,[A,B]);
IsLocallySolvable(S,2);
\end{lstlisting}

\subsection{The case of $n=1610$.}%\[n=1610.\] 
According to Table \ref{tb:factorisation}, in the case $n=1610$ it remains to
show that equation
\begin{equation}\label{1610}
25921 u^8 + 400 v^8= w^4
\end{equation}
has no integer solutions satisfying \eqref{eq:cond}, which in this case reduces to \[\gcd(161u,20v)=\gcd(161u,w)=\gcd(20v,w)=1.\] Write \eqref{1610} as
\begin{equation} (161 u^4+20v^4i)(161 u^4-20v^4i)=w^4.
\end{equation}
This implies that there exist integers $s,t$ such that 
\[161 u^4+20v^4i=i^{\epsilon}(s+ti)^4,\]
with $\epsilon \in \{0,1,2,3\}$.

{\bf Case 1:} $161 u^4+20v^4i=(s+ti)^4$. Equating the real and imaginary parts on both sides of this equation gives
\begin{equation}\label{1610-1}
\begin{cases}
s^4 - 6s^2t^2 + t^4-161u^4=0,\\
4s^3t - 4st^3-20v^4=0.
\end{cases}
\end{equation}
The scheme defined by system \eqref{1610-1} is locally insoluble at $17$.
%\begin{lstlisting}
%_<x>:=PolynomialRing(Rationals());
%k<i>:=NumberField(x^2+1);
%K<s,t>:=PolynomialRing(k,2);
%e:=0;
%F:=i^e*(s+t*i)^4;
%F;
%L<s,t>:=PolynomialRing(Rationals(),2);
%_<i>:=PolynomialRing(L);
%F:=s^4 + 4*i*s^3*t - 6*s^2*t^2 - 4*i*s*t^3 + t^4;
%F;
%A:=s^4 - 6*s^2*t^2 + t^4;
%B:=(4*s^3*t - 4*s*t^3);
%F-A-i*B;

\begin{lstlisting}
P<s,t,u,v>:=ProjectiveSpace(Rationals(),3);
A:=s^4 - 6*s^2*t^2 + t^4-161*u^4;
B:=4*s^3*t - 4*s*t^3-20*v^4;
S:=Scheme(P,[A,B]);
IsLocallySolvable(S,17);
\end{lstlisting}

{\bf Case 2:} $161 u^4+20v^4i=i(s+ti)^4$. Equating the real and imaginary parts on both sides of this equation gives
\begin{equation}\label{1610-2}
\begin{cases}
- 4s^3t + 4st^3-161u^4=0,\\
s^4 - 6s^2t^2 + t^4-20v^4=0.
\end{cases}
\end{equation}
The scheme defined by system \eqref{1610-2} is locally insoluble at $2$.

%\begin{lstlisting}
%_<x>:=PolynomialRing(Rationals());
%k<i>:=NumberField(x^2+1);
%K<s,t>:=PolynomialRing(k,2);
%e:=1;
%F:=i^e*(s+t*i)^4;
%F;
%L<s,t>:=PolynomialRing(Rationals(),2);
%_<i>:=PolynomialRing(L);
%F:=i*s^4 - 4*s^3*t - 6*i*s^2*t^2 + 4*s*t^3 + i*t^4;
%F;
%A:=- 4*s^3*t + 4*s*t^3;
%B:=(s^4 - 6*s^2*t^2 + t^4);
%F-A-i*B;

\begin{lstlisting}
P<s,t,u,v>:=ProjectiveSpace(Rationals(),3);
A:=- 4*s^3*t + 4*s*t^3-161*u^4;
B:=s^4 - 6*s^2*t^2 + t^4-20*v^4;
S:=Scheme(P,[A,B]);
IsLocallySolvable(S,2);
\end{lstlisting}

{\bf Case 3:} $161 u^4+20v^4i=i^2(s+ti)^4$. Equating the real and imaginary parts on both sides of this equation gives
\begin{equation}\label{1610-3}
\begin{cases}
- s^4 + 6s^2t^2 - t^4-161u^4=0,\\
-4s^3t + 4st^3-20v^4=0.
\end{cases}
\end{equation}
The scheme defined by system \eqref{1610-3} is locally insoluble at $2$.

%\begin{lstlisting}
%_<x>:=PolynomialRing(Rationals());
%k<i>:=NumberField(x^2+1);
%K<s,t>:=PolynomialRing(k,2);
%e:=2;
%F:=i^e*(s+t*i)^4;
%F;
%L<s,t>:=PolynomialRing(Rationals(),2);
%_<i>:=PolynomialRing(L);
%F:=-s^4 - 4*i*s^3*t + 6*s^2*t^2 + 4*i*s*t^3 - t^4;
%F;
%A:=- s^4 + 6*s^2*t^2 - t^4;
%B:=-4*s^3*t + 4*s*t^3;
%F-A-i*B;

\begin{lstlisting}
P<s,t,u,v>:=ProjectiveSpace(Rationals(),3);
A:=- s^4 + 6*s^2*t^2 - t^4-161*u^4;
B:=-4*s^3*t + 4*s*t^3-20*v^4;
S:=Scheme(P,[A,B]);
IsLocallySolvable(S,2);
\end{lstlisting}

{\bf Case 4:} $161 u^4+20v^4i=i^3(s+ti)^4$. Equating the real and imaginary parts on both sides of this equation gives
\begin{equation}\label{1610-4}
\begin{cases}
4s^3t - 4st^3-161u^4=0,\\
-s^4 + 6s^2t^2 - t^4-20v^4=0.
\end{cases}
\end{equation}
The scheme defined by system \eqref{1610-4} is locally insoluble at $2$.

%\begin{lstlisting}
%_<x>:=PolynomialRing(Rationals());
%k<i>:=NumberField(x^2+1);
%K<s,t>:=PolynomialRing(k,2);
%e:=3;
%F:=i^e*(s+t*i)^4;
%F;
%L<s,t>:=PolynomialRing(Rationals(),2);
%_<i>:=PolynomialRing(L);
%F:=-i*s^4 + 4*s^3*t + 6*i*s^2*t^2 - 4*s*t^3 - i*t^4;
%F;
%A:=4*s^3*t - 4*s*t^3;
%B:=-s^4 + 6*s^2*t^2 - t^4;
%F-A-i*B;

\begin{lstlisting}
P<s,t,u,v>:=ProjectiveSpace(Rationals(),3);
A:= 4*s^3*t - 4*s*t^3-161*u^4;
B:=-s^4 + 6*s^2*t^2 - t^4-20*v^4;
S:=Scheme(P,[A,B]);
IsLocallySolvable(S,2);
\end{lstlisting}

\subsection{The case of $n=1640$.}%\[n=1640.\] 
According to Table \ref{tb:factorisation}, in the case $n=1640$ it remains to
show that equation
\begin{equation}\label{1640}
25 u^8 + v^8= 41w^4
\end{equation}
has no integer solutions satisfying \eqref{eq:cond}, which in this case reduces to
$$
\gcd(5u,v)=\gcd(5u,41w)=\gcd(v,41w)=1.
$$
Analysis modulo $4$ shows that $2\nmid w$, while modulo $16$ analysis shows that $2\nmid u$ and $2|v$.
%Taking mod 4 shows that $2\nmid w$. Taking mod 16 shows that $2\nmid u$ and $2|v$.  
Write \eqref{1640} as
\begin{equation} (5u^4+v^4i)(5u^4-v^4i)=(5+4i)(5-4i)w^4.
\end{equation}
This implies that there exist integers $s,t$ such that 
\[5u^4+v^4i=i^{\epsilon}(5\pm 4i)(s+ti)^4,\]
with $\epsilon \in \{0,1,2,3\}$.

{\bf Case 1:} $5u^4+v^4i=(5+4i)(s+ti)^4$. Equating the real and imaginary parts on both sides of this equation gives
\begin{equation}\label{1640-1}
\begin{cases}
5s^4 - 16s^3t - 30s^2t^2 + 16st^3 + 5t^4 - 5u^4=0,\\
4s^4 + 20s^3t - 24s^2t^2 - 20st^3 + 4t^4 - v^4=0.
\end{cases}
\end{equation}
The scheme defined by system \eqref{1640-1} is locally insoluble at $2$.

%\begin{lstlisting}
%_<x>:=PolynomialRing(Rationals());
%k<i>:=NumberField(x^2+1);
%K<s,t>:=PolynomialRing(k,2);
%e:=0;
%F:=i^e*(5+4*i)*(s+t*i)^4;
%F;
%L<s,t>:=PolynomialRing(Rationals(),2);
%_<i>:=PolynomialRing(L);
%F:=(4*i + 5)*s^4 + (20*i - 16)*s^3*t + (-24*i - 30)*s^2*t^2 + (-20*i + 16)*s*t^3 +
%    (4*i + 5)*t^4;
%F;
%A:=5*s^4 - 16*s^3*t - 30*s^2*t^2 + 16*s*t^3 + 5*t^4;
%B:=4*s^4 + 20*s^3*t - 24*s^2*t^2 - 20*s*t^3 + 4*t^4;
%F-A-i*B;

\begin{lstlisting}
P<s,t,u,v>:=ProjectiveSpace(Rationals(),3);
A:=5*s^4 - 16*s^3*t - 30*s^2*t^2 + 16*s*t^3 + 5*t^4-5*u^4;
B:=4*s^4 + 20*s^3*t - 24*s^2*t^2 - 20*s*t^3 + 4*t^4-v^4;
S:=Scheme(P,[A,B]);
IsLocallySolvable(S,2);
\end{lstlisting}

{\bf Case 2:} $5u^4+v^4i=i(5+4i)(s+ti)^4$. Equating the real and imaginary parts on both sides of this equation gives
\begin{equation}\label{1640-2}
\begin{cases}
-4s^4 - 20s^3t + 24s^2t^2 + 20st^3 - 4t^4 - 5u^4=0,\\
5s^4 - 16s^3t - 30s^2t^2 + 16st^3 + 5t^4 - v^4=0.
\end{cases}
\end{equation}
The scheme defined by system \eqref{1640-2} is locally insoluble at $2$.

%\begin{lstlisting}
%_<x>:=PolynomialRing(Rationals());
%k<i>:=NumberField(x^2+1);
%K<s,t>:=PolynomialRing(k,2);
%e:=1;
%F:=i^e*(5+4*i)*(s+t*i)^4;
%F;
%L<s,t>:=PolynomialRing(Rationals(),2);
%_<i>:=PolynomialRing(L);
%F:=(5*i - 4)*s^4 + (-16*i - 20)*s^3*t + (-30*i + 24)*s^2*t^2 + (16*i + 20)*s*t^3 +
%   (5*i - 4)*t^4;
%F;
%A:=- 4*s^4 - 20*s^3*t +24*s^2*t^2 + 20*s*t^3 - 4*t^4;
%B:=5*s^4 - 16*s^3*t - 30*s^2*t^2 + 16*s*t^3 + 5*t^4;
%F-A-i*B;

\begin{lstlisting}
P<s,t,u,v>:=ProjectiveSpace(Rationals(),3);
A:=- 4*s^4 - 20*s^3*t +24*s^2*t^2 + 20*s*t^3 - 4*t^4-5*u^4;
B:=5*s^4 - 16*s^3*t - 30*s^2*t^2 + 16*s*t^3 + 5*t^4-v^4;
S:=Scheme(P,[A,B]);
IsLocallySolvable(S,2);
\end{lstlisting}

{\bf Case 3:} $5u^4+v^4i=i^2(5+4i)(s+ti)^4$. Equating the real and imaginary parts on both sides of this equation gives
\begin{equation}\label{1640-3}
\begin{cases}
-5s^4 + 16s^3t + 30s^2t^2 - 16st^3 - 5t^4 - 5u^4=0,\\
-4s^4 - 20s^3t + 24s^2t^2 + 20st^3 - 4t^4 - v^4=0.
\end{cases}
\end{equation}
The scheme defined by system \eqref{1640-3} is locally insoluble at $2$.

%\begin{lstlisting}
%_<x>:=PolynomialRing(Rationals());
%k<i>:=NumberField(x^2+1);
%K<s,t>:=PolynomialRing(k,2);
%e:=2;
%F:=i^e*(5+4*i)*(s+t*i)^4;
%F;
%L<s,t>:=PolynomialRing(Rationals(),2);
%_<i>:=PolynomialRing(L);
%F:=(-4*i - 5)*s^4 + (-20*i + 16)*s^3*t + (24*i + 30)*s^2*t^2 + (20*i - 16)*s*t^3 +
%    (-4*i - 5)*t^4;
%F;
%A:=- 5*s^4 + 16*s^3*t +30*s^2*t^2 - 16*s*t^3 - 5*t^4;
%B:=(-4*s^4 - 20*s^3*t + 24*s^2*t^2 + 20*s*t^3 - 4*t^4);
%F-A-i*B;

\begin{lstlisting}
P<s,t,u,v>:=ProjectiveSpace(Rationals(),3);
A:=- 5*s^4 + 16*s^3*t +30*s^2*t^2 - 16*s*t^3 - 5*t^4-5*u^4;
B:=-4*s^4 - 20*s^3*t + 24*s^2*t^2 + 20*s*t^3 - 4*t^4-v^4;
S:=Scheme(P,[A,B]);
IsLocallySolvable(S,2);
\end{lstlisting}

{\bf Case 4:} $5u^4+v^4i=i^3(5+4i)(s+ti)^4$. Equating the real and imaginary parts on both sides of this equation gives
\begin{equation}\label{1640-4}
\begin{cases}
4s^4 + 20s^3t - 24s^2t^2 - 20st^3 + 4t^4 - 5u^4=0,\\
-5s^4 + 16s^3t + 30s^2t^2 - 16st^3 - 5t^4 - v^4=0.
\end{cases}
\end{equation}
The scheme defined by system \eqref{1640-4} is locally insoluble at $2$.

%\begin{lstlisting}
%_<x>:=PolynomialRing(Rationals());
%k<i>:=NumberField(x^2+1);
%K<s,t>:=PolynomialRing(k,2);
%e:=3;
%F:=i^e*(5+4*i)*(s+t*i)^4;
%F;
%L<s,t>:=PolynomialRing(Rationals(),2);
%_<i>:=PolynomialRing(L);
%F:=(-5*i + 4)*s^4 + (16*i + 20)*s^3*t + (30*i - 24)*s^2*t^2 + (-16*i - 20)*s*t^3 +
%    (-5*i + 4)*t^4;
%F;
%A:=4*s^4 + 20*s^3*t -24*s^2*t^2 - 20*s*t^3 + 4*t^4;
%B:=-5*s^4 + 16*s^3*t + 30*s^2*t^2 - 16*s*t^3 - 5*t^4;
%F-A-i*B;

\begin{lstlisting}
P<s,t,u,v>:=ProjectiveSpace(Rationals(),3);
A:=4*s^4 + 20*s^3*t -24*s^2*t^2 - 20*s*t^3 + 4*t^4-5*u^4;
B:=-5*s^4 + 16*s^3*t + 30*s^2*t^2 - 16*s*t^3 - 5*t^4-v^4;
S:=Scheme(P,[A,B]);
IsLocallySolvable(S,2);
\end{lstlisting}
{\bf Case 5:} $5u^4+v^4i=(5-4i)(s+ti)^4$. Equating the real and imaginary parts on both sides of this equation gives
\begin{equation}\label{1640-5}
\begin{cases}
5s^4 + 16s^3t - 30s^2t^2 - 16st^3 + 5t^4 - 5u^4=0,\\
-4s^4 + 20s^3t + 24s^2t^2 - 20st^3 - 4t^4 - v^4=0.
\end{cases}
\end{equation}
The scheme defined by system \eqref{1640-5} is locally insoluble at $2$.

%\begin{lstlisting}
%_<x>:=PolynomialRing(Rationals());
%k<i>:=NumberField(x^2+1);
%K<s,t>:=PolynomialRing(k,2);
%e:=0;
%F:=i^e*(5-4*i)*(s+t*i)^4;
%F;
%L<s,t>:=PolynomialRing(Rationals(),2);
%_<i>:=PolynomialRing(L);
%F:=(-4*i + 5)*s^4 + (20*i + 16)*s^3*t + (24*i - 30)*s^2*t^2 + (-20*i - 16)*s*t^3 +
%    (-4*i + 5)*t^4;
%F;
%A:=5*s^4 + 16*s^3*t - 30*s^2*t^2 - 16*s*t^3 + 5*t^4;
%B:=-4*s^4 + 20*s^3*t + 24*s^2*t^2 - 20*s*t^3 - 4*t^4;
%F-A-i*B;

\begin{lstlisting}
P<s,t,u,v>:=ProjectiveSpace(Rationals(),3);
A:=5*s^4 + 16*s^3*t - 30*s^2*t^2 - 16*s*t^3 + 5*t^4-5*u^4;
B:=-4*s^4 + 20*s^3*t + 24*s^2*t^2 - 20*s*t^3 - 4*t^4-v^4;
S:=Scheme(P,[A,B]);
IsLocallySolvable(S,2);
\end{lstlisting}

{\bf Case 6:} $5u^4+v^4i=i(5-4i)(s+ti)^4$. Equating the real and imaginary parts on both sides of this equation gives
\begin{equation}\label{1640-6}
\begin{cases}
4s^4 - 20s^3t - 24s^2t^2 + 20st^3 + 4t^4 - 5u^4=0,\\
5s^4 + 16s^3t - 30s^2t^2 - 16st^3 + 5t^4 - v^4=0.
\end{cases}
\end{equation}
The scheme defined by system \eqref{1640-6} is locally insoluble at $2$.

%\begin{lstlisting}
%_<x>:=PolynomialRing(Rationals());
%k<i>:=NumberField(x^2+1);
%K<s,t>:=PolynomialRing(k,2);
%e:=1;
%F:=i^e*(5-4*i)*(s+t*i)^4;
%F;
%L<s,t>:=PolynomialRing(Rationals(),2);
%_<i>:=PolynomialRing(L);
%F:=(5*i + 4)*s^4 + (16*i - 20)*s^3*t + (-30*i - 24)*s^2*t^2 + (-16*i + 20)*s*t^3 +
%    (5*i + 4)*t^4;
%F;
%A:=4*s^4 - 20*s^3*t - 24*s^2*t^2 + 20*s*t^3 + 4*t^4;
%B:=5*s^4 + 16*s^3*t - 30*s^2*t^2 - 16*s*t^3 + 5*t^4;
%F-A-i*B;

\begin{lstlisting}
P<s,t,u,v>:=ProjectiveSpace(Rationals(),3);
A:=4*s^4 - 20*s^3*t - 24*s^2*t^2 + 20*s*t^3 + 4*t^4-5*u^4;
B:=5*s^4 + 16*s^3*t - 30*s^2*t^2 - 16*s*t^3 + 5*t^4-v^4;
S:=Scheme(P,[A,B]);
IsLocallySolvable(S,2);
\end{lstlisting}

{\bf Case 7:} $5u^4+v^4i=i^2(5-4i)(s+ti)^4$. Equating the real and imaginary parts on both sides of this equation gives
\begin{equation}\label{1640-7}
\begin{cases}
-5s^4 - 16s^3t + 30s^2t^2 + 16st^3 - 5t^4 - 5u^4=0,\\
4s^4 - 20s^3t - 24s^2t^2 + 20st^3 + 4t^4 - v^4=0.
\end{cases}
\end{equation}
The scheme defined by system \eqref{1640-7} is locally insoluble at $2$.

%\begin{lstlisting}
%_<x>:=PolynomialRing(Rationals());
%k<i>:=NumberField(x^2+1);
%K<s,t>:=PolynomialRing(k,2);
%e:=2;
%F:=i^e*(5-4*i)*(s+t*i)^4;
%F;
%L<s,t>:=PolynomialRing(Rationals(),2);
%_<i>:=PolynomialRing(L);
%F:=(4*i - 5)*s^4 + (-20*i - 16)*s^3*t + (-24*i + 30)*s^2*t^2 + (20*i + 16)*s*t^3 +
%    (4*i - 5)*t^4;
%F;
%A:=- 5*s^4 - 16*s^3*t +30*s^2*t^2 + 16*s*t^3 - 5*t^4;
%B:=4*s^4 - 20*s^3*t - 24*s^2*t^2 + 20*s*t^3 + 4*t^4;
%F-A-i*B;

\begin{lstlisting}
P<s,t,u,v>:=ProjectiveSpace(Rationals(),3);
A:=- 5*s^4 - 16*s^3*t +30*s^2*t^2 + 16*s*t^3 - 5*t^4-5*u^4;
B:=4*s^4 - 20*s^3*t - 24*s^2*t^2 + 20*s*t^3 + 4*t^4-v^4;
S:=Scheme(P,[A,B]);
IsLocallySolvable(S,2);
\end{lstlisting}

{\bf Case 8:} $5u^4+v^4i=i^3(5-4i)(s+ti)^4$. Equating the real and imaginary parts on both sides of this equation gives
\begin{equation}\label{1640-8}
\begin{cases}
-4s^4 + 20s^3t + 24s^2t^2 - 20st^3 - 4t^4 - 5u^4=0,\\
-5s^4 - 16s^3t + 30s^2t^2 + 16st^3 - 5t^4 - v^4=0.
\end{cases}
\end{equation}
The scheme defined by system \eqref{1640-8} is locally insoluble at $2$.

%\begin{lstlisting}
%_<x>:=PolynomialRing(Rationals());
%k<i>:=NumberField(x^2+1);
%K<s,t>:=PolynomialRing(k,2);
%e:=3;
%F:=i^e*(5-4*i)*(s+t*i)^4;
%F;
%L<s,t>:=PolynomialRing(Rationals(),2);
%_<i>:=PolynomialRing(L);
%F:=(-5*i - 4)*s^4 + (-16*i + 20)*s^3*t + (30*i + 24)*s^2*t^2 + (16*i - 20)*s*t^3 +
%    (-5*i - 4)*t^4;
%F;
%A:=- 4*s^4 + 20*s^3*t +24*s^2*t^2 - 20*s*t^3 - 4*t^4;
%B:=-5*s^4 - 16*s^3*t + 30*s^2*t^2 + 16*s*t^3 - 5*t^4;
%F-A-i*B;

\begin{lstlisting}
P<s,t,u,v>:=ProjectiveSpace(Rationals(),3);
A:=- 4*s^4 + 20*s^3*t +24*s^2*t^2 - 20*s*t^3 - 4*t^4-5*u^4;
B:=-5*s^4 - 16*s^3*t + 30*s^2*t^2 + 16*s*t^3 - 5*t^4-v^4;
S:=Scheme(P,[A,B]);
IsLocallySolvable(S,2);
\end{lstlisting}

\subsection{The case of $n=1731$.}%\[n=1731.\]

According to Table \ref{tb:factorisation}, in the case $n=1731$ it remains to
show that equation
\begin{equation}\label{1731}
36 u^8 + v^8=577  w^4
\end{equation}
has no integer solutions satisfying \eqref{eq:cond}, which in this case reduces to \[\gcd(6u,v)=\gcd(6u,577w)=\gcd(v,577w)=1.\]   Write \eqref{1731} as
\begin{equation} (6u^4+v^4i)(6u^4-v^4i)=(24+i)(24-i)w^4.
\end{equation}
This implies that there exist integers $s,t$ such that 
\[6u^4+v^4i=i^{\epsilon}(24\pm i)(s+ti)^4,\]
with $\epsilon \in \{0,1,2,3\}$.

{\bf Case 1:} $6u^4+v^4i=(24+i)(s+ti)^4$. Equating the real and imaginary parts on both sides of this equation gives
\begin{equation}\label{1731-1}
\begin{cases}
24s^4 - 4s^3t - 144s^2t^2 + 4st^3 + 24t^4 - 6u^4=0,\\
s^4 + 96s^3t - 6s^2t^2 - 96st^3 + t^4 - v^4=0.
\end{cases}
\end{equation}
The scheme defined by system \eqref{1731-1} is locally insoluble at $2$.

%\begin{lstlisting}
%_<x>:=PolynomialRing(Rationals());
%k<i>:=NumberField(x^2+1);
%K<s,t>:=PolynomialRing(k,2);
%e:=0;
%F:=i^e*(24+i)*(s+t*i)^4;
%F;
%L<s,t>:=PolynomialRing(Rationals(),2);
%_<i>:=PolynomialRing(L);
%F:=(i + 24)*s^4 + (96*i - 4)*s^3*t + (-6*i - 144)*s^2*t^2 + (-96*i + 4)*s*t^3 + (i
%    + 24)*t^4;
%F;
%A:=24*s^4 - 4*s^3*t - 144*s^2*t^2  + 4*s*t^3 + 24*t^4;
%B:=s^4 + 96*s^3*t - 6*s^2*t^2 - 96*s*t^3 + t^4;
%F-A-i*B;

\begin{lstlisting}
P<s,t,u,v>:=ProjectiveSpace(Rationals(),3);
A:=24*s^4 - 4*s^3*t - 144*s^2*t^2  + 4*s*t^3 + 24*t^4-6*u^4;
B:=s^4 + 96*s^3*t - 6*s^2*t^2 - 96*s*t^3 + t^4-v^4;
S:=Scheme(P,[A,B]);
IsLocallySolvable(S,2);
\end{lstlisting}

{\bf Case 2:} $6u^4+v^4i=i(24+i)(s+ti)^4$. Equating the real and imaginary parts on both sides of this equation gives
\begin{equation}\label{1731-2}
\begin{cases}
-s^4 - 96s^3t + 6s^2t^2 + 96st^3 - t^4 - 6u^4=0,\\
24s^4 - 4s^3t - 144s^2t^2 + 4st^3 + 24t^4 - v^4=0.
\end{cases}
\end{equation}
The scheme defined by system \eqref{1731-2} is locally insoluble at $2$.

%\begin{lstlisting}
%_<x>:=PolynomialRing(Rationals());
%k<i>:=NumberField(x^2+1);
%K<s,t>:=PolynomialRing(k,2);
%e:=1;
%F:=i^e*(24+i)*(s+t*i)^4;
%F;
%L<s,t>:=PolynomialRing(Rationals(),2);
%_<i>:=PolynomialRing(L);
%F:=(24*i - 1)*s^4 + (-4*i - 96)*s^3*t + (-144*i + 6)*s^2*t^2 + (4*i + 96)*s*t^3 +
%    (24*i - 1)*t^4;
%F;
%A:=- s^4 - 96*s^3*t +6*s^2*t^2 + 96*s*t^3 - t^4;
%B:=24*s^4 - 4*s^3*t - 144*s^2*t^2 + 4*s*t^3 + 24*t^4;
%F-A-i*B;

\begin{lstlisting}
P<s,t,u,v>:=ProjectiveSpace(Rationals(),3);
A:=- s^4 - 96*s^3*t +6*s^2*t^2 + 96*s*t^3 - t^4-6*u^4;
B:=24*s^4 - 4*s^3*t - 144*s^2*t^2 + 4*s*t^3 + 24*t^4-v^4;
S:=Scheme(P,[A,B]);
IsLocallySolvable(S,2);
\end{lstlisting}

{\bf Case 3:} $6u^4+v^4i=i^2(24+i)(s+ti)^4$. Equating the real and imaginary parts on both sides of this equation gives
\begin{equation}\label{1731-3}
\begin{cases}
-24s^4 + 4s^3t + 144s^2t^2 - 4st^3 - 24t^4 - 6u^4=0,\\
-s^4 - 96s^3t + 6s^2t^2 + 96st^3 - t^4 - v^4=0.
\end{cases}
\end{equation}
The scheme defined by system \eqref{1731-3} is locally insoluble at $2$.
%\begin{lstlisting}
%_<x>:=PolynomialRing(Rationals());
%k<i>:=NumberField(x^2+1);
%K<s,t>:=PolynomialRing(k,2);
%e:=2;
%F:=i^e*(24+i)*(s+t*i)^4;
%F;
%L<s,t>:=PolynomialRing(Rationals(),2);
%_<i>:=PolynomialRing(L);
%F:=(-i - 24)*s^4 + (-96*i + 4)*s^3*t + (6*i + 144)*s^2*t^2 + (96*i - 4)*s*t^3 + (-i
%   - 24)*t^4;
%F;
%A:=- 24*s^4 + 4*s^3*t +  144*s^2*t^2 - 4*s*t^3 - 24*t^4;
%B:=-s^4 - 96*s^3*t + 6*s^2*t^2 + 96*s*t^3 - t^4;
%F-A-i*B;

\begin{lstlisting}
P<s,t,u,v>:=ProjectiveSpace(Rationals(),3);
A:=- 24*s^4 + 4*s^3*t +  144*s^2*t^2 - 4*s*t^3 - 24*t^4-6*u^4;
B:=-s^4 - 96*s^3*t + 6*s^2*t^2 + 96*s*t^3 - t^4-v^4;
S:=Scheme(P,[A,B]);
IsLocallySolvable(S,2);
\end{lstlisting}

{\bf Case 4:} $6u^4+v^4i=i^3(24+i)(s+ti)^4$. Equating the real and imaginary parts on both sides of this equation gives
\begin{equation}\label{1731-4}
\begin{cases}
s^4 + 96s^3t - 6s^2t^2 - 96st^3 + t^4 - 6u^4=0,\\
-24s^4 + 4s^3t + 144s^2t^2 - 4st^3 - 24t^4 - v^4=0.
\end{cases}
\end{equation}
The scheme defined by system \eqref{1731-4} is locally insoluble at $2$.

%\begin{lstlisting}
%_<x>:=PolynomialRing(Rationals());
%k<i>:=NumberField(x^2+1);
%K<s,t>:=PolynomialRing(k,2);
%e:=3;
%F:=i^e*(24+i)*(s+t*i)^4;
%F;
%L<s,t>:=PolynomialRing(Rationals(),2);
%_<i>:=PolynomialRing(L);
%F:=(-24*i + 1)*s^4 + (4*i + 96)*s^3*t + (144*i - 6)*s^2*t^2 + (-4*i - 96)*s*t^3 +
%    (-24*i + 1)*t^4;
%F;
%A:=s^4 + 96*s^3*t - 6*s^2*t^2 - 96*s*t^3 + t^4;
%B:=-24*s^4 + 4*s^3*t + 144*s^2*t^2 - 4*s*t^3 - 24*t^4;
%F-A-i*B;

\begin{lstlisting}
P<s,t,u,v>:=ProjectiveSpace(Rationals(),3);
A:=s^4 + 96*s^3*t - 6*s^2*t^2 - 96*s*t^3 + t^4-6*u^4;
B:=-24*s^4 + 4*s^3*t + 144*s^2*t^2 - 4*s*t^3 - 24*t^4-v^4;
S:=Scheme(P,[A,B]);
IsLocallySolvable(S,2);
\end{lstlisting}

{\bf Case 5:} $6u^4+v^4i=(24-i)(s+ti)^4$. Equating the real and imaginary parts on both sides of this equation gives
\begin{equation}\label{1731-5}
\begin{cases}
24s^4 + 4s^3t - 144s^2t^2 - 4st^3 + 24t^4 - 6u^4=0,\\
-s^4 + 96s^3t + 6s^2t^2 - 96st^3 - t^4 - v^4=0.
\end{cases}
\end{equation}
The scheme defined by system \eqref{1731-5} is locally insoluble at $2$.

%\begin{lstlisting}
%_<x>:=PolynomialRing(Rationals());
%k<i>:=NumberField(x^2+1);
%K<s,t>:=PolynomialRing(k,2);
%e:=0;
%F:=i^e*(24-i)*(s+t*i)^4;
%F;
%L<s,t>:=PolynomialRing(Rationals(),2);
%_<i>:=PolynomialRing(L);
%F:=(-i + 24)*s^4 + (96*i + 4)*s^3*t + (6*i - 144)*s^2*t^2 + (-96*i - 4)*s*t^3 + (-i
%    + 24)*t^4;
%F;
%A:=24*s^4 + 4*s^3*t -  144*s^2*t^2 - 4*s*t^3 + 24*t^4;
%B:=-s^4 + 96*s^3*t + 6*s^2*t^2 - 96*s*t^3 - t^4;
%F-A-i*B;

\begin{lstlisting}
P<s,t,u,v>:=ProjectiveSpace(Rationals(),3);
A:=24*s^4 + 4*s^3*t -  144*s^2*t^2 - 4*s*t^3 + 24*t^4-6*u^4;
B:=-s^4 + 96*s^3*t + 6*s^2*t^2 - 96*s*t^3 - t^4-v^4;
S:=Scheme(P,[A,B]);
IsLocallySolvable(S,2);
\end{lstlisting}

{\bf Case 6:} $6u^4+v^4i=i(24-i)(s+ti)^4$. Equating the real and imaginary parts on both sides of this equation gives
\begin{equation}\label{1731-6}
\begin{cases}
s^4 - 96s^3t - 6s^2t^2 + 96st^3 + t^4 - 6u^4=0,\\
24s^4 + 4s^3t - 144s^2t^2 - 4st^3 + 24t^4 - v^4=0,
\end{cases}
\end{equation}
The scheme defined by system \eqref{1731-6} is locally insoluble at $2$.

%\begin{lstlisting}
%_<x>:=PolynomialRing(Rationals());
%k<i>:=NumberField(x^2+1);
%K<s,t>:=PolynomialRing(k,2);
%e:=1;
%F:=i^e*(24-i)*(s+t*i)^4;
%F;
%L<s,t>:=PolynomialRing(Rationals(),2);
%_<i>:=PolynomialRing(L);
%F:=(24*i + 1)*s^4 + (4*i - 96)*s^3*t + (-144*i - 6)*s^2*t^2 + (-4*i + 96)*s*t^3 +
%    (24*i + 1)*t^4;
%F;
%A:=s^4 - 96*s^3*t -6*s^2*t^2 + 96*s*t^3 + t^4;
%B:=24*s^4 + 4*s^3*t - 144*s^2*t^2 - 4*s*t^3 + 24*t^4;
%F-A-i*B;

\begin{lstlisting}
P<s,t,u,v>:=ProjectiveSpace(Rationals(),3);
A:=s^4 - 96*s^3*t -6*s^2*t^2 + 96*s*t^3 + t^4-6*u^4;
B:=24*s^4 + 4*s^3*t - 144*s^2*t^2 - 4*s*t^3 + 24*t^4-v^4;
S:=Scheme(P,[A,B]);
IsLocallySolvable(S,2);
\end{lstlisting}

{\bf Case 7:} $6u^4+v^4i=i^2(24-i)(s+ti)^4$. Equating the real and imaginary parts on both sides of this equation gives
\begin{equation}\label{1731-7}
\begin{cases}
-24s^4 - 4s^3t + 144s^2t^2 + 4st^3 - 24t^4 - 6u^4=0,\\
s^4 - 96s^3t - 6s^2t^2 + 96st^3 + t^4 - v^4=0
\end{cases}
\end{equation}
The scheme defined by system \eqref{1731-7} is locally insoluble at $2$.

%\begin{lstlisting}
%_<x>:=PolynomialRing(Rationals());
%k<i>:=NumberField(x^2+1);
%K<s,t>:=PolynomialRing(k,2);
%e:=2;
%F:=i^e*(24-i)*(s+t*i)^4;
%F;
%L<s,t>:=PolynomialRing(Rationals(),2);
%_<i>:=PolynomialRing(L);
%F:=(i - 24)*s^4 + (-96*i - 4)*s^3*t + (-6*i + 144)*s^2*t^2 + (96*i + 4)*s*t^3 + (i- 24)*t^4;
%F;
%A:=- 24*s^4 - 4*s^3*t + 144*s^2*t^2 + 4*s*t^3 - 24*t^4;
%B:=s^4 - 96*s^3*t - 6*s^2*t^2 + 96*s*t^3 + t^4;
%F-A-i*B;

\begin{lstlisting}
P<s,t,u,v>:=ProjectiveSpace(Rationals(),3);
A:=- 24*s^4 - 4*s^3*t + 144*s^2*t^2 + 4*s*t^3 - 24*t^4-6*u^4;
B:=s^4 - 96*s^3*t - 6*s^2*t^2 + 96*s*t^3 + t^4-v^4;
S:=Scheme(P,[A,B]);
IsLocallySolvable(S,2);
\end{lstlisting}

{\bf Case 8:} $6u^4+v^4i=i^3(24-i)(s+ti)^4$. Equating the real and imaginary parts on both sides of this equation gives
\begin{equation}\label{1731-8}
\begin{cases}
-s^4 + 96s^3t + 6s^2t^2 - 96st^3 - t^4 - 6u^4=0,\\
-24s^4 - 4s^3t + 144s^2t^2 + 4st^3 - 24t^4 - v^4=0.
\end{cases}
\end{equation}
The scheme defined by system \eqref{1731-8} is locally insoluble at $2$.

%\begin{lstlisting}
%_<x>:=PolynomialRing(Rationals());
%k<i>:=NumberField(x^2+1);
%K<s,t>:=PolynomialRing(k,2);
%e:=3;
%F:=i^e*(24-i)*(s+t*i)^4;
%F;
%L<s,t>:=PolynomialRing(Rationals(),2);
%_<i>:=PolynomialRing(L);
%F:=(-24*i - 1)*s^4 + (-4*i + 96)*s^3*t + (144*i + 6)*s^2*t^2 + (4*i - 96)*s*t^3 +
%    (-24*i - 1)*t^4;
%F;
%A:=- s^4 + 96*s^3*t +6*s^2*t^2 - 96*s*t^3 - t^4;
%B:=-24*s^4 - 4*s^3*t + 144*s^2*t^2 + 4*s*t^3 - 24*t^4;
%F-A-i*B;

\begin{lstlisting}
P<s,t,u,v>:=ProjectiveSpace(Rationals(),3);
A:=- s^4 + 96*s^3*t +6*s^2*t^2 - 96*s*t^3 - t^4-6*u^4;
B:=-24*s^4 - 4*s^3*t + 144*s^2*t^2 + 4*s*t^3 - 24*t^4-v^4;
S:=Scheme(P,[A,B]);
IsLocallySolvable(S,2);
\end{lstlisting}

\subsection{The case of $n=2035$.}%\[n=2035.\]

According to Table \ref{tb:factorisation}, in the case $n=2035$ it remains to
show that equation
\begin{equation}\label{2035}
121 u^8 + 4 v^8=185  w^4
\end{equation}
has no integer solutions satisfying \eqref{eq:cond}, which in this case reduces to \[\gcd(11u,2v)=\gcd(11u,185w)=\gcd(2v,185w)=1.\]   Write \eqref{2035} as
\begin{equation} (11u^4+2v^4i)(11u^4-2v^4i)=(1+2i)(1-2i)(6+i)(6-i)w^4.
\end{equation}
Since $5\nmid u,v$, we have \[11u^4+2v^4i\equiv 1+2i\pmod{5}.\]
Thus $1+2i|11u^4+2v^4i$. This implies that there exist integers $s,t$ such that 
\[11u^4+2v^4i=i^{\epsilon}(1+2i)(6\pm i)(s+ti)^4,\]
with $\epsilon \in \{0,1,2,3\}$.

{\bf Case 1:} $11u^4+2v^4i=(1+2i)(6+i)(s+ti)^4$. Equating the real and imaginary parts on both sides of this equation gives
\begin{equation}\label{2035-1}
\begin{cases}
4s^4 - 52s^3t - 24s^2t^2 + 52st^3 + 4t^4 - 11u^4=0,\\
13s^4 + 16s^3t - 78s^2t^2 - 16st^3 + 13t^4 - 2v^4=0.
\end{cases}
\end{equation}
The scheme defined by system \eqref{2035-1} is locally insoluble at $2$.

%\begin{lstlisting}
%_<x>:=PolynomialRing(Rationals());
%k<i>:=NumberField(x^2+1);
%K<s,t>:=PolynomialRing(k,2);
%e:=0;
%F:=i^e*(1+2*i)*(6+i)*(s+t*i)^4;
%F;
%L<s,t>:=PolynomialRing(Rationals(),2);
%_<i>:=PolynomialRing(L);
%F:=(13*i + 4)*s^4 + (16*i - 52)*s^3*t + (-78*i - 24)*s^2*t^2 + (-16*i + 52)*s*t^3 +
%    (13*i + 4)*t^4;
%F;
%A:=4*s^4 - 52*s^3*t - 24*s^2*t^2 + 52*s*t^3 + 4*t^4;
%B:=13*s^4 + 16*s^3*t - 78*s^2*t^2 - 16*s*t^3 + 13*t^4;
%F-A-i*B;

\begin{lstlisting}
P<s,t,u,v>:=ProjectiveSpace(Rationals(),3);
A:=4*s^4 - 52*s^3*t - 24*s^2*t^2 + 52*s*t^3 + 4*t^4-11*u^4;
B:=13*s^4 + 16*s^3*t - 78*s^2*t^2 - 16*s*t^3 + 13*t^4-2*v^4;
S:=Scheme(P,[A,B]);
IsLocallySolvable(S,2);
\end{lstlisting}

{\bf Case 2:} $11u^4+2v^4i=i(1+2i)(6+i)(s+ti)^4$. Equating the real and imaginary parts on both sides of this equation gives
\begin{equation}\label{2035-2}
\begin{cases}
-13s^4 - 16s^3t + 78s^2t^2 + 16st^3 - 13t^4 - 11u^4=0,\\
4s^4 - 52s^3t - 24s^2t^2 + 52st^3 + 4t^4 - 2v^4=0.
\end{cases}
\end{equation}
The scheme defined by system \eqref{2035-2} is locally insoluble at $2$.

%\begin{lstlisting}
%_<x>:=PolynomialRing(Rationals());
%k<i>:=NumberField(x^2+1);
%K<s,t>:=PolynomialRing(k,2);
%e:=1;
%F:=i^e*(1+2*i)*(6+i)*(s+t*i)^4;
%F;
%L<s,t>:=PolynomialRing(Rationals(),2);
%_<i>:=PolynomialRing(L);
%F:=(4*i - 13)*s^4 + (-52*i - 16)*s^3*t + (-24*i + 78)*s^2*t^2 + (52*i + 16)*s*t^3 +
%    (4*i - 13)*t^4;
%F;
%A:=- 13*s^4 - 16*s^3*t +78*s^2*t^2 + 16*s*t^3 - 13*t^4;
%B:=4*s^4 - 52*s^3*t - 24*s^2*t^2 + 52*s*t^3 + 4*t^4;
%F-A-i*B;

\begin{lstlisting}
P<s,t,u,v>:=ProjectiveSpace(Rationals(),3);
A:=- 13*s^4 - 16*s^3*t +78*s^2*t^2 + 16*s*t^3 - 13*t^4-11*u^4;
B:=4*s^4 - 52*s^3*t - 24*s^2*t^2 + 52*s*t^3 + 4*t^4-2*v^4;
S:=Scheme(P,[A,B]);
IsLocallySolvable(S,2);
\end{lstlisting}

{\bf Case 3:} $11u^4+2v^4i=i^2(1+2i)(6+i)(s+ti)^4$. Equating the real and imaginary parts on both sides of this equation gives
\begin{equation}\label{2035-3}
\begin{cases}
-4s^4 + 52s^3t + 24s^2t^2 - 52st^3 - 4t^4 - 11u^4=0,\\
-13s^4 - 16s^3t + 78s^2t^2 + 16st^3 - 13t^4 - 2v^4=0.
\end{cases}
\end{equation}
The scheme defined by system \eqref{2035-3} is locally insoluble at $2$.

%\begin{lstlisting}
%_<x>:=PolynomialRing(Rationals());
%k<i>:=NumberField(x^2+1);
%K<s,t>:=PolynomialRing(k,2);
%e:=2;
%F:=i^e*(1+2*i)*(6+i)*(s+t*i)^4;
%F;
%L<s,t>:=PolynomialRing(Rationals(),2);
%_<i>:=PolynomialRing(L);
%F:=(-13*i - 4)*s^4 + (-16*i + 52)*s^3*t + (78*i + 24)*s^2*t^2 + (16*i - 52)*s*t^3 +
%    (-13*i - 4)*t^4;
%F;
%A:=- 4*s^4 + 52*s^3*t +24*s^2*t^2 - 52*s*t^3 - 4*t^4;
%B:=-13*s^4 - 16*s^3*t + 78*s^2*t^2 + 16*s*t^3 - 13*t^4;
%F-A-i*B;

\begin{lstlisting}
P<s,t,u,v>:=ProjectiveSpace(Rationals(),3);
A:=- 4*s^4 + 52*s^3*t +24*s^2*t^2 - 52*s*t^3 - 4*t^4-11*u^4;
B:=-13*s^4 - 16*s^3*t + 78*s^2*t^2 + 16*s*t^3 - 13*t^4-2*v^4;
S:=Scheme(P,[A,B]);
IsLocallySolvable(S,2);
\end{lstlisting}
{\bf Case 4:} $11u^4+2v^4i=i^3(1+2i)(6+i)(s+ti)^4$. Equating the real and imaginary parts on both sides of this equation gives
\begin{equation}\label{2035-4}
\begin{cases}
13s^4 + 16s^3t - 78s^2t^2 - 16st^3 + 13t^4 - 11u^4=0,\\
-4s^4 + 52s^3t + 24s^2t^2 - 52st^3 - 4t^4 - 2v^4=0.
\end{cases}
\end{equation}
The scheme defined by system \eqref{2035-4} is locally insoluble at $2$.

%\begin{lstlisting}
%_<x>:=PolynomialRing(Rationals());
%k<i>:=NumberField(x^2+1);
%K<s,t>:=PolynomialRing(k,2);
%e:=3;
%F:=i^e*(1+2*i)*(6+i)*(s+t*i)^4;
%F;
%L<s,t>:=PolynomialRing(Rationals(),2);
%_<i>:=PolynomialRing(L);
%F:=(-4*i + 13)*s^4 + (52*i + 16)*s^3*t + (24*i - 78)*s^2*t^2 + (-52*i - 16)*s*t^3 +
%    (-4*i + 13)*t^4;
%F;
%A:=13*s^4 + 16*s^3*t -78*s^2*t^2 - 16*s*t^3 + 13*t^4;
%B:=-4*s^4 + 52*s^3*t + 24*s^2*t^2 - 52*s*t^3 - 4*t^4;
%F-A-i*B;

\begin{lstlisting}
P<s,t,u,v>:=ProjectiveSpace(Rationals(),3);
A:=13*s^4 + 16*s^3*t -78*s^2*t^2 - 16*s*t^3 + 13*t^4-11*u^4;
B:=-4*s^4 + 52*s^3*t + 24*s^2*t^2 - 52*s*t^3 - 4*t^4-2*v^4;
S:=Scheme(P,[A,B]);
IsLocallySolvable(S,2);
\end{lstlisting}

{\bf Case 5:} $11u^4+2v^4i=(1+2i)(6-i)(s+ti)^4$. Equating the real and imaginary parts on both sides of this equation gives
\begin{equation}\label{2035-5}
\begin{cases}
8s^4 - 44s^3t - 48s^2t^2 + 44st^3 + 8t^4 - 11u^4=0,\\
11s^4 + 32s^3t - 66s^2t^2 - 32st^3 + 11t^4 - 2v^4=0.
\end{cases}
\end{equation}
The scheme defined by system \eqref{2035-5} is locally insoluble at $2$.

%\begin{lstlisting}
%_<x>:=PolynomialRing(Rationals());
%k<i>:=NumberField(x^2+1);
%K<s,t>:=PolynomialRing(k,2);
%e:=0;
%F:=i^e*(1+2*i)*(6-i)*(s+t*i)^4;
%F;
%L<s,t>:=PolynomialRing(Rationals(),2);
%_<i>:=PolynomialRing(L);
%F:=(11*i + 8)*s^4 + (32*i - 44)*s^3*t + (-66*i - 48)*s^2*t^2 + (-32*i + 44)*s*t^3 +
%    (11*i + 8)*t^4;
%F;
%A:= 8*s^4 - 44*s^3*t -48*s^2*t^2 + 44*s*t^3 + 8*t^4;
%B:=11*s^4 + 32*s^3*t - 66*s^2*t^2 - 32*s*t^3 + 11*t^4;
%F-A-i*B;

\begin{lstlisting}
P<s,t,u,v>:=ProjectiveSpace(Rationals(),3);
A:=8*s^4 - 44*s^3*t -48*s^2*t^2 + 44*s*t^3 + 8*t^4-11*u^4;
B:=11*s^4 + 32*s^3*t - 66*s^2*t^2 - 32*s*t^3 + 11*t^4-2*v^4;
S:=Scheme(P,[A,B]);
IsLocallySolvable(S,2);
\end{lstlisting}
{\bf Case 6:} $11u^4+2v^4i=i(1+2i)(6-i)(s+ti)^4$. Equating the real and imaginary parts on both sides of this equation gives
\begin{equation}\label{2035-6}
\begin{cases}
-11s^4 - 32s^3t + 66s^2t^2 + 32st^3 - 11t^4 - 11u^4=0,\\
8s^4 - 44s^3t - 48s^2t^2 + 44st^3 + 8t^4 - 2v^4=0,
\end{cases}
\end{equation}
The scheme defined by system \eqref{2035-6} is locally insoluble at $2$.

%\begin{lstlisting}
%_<x>:=PolynomialRing(Rationals());
%k<i>:=NumberField(x^2+1);
%K<s,t>:=PolynomialRing(k,2);
%e:=1;
%F:=i^e*(1+2*i)*(6-i)*(s+t*i)^4;
%F;
%L<s,t>:=PolynomialRing(Rationals(),2);
%_<i>:=PolynomialRing(L);
%F:=(8*i - 11)*s^4 + (-44*i - 32)*s^3*t + (-48*i + 66)*s^2*t^2 + (44*i + 32)*s*t^3 +
%    (8*i - 11)*t^4;
%F;
%A:=- 11*s^4 - 32*s^3*t +   66*s^2*t^2 + 32*s*t^3 - 11*t^4;
%B:=8*s^4 - 44*s^3*t - 48*s^2*t^2 + 44*s*t^3 + 8*t^4;
%F-A-i*B;
\begin{lstlisting}
P<s,t,u,v>:=ProjectiveSpace(Rationals(),3);
A:=- 11*s^4 - 32*s^3*t +   66*s^2*t^2 + 32*s*t^3 - 11*t^4-11*u^4;
B:=8*s^4 - 44*s^3*t - 48*s^2*t^2 + 44*s*t^3 + 8*t^4-2*v^4;
S:=Scheme(P,[A,B]);
IsLocallySolvable(S,2);
\end{lstlisting}

{\bf Case 7:} $11u^4+2v^4i=i^2(1+2i)(6-i)(s+ti)^4$. Equating the real and imaginary parts on both sides of this equation gives
\begin{equation}\label{2035-7}
\begin{cases}
-8s^4 + 44s^3t + 48s^2t^2 - 44st^3 - 8t^4 - 11u^4=0,\\
-11s^4 - 32s^3t + 66s^2t^2 + 32st^3 - 11t^4 - 2v^4=0.
\end{cases}
\end{equation}
The scheme defined by system \eqref{2035-7} is locally insoluble at $2$.

%\begin{lstlisting}
%_<x>:=PolynomialRing(Rationals());
%k<i>:=NumberField(x^2+1);
%K<s,t>:=PolynomialRing(k,2);
%e:=2;
%F:=i^e*(1+2*i)*(6-i)*(s+t*i)^4;
%F;
%L<s,t>:=PolynomialRing(Rationals(),2);
%_<i>:=PolynomialRing(L);
%F:=(-11*i - 8)*s^4 + (-32*i + 44)*s^3*t + (66*i + 48)*s^2*t^2 + (32*i - 44)*s*t^3 +
%    (-11*i - 8)*t^4;
%F;
%A:=- 8*s^4 + 44*s^3*t +48*s^2*t^2 - 44*s*t^3 - 8*t^4;
%B:=-11*s^4 - 32*s^3*t + 66*s^2*t^2 + 32*s*t^3 - 11*t^4;
%F-A-i*B;

\begin{lstlisting}
P<s,t,u,v>:=ProjectiveSpace(Rationals(),3);
A:=- 8*s^4 + 44*s^3*t +48*s^2*t^2 - 44*s*t^3 - 8*t^4-11*u^4;
B:=-11*s^4 - 32*s^3*t + 66*s^2*t^2 + 32*s*t^3 - 11*t^4-2*v^4;
S:=Scheme(P,[A,B]);
IsLocallySolvable(S,2);
\end{lstlisting}
{\bf Case 8:} $11u^4+2v^4i=i^3(1+2i)(6-i)(s+ti)^4$. Equating the real and imaginary parts on both sides of this equation gives
\begin{equation}\label{2035-8}
\begin{cases}
11s^4 + 32s^3t - 66s^2t^2 - 32st^3 + 11t^4 - 11u^4=0,\\
-8s^4 + 44s^3t + 48s^2t^2 - 44st^3 - 8t^4 - 2v^4=0.
\end{cases}
\end{equation}
The scheme defined by system \eqref{2035-8} is locally insoluble at $2$.

%\begin{lstlisting}
%_<x>:=PolynomialRing(Rationals());
%k<i>:=NumberField(x^2+1);
%K<s,t>:=PolynomialRing(k,2);
%e:=3;
%F:=i^e*(1+2*i)*(6-i)*(s+t*i)^4;
%F;
%L<s,t>:=PolynomialRing(Rationals(),2);
%_<i>:=PolynomialRing(L);
%F:=(-8*i + 11)*s^4 + (44*i + 32)*s^3*t + (48*i - 66)*s^2*t^2 + (-44*i - 32)*s*t^3 +
%    (-8*i + 11)*t^4;
%F;
%A:=11*s^4 + 32*s^3*t -66*s^2*t^2 - 32*s*t^3 + 11*t^4;
%B:=-8*s^4 + 44*s^3*t + 48*s^2*t^2 - 44*s*t^3 - 8*t^4;
%F-A-i*B;

\begin{lstlisting}
P<s,t,u,v>:=ProjectiveSpace(Rationals(),3);
A:=11*s^4 + 32*s^3*t -66*s^2*t^2 - 32*s*t^3 + 11*t^4-11*u^4;
B:=-8*s^4 + 44*s^3*t + 48*s^2*t^2 - 44*s*t^3 - 8*t^4-2*v^4;
S:=Scheme(P,[A,B]);
IsLocallySolvable(S,2);
\end{lstlisting}

\subsection{The case of $n=2213$.}%\[n=2213.\]
According to Table \ref{tb:factorisation}, in the case $n=2213$ it remains to
show that equation 
\begin{equation}\label{2213}
4 u^8+v^8=2213  w^4
\end{equation}
has no integer solutions satisfying \eqref{eq:cond}, which in this case reduces to \[\gcd(2u,v)=\gcd(2u,2213w)=\gcd(v,2213w)=1.\]   Write \eqref{2213} as
\begin{equation} (2u^4+v^4i)(2u^4-v^4i)=(47+2i)(47-2i)w^4.
\end{equation}
This implies that there exist integers $s,t$ such that 
\[2u^4+v^4i=i^{\epsilon}(47\pm 2i)(s+ti)^4,\]
with $\epsilon \in \{0,1,2,3\}$.

{\bf Case 1:} $2u^4+v^4i=(47+2i)(s+ti)^4$. Equating the real and imaginary parts on both sides of this equation gives
\begin{equation}\label{2213-1}
\begin{cases}
47s^4 - 8s^3t - 282s^2t^2 + 8st^3 + 47t^4 - 2u^4=0,\\
2s^4 + 188s^3t - 12s^2t^2 - 188st^3 + 2t^4 - v^4=0.
\end{cases}
\end{equation}
The scheme defined by system \eqref{2213-1} is locally insoluble at $2$.

%\begin{lstlisting}
%_<x>:=PolynomialRing(Rationals());
%k<i>:=NumberField(x^2+1);
%K<s,t>:=PolynomialRing(k,2);
%e:=0;
%F:=i^e*(47+2*i)*(s+t*i)^4;
%F;
%L<s,t>:=PolynomialRing(Rationals(),2);
%_<i>:=PolynomialRing(L);
%F:=(2*i + 47)*s^4 + (188*i - 8)*s^3*t + (-12*i - 282)*s^2*t^2 + (-188*i + 8)*s*t^3
%    + (2*i + 47)*t^4;
%F;
%A:= 47*s^4 - 8*s^3*t -282*s^2*t^2 + 8*s*t^3 + 47*t^4;
%B:=2*s^4 + 188*s^3*t - 12*s^2*t^2 - 188*s*t^3 + 2*t^4;
%F-A-i*B;

\begin{lstlisting}
P<s,t,u,v>:=ProjectiveSpace(Rationals(),3);
A:= 47*s^4 - 8*s^3*t -282*s^2*t^2 + 8*s*t^3 + 47*t^4-2*u^4;
B:=2*s^4 + 188*s^3*t - 12*s^2*t^2 - 188*s*t^3 + 2*t^4-v^4;
S:=Scheme(P,[A,B]);
IsLocallySolvable(S,2);
\end{lstlisting}

{\bf Case 2:} $2u^4+v^4i=i(47+2i)(s+ti)^4$. Equating the real and imaginary parts on both sides of this equation gives
\begin{equation}\label{2213-2}
\begin{cases}
-2s^4 - 188s^3t + 12s^2t^2 + 188st^3 - 2t^4 - 2u^4=0,\\
47s^4 - 8s^3t - 282s^2t^2 + 8st^3 + 47t^4 - v^4=0.
\end{cases}
\end{equation}
The scheme defined by system \eqref{2213-2} is locally insoluble at $2$.

%\begin{lstlisting}
%_<x>:=PolynomialRing(Rationals());
%k<i>:=NumberField(x^2+1);
%K<s,t>:=PolynomialRing(k,2);
%e:=1;
%F:=i^e*(47+2*i)*(s+t*i)^4;
%F;
%L<s,t>:=PolynomialRing(Rationals(),2);
%_<i>:=PolynomialRing(L);
%F:=(47*i - 2)*s^4 + (-8*i - 188)*s^3*t + (-282*i + 12)*s^2*t^2 + (8*i + 188)*s*t^3
%    + (47*i - 2)*t^4;
%F;
%A:= - 2*s^4 - 188*s^3*t +12*s^2*t^2 + 188*s*t^3 - 2*t^4;
%B:=47*s^4 - 8*s^3*t - 282*s^2*t^2 + 8*s*t^3 + 47*t^4;
%F-A-i*B;

\begin{lstlisting}
P<s,t,u,v>:=ProjectiveSpace(Rationals(),3);
A:=  - 2*s^4 - 188*s^3*t +12*s^2*t^2 + 188*s*t^3 - 2*t^4-2*u^4;
B:=47*s^4 - 8*s^3*t - 282*s^2*t^2 + 8*s*t^3 + 47*t^4-v^4;
S:=Scheme(P,[A,B]);
IsLocallySolvable(S,2);
\end{lstlisting}

{\bf Case 3:} $2u^4+v^4i=i^2(47+2i)(s+ti)^4$. Equating the real and imaginary parts on both sides of this equation gives
\begin{equation}\label{2213-3}
\begin{cases}
-47s^4 + 8s^3t + 282s^2t^2 - 8st^3 - 47t^4 - 2u^4=0,\\
-2s^4 - 188s^3t + 12s^2t^2 + 188st^3 - 2t^4 - v^4=0.
\end{cases}
\end{equation}
The scheme defined by system \eqref{2213-3} is locally insoluble at $2$.

%\begin{lstlisting}
%_<x>:=PolynomialRing(Rationals());
%k<i>:=NumberField(x^2+1);
%K<s,t>:=PolynomialRing(k,2);
%e:=2;
%F:=i^e*(47+2*i)*(s+t*i)^4;
%F;
%L<s,t>:=PolynomialRing(Rationals(),2);
%_<i>:=PolynomialRing(L);
%F:=(-2*i - 47)*s^4 + (-188*i + 8)*s^3*t + (12*i + 282)*s^2*t^2 + (188*i - 8)*s*t^3
%    + (-2*i - 47)*t^4;
%F;
%A:=- 47*s^4 + 8*s^3*t +
%    282*s^2*t^2 - 8*s*t^3 - 47*t^4;
%B:=-2*s^4 - 188*s^3*t + 12*s^2*t^2 + 188*s*t^3 - 2*t^4;
%F-A-i*B;
\begin{lstlisting}
P<s,t,u,v>:=ProjectiveSpace(Rationals(),3);
A:=- 47*s^4 + 8*s^3*t +
    282*s^2*t^2 - 8*s*t^3 - 47*t^4-2*u^4;
B:=-2*s^4 - 188*s^3*t + 12*s^2*t^2 + 188*s*t^3 - 2*t^4-v^4;
S:=Scheme(P,[A,B]);
IsLocallySolvable(S,2);
\end{lstlisting}

{\bf Case 4:} $2u^4+v^4i=i^3(47+2i)(s+ti)^4$. Equating the real and imaginary parts on both sides of this equation gives
\begin{equation}\label{2213-4}
\begin{cases}
2s^4 + 188s^3t - 12s^2t^2 - 188st^3 + 2t^4-2u^4=0,\\
-47s^4 + 8s^3t + 282s^2t^2 - 8st^3 - 47t^4-v^4=0.
\end{cases}
\end{equation}
The scheme defined by system \eqref{2213-4} is locally insoluble at $5$.

%\begin{lstlisting}
%_<x>:=PolynomialRing(Rationals());
%k<i>:=NumberField(x^2+1);
%K<s,t>:=PolynomialRing(k,2);
%e:=3;
%F:=i^e*(47+2*i)*(s+t*i)^4;
%F;
%L<s,t>:=PolynomialRing(Rationals(),2);
%_<i>:=PolynomialRing(L);
%F:=(-47*i + 2)*s^4 + (8*i + 188)*s^3*t + (282*i - 12)*s^2*t^2 + (-8*i - 188)*s*t^3
%    + (-47*i + 2)*t^4;
%F;
%A:=2*s^4 + 188*s^3*t - 12*s^2*t^2 - 188*s*t^3 + 2*t^4;
%B:=-47*s^4 + 8*s^3*t + 282*s^2*t^2 - 8*s*t^3 - 47*t^4;
%F-A-i*B;

\begin{lstlisting}
P<s,t,u,v>:=ProjectiveSpace(Rationals(),3);
A:=2*s^4 + 188*s^3*t - 12*s^2*t^2 - 188*s*t^3 + 2*t^4-2*u^4;
B:=-47*s^4 + 8*s^3*t + 282*s^2*t^2 - 8*s*t^3 - 47*t^4-v^4;
S:=Scheme(P,[A,B]);
IsLocallySolvable(S,5);
\end{lstlisting}

{\bf Case 5:} $2u^4+v^4i=(47-2i)(s+ti)^4$. Equating the real and imaginary parts on both sides of this equation gives
\begin{equation}\label{2213-5}
\begin{cases}
47s^4 + 8s^3t - 282s^2t^2 - 8st^3 + 47t^4 - 2u^4=0,\\
-2s^4 + 188s^3t + 12s^2t^2 - 188st^3 - 2t^4 - v^4=0.
\end{cases}
\end{equation}
The scheme defined by system \eqref{2213-5} is locally insoluble at $2$.

%\begin{lstlisting}
%_<x>:=PolynomialRing(Rationals());
%k<i>:=NumberField(x^2+1);
%K<s,t>:=PolynomialRing(k,2);
%e:=0;
%F:=i^e*(47-2*i)*(s+t*i)^4;
%F;
%L<s,t>:=PolynomialRing(Rationals(),2);
%_<i>:=PolynomialRing(L);
%F:=(-2*i + 47)*s^4 + (188*i + 8)*s^3*t + (12*i - 282)*s^2*t^2 + (-188*i - 8)*s*t^3
%    + (-2*i + 47)*t^4;
%F;
%A:=47*s^4 + 8*s^3*t -
%    282*s^2*t^2 - 8*s*t^3 + 47*t^4;
%B:=-2*s^4 + 188*s^3*t + 12*s^2*t^2 - 188*s*t^3 - 2*t^4;
%F-A-i*B;

\begin{lstlisting}
P<s,t,u,v>:=ProjectiveSpace(Rationals(),3);
A:= 47*s^4 + 8*s^3*t -
    282*s^2*t^2 - 8*s*t^3 + 47*t^4-2*u^4;
B:=-2*s^4 + 188*s^3*t + 12*s^2*t^2 - 188*s*t^3 - 2*t^4-v^4;
S:=Scheme(P,[A,B]);
IsLocallySolvable(S,2);
\end{lstlisting}

{\bf Case 6:} $2u^4+v^4i=i(47-2i)(s+ti)^4$. Equating the real and imaginary parts on both sides of this equation gives
\begin{equation}\label{2213-6}
\begin{cases}
2s^4 - 188s^3t - 12s^2t^2 + 188st^3 + 2t^4 - 2u^4=0,\\
47s^4 + 8s^3t - 282s^2t^2 - 8st^3 + 47t^4 - v^4=0.
\end{cases}
\end{equation}
The scheme defined by system \eqref{2213-6} is locally insoluble at $2$.

%\begin{lstlisting}
%_<x>:=PolynomialRing(Rationals());
%k<i>:=NumberField(x^2+1);
%K<s,t>:=PolynomialRing(k,2);
%e:=1;
%F:=i^e*(47-2*i)*(s+t*i)^4;
%F;
%L<s,t>:=PolynomialRing(Rationals(),2);
%_<i>:=PolynomialRing(L);
%F:=(47*i + 2)*s^4 + (8*i - 188)*s^3*t + (-282*i - 12)*s^2*t^2 + (-8*i + 188)*s*t^3
%    + (47*i + 2)*t^4;
%F;
%A:= 2*s^4 - 188*s^3*t -
%    12*s^2*t^2 + 188*s*t^3 + 2*t^4;
%B:=47*s^4 + 8*s^3*t - 282*s^2*t^2 - 8*s*t^3 + 47*t^4;
%F-A-i*B;

\begin{lstlisting}
P<s,t,u,v>:=ProjectiveSpace(Rationals(),3);
A:= 2*s^4 - 188*s^3*t -
    12*s^2*t^2 + 188*s*t^3 + 2*t^4-2*u^4;
B:=47*s^4 + 8*s^3*t - 282*s^2*t^2 - 8*s*t^3 + 47*t^4-v^4;
S:=Scheme(P,[A,B]);
IsLocallySolvable(S,2);
\end{lstlisting}

{\bf Case 7:} $2u^4+v^4i=i^2(47-2i)(s+ti)^4$. Equating the real and imaginary parts on both sides of this equation gives
\begin{equation}\label{2213-7}
\begin{cases}
-47s^4 - 8s^3t + 282s^2t^2 + 8st^3 - 47t^4 - 2u^4=0,\\
2s^4 - 188s^3t - 12s^2t^2 + 188st^3 + 2t^4 - v^4=0.
\end{cases}
\end{equation}
The scheme defined by system \eqref{2213-7} is locally insoluble at $2$.

%\begin{lstlisting}
%_<x>:=PolynomialRing(Rationals());
%k<i>:=NumberField(x^2+1);
%K<s,t>:=PolynomialRing(k,2);
%e:=2;
%F:=i^e*(47-2*i)*(s+t*i)^4;
%F;
%L<s,t>:=PolynomialRing(Rationals(),2);
%_<i>:=PolynomialRing(L);
%F:=(2*i - 47)*s^4 + (-188*i - 8)*s^3*t + (-12*i + 282)*s^2*t^2 + (188*i + 8)*s*t^3
%    + (2*i - 47)*t^4;
%F;
%A:=- 47*s^4 - 8*s^3*t + 282*s^2*t^2 + 8*s*t^3 - 47*t^4;
%B:=2*s^4 - 188*s^3*t - 12*s^2*t^2 + 188*s*t^3 + 2*t^4;
%F-A-i*B;

\begin{lstlisting}
P<s,t,u,v>:=ProjectiveSpace(Rationals(),3);
A:=- 47*s^4 - 8*s^3*t + 282*s^2*t^2 + 8*s*t^3 - 47*t^4-2*u^4;
B:=2*s^4 - 188*s^3*t - 12*s^2*t^2 + 188*s*t^3 + 2*t^4-v^4;
S:=Scheme(P,[A,B]);
IsLocallySolvable(S,2);
\end{lstlisting}

{\bf Case 8:} $2u^4+v^4i=i^3(47-2i)(s+ti)^4$. Equating the real and imaginary parts on both sides of this equation gives
\begin{equation}\label{2213-8}
\begin{cases}
-2s^4 + 188s^3t + 12s^2t^2 - 188st^3 - 2t^4 - 2u^4=0,\\
-47s^4 - 8s^3t + 282s^2t^2 + 8st^3 - 47t^4 - v^4=0.
\end{cases}
\end{equation}
The scheme defined by system \eqref{2213-8} is locally insoluble at $2$.

%\begin{lstlisting}
%_<x>:=PolynomialRing(Rationals());
%k<i>:=NumberField(x^2+1);
%K<s,t>:=PolynomialRing(k,2);
%e:=3;
%F:=i^e*(47-2*i)*(s+t*i)^4;
%F;
%L<s,t>:=PolynomialRing(Rationals(),2);
%_<i>:=PolynomialRing(L);
%F:=(-47*i - 2)*s^4 + (-8*i + 188)*s^3*t + (282*i + 12)*s^2*t^2 + (8*i - 188)*s*t^3
%    + (-47*i - 2)*t^4;
%F;
%A:=- 2*s^4 + 188*s^3*t +
%    12*s^2*t^2 - 188*s*t^3 - 2*t^4;
%B:=-47*s^4 - 8*s^3*t + 282*s^2*t^2 + 8*s*t^3 - 47*t^4;
%F-A-i*B;

\begin{lstlisting}
P<s,t,u,v>:=ProjectiveSpace(Rationals(),3);
A:=- 2*s^4 + 188*s^3*t +   12*s^2*t^2 - 188*s*t^3 - 2*t^4-2*u^4;
B:=-47*s^4 - 8*s^3*t + 282*s^2*t^2 + 8*s*t^3 - 47*t^4-v^4;
S:=Scheme(P,[A,B]);
IsLocallySolvable(S,2);
\end{lstlisting}

\subsection{The case of $n=2719$.}%\[n= 2719.\]
According to Table \ref{tb:factorisation}, in the case $n=2719$ it remains to
show that equation 
\begin{equation}\label{2719}
7392961 u^8 + v^8 =2  w^4
\end{equation}
has no integer solutions satisfying \eqref{eq:cond}, which in this case reduces to
\begin{equation}\label{2719:cond}
\gcd(2719u,v)=\gcd(2719u,2w)=\gcd(v,2w)=1.
\end{equation}  Write \eqref{2719} as
\begin{equation} (2719u^4+v^4i)(2719u^4-v^4i)=(1+i)(1-i)w^4.
\end{equation}
This implies that there exist integers $s,t$ such that 
\[2719u^4+v^4i=i^{\epsilon}(s+ti)^4,\]
with $\epsilon \in \{0,1,2,3\}$.

{\bf Case 1:} $2719u^4+v^4i=(s+ti)^4$. Equating the real and imaginary parts on both sides of this equation gives
\begin{equation}\label{2719-1}
\begin{cases}
s^4 - 6s^2t^2 + t^4 - 2719u^4=0,\\
4s^3t - 4st^3 - v^4=0.
\end{cases}
\end{equation}
The scheme defined by system \eqref{2719-1} is locally insoluble at $2$.

%\begin{lstlisting}
%_<x>:=PolynomialRing(Rationals());
%k<i>:=NumberField(x^2+1);
%K<s,t>:=PolynomialRing(k,2);
%e:=0;
%F:=i^e*(s+t*i)^4;
%F;
%L<s,t>:=PolynomialRing(Rationals(),2);
%_<i>:=PolynomialRing(L);
%F:=s^4 + 4*i*s^3*t - 6*s^2*t^2 - 4*i*s*t^3 + t^4;
%F;
%A:= s^4 - 6*s^2*t^2 + t^4;
%B:=4*s^3*t - 4*s*t^3;
%F-A-i*B;
\begin{lstlisting}
P<s,t,u,v>:=ProjectiveSpace(Rationals(),3);
A:= s^4 - 6*s^2*t^2 + t^4-2719*u^4;
B:=4*s^3*t - 4*s*t^3-v^4;
S:=Scheme(P,[A,B]);
IsLocallySolvable(S,2);
\end{lstlisting}

{\bf Case 2:}  $2719u^4+v^4i=i(s+ti)^4$. Equating the real and imaginary parts on both sides of this equation gives
\begin{equation}\label{2719-2}
\begin{cases}
-4s^3t + 4st^3 - 2719u^4=0,\\
s^4 - 6s^2t^2 + t^4 - v^4=0.
\end{cases}
\end{equation}
%The scheme defined by system \eqref{2719-2} is locally solvable everywhere, but since t
The equation $X^4-6X^2Y^2+Y^4=Z^2$ has no integer solutions $(X,Y,Z)$ with $XY\neq 0$ (see Mordell \cite[Theorem~3,page 18]{Mordell}. It follows from this and the second equation in \eqref{2719-2} that $st=0$. Then the first equation forces $u=0$, which is not allowed. 

%\begin{lstlisting}
%_<x>:=PolynomialRing(Rationals());
%k<i>:=NumberField(x^2+1);
%K<s,t>:=PolynomialRing(k,2);
%e:=1;
%F:=i^e*(s+t*i)^4;
%F;
%L<s,t>:=PolynomialRing(Rationals(),2);
%_<i>:=PolynomialRing(L);
%F:=i*s^4 - 4*s^3*t - 6*i*s^2*t^2 + 4*s*t^3 + i*t^4;
%F;
%A:=- 4*s^3*t + 4*s*t^3;
%B:=s^4 - 6*s^2*t^2 + t^4;
%F-A-i*B;
%\end{lstlisting}

{\bf Case 3:}  $2719u^4+v^4i=i^2(s+ti)^4$. Equating the real and imaginary parts on both sides of this equation gives
\begin{equation}\label{2719-3}
\begin{cases}
-s^4 + 6s^2t^2 - t^4 - 2719u^4=0,\\
-4s^3t + 4st^3 - v^4=0
\end{cases}
\end{equation}
%The scheme defined by system \eqref{2719-3} is locally everywhere, but i
Modulo $2$ analysis on the second equation returns that any integer solution must have $v$ even, however this is a contradiction with \eqref{2719:cond}.
%It follows from the second equation that $4st(t^2-s^2)=v^4$. Hence, $(X,Y)=(t/s,v^2/2s^2)$ satisfies \[Y^2=X(X^2-1).\]
%Magma shows that the only rational points on this curve are $(0,0)$ and $(0,\pm 1)$. It follows that $t=0$. Hence, $v=0$, which is not allowed. 
%\begin{lstlisting}
%_<x>:=PolynomialRing(Rationals());
%k<i>:=NumberField(x^2+1);
%K<s,t>:=PolynomialRing(k,2);
%e:=2;
%F:=i^e*(s+t*i)^4;
%F;
%L<s,t>:=PolynomialRing(Rationals(),2);
%_<i>:=PolynomialRing(L);
%F:=-s^4 - 4*i*s^3*t + 6*s^2*t^2 + 4*i*s*t^3 - t^4;
%F;
%A:=- s^4 + 6*s^2*t^2 - t^4;
%B:=-4*s^3*t + 4*s*t^3;
%F-A-i*B;
%\end{lstlisting}

{\bf Case 4:}  $2719u^4+v^4i=i^3(s+ti)^4$. Equating the real and imaginary parts on both sides of this equation gives
\begin{equation}\label{2719-4}
\begin{cases}
4s^3t - 4st^3 - 2719u^4=0,\\
-s^4 + 6s^2t^2 - t^4 - v^4=0.
\end{cases}
\end{equation}
The scheme defined by system \eqref{2719-4} is locally insoluble at $2$.

%\begin{lstlisting}
%_<x>:=PolynomialRing(Rationals());
%k<i>:=NumberField(x^2+1);
%K<s,t>:=PolynomialRing(k,2);
%e:=3;
%F:=i^e*(s+t*i)^4;
%F;
%L<s,t>:=PolynomialRing(Rationals(),2);
%_<i>:=PolynomialRing(L);
%F:=-i*s^4 + 4*s^3*t + 6*i*s^2*t^2 - 4*s*t^3 - i*t^4;
%F;
%A:=4*s^3*t - 4*s*t^3;
%B:=-s^4 + 6*s^2*t^2 - t^4;
%F-A-i*B;

\begin{lstlisting}
P<s,t,u,v>:=ProjectiveSpace(Rationals(),3);
A:=4*s^3*t - 4*s*t^3-2719*u^4;
B:=-s^4 + 6*s^2*t^2 - t^4- v^4;
S:=Scheme(P,[A,B]);
IsLocallySolvable(S,2);
\end{lstlisting}

\subsection{The case of $n=2810$.}%\[n= 2810.\]
According to Table \ref{tb:factorisation}, in the case $n=2810$ it remains to
show that equation 
\begin{equation}\label{2810}
25 u^8 + 16 v^8=281  w^4
\end{equation}
has no integer solutions satisfying \eqref{eq:cond}, which in this case reduces to \[\gcd(5u,4v)=\gcd(5u,281w)=\gcd(4v,281w)=1.\]   Write \eqref{2810} as
\begin{equation} (5u^4+4v^4i)(5u^4-4v^4i)=(16+5i)(16-5i)w^4.
\end{equation}
This implies that there exist integers $s,t$ such that 
\[5u^4+4v^4i=i^{\epsilon}(16\pm 5i)(s+ti)^4,\]
with $\epsilon \in \{0,1,2,3\}$.

{\bf Case 1:} $5u^4+4v^4i=(16+5i)(s+ti)^4$. Equating the real and imaginary parts on both sides of this equation gives
\begin{equation}\label{2810-1}
\begin{cases}
16s^4 - 20s^3t - 96s^2t^2 + 20st^3 + 16t^4 - 5u^4=0,\\
5s^4 + 64s^3t - 30s^2t^2 - 64st^3 + 5t^4 - 4v^4=0
\end{cases}
\end{equation}
The scheme defined by system \eqref{2810-1} is locally insoluble at $2$.
%\begin{lstlisting}
%_<x>:=PolynomialRing(Rationals());
%k<i>:=NumberField(x^2+1);
%K<s,t>:=PolynomialRing(k,2);
%e:=0;
%F:=i^e*(16+5*i)*(s+t*i)^4;
%F;
%L<s,t>:=PolynomialRing(Rationals(),2);
%_<i>:=PolynomialRing(L);
%F:=(5*i + 16)*s^4 + (64*i - 20)*s^3*t + (-30*i - 96)*s^2*t^2 + (-64*i + 20)*s*t^3 +
%    (5*i + 16)*t^4;
%F;
%A:=16*s^4 - 20*s^3*t - 96*s^2*t^2 + 20*s*t^3 + 16*t^4;
%B:=5*s^4 + 64*s^3*t - 30*s^2*t^2 - 64*s*t^3 + 5*t^4;
%F-A-i*B;

\begin{lstlisting}
P<s,t,u,v>:=ProjectiveSpace(Rationals(),3);
A:=16*s^4 - 20*s^3*t - 96*s^2*t^2 + 20*s*t^3 + 16*t^4-5*u^4;
B:=5*s^4 + 64*s^3*t - 30*s^2*t^2 - 64*s*t^3 + 5*t^4-4*v^4;
S:=Scheme(P,[A,B]);
IsLocallySolvable(S,2);
\end{lstlisting}

{\bf Case 2:} $5u^4+4v^4i=i(16+5i)(s+ti)^4$. Equating the real and imaginary parts on both sides of this equation gives
\begin{equation}\label{2810-2}
\begin{cases}
-5s^4 - 64s^3t + 30s^2t^2 + 64st^3 - 5t^4 - 5u^4=0,\\
16s^4 - 20s^3t - 96s^2t^2 + 20st^3 + 16t^4 - 4v^4=0.
\end{cases}
\end{equation}
The scheme defined by system \eqref{2810-2} is locally insoluble at $2$.

%\begin{lstlisting}
%_<x>:=PolynomialRing(Rationals());
%k<i>:=NumberField(x^2+1);
%K<s,t>:=PolynomialRing(k,2);
%e:=1;
%F:=i^e*(16+5*i)*(s+t*i)^4;
%F;
%L<s,t>:=PolynomialRing(Rationals(),2);
%_<i>:=PolynomialRing(L);
%F:=(16*i - 5)*s^4 + (-20*i - 64)*s^3*t + (-96*i + 30)*s^2*t^2 + (20*i + 64)*s*t^3 +
%    (16*i - 5)*t^4;
%F;
%A:=- 5*s^4 - 64*s^3*t +
%    30*s^2*t^2 + 64*s*t^3 - 5*t^4;
%B:=16*s^4 - 20*s^3*t - 96*s^2*t^2 + 20*s*t^3 + 16*t^4;
%F-A-i*B;

\begin{lstlisting}
P<s,t,u,v>:=ProjectiveSpace(Rationals(),3);
A:= - 5*s^4 - 64*s^3*t + 30*s^2*t^2 + 64*s*t^3 - 5*t^4-5*u^4;
B:=16*s^4 - 20*s^3*t - 96*s^2*t^2 + 20*s*t^3 + 16*t^4-4*v^4;
S:=Scheme(P,[A,B]);
IsLocallySolvable(S,2);
\end{lstlisting}

{\bf Case 3:} $5u^4+4v^4i=i^{2}(16+5i)(s+ti)^4$. Equating the real and imaginary parts on both sides of this equation gives
\begin{equation}\label{2810-3}
\begin{cases}
-16s^4 + 20s^3t + 96s^2t^2 - 20st^3 - 16t^4 - 5u^4=0,\\
-5s^4 - 64s^3t + 30s^2t^2 + 64st^3 - 5t^4 - 4v^4=0.
\end{cases}
\end{equation}
The scheme defined by system \eqref{2810-3} is locally insoluble at $2$.
%\begin{lstlisting}
%_<x>:=PolynomialRing(Rationals());
%k<i>:=NumberField(x^2+1);
%K<s,t>:=PolynomialRing(k,2);
%e:=2;
%F:=i^e*(16+5*i)*(s+t*i)^4;
%F;
%L<s,t>:=PolynomialRing(Rationals(),2);
%_<i>:=PolynomialRing(L);
%F:=(-5*i - 16)*s^4 + (-64*i + 20)*s^3*t + (30*i + 96)*s^2*t^2 + (64*i - 20)*s*t^3 +
%    (-5*i - 16)*t^4;
%F;
%A:=- 16*s^4 + 20*s^3*t +96*s^2*t^2 - 20*s*t^3 - 16*t^4;
%B:=-5*s^4 - 64*s^3*t + 30*s^2*t^2 + 64*s*t^3 - 5*t^4;
%F-A-i*B;

\begin{lstlisting}
P<s,t,u,v>:=ProjectiveSpace(Rationals(),3);
A:=- 16*s^4 + 20*s^3*t +96*s^2*t^2 - 20*s*t^3 - 16*t^4-5*u^4;
B:=-5*s^4 - 64*s^3*t + 30*s^2*t^2 + 64*s*t^3 - 5*t^4-4*v^4;
S:=Scheme(P,[A,B]);
IsLocallySolvable(S,2);
\end{lstlisting}

{\bf Case 4:} $5u^4+4v^4i=i^{3}(16+5i)(s+ti)^4$. Equating the real and imaginary parts on both sides of this equation gives
\begin{equation}\label{2810-4}
\begin{cases}
5s^4 + 64s^3t - 30s^2t^2 - 64st^3 + 5t^4 - 5u^4=0,\\
-16s^4 + 20s^3t + 96s^2t^2 - 20st^3 - 16t^4 - 4v^4=0.
\end{cases}
\end{equation}
The scheme defined by system \eqref{2810-4} is locally insoluble at $3$.

%\begin{lstlisting}
%_<x>:=PolynomialRing(Rationals());
%k<i>:=NumberField(x^2+1);
%K<s,t>:=PolynomialRing(k,2);
%e:=3;
%F:=i^e*(16+5*i)*(s+t*i)^4;
%F;
%L<s,t>:=PolynomialRing(Rationals(),2);
%_<i>:=PolynomialRing(L);
%F:=(-16*i + 5)*s^4 + (20*i + 64)*s^3*t + (96*i - 30)*s^2*t^2 + (-20*i - 64)*s*t^3 +
%    (-16*i + 5)*t^4;
%F;
%A:=5*s^4 + 64*s^3*t - 30*s^2*t^2 - 64*s*t^3 + 5*t^4;
%B:=-16*s^4 + 20*s^3*t + 96*s^2*t^2 - 20*s*t^3 - 16*t^4;
%F-A-i*B;

\begin{lstlisting}
P<s,t,u,v>:=ProjectiveSpace(Rationals(),3);
A:=5*s^4 + 64*s^3*t - 30*s^2*t^2 - 64*s*t^3 + 5*t^4-5*u^4;
B:=-16*s^4 + 20*s^3*t + 96*s^2*t^2 - 20*s*t^3 - 16*t^4-4*v^4;
S:=Scheme(P,[A,B]);
IsLocallySolvable(S,3);
\end{lstlisting}

{\bf Case 5:} $5u^4+4v^4i=(16-5i)(s+ti)^4$. Equating the real and imaginary parts on both sides of this equation gives
\begin{equation}\label{2810-5}
\begin{cases}
16s^4 + 20s^3t - 96s^2t^2 - 20st^3 + 16t^4 - 5u^4=0,\\
-5s^4 + 64s^3t + 30s^2t^2 - 64st^3 - 5t^4 - 4v^4=0.
\end{cases}
\end{equation}
The scheme defined by system \eqref{2810-5} is locally insoluble at $2$.

%\begin{lstlisting}
%_<x>:=PolynomialRing(Rationals());
%k<i>:=NumberField(x^2+1);
%K<s,t>:=PolynomialRing(k,2);
%e:=0;
%F:=i^e*(16-5*i)*(s+t*i)^4;
%F;
%L<s,t>:=PolynomialRing(Rationals(),2);
%_<i>:=PolynomialRing(L);
%F:=(-5*i + 16)*s^4 + (64*i + 20)*s^3*t + (30*i - 96)*s^2*t^2 + (-64*i - 20)*s*t^3 +
%    (-5*i + 16)*t^4;
%F;
%A:=16*s^4 + 20*s^3*t -96*s^2*t^2 - 20*s*t^3 + 16*t^4;
%B:=-5*s^4 + 64*s^3*t + 30*s^2*t^2 - 64*s*t^3 - 5*t^4;
%F-A-i*B;

\begin{lstlisting}
P<s,t,u,v>:=ProjectiveSpace(Rationals(),3);
A:=16*s^4 + 20*s^3*t -96*s^2*t^2 - 20*s*t^3 + 16*t^4-5*u^4;
B:=-5*s^4 + 64*s^3*t + 30*s^2*t^2 - 64*s*t^3 - 5*t^4-4*v^4;
S:=Scheme(P,[A,B]);
IsLocallySolvable(S,2);
\end{lstlisting}
{\bf Case 6:} $5u^4+4v^4i=i(16-5i)(s+ti)^4$. Equating the real and imaginary parts on both sides of this equation gives

\begin{equation}\label{2810-6}
\begin{cases}
5s^4 - 64s^3t - 30s^2t^2 + 64st^3 + 5t^4 - 5u^4=0,\\
16s^4 + 20s^3t - 96s^2t^2 - 20st^3 + 16t^4 - 4v^4=0.
\end{cases}
\end{equation}
The scheme defined by system \eqref{2810-6} is locally insoluble at $5$.

%\begin{lstlisting}
%_<x>:=PolynomialRing(Rationals());
%k<i>:=NumberField(x^2+1);
%K<s,t>:=PolynomialRing(k,2);
%e:=1;
%F:=i^e*(16-5*i)*(s+t*i)^4;
%F;
%L<s,t>:=PolynomialRing(Rationals(),2);
%_<i>:=PolynomialRing(L);
%F:=(16*i + 5)*s^4 + (20*i - 64)*s^3*t + (-96*i - 30)*s^2*t^2 + (-20*i + 64)*s*t^3 +
%    (16*i + 5)*t^4;
%F;
%A:= 5*s^4 - 64*s^3*t -  30*s^2*t^2 + 64*s*t^3 + 5*t^4;
%B:=16*s^4 + 20*s^3*t - 96*s^2*t^2 - 20*s*t^3 + 16*t^4;
%F-A-i*B;

\begin{lstlisting}
P<s,t,u,v>:=ProjectiveSpace(Rationals(),3);
A:=5*s^4 - 64*s^3*t -  30*s^2*t^2 + 64*s*t^3 + 5*t^4-5*u^4;
B:=16*s^4 + 20*s^3*t - 96*s^2*t^2 - 20*s*t^3 + 16*t^4-4*v^4;
S:=Scheme(P,[A,B]);
IsLocallySolvable(S,5);
\end{lstlisting}

{\bf Case 7:} $5u^4+4v^4i=i^{2}(16-5i)(s+ti)^4$. Equating the real and imaginary parts on both sides of this equation gives
\begin{equation}\label{2810-7}
\begin{cases}
-16s^4 - 20s^3t + 96s^2t^2 + 20st^3 - 16t^4 - 5u^4=0,\\
5s^4 - 64s^3t - 30s^2t^2 + 64st^3 + 5t^4 - 4v^4=0.
\end{cases}
\end{equation}
The scheme defined by system \eqref{2810-3} is locally insoluble at $2$.

%\begin{lstlisting}
%_<x>:=PolynomialRing(Rationals());
%k<i>:=NumberField(x^2+1);
%K<s,t>:=PolynomialRing(k,2);
%e:=2;
%F:=i^e*(16-5*i)*(s+t*i)^4;
%F;
%L<s,t>:=PolynomialRing(Rationals(),2);
%_<i>:=PolynomialRing(L);
%F:=(5*i - 16)*s^4 + (-64*i - 20)*s^3*t + (-30*i + 96)*s^2*t^2 + (64*i + 20)*s*t^3 +    (5*i - 16)*t^4;
%F;
%A:=- 16*s^4 - 20*s^3*t +96*s^2*t^2 + 20*s*t^3 - 16*t^4;
%B:=5*s^4 - 64*s^3*t - 30*s^2*t^2 + 64*s*t^3 + 5*t^4;
%F-A-i*B;

\begin{lstlisting}
P<s,t,u,v>:=ProjectiveSpace(Rationals(),3);
A:=- 16*s^4 - 20*s^3*t +96*s^2*t^2 + 20*s*t^3 - 16*t^4-5*u^4;
B:=5*s^4 - 64*s^3*t - 30*s^2*t^2 + 64*s*t^3 + 5*t^4-4*v^4;
S:=Scheme(P,[A,B]);
IsLocallySolvable(S,2);
\end{lstlisting}

{\bf Case 8:} $5u^4+4v^4i=i^{3}(16-5i)(s+ti)^4$. Equating the real and imaginary parts on both sides of this equation gives
\begin{equation}\label{2810-8}
\begin{cases}
-5s^4 + 64s^3t + 30s^2t^2 - 64st^3 - 5t^4 - 5u^4=0,\\
-16s^4 - 20s^3t + 96s^2t^2 + 20st^3 - 16t^4 - 4v^4=0.
\end{cases}
\end{equation}
The scheme defined by system \eqref{2810-8} is locally insoluble at $2$.

%\begin{lstlisting}
%_<x>:=PolynomialRing(Rationals());
%k<i>:=NumberField(x^2+1);
%K<s,t>:=PolynomialRing(k,2);
%e:=3;
%F:=i^e*(16-5*i)*(s+t*i)^4;
%F;
%L<s,t>:=PolynomialRing(Rationals(),2);
%_<i>:=PolynomialRing(L);
%F:=(-16*i - 5)*s^4 + (-20*i + 64)*s^3*t + (96*i + 30)*s^2*t^2 + (20*i - 64)*s*t^3 +
%    (-16*i - 5)*t^4;
%F;
%A:=- 5*s^4 + 64*s^3*t +  30*s^2*t^2 - 64*s*t^3 - 5*t^4;
%B:=-16*s^4 - 20*s^3*t + 96*s^2*t^2 + 20*s*t^3 - 16*t^4;
%F-A-i*B;

\begin{lstlisting}
P<s,t,u,v>:=ProjectiveSpace(Rationals(),3);
A:=- 5*s^4 + 64*s^3*t +  30*s^2*t^2 - 64*s*t^3 - 5*t^4-5*u^4;
B:=-16*s^4 - 20*s^3*t + 96*s^2*t^2 + 20*s*t^3 - 16*t^4-4*v^4;
S:=Scheme(P,[A,B]);
IsLocallySolvable(S,2);
\end{lstlisting}

\subsection{The case of $n=3185$.}
%\[n=3185 \quad (I)\]

According to Table \ref{tb:factorisation}, in the case $n=3185$ it remains to
show that equation
\begin{equation}\label{3185.1}
2401 u^8 + 4 v^8= 65  w^4 
\end{equation}
has no integer solutions satisfying \eqref{eq:cond}, which in this case reduces to \[\gcd(7u,2v)=\gcd(7u,65w)=\gcd(2v,65w)=1,\]  and to show that equation
\begin{equation}\label{3185.2}
9604 u^8 + v^8=65  w^4
\end{equation}
has no integer solutions satisfying \eqref{eq:cond}, which in this case reduces to \[\gcd(98u,v)=\gcd(98u,65w)=\gcd(v,65w)=1.\]

\subsubsection{The case of \eqref{3185.1}.}
We first consider \eqref{3185.1}, which we can write as
\begin{equation} (49u^4+2v^4i)(49u^4-2v^4i)=(1+2i)(1-2i)(3+2i)(3-2i)w^4.
\end{equation}
Since $5\nmid u,v$, we have
\[49u^4+2v^4i\equiv -1+2i\pmod{5}.\]
Hence, $1-2i|49u^4+2v^4i$. This implies that there exist integers $s,t$ such that 
\[49u^4+2v^4i=i^{\epsilon}(1-2i)(3\pm 2i)(s+ti)^4,\]
with $\epsilon \in \{0,1,2,3\}$.

{\bf Case 1:} $49u^4+2v^4i=(1-2i)(3+2i)(s+ti)^4$. Equating the real and imaginary parts on both sides of this equation gives
\begin{equation}\label{3185.1-1}
\begin{cases}
7s^4 + 16s^3t - 42s^2t^2 - 16st^3 + 7t^4 - 49u^4=0,\\
-4s^4 + 28s^3t + 24s^2t^2 - 28st^3 - 4t^4 - 2v^4=0.
\end{cases}
\end{equation}
The scheme defined by system \eqref{3185.1-1} is locally insoluble at $2$.
%\begin{lstlisting}
%_<x>:=PolynomialRing(Rationals());
%k<i>:=NumberField(x^2+1);
%K<s,t>:=PolynomialRing(k,2);
%e:=0;
%F:=i^e*(1-2*i)*(3+2*i)*(s+t*i)^4;
%F;
%L<s,t>:=PolynomialRing(Rationals(),2);
%_<i>:=PolynomialRing(L);
%F:=(-4*i + 7)*s^4 + (28*i + 16)*s^3*t + (24*i - 42)*s^2*t^2 + (-28*i - 16)*s*t^3 +
%    (-4*i + 7)*t^4;
%F;
%A:=7*s^4 + 16*s^3*t -42*s^2*t^2 - 16*s*t^3 + 7*t^4;
%B:=-4*s^4 + 28*s^3*t + 24*s^2*t^2 - 28*s*t^3 - 4*t^4;
%F-A-i*B;
\begin{lstlisting}
P<s,t,u,v>:=ProjectiveSpace(Rationals(),3);
A:=7*s^4 + 16*s^3*t -42*s^2*t^2 - 16*s*t^3 + 7*t^4-49*u^4;
B:=-4*s^4 + 28*s^3*t + 24*s^2*t^2 - 28*s*t^3 - 4*t^4-2*v^4;
S:=Scheme(P,[A,B]);
IsLocallySolvable(S,2);
\end{lstlisting}

{\bf Case 2:} $49u^4+2v^4i=i(1-2i)(3+2i)(s+ti)^4$. Equating the real and imaginary parts on both sides of this equation gives
\begin{equation}\label{3185.1-2}
\begin{cases}
4s^4 - 28s^3t - 24s^2t^2 + 28st^3 + 4t^4 - 49u^4=0,\\
7s^4 + 16s^3t - 42s^2t^2 - 16st^3 + 7t^4 - 2v^4=0.
\end{cases}
\end{equation}
The scheme defined by system \eqref{3185.1-2} is locally insoluble at $2$.
%\begin{lstlisting}
%_<x>:=PolynomialRing(Rationals());
%k<i>:=NumberField(x^2+1);
%K<s,t>:=PolynomialRing(k,2);
%e:=1;
%F:=i^e*(1-2*i)*(3+2*i)*(s+t*i)^4;
%F;
%L<s,t>:=PolynomialRing(Rationals(),2);
%_<i>:=PolynomialRing(L);
%F:=(7*i + 4)*s^4 + (16*i - 28)*s^3*t + (-42*i - 24)*s^2*t^2 + (-16*i + 28)*s*t^3 +
%    (7*i + 4)*t^4;
%F;
%A:= 4*s^4 - 28*s^3*t -  24*s^2*t^2 + 28*s*t^3 + 4*t^4;
%B:=7*s^4 + 16*s^3*t - 42*s^2*t^2 - 16*s*t^3 + 7*t^4;
%F-A-i*B;

\begin{lstlisting}
P<s,t,u,v>:=ProjectiveSpace(Rationals(),3);
A:=4*s^4 - 28*s^3*t -  24*s^2*t^2 + 28*s*t^3 + 4*t^4-49*u^4;
B:=7*s^4 + 16*s^3*t - 42*s^2*t^2 - 16*s*t^3 + 7*t^4-2*v^4;
S:=Scheme(P,[A,B]);
IsLocallySolvable(S,2);
\end{lstlisting}

{\bf Case 3:} $49u^4+2v^4i=i^2(1-2i)(3+2i)(s+ti)^4$. Equating the real and imaginary parts on both sides of this equation gives
\begin{equation}\label{3185.1-3}
\begin{cases}
-7s^4 - 16s^3t + 42s^2t^2 + 16st^3 - 7t^4 - 49u^4=0,\\
4s^4 - 28s^3t - 24s^2t^2 + 28st^3 + 4t^4 - 2v^4=0.
\end{cases}
\end{equation}
The scheme defined by system \eqref{3185.1-3} is locally insoluble at $2$.
%\begin{lstlisting}
%_<x>:=PolynomialRing(Rationals());
%k<i>:=NumberField(x^2+1);
%K<s,t>:=PolynomialRing(k,2);
%e:=2;
%F:=i^e*(1-2*i)*(3+2*i)*(s+t*i)^4;
%F;
%L<s,t>:=PolynomialRing(Rationals(),2);
%_<i>:=PolynomialRing(L);
%F:=(4*i - 7)*s^4 + (-28*i - 16)*s^3*t + (-24*i + 42)*s^2*t^2 + (28*i + 16)*s*t^3 +
%    (4*i - 7)*t^4;
%F;
%A:=- 7*s^4 - 16*s^3*t +42*s^2*t^2 + 16*s*t^3 - 7*t^4;
%B:=4*s^4 - 28*s^3*t - 24*s^2*t^2 + 28*s*t^3 + 4*t^4;
%F-A-i*B;

\begin{lstlisting}
P<s,t,u,v>:=ProjectiveSpace(Rationals(),3);
A:=- 7*s^4 - 16*s^3*t +42*s^2*t^2 + 16*s*t^3 - 7*t^4-49*u^4;
B:=4*s^4 - 28*s^3*t - 24*s^2*t^2 + 28*s*t^3 + 4*t^4-2*v^4;
S:=Scheme(P,[A,B]);
IsLocallySolvable(S,2);
\end{lstlisting}
{\bf Case 4:} $49u^4+2v^4i=i^3(1-2i)(3+2i)(s+ti)^4$. Equating the real and imaginary parts on both sides of this equation gives
\begin{equation}\label{3185.1-4}
\begin{cases}
-4s^4 + 28s^3t + 24s^2t^2 - 28st^3 - 4t^4 - 49u^4=0,\\
-7s^4 - 16s^3t + 42s^2t^2 + 16st^3 - 7t^4 - 2v^4=0.
\end{cases}
\end{equation}
The scheme defined by system \eqref{3185.1-4} is locally insoluble at $2$.
%
%\begin{lstlisting}
%_<x>:=PolynomialRing(Rationals());
%k<i>:=NumberField(x^2+1);
%K<s,t>:=PolynomialRing(k,2);
%e:=3;
%F:=i^e*(1-2*i)*(3+2*i)*(s+t*i)^4;
%F;
%L<s,t>:=PolynomialRing(Rationals(),2);
%_<i>:=PolynomialRing(L);
%F:=(-7*i - 4)*s^4 + (-16*i + 28)*s^3*t + (42*i + 24)*s^2*t^2 + (16*i - 28)*s*t^3 +
%    (-7*i - 4)*t^4;
%F;
%A:=- 4*s^4 + 28*s^3*t +  24*s^2*t^2 - 28*s*t^3 - 4*t^4;
%B:=-7*s^4 - 16*s^3*t + 42*s^2*t^2 + 16*s*t^3 - 7*t^4;
%F-A-i*B;

\begin{lstlisting}
P<s,t,u,v>:=ProjectiveSpace(Rationals(),3);
A:=- 4*s^4 + 28*s^3*t +  24*s^2*t^2 - 28*s*t^3 - 4*t^4-49*u^4;
B:=-7*s^4 - 16*s^3*t + 42*s^2*t^2 + 16*s*t^3 - 7*t^4-2*v^4;
S:=Scheme(P,[A,B]);
IsLocallySolvable(S,2);
\end{lstlisting}

{\bf Case 5:} $49u^4+2v^4i=(1-2i)(3-2i)(s-ti)^4$. Equating the real and imaginary parts on both sides of this equation gives
\begin{equation}\label{3185.1-5}
\begin{cases}
-s^4 + 32s^3t + 6s^2t^2 - 32st^3 - t^4 - 49u^4=0,\\
-8s^4 - 4s^3t + 48s^2t^2 + 4st^3 - 8t^4 - 2v^4=0.
\end{cases}
\end{equation}
The scheme defined by system \eqref{3185.1-5} is locally insoluble at $2$.

%\begin{lstlisting}
%_<x>:=PolynomialRing(Rationals());
%k<i>:=NumberField(x^2+1);
%K<s,t>:=PolynomialRing(k,2);
%e:=0;
%F:=i^e*(1-2*i)*(3-2*i)*(s+t*i)^4;
%F;
%L<s,t>:=PolynomialRing(Rationals(),2);
%_<i>:=PolynomialRing(L);
%F:=(-8*i - 1)*s^4 + (-4*i + 32)*s^3*t + (48*i + 6)*s^2*t^2 + (4*i - 32)*s*t^3 +
%    (-8*i - 1)*t^4;
%F;
%A:=- s^4 + 32*s^3*t + 6*s^2*t^2  - 32*s*t^3 - t^4;
%B:=-8*s^4 - 4*s^3*t + 48*s^2*t^2 + 4*s*t^3 - 8*t^4;
%F-A-i*B;
%
\begin{lstlisting}
P<s,t,u,v>:=ProjectiveSpace(Rationals(),3);
A:=- s^4 + 32*s^3*t + 6*s^2*t^2  - 32*s*t^3 - t^4-49*u^4;
B:=-8*s^4 - 4*s^3*t + 48*s^2*t^2 + 4*s*t^3 - 8*t^4-2*v^4;
S:=Scheme(P,[A,B]);
IsLocallySolvable(S,2);
\end{lstlisting}

{\bf Case 6:} $49u^4+2v^4i=i(1-2i)(3-2i)(s+ti)^4$. Equating the real and imaginary parts on both sides of this equation gives
\begin{equation}\label{3185.1-6}
\begin{cases}
8s^4 + 4s^3t - 48s^2t^2 - 4st^3 + 8t^4 - 49u^4=0,\\
-s^4 + 32s^3t + 6s^2t^2 - 32st^3 - t^4 - 2v^4=0.
\end{cases}
\end{equation}
The scheme defined by system \eqref{3185.1-6} is locally insoluble at $2$.

%\begin{lstlisting}
%_<x>:=PolynomialRing(Rationals());
%k<i>:=NumberField(x^2+1);
%K<s,t>:=PolynomialRing(k,2);
%e:=1;
%F:=i^e*(1-2*i)*(3-2*i)*(s+t*i)^4;
%F;
%L<s,t>:=PolynomialRing(Rationals(),2);
%_<i>:=PolynomialRing(L);
%F:=(-i + 8)*s^4 + (32*i + 4)*s^3*t + (6*i - 48)*s^2*t^2 + (-32*i - 4)*s*t^3 + (-i +
%    8)*t^4;
%F;
%A:=8*s^4 + 4*s^3*t - 48*s^2*t^2- 4*s*t^3 + 8*t^4;
%B:=-s^4 + 32*s^3*t + 6*s^2*t^2 - 32*s*t^3 - t^4;
%F-A-i*B;

\begin{lstlisting}
P<s,t,u,v>:=ProjectiveSpace(Rationals(),3);
A:=8*s^4 + 4*s^3*t - 48*s^2*t^2- 4*s*t^3 + 8*t^4-49*u^4;
B:=-s^4 + 32*s^3*t + 6*s^2*t^2 - 32*s*t^3 - t^4-2*v^4;
S:=Scheme(P,[A,B]);
IsLocallySolvable(S,2);
\end{lstlisting}

{\bf Case 7:} $49u^4+2v^4i=i^2(1-2i)(3-2i)(s+ti)^4$. Equating the real and imaginary parts on both sides of this equation gives
\begin{equation}\label{3185.1-7}
\begin{cases}
s^4 - 32s^3t - 6s^2t^2 + 32st^3 + t^4 - 49u^4=0,\\
8s^4 + 4s^3t - 48s^2t^2 - 4st^3 + 8t^4 - 2v^4=0.
\end{cases}
\end{equation}
The scheme defined by system \eqref{3185.1-7} is locally insoluble at $2$.

%\begin{lstlisting}
%_<x>:=PolynomialRing(Rationals());
%k<i>:=NumberField(x^2+1);
%K<s,t>:=PolynomialRing(k,2);
%e:=2;
%F:=i^e*(1-2*i)*(3-2*i)*(s+t*i)^4;
%F;
%L<s,t>:=PolynomialRing(Rationals(),2);
%_<i>:=PolynomialRing(L);
%F:=(8*i + 1)*s^4 + (4*i - 32)*s^3*t + (-48*i - 6)*s^2*t^2 + (-4*i + 32)*s*t^3 +
%    (8*i + 1)*t^4;
%F;
%A:=s^4 - 32*s^3*t - 6*s^2*t^2 + 32*s*t^3 + t^4;
%B:=8*s^4 + 4*s^3*t - 48*s^2*t^2 - 4*s*t^3 + 8*t^4;
%F-A-i*B;

\begin{lstlisting}
P<s,t,u,v>:=ProjectiveSpace(Rationals(),3);
A:=s^4 - 32*s^3*t - 6*s^2*t^2 + 32*s*t^3 + t^4-49*u^4;
B:=8*s^4 + 4*s^3*t - 48*s^2*t^2 - 4*s*t^3 + 8*t^4-2*v^4;
S:=Scheme(P,[A,B]);
IsLocallySolvable(S,2);
\end{lstlisting}

{\bf Case 8:} $49u^4+2v^4i=i^3(1-2i)(3-2i)(s+ti)^4$. Equating the real and imaginary parts on both sides of this equation gives
\begin{equation}\label{3185.1-8}
\begin{cases}
-8s^4 - 4s^3t + 48s^2t^2 + 4st^3 - 8t^4 - 49u^4=0,\\
s^4 - 32s^3t - 6s^2t^2 + 32st^3 + t^4 - 2v^4=0.
\end{cases}
\end{equation}
The scheme defined by system \eqref{3185.1-8} is locally insoluble at $2$.
%\begin{lstlisting}
%_<x>:=PolynomialRing(Rationals());
%k<i>:=NumberField(x^2+1);
%K<s,t>:=PolynomialRing(k,2);
%e:=3;
%F:=i^e*(1-2*i)*(3-2*i)*(s+t*i)^4;
%F;
%L<s,t>:=PolynomialRing(Rationals(),2);
%_<i>:=PolynomialRing(L);
%F:=(i - 8)*s^4 + (-32*i - 4)*s^3*t + (-6*i + 48)*s^2*t^2 + (32*i + 4)*s*t^3 + (i -
%    8)*t^4;
%F;
%A:=- 8*s^4 - 4*s^3*t + 48*s^2*t^2 +   4*s*t^3 - 8*t^4;
%B:=s^4 - 32*s^3*t - 6*s^2*t^2 + 32*s*t^3 + t^4;
%F-A-i*B;

\begin{lstlisting}
P<s,t,u,v>:=ProjectiveSpace(Rationals(),3);
A:=- 8*s^4 - 4*s^3*t + 48*s^2*t^2 + 4*s*t^3 - 8*t^4-49*u^4;
B:=s^4 - 32*s^3*t - 6*s^2*t^2 + 32*s*t^3 + t^4-2*v^4;
S:=Scheme(P,[A,B]);
IsLocallySolvable(S,2);
\end{lstlisting}

\subsubsection{\bf{The case of \eqref{3185.2}.}}
We next consider \eqref{3185.2}, which we can write as
\begin{equation} (98u^4+v^4i)(98u^4-v^4i)=(1+2i)(1-2i)(3+2i)(3-2i)w^4.
\end{equation}
Since $5\nmid u,v$, we have
\[98u^4+v^4i\equiv i-2\pmod{5}.\]
Hence, $1+2i|98u^4+v^4i$. This implies that there exist integers $s,t$ such that 
\[98u^4+v^4i=i^{\epsilon}(1+2i)(3\pm 2i)(s+ti)^4,\]
with $\epsilon \in \{0,1,2,3\}$.

{\bf Case 1:} $98u^4+v^4i=(1+2i)(3+2i)(s+ti)^4$. Equating the real and imaginary parts on both sides of this equation gives
\begin{equation}\label{3185.2-1}
\begin{cases}
-s^4 - 32s^3t + 6s^2t^2 + 32st^3 - t^4 - 98u^4=0,\\
8s^4 - 4s^3t - 48s^2t^2 + 4st^3 + 8t^4 - v^4=0.
\end{cases}
\end{equation}
The scheme defined by system \eqref{3185.2-1} is locally insoluble at $2$.
%\begin{lstlisting}
%    _<x>:=PolynomialRing(Rationals());
%k<i>:=NumberField(x^2+1);
%K<s,t>:=PolynomialRing(k,2);
%e:=0;
%F:=i^e*(1+2*i)*(3+2*i)*(s+t*i)^4;
%F;
%L<s,t>:=PolynomialRing(Rationals(),2);
%_<i>:=PolynomialRing(L);
%F:=(8*i - 1)*s^4 + (-4*i - 32)*s^3*t + (-48*i + 6)*s^2*t^2 + (4*i + 32)*s*t^3 +
%    (8*i - 1)*t^4;
%F;
%A:=- s^4 - 32*s^3*t + 6*s^2*t^2 + 32*s*t^3 - t^4;
%B:=8*s^4 - 4*s^3*t - 48*s^2*t^2 + 4*s*t^3 + 8*t^4;
%F-A-i*B;

\begin{lstlisting}
P<s,t,u,v>:=ProjectiveSpace(Rationals(),3);
A:=- s^4 - 32*s^3*t + 6*s^2*t^2 + 32*s*t^3 - t^4-98*u^4;
B:=8*s^4 - 4*s^3*t - 48*s^2*t^2 + 4*s*t^3 + 8*t^4-v^4;
S:=Scheme(P,[A,B]);
IsLocallySolvable(S,2);
\end{lstlisting}

{\bf Case 2:} $98u^4+v^4i=i(1+2i)(3+2i)(s+ti)^4$. Equating the real and imaginary parts on both sides of this equation gives
\begin{equation}\label{3185.2-2}
\begin{cases}
-8s^4 + 4s^3t + 48s^2t^2 - 4st^3 - 8t^4 - 98u^4=0,\\
-s^4 - 32s^3t + 6s^2t^2 + 32st^3 - t^4 - v^4=0.
\end{cases}
\end{equation}
The scheme defined by system \eqref{3185.2-2} is locally insoluble at $2$.

%\begin{lstlisting}
%    _<x>:=PolynomialRing(Rationals());
%k<i>:=NumberField(x^2+1);
%K<s,t>:=PolynomialRing(k,2);
%e:=1;
%F:=i^e*(1+2*i)*(3+2*i)*(s+t*i)^4;
%F;
%L<s,t>:=PolynomialRing(Rationals(),2);
%_<i>:=PolynomialRing(L);
%F:= (-i - 8)*s^4 + (-32*i + 4)*s^3*t + (6*i + 48)*s^2*t^2 + (32*i - 4)*s*t^3 + (-i -
%    8)*t^4;
%F;
%A:=- 8*s^4 + 4*s^3*t + 48*s^2*t^2 - 4*s*t^3 - 8*t^4;
%B:=-s^4 - 32*s^3*t + 6*s^2*t^2 + 32*s*t^3 - t^4;
%F-A-i*B;

\begin{lstlisting}
P<s,t,u,v>:=ProjectiveSpace(Rationals(),3);
A:=- 8*s^4 + 4*s^3*t + 48*s^2*t^2 - 4*s*t^3 - 8*t^4-98*u^4;
B:=-s^4 - 32*s^3*t + 6*s^2*t^2 + 32*s*t^3 - t^4-v^4;
S:=Scheme(P,[A,B]);
IsLocallySolvable(S,2);
\end{lstlisting}

{\bf Case 3:}  $98u^4+v^4i=i^2(1+2i)(3+2i)(s+ti)^4$. Equating the real and imaginary parts on both sides of this equation gives
\begin{equation}\label{3185.2-3}
\begin{cases}
s^4 + 32s^3t - 6s^2t^2 - 32st^3 + t^4 - 98u^4=0, \\
-8s^4 + 4s^3t + 48s^2t^2 - 4st^3 - 8t^4 - v^4=0.
%-8*s^4 + 4*s^3*t + 48*s^2*t^2 - 4*s*t^3 - 8*t^4 - 98*u^4=0,\\
%-s^4 - 32*s^3*t + 6*s^2*t^2 + 32*s*t^3 - t^4 - v^4=0.
\end{cases}
\end{equation}
The scheme defined by system \eqref{3185.2-3} is locally insoluble at $2$.
%\begin{lstlisting}
%    _<x>:=PolynomialRing(Rationals());
%k<i>:=NumberField(x^2+1);
%K<s,t>:=PolynomialRing(k,2);
%e:=2;
%F:=i^e*(1+2*i)*(3+2*i)*(s+t*i)^4;
%F;
%L<s,t>:=PolynomialRing(Rationals(),2);
%_<i>:=PolynomialRing(L);
%F:= (-i - 8)*s^4 + (-32*i + 4)*s^3*t + (6*i + 48)*s^2*t^2 + (32*i - 4)*s*t^3 + (-i -
%    8)*t^4;
%F;
%A:=- 8*s^4 + 4*s^3*t + 48*s^2*t^2 - 4*s*t^3 - 8*t^4;
%B:=-s^4 - 32*s^3*t + 6*s^2*t^2 + 32*s*t^3 - t^4;
%F-A-i*B;

\begin{lstlisting}
P<s,t,u,v>:=ProjectiveSpace(Rationals(),3);
A:=s^4 + 32*s^3*t - 6*s^2*t^2 - 32*s*t^3 + t^4 - 98*u^4;
B:=-8*s^4 + 4*s^3*t + 48*s^2*t^2 - 4*s*t^3 - 8*t^4 - v^4;
S:=Scheme(P,[A,B]);
IsLocallySolvable(S,2);
\end{lstlisting}

{\bf Case 4:}  $98u^4+v^4i=i^3(1+2i)(3+2i)(s+ti)^4$. Equating the real and imaginary parts on both sides of this equation gives
\begin{equation}\label{3185.2-4}
\begin{cases}
8s^4 - 4s^3t - 48s^2t^2 + 4st^3 + 8t^4 - 98u^4=0,\\
s^4 + 32s^3t - 6s^2t^2 - 32st^3 + t^4 - v^4=0.
\end{cases}
\end{equation}
The scheme defined by system \eqref{3185.2-4} is locally insoluble at $2$.

%\begin{lstlisting}
%    _<x>:=PolynomialRing(Rationals());
%k<i>:=NumberField(x^2+1);
%K<s,t>:=PolynomialRing(k,2);
%e:=3;
%F:=i^e*(1+2*i)*(3+2*i)*(s+t*i)^4;
%F;
%L<s,t>:=PolynomialRing(Rationals(),2);
%_<i>:=PolynomialRing(L);
%F:= (i + 8)*s^4 + (32*i - 4)*s^3*t + (-6*i - 48)*s^2*t^2 + (-32*i + 4)*s*t^3 + (i +
%    8)*t^4;
%F;
%A:=8*s^4 - 4*s^3*t - 48*s^2*t^2 +4*s*t^3 + 8*t^4;
%B:=s^4 + 32*s^3*t - 6*s^2*t^2 - 32*s*t^3 + t^4;
%F-A-i*B;
\begin{lstlisting}
P<s,t,u,v>:=ProjectiveSpace(Rationals(),3);
A:=8*s^4 - 4*s^3*t - 48*s^2*t^2 +4*s*t^3 + 8*t^4-98*u^4;
B:=s^4 + 32*s^3*t - 6*s^2*t^2 - 32*s*t^3 + t^4-v^4;
S:=Scheme(P,[A,B]);
IsLocallySolvable(S,2);
\end{lstlisting}

{\bf Case 5:} $98u^4+v^4i=(1-2i)(3-2i)(s-ti)^4$.  Equating the real and imaginary parts on both sides of this equation gives
\begin{equation}\label{3185.2-5}
\begin{cases}
-s^4 + 32*s^3*t + 6*s^2*t^2 - 32*s*t^3 - t^4 - 98*u^4=0,\\
-8*s^4 - 4*s^3*t + 48*s^2*t^2 + 4*s*t^3 - 8*t^4 - v^4=0.
\end{cases}
\end{equation}
The scheme defined by system \eqref{3185.2-5} is locally insoluble at $2$.

%\begin{lstlisting}
%_<x>:=PolynomialRing(Rationals());
%k<i>:=NumberField(x^2+1);
%K<s,t>:=PolynomialRing(k,2);
%e:=0;
%F:=i^e*(1-2*i)*(3-2*i)*(s+t*i)^4;
%F;
%L<s,t>:=PolynomialRing(Rationals(),2);
%_<i>:=PolynomialRing(L);
%F:=(-8*i - 1)*s^4 + (-4*i + 32)*s^3*t + (48*i + 6)*s^2*t^2 + (4*i - 32)*s*t^3 +
%    (-8*i - 1)*t^4;
%F;
%A:=- s^4 + 32*s^3*t + 6*s^2*t^2  - 32*s*t^3 - t^4;
%B:=-8*s^4 - 4*s^3*t + 48*s^2*t^2 + 4*s*t^3 - 8*t^4;
%F-A-i*B;

\begin{lstlisting}
P<s,t,u,v>:=ProjectiveSpace(Rationals(),3);
A:=- s^4 + 32*s^3*t + 6*s^2*t^2  - 32*s*t^3 - t^4-98*u^4;
B:=-8*s^4 - 4*s^3*t + 48*s^2*t^2 + 4*s*t^3 - 8*t^4-v^4;
S:=Scheme(P,[A,B]);
IsLocallySolvable(S,2);
\end{lstlisting}

{\bf Case 6:} $98u^4+v^4i=i(1-2i)(3-2i)(s+ti)^4$.  Equating the real and imaginary parts on both sides of this equation gives
\begin{equation}\label{3185.2-6}
\begin{cases}
8*s^4 + 4*s^3*t - 48*s^2*t^2 - 4*s*t^3 + 8*t^4 - 98*u^4=0,\\
-s^4 + 32*s^3*t + 6*s^2*t^2 - 32*s*t^3 - t^4 - v^4=0.
\end{cases}
\end{equation}
The scheme defined by system \eqref{3185.2-6} is locally insoluble at $2$.

%\begin{lstlisting}
%    _<x>:=PolynomialRing(Rationals());
%k<i>:=NumberField(x^2+1);
%K<s,t>:=PolynomialRing(k,2);
%e:=1;
%F:=i^e*(1-2*i)*(3-2*i)*(s+t*i)^4;
%F;
%L<s,t>:=PolynomialRing(Rationals(),2);
%_<i>:=PolynomialRing(L);
%F:=(-i + 8)*s^4 + (32*i + 4)*s^3*t + (6*i - 48)*s^2*t^2 + (-32*i - 4)*s*t^3 + (-i +
%    8)*t^4;
%F;
%A:=8*s^4 + 4*s^3*t - 48*s^2*t^2- 4*s*t^3 + 8*t^4;
%B:=-s^4 + 32*s^3*t + 6*s^2*t^2 - 32*s*t^3 - t^4;
%F-A-i*B;

\begin{lstlisting}
P<s,t,u,v>:=ProjectiveSpace(Rationals(),3);
A:=8*s^4 + 4*s^3*t - 48*s^2*t^2- 4*s*t^3 + 8*t^4-98*u^4;
B:=-s^4 + 32*s^3*t + 6*s^2*t^2 - 32*s*t^3 - t^4-v^4;
S:=Scheme(P,[A,B]);
IsLocallySolvable(S,2);
\end{lstlisting}

{\bf Case 7:} $98u^4+v^4i=i^2(1-2i)(3-2i)(s+ti)^4$.  Equating the real and imaginary parts on both sides of this equation gives
\begin{equation}\label{3185.2-7}
\begin{cases}
s^4 - 32*s^3*t - 6*s^2*t^2 + 32*s*t^3 + t^4 - 98*u^4=0,\\
8*s^4 + 4*s^3*t - 48*s^2*t^2 - 4*s*t^3 + 8*t^4 - v^4=0.
\end{cases}
\end{equation}
The scheme defined by system \eqref{3185.1-7} is locally insoluble at $2$.
%
%\begin{lstlisting}
%_<x>:=PolynomialRing(Rationals());
%k<i>:=NumberField(x^2+1);
%K<s,t>:=PolynomialRing(k,2);
%e:=2;
%F:=i^e*(1-2*i)*(3-2*i)*(s+t*i)^4;
%F;
%L<s,t>:=PolynomialRing(Rationals(),2);
%_<i>:=PolynomialRing(L);
%F:=(8*i + 1)*s^4 + (4*i - 32)*s^3*t + (-48*i - 6)*s^2*t^2 + (-4*i + 32)*s*t^3 +
%    (8*i + 1)*t^4;
%F;
%A:=s^4 - 32*s^3*t - 6*s^2*t^2 + 32*s*t^3 + t^4;
%B:=8*s^4 + 4*s^3*t - 48*s^2*t^2 - 4*s*t^3 + 8*t^4;
%F-A-i*B;

\begin{lstlisting}
P<s,t,u,v>:=ProjectiveSpace(Rationals(),3);
A:=s^4 - 32*s^3*t - 6*s^2*t^2 + 32*s*t^3 + t^4-98*u^4;
B:=8*s^4 + 4*s^3*t - 48*s^2*t^2 - 4*s*t^3 + 8*t^4-v^4;
S:=Scheme(P,[A,B]);
IsLocallySolvable(S,2);
\end{lstlisting}

{\bf Case 8:} $98u^4+v^4i=i^3(1-2i)(3-2i)(s+ti)^4$.  Equating the real and imaginary parts on both sides of this equation gives
\begin{equation}\label{3185.2-8}
\begin{cases}
-8*s^4 - 4*s^3*t + 48*s^2*t^2 + 4*s*t^3 - 8*t^4 - 98*u^4=0,\\
s^4 - 32*s^3*t - 6*s^2*t^2 + 32*s*t^3 + t^4 - v^4=0.
\end{cases}
\end{equation}
The scheme defined by system \eqref{3185.2-8} is locally insoluble at $2$.
%\begin{lstlisting}
%_<x>:=PolynomialRing(Rationals());
%k<i>:=NumberField(x^2+1);
%K<s,t>:=PolynomialRing(k,2);
%e:=3;
%F:=i^e*(1-2*i)*(3-2*i)*(s+t*i)^4;
%F;
%L<s,t>:=PolynomialRing(Rationals(),2);
%_<i>:=PolynomialRing(L);
%F:=(i - 8)*s^4 + (-32*i - 4)*s^3*t + (-6*i + 48)*s^2*t^2 + (32*i + 4)*s*t^3 + (i -
%    8)*t^4;
%F;
%A:=- 8*s^4 - 4*s^3*t + 48*s^2*t^2 +   4*s*t^3 - 8*t^4;
%B:=s^4 - 32*s^3*t - 6*s^2*t^2 + 32*s*t^3 + t^4;
%F-A-i*B;

\begin{lstlisting}
P<s,t,u,v>:=ProjectiveSpace(Rationals(),3);
A:=- 8*s^4 - 4*s^3*t + 48*s^2*t^2 +   4*s*t^3 - 8*t^4-98*u^4;
B:=s^4 - 32*s^3*t - 6*s^2*t^2 + 32*s*t^3 + t^4-v^4;
S:=Scheme(P,[A,B]);
IsLocallySolvable(S,2);
\end{lstlisting}

\subsection{The case of $n=3265$.}% \[n=3265.\]
According to Table \ref{tb:factorisation}, in the case $n=3265$ it remains to
show that equation 
\begin{equation}\label{3265}
4 u^8 + v^8=3265  w^4
\end{equation}
has no integer solutions satisfying \eqref{eq:cond}, which in this case reduces to \[\gcd(2u,v)=\gcd(2u,3265w)=\gcd(v,3265w)=1.\]
Write \eqref{3265} as
\begin{equation} (2u^4+v^4i)(2u^4-v^4i)=(2+i)(2-i)(22+13i)(22-13i)w^4.
\end{equation}
Since $5\nmid u,v$, we have 
\[2u^4+v^4i\equiv 2+i\pmod{5}.\]

Hence, $2+i|2u^4+v^4i$. This implies that there exist integers $s,t$ such that 
\[2u^4+v^4i=i^{\epsilon}(2+i)(22\pm 13i)(s+ti)^4,\]
with $\epsilon \in \{0,1,2,3\}$.

{\bf Case 1:} $2u^4+v^4i=(2+i)(22+13i)(s+ti)^4$. Equating the real and imaginary parts on both sides of this equation gives
\begin{equation}\label{3265-1}
\begin{cases}
31s^4 - 192s^3t - 186s^2t^2 + 192st^3 + 31t^4 - 2u^4=0,\\
48s^4 + 124s^3t - 288s^2t^2 - 124st^3 + 48t^4 - v^4=0.
\end{cases}
\end{equation}
The scheme defined by system \eqref{3265-1} is locally insoluble at $2$.

%\begin{lstlisting}
%    _<x>:=PolynomialRing(Rationals());
%k<i>:=NumberField(x^2+1);
%K<s,t>:=PolynomialRing(k,2);
%e:=0;
%F:=i^e*(2+i)*(22+13*i)*(s+t*i)^4;
%F;
%L<s,t>:=PolynomialRing(Rationals(),2);
%_<i>:=PolynomialRing(L);
%F:=(48*i + 31)*s^4 + (124*i - 192)*s^3*t + (-288*i - 186)*s^2*t^2 + (-124*i + 192)*s*t^3 + (48*i + 31)*t^4;
%F;
%A:=31*s^4 - 192*s^3*t - 186*s^2*t^2 + 192*s*t^3 + 31*t^4;
%B:=48*s^4 + 124*s^3*t - 288*s^2*t^2 - 124*s*t^3 + 48*t^4;
%F-A-i*B;

\begin{lstlisting}
P<s,t,u,v>:=ProjectiveSpace(Rationals(),3);
A:=31*s^4 - 192*s^3*t - 186*s^2*t^2 + 192*s*t^3 + 31*t^4-2*u^4;
B:=48*s^4 + 124*s^3*t - 288*s^2*t^2 - 124*s*t^3 + 48*t^4-v^4;
S:=Scheme(P,[A,B]);
IsLocallySolvable(S,2);
\end{lstlisting}

{\bf Case 2:} $2u^4+v^4i=i(2+i)(22+13i)(s+ti)^4$. Equating the real and imaginary parts on both sides of this equation gives
\begin{equation}\label{3265-2}
\begin{cases}
-48s^4 - 124s^3t + 288s^2t^2 + 124st^3 - 48t^4 - 2u^4=0,\\
31s^4 - 192s^3t - 186s^2t^2 + 192st^3 + 31t^4 - v^4=0
\end{cases}
\end{equation}
The scheme defined by system \eqref{3265-2} is locally insoluble at $2$.

%\begin{lstlisting}
%_<x>:=PolynomialRing(Rationals());
%k<i>:=NumberField(x^2+1);
%K<s,t>:=PolynomialRing(k,2);
%e:=1;
%F:=i^e*(2+i)*(22+13*i)*(s+t*i)^4;
%F;
%L<s,t>:=PolynomialRing(Rationals(),2);
%_<i>:=PolynomialRing(L);
%F:=(31*i - 48)*s^4 + (-192*i - 124)*s^3*t + (-186*i + 288)*s^2*t^2 + (192*i +
%    124)*s*t^3 + (31*i - 48)*t^4;
%F;
%A:=- 48*s^4 - 124*s^3*t +288*s^2*t^2 + 124*s*t^3 - 48*t^4;
%B:=31*s^4 - 192*s^3*t - 186*s^2*t^2 + 192*s*t^3 + 31*t^4;
%F-A-i*B;

\begin{lstlisting}
P<s,t,u,v>:=ProjectiveSpace(Rationals(),3);
A:=- 48*s^4 - 124*s^3*t +288*s^2*t^2 + 124*s*t^3 - 48*t^4-2*u^4;
B:=31*s^4 - 192*s^3*t - 186*s^2*t^2 + 192*s*t^3 + 31*t^4-v^4;
S:=Scheme(P,[A,B]);
IsLocallySolvable(S,2);
\end{lstlisting}

{\bf Case 3:} $2u^4+v^4i=i^2(2+i)(22+13i)(s+ti)^4$. Equating the real and imaginary parts on both sides of this equation gives
\begin{equation}\label{3265-3}
\begin{cases}
-31s^4 + 192s^3t + 186s^2t^2 - 192st^3 - 31t^4 - 2u^4=0,\\
-48s^4 - 124s^3t + 288s^2t^2 + 124st^3 - 48t^4 - v^4=0.
\end{cases}
\end{equation}
The scheme defined by system \eqref{3265-3} is locally insoluble at $2$.

%\begin{lstlisting}
%_<x>:=PolynomialRing(Rationals());
%k<i>:=NumberField(x^2+1);
%K<s,t>:=PolynomialRing(k,2);
%e:=2;
%F:=i^e*(2+i)*(22+13*i)*(s+t*i)^4;
%F;
%L<s,t>:=PolynomialRing(Rationals(),2);
%_<i>:=PolynomialRing(L);
%F:=(-48*i - 31)*s^4 + (-124*i + 192)*s^3*t + (288*i + 186)*s^2*t^2 + (124*i -
%    192)*s*t^3 + (-48*i - 31)*t^4;
%F;
%A:=- 31*s^4 + 192*s^3*t + 186*s^2*t^2 - 192*s*t^3 - 31*t^4;
%B:=-48*s^4 - 124*s^3*t + 288*s^2*t^2 + 124*s*t^3 - 48*t^4;
%F-A-i*B;

\begin{lstlisting}
P<s,t,u,v>:=ProjectiveSpace(Rationals(),3);
A:=- 31*s^4 + 192*s^3*t + 186*s^2*t^2 - 192*s*t^3 - 31*t^4-2*u^4;
B:=-48*s^4 - 124*s^3*t + 288*s^2*t^2 + 124*s*t^3 - 48*t^4-v^4;
S:=Scheme(P,[A,B]);
IsLocallySolvable(S,2);
\end{lstlisting}

{\bf Case 4:} $2u^4+v^4i=i^3(2+i)(22+13i)(s+ti)^4$. Equating the real and imaginary parts on both sides of this equation gives
\begin{equation}\label{3265-4}
\begin{cases}
48s^4 + 124s^3t - 288s^2t^2 - 124st^3 + 48t^4 - 2u^4=0,\\
-31s^4 + 192s^3t + 186s^2t^2 - 192st^3 - 31t^4 - v^4=0.
\end{cases}
\end{equation}
The scheme defined by system \eqref{3265-4} is locally insoluble at $5$.

%\begin{lstlisting}
%_<x>:=PolynomialRing(Rationals());
%k<i>:=NumberField(x^2+1);
%K<s,t>:=PolynomialRing(k,2);
%e:=3;
%F:=i^e*(2+i)*(22+13*i)*(s+t*i)^4;
%F;
%L<s,t>:=PolynomialRing(Rationals(),2);
%_<i>:=PolynomialRing(L);
%F:=(-31*i + 48)*s^4 + (192*i + 124)*s^3*t + (186*i - 288)*s^2*t^2 + (-192*i -
%    124)*s*t^3 + (-31*i + 48)*t^4;
%F;
%A:=48*s^4 + 124*s^3*t - 288*s^2*t^2 - 124*s*t^3 + 48*t^4;
%B:=-31*s^4 + 192*s^3*t + 186*s^2*t^2 - 192*s*t^3 - 31*t^4;
%F-A-i*B;

\begin{lstlisting}
P<s,t,u,v>:=ProjectiveSpace(Rationals(),3);
A:=48*s^4 + 124*s^3*t - 288*s^2*t^2 - 124*s*t^3 + 48*t^4-2*u^4;
B:=-31*s^4 + 192*s^3*t + 186*s^2*t^2 - 192*s*t^3 - 31*t^4-v^4;
S:=Scheme(P,[A,B]);
IsLocallySolvable(S,5);
\end{lstlisting}

{\bf Case 5:} $2u^4+v^4i=(2+i)(22-13i)(s+ti)^4$. Equating the real and imaginary parts on both sides of this equation gives
\begin{equation}\label{3265-5}
\begin{cases}
57s^4 + 16s^3t - 342s^2t^2 - 16st^3 + 57t^4 - 2u^4=0,\\
-4s^4 + 228s^3t + 24s^2t^2 - 228st^3 - 4t^4 - v^4=0.
\end{cases}
\end{equation}
The scheme defined by system \eqref{3265-5} is locally insoluble at $2$.
%\begin{lstlisting}
%_<x>:=PolynomialRing(Rationals());
%k<i>:=NumberField(x^2+1);
%K<s,t>:=PolynomialRing(k,2);
%e:=0;
%F:=i^e*(2+i)*(22-13*i)*(s+t*i)^4;
%F;
%L<s,t>:=PolynomialRing(Rationals(),2);
%_<i>:=PolynomialRing(L);
%F:=(-4*i + 57)*s^4 + (228*i + 16)*s^3*t + (24*i - 342)*s^2*t^2 + (-228*i -
%    16)*s*t^3 + (-4*i + 57)*t^4;
%F;
%A:=57*s^4 + 16*s^3*t - 342*s^2*t^2 - 16*s*t^3 + 57*t^4;
%B:=-4*s^4 + 228*s^3*t + 24*s^2*t^2 - 228*s*t^3 - 4*t^4;
%F-A-i*B;

\begin{lstlisting}
P<s,t,u,v>:=ProjectiveSpace(Rationals(),3);
A:=57*s^4 + 16*s^3*t - 342*s^2*t^2 - 16*s*t^3 + 57*t^4-2*u^4;
B:=-4*s^4 + 228*s^3*t + 24*s^2*t^2 - 228*s*t^3 - 4*t^4-v^4;
S:=Scheme(P,[A,B]);
IsLocallySolvable(S,2);
\end{lstlisting}

{\bf Case 6:} $2u^4+v^4i=i(2+i)(22-13i)(s+ti)^4$. Equating the real and imaginary parts on both sides of this equation gives
\begin{equation}\label{3265-6}
\begin{cases}
4s^4 - 228s^3t - 24s^2t^2 + 228st^3 + 4t^4 - 2u^4=0,\\
57s^4 + 16s^3t - 342s^2t^2 - 16st^3 + 57t^4 - v^4=0.
\end{cases}
\end{equation}
The scheme defined by system \eqref{3265-6} is locally insoluble at $2$.

%\begin{lstlisting}
%_<x>:=PolynomialRing(Rationals());
%k<i>:=NumberField(x^2+1);
%K<s,t>:=PolynomialRing(k,2);
%e:=1;
%F:=i^e*(2+i)*(22-13*i)*(s+t*i)^4;
%F;
%L<s,t>:=PolynomialRing(Rationals(),2);
%_<i>:=PolynomialRing(L);
%F:=(57*i + 4)*s^4 + (16*i - 228)*s^3*t + (-342*i - 24)*s^2*t^2 + (-16*i +
%    228)*s*t^3 + (57*i + 4)*t^4;
%F;
%A:=4*s^4 - 228*s^3*t -24*s^2*t^2 + 228*s*t^3 + 4*t^4;
%B:=57*s^4 + 16*s^3*t - 342*s^2*t^2 - 16*s*t^3 + 57*t^4;
%F-A-i*B;

\begin{lstlisting}
P<s,t,u,v>:=ProjectiveSpace(Rationals(),3);
A:=4*s^4 - 228*s^3*t -24*s^2*t^2 + 228*s*t^3 + 4*t^4-2*u^4;
B:=57*s^4 + 16*s^3*t - 342*s^2*t^2 - 16*s*t^3 + 57*t^4-v^4;
S:=Scheme(P,[A,B]);
IsLocallySolvable(S,2);
\end{lstlisting}

{\bf Case 7:} $2u^4+v^4i=i^2(2+i)(22-13i)(s+ti)^4$. Equating the real and imaginary parts on both sides of this equation gives
\begin{equation}\label{3265-7}
\begin{cases}
-57s^4 - 16s^3t + 342s^2t^2 + 16st^3 - 57t^4 - 2u^4=0,\\
4s^4 - 228s^3t - 24s^2t^2 + 228st^3 + 4t^4 - v^4=0.
\end{cases}
\end{equation}
The scheme defined by system \eqref{3265-7} is locally insoluble at $2$.

%\begin{lstlisting}
%_<x>:=PolynomialRing(Rationals());
%k<i>:=NumberField(x^2+1);
%K<s,t>:=PolynomialRing(k,2);
%e:=2;
%F:=i^e*(2+i)*(22-13*i)*(s+t*i)^4;
%F;
%L<s,t>:=PolynomialRing(Rationals(),2);
%_<i>:=PolynomialRing(L);
%F:=(4*i - 57)*s^4 + (-228*i - 16)*s^3*t + (-24*i + 342)*s^2*t^2 + (228*i +
%    16)*s*t^3 + (4*i - 57)*t^4;
%F;
%A:=- 57*s^4 - 16*s^3*t +342*s^2*t^2 + 16*s*t^3 - 57*t^4;
%B:=4*s^4 - 228*s^3*t - 24*s^2*t^2 + 228*s*t^3 + 4*t^4;
%F-A-i*B;

\begin{lstlisting}
P<s,t,u,v>:=ProjectiveSpace(Rationals(),3);
A:=- 57*s^4 - 16*s^3*t +342*s^2*t^2 + 16*s*t^3 - 57*t^4-2*u^4;
B:=4*s^4 - 228*s^3*t - 24*s^2*t^2 + 228*s*t^3 + 4*t^4-v^4;
S:=Scheme(P,[A,B]);
IsLocallySolvable(S,2);
\end{lstlisting}

{\bf Case 8:} $2u^4+v^4i=i^3(2+i)(22-13i)(s+ti)^4$. Equating the real and imaginary parts on both sides of this equation gives
\begin{equation}\label{3265-8}
\begin{cases}
-4s^4 + 228s^3t + 24s^2t^2 - 228st^3 - 4t^4 - 2u^4=0,\\
-57s^4 - 16s^3t + 342s^2t^2 + 16st^3 - 57t^4 - v^4=0.
\end{cases}
\end{equation}
The scheme defined by system \eqref{3265-8} is locally insoluble at $2$.
%\begin{lstlisting}
%_<x>:=PolynomialRing(Rationals());
%k<i>:=NumberField(x^2+1);
%K<s,t>:=PolynomialRing(k,2);
%e:=3;
%F:=i^e*(2+i)*(22-13*i)*(s+t*i)^4;
%F;
%L<s,t>:=PolynomialRing(Rationals(),2);
%_<i>:=PolynomialRing(L);
%F:=(-57*i - 4)*s^4 + (-16*i + 228)*s^3*t + (342*i + 24)*s^2*t^2 + (16*i -
%    228)*s*t^3 + (-57*i - 4)*t^4;
%F;
%A:=- 4*s^4 + 228*s^3*t + 24*s^2*t^2 - 228*s*t^3 - 4*t^4;
%B:=-57*s^4 - 16*s^3*t + 342*s^2*t^2 + 16*s*t^3 - 57*t^4;
%F-A-i*B;
\begin{lstlisting}
P<s,t,u,v>:=ProjectiveSpace(Rationals(),3);
A:=- 4*s^4 + 228*s^3*t + 24*s^2*t^2 - 228*s*t^3 - 4*t^4-2*u^4;
B:=-57*s^4 - 16*s^3*t + 342*s^2*t^2 + 16*s*t^3 - 57*t^4-v^4;
S:=Scheme(P,[A,B]);
IsLocallySolvable(S,2);
\end{lstlisting}

\subsection{The case of $n=3906$.} %\[n=3906\text{ (I) }\]
According to Table \ref{tb:factorisation}, in the case $n=3906$ it remains to
show that equation 
\begin{equation}\label{3906.1}
15376 u^8 + 3969 v^8=w^4
\end{equation}
has no integer solutions satisfying \eqref{eq:cond}, which in this case reduces to \[\gcd(124u,63v)=\gcd(124u,w)=\gcd(63v,w)=1,\]
and to show that equation 
\begin{equation}\label{3906.2}
63504 u^8 + 961 v^8=w^4
\end{equation}
has no integer solutions satisfying \eqref{eq:cond}, which in this case reduces to \[\gcd(252u,31v)=\gcd(252u,w)=\gcd(31v,w)=1.\]

\subsubsection{The case of \eqref{3906.1}.}
We first consider \eqref{3906.1}, which we can write as
\begin{equation} (124u^4+63v^4i)(124u^4-63v^4i)=w^4.
\end{equation}
This implies that there exist integers $s,t$ such that 
\[124u^4+63v^4i=i^{\epsilon}(s+ti)^4,\]
with $\epsilon \in \{0,1,2,3\}$.

{\bf Case 1:} $124u^4+63v^4i=(s+ti)^4$. Equating the real and imaginary parts on both sides of this equation gives
\begin{equation}\label{3906.1-1}
\begin{cases}
s^4 - 6s^2t^2 + t^4 - 124u^4=0,\\
4s^3t - 4st^3 - 63v^4=0.
\end{cases}
\end{equation}
The scheme defined by system \eqref{3906.1-1} is locally insoluble at $5$.

%\begin{lstlisting}
%    _<x>:=PolynomialRing(Rationals());
%k<i>:=NumberField(x^2+1);
%K<s,t>:=PolynomialRing(k,2);
%e:=0;
%F:=i^e*(s+t*i)^4;
%F;
%L<s,t>:=PolynomialRing(Rationals(),2);
%_<i>:=PolynomialRing(L);
%F:=s^4 + 4*i*s^3*t - 6*s^2*t^2 - 4*i*s*t^3 + t^4;
%F;
%A:=s^4 - 6*s^2*t^2 + t^4;
%B:=4*s^3*t - 4*s*t^3;
%F-A-i*B;
\begin{lstlisting}
P<s,t,u,v>:=ProjectiveSpace(Rationals(),3);
A:=s^4 - 6*s^2*t^2 + t^4-124*u^4;
B:=4*s^3*t - 4*s*t^3-63*v^4;
S:=Scheme(P,[A,B]);
IsLocallySolvable(S,5);
\end{lstlisting}

{\bf Case 2:} $124u^4+63v^4i=i(s+ti)^4$. Equating the real and imaginary parts on both sides of this equation gives
\begin{equation}\label{3906.1-2}
\begin{cases}
-4s^3t + 4st^3 - 124u^4=0,\\
s^4 - 6s^2t^2 + t^4 - 63v^4=0.
\end{cases}
\end{equation}
The scheme defined by system \eqref{3906.1-2} is locally insoluble at $2$.

%
%\begin{lstlisting}
%    _<x>:=PolynomialRing(Rationals());
%k<i>:=NumberField(x^2+1);
%K<s,t>:=PolynomialRing(k,2);
%e:=1;
%F:=i^e*(s+t*i)^4;
%F;
%L<s,t>:=PolynomialRing(Rationals(),2);
%_<i>:=PolynomialRing(L);
%F:=i*s^4 - 4*s^3*t - 6*i*s^2*t^2 + 4*s*t^3 + i*t^4;
%F;
%A:= - 4*s^3*t + 4*s*t^3;
%B:=s^4 - 6*s^2*t^2 + t^4;
%F-A-i*B;

\begin{lstlisting}
P<s,t,u,v>:=ProjectiveSpace(Rationals(),3);
A:=- 4*s^3*t + 4*s*t^3-124*u^4;
B:=s^4 - 6*s^2*t^2 + t^4-63*v^4;
S:=Scheme(P,[A,B]);
IsLocallySolvable(S,2);
\end{lstlisting}

{\bf Case 3:} $124u^4+63v^4i=i^2(s+ti)^4$. Equating the real and imaginary parts on both sides of this equation gives
\begin{equation}\label{3906.1-3}
\begin{cases}
-s^4 + 6s^2t^2 - t^4 -124u^4=0,\\
-4s^3t + 4st^3 - 63v^4=0.
\end{cases}
\end{equation}
The scheme defined by system \eqref{3906.1-3} is locally insoluble at $2$.

%\begin{lstlisting}
%    _<x>:=PolynomialRing(Rationals());
%k<i>:=NumberField(x^2+1);
%K<s,t>:=PolynomialRing(k,2);
%e:=2;
%F:=i^e*(s+t*i)^4;
%F;
%L<s,t>:=PolynomialRing(Rationals(),2);
%_<i>:=PolynomialRing(L);
%F:=-s^4 - 4*i*s^3*t + 6*s^2*t^2 + 4*i*s*t^3 - t^4;
%F;
%A:=- s^4 + 6*s^2*t^2 - t^4;
%B:=-4*s^3*t + 4*s*t^3;
%F-A-i*B;
\begin{lstlisting}
P<s,t,u,v>:=ProjectiveSpace(Rationals(),3);
A:=- s^4 + 6*s^2*t^2 - t^4-124*u^4;
B:=-4*s^3*t + 4*s*t^3-63*v^4;
S:=Scheme(P,[A,B]);
IsLocallySolvable(S,2);
\end{lstlisting}

{\bf Case 4:} $124u^4+63v^4i=i^3(s+ti)^4$. Equating the real and imaginary parts on both sides of this equation gives
\begin{equation}\label{3906.1-4}
\begin{cases}
4s^3t - 4st^3 - 124u^4=0,\\
-s^4 + 6s^2t^2 - t^4 - 63v^4=0.
\end{cases}
\end{equation}
The scheme defined by system \eqref{3906.1-4} is locally insoluble at $3$.

%\begin{lstlisting}
%    _<x>:=PolynomialRing(Rationals());
%k<i>:=NumberField(x^2+1);
%K<s,t>:=PolynomialRing(k,2);
%e:=3;
%F:=i^e*(s+t*i)^4;
%F;
%L<s,t>:=PolynomialRing(Rationals(),2);
%_<i>:=PolynomialRing(L);
%F:=-i*s^4 + 4*s^3*t + 6*i*s^2*t^2 - 4*s*t^3 - i*t^4;
%F;
%A:=4*s^3*t - 4*s*t^3;
%B:=-s^4 + 6*s^2*t^2 - t^4;
%F-A-i*B;
\begin{lstlisting}
P<s,t,u,v>:=ProjectiveSpace(Rationals(),3);
A:=4*s^3*t - 4*s*t^3-124*u^4;
B:=-s^4 + 6*s^2*t^2 - t^4-63*v^4;
S:=Scheme(P,[A,B]);
IsLocallySolvable(S,3);
\end{lstlisting}

\subsubsection{\bf{The case of \eqref{3906.2}}} %\[n=3906\text{ (II) }\]
We now consider \eqref{3906.2}, which we can write as
\begin{equation} (252u^4+31v^4i)(252u^4-31v^4i)=w^4.
\end{equation}
This implies that there exist integers $s,t$ such that 
\[252u^4+31v^4i=i^{\epsilon}(s+ti)^4,\]
with $\epsilon \in \{0,1,2,3\}$.

{\bf Case 1:} $252u^4+31v^4i=(s+ti)^4$. Equating the real and imaginary parts on both sides of this equation gives
\begin{equation}\label{3906.2-1}
\begin{cases}
s^4 - 6s^2t^2 + t^4 - 252u^4=0,\\
4s^3t - 4st^3 - 31v^4=0.
\end{cases}
\end{equation}
The scheme defined by system \eqref{3906.2-1} is locally insoluble at $5$.

%
%\begin{lstlisting}
%    _<x>:=PolynomialRing(Rationals());
%k<i>:=NumberField(x^2+1);
%K<s,t>:=PolynomialRing(k,2);
%e:=0;
%F:=i^e*(s+t*i)^4;
%F;
%L<s,t>:=PolynomialRing(Rationals(),2);
%_<i>:=PolynomialRing(L);
%F:=s^4 + 4*i*s^3*t - 6*s^2*t^2 - 4*i*s*t^3 + t^4;
%F;
%A:=s^4 - 6*s^2*t^2 + t^4;
%B:=4*s^3*t - 4*s*t^3;
%F-A-i*B;
\begin{lstlisting}
P<s,t,u,v>:=ProjectiveSpace(Rationals(),3);
A:=s^4 - 6*s^2*t^2 + t^4-252*u^4;
B:=4*s^3*t - 4*s*t^3-31*v^4;
S:=Scheme(P,[A,B]);
IsLocallySolvable(S,5);
\end{lstlisting}

{\bf Case 2:} $252u^4+31v^4i=i(s+ti)^4$. Equating the real and imaginary parts on both sides of this equation gives
\begin{equation}\label{3906.2-2}
\begin{cases}
-4s^3t + 4st^3 - 252u^4=0,\\
s^4 - 6s^2t^2 + t^4 - 31v^4=0.
\end{cases}
\end{equation}
The scheme defined by system \eqref{3906.2-2} is locally insoluble at $2$.

%
%\begin{lstlisting}
%    _<x>:=PolynomialRing(Rationals());
%k<i>:=NumberField(x^2+1);
%K<s,t>:=PolynomialRing(k,2);
%e:=1;
%F:=i^e*(s+t*i)^4;
%F;
%L<s,t>:=PolynomialRing(Rationals(),2);
%_<i>:=PolynomialRing(L);
%F:=i*s^4 - 4*s^3*t - 6*i*s^2*t^2 + 4*s*t^3 + i*t^4;
%F;
%A:= - 4*s^3*t + 4*s*t^3;
%B:=s^4 - 6*s^2*t^2 + t^4;
%F-A-i*B;
\begin{lstlisting}
P<s,t,u,v>:=ProjectiveSpace(Rationals(),3);
A:=- 4*s^3*t + 4*s*t^3-252*u^4;
B:=s^4 - 6*s^2*t^2 + t^4-31*v^4;
S:=Scheme(P,[A,B]);
IsLocallySolvable(S,2);
\end{lstlisting}

{\bf Case 3:} $124u^4+63v^4i=i^2(s+ti)^4$. Equating the real and imaginary parts on both sides of this equation gives
\begin{equation}\label{3906.2-3}
\begin{cases}
-s^4 + 6s^2t^2 - t^4 -252u^4=0,\\
-4s^3t + 4st^3 - 31v^4=0.
\end{cases}
\end{equation}
The scheme defined by system \eqref{3906.2-3} is locally insoluble at $2$.
%
%\begin{lstlisting}
%    _<x>:=PolynomialRing(Rationals());
%k<i>:=NumberField(x^2+1);
%K<s,t>:=PolynomialRing(k,2);
%e:=2;
%F:=i^e*(s+t*i)^4;
%F;
%L<s,t>:=PolynomialRing(Rationals(),2);
%_<i>:=PolynomialRing(L);
%F:=-s^4 - 4*i*s^3*t + 6*s^2*t^2 + 4*i*s*t^3 - t^4;
%F;
%A:=- s^4 + 6*s^2*t^2 - t^4;
%B:=-4*s^3*t + 4*s*t^3;
%F-A-i*B;
\begin{lstlisting}
P<s,t,u,v>:=ProjectiveSpace(Rationals(),3);
A:=- s^4 + 6*s^2*t^2 - t^4-252*u^4;
B:=-4*s^3*t + 4*s*t^3-31*v^4;
S:=Scheme(P,[A,B]);
IsLocallySolvable(S,2);
\end{lstlisting}

{\bf Case 4:} $252u^4+31v^4i=i^3(s+ti)^4$. Equating the real and imaginary parts on both sides of this equation gives
\begin{equation}\label{3906.2-4}
\begin{cases}
4s^3t - 4st^3 - 252u^4=0,\\
-s^4 + 6s^2t^2 - t^4 - 31v^4=0.
\end{cases}
\end{equation}
The scheme defined by system \eqref{3906.2-4} is locally insoluble at $5$.

%\begin{lstlisting}
%    _<x>:=PolynomialRing(Rationals());
%k<i>:=NumberField(x^2+1);
%K<s,t>:=PolynomialRing(k,2);
%e:=3;
%F:=i^e*(s+t*i)^4;
%F;
%L<s,t>:=PolynomialRing(Rationals(),2);
%_<i>:=PolynomialRing(L);
%F:=-i*s^4 + 4*s^3*t + 6*i*s^2*t^2 - 4*s*t^3 - i*t^4;
%F;
%A:=4*s^3*t - 4*s*t^3;
%B:=-s^4 + 6*s^2*t^2 - t^4;
%F-A-i*B;
\begin{lstlisting}
P<s,t,u,v>:=ProjectiveSpace(Rationals(),3);
A:=4*s^3*t - 4*s*t^3-252*u^4;
B:=-s^4 + 6*s^2*t^2 - t^4-31*v^4;
S:=Scheme(P,[A,B]);
IsLocallySolvable(S,5);
\end{lstlisting}

\subsection{The case of $n=4100$.} %\[n=4100.\]
According to Table \ref{tb:factorisation}, in the case $n=4100$ it remains to
show that equation
\begin{equation}\label{4100}
64 u^8 + v^8=1025  w^4
\end{equation}
has no integer solutions satisfying \eqref{eq:cond}, which in this case reduces to \[\gcd(8u,v)=\gcd(8u,1025w)=\gcd(v,1025w)=1.\]
Write \eqref{4100} as
\begin{equation} (8u^4+v^4i)(8u^4-v^4i)=(2+i)^2(2-i)^2(5+4i)(5-4i)w^4.
\end{equation}
Since $5\nmid u,v$, we have 
\[\begin{split}8u^4+v^4i&\equiv -2+i\pmod{5}\\
&\equiv 0\pmod{2-i}\\
&\not\equiv 0\pmod{2+i}.\end{split}\]
Hence, $(2-i)^2|8u^4+v^4i$. This implies that there exist integers $s,t$ such that 
\[8u^4+v^4i=i^{\epsilon}(2-i)^2(5\pm 4i)(s+ti)^4,\]
with $\epsilon \in \{0,1,2,3\}$.

{\bf Case 1:} $8u^4+v^4i=(2-i)^2(5+4i)(s+ti)^4$. Equating the real and imaginary parts on both sides of this equation gives
\begin{equation}\label{4100-1}
\begin{cases}
31s^4 + 32s^3t -186s^2t^2 - 32st^3 + 31t^4-8u^4=0,\\
-8s^4 + 124s^3t + 48s^2t^2 - 124st^3 - 8t^4-v^4=0.
\end{cases}
\end{equation}
The scheme defined by system \eqref{4100-1} is locally insoluble at $2$.

% _<x>:=PolynomialRing(Rationals());
%k<i>:=NumberField(x^2+1);
%K<s,t>:=PolynomialRing(k,2);
%e:=0;
%F:=i^e*(2-i)^2*(5+4*i)*(s+t*i)^4;
%F;
%L<s,t>:=PolynomialRing(Rationals(),2);
%_<i>:=PolynomialRing(L);
%F:=(-8*i + 31)*s^4 + (124*i + 32)*s^3*t + (48*i - 186)*s^2*t^2 + (-124*i -
%    32)*s*t^3 + (-8*i + 31)*t^4;
%F;
%A:=31*s^4 + 32*s^3*t -186*s^2*t^2 - 32*s*t^3 + 31*t^4;
%B:=-8*s^4 + 124*s^3*t + 48*s^2*t^2 - 124*s*t^3 - 8*t^4;
%F-A-i*B;

\begin{lstlisting}
P<s,t,u,v>:=ProjectiveSpace(Rationals(),3);
A:=31*s^4 + 32*s^3*t -186*s^2*t^2 - 32*s*t^3 + 31*t^4-8*u^4;
B:=-8*s^4 + 124*s^3*t + 48*s^2*t^2 - 124*s*t^3 - 8*t^4-v^4;
S:=Scheme(P,[A,B]);
IsLocallySolvable(S,2);
\end{lstlisting}

{\bf Case 2:} $8u^4+v^4i=i(2-i)^2(5+4i)(s+ti)^4$. Equating the real and imaginary parts on both sides of this equation gives
\begin{equation}\label{4100-2}
\begin{cases}
8s^4 - 124s^3t -48s^2t^2 + 124st^3 + 8t^4-8u^4=0.\\
31s^4 + 32s^3t - 186s^2t^2 - 32st^3 + 31t^4-v^4=0.
\end{cases}
\end{equation}
The scheme defined by system \eqref{4100-2} is locally insoluble at $2$.

% _<x>:=PolynomialRing(Rationals());
%k<i>:=NumberField(x^2+1);
%K<s,t>:=PolynomialRing(k,2);
%e:=1;
%F:=i^e*(2-i)^2*(5+4*i)*(s+t*i)^4;
%F;
%L<s,t>:=PolynomialRing(Rationals(),2);
%_<i>:=PolynomialRing(L);
%F:=(31*i + 8)*s^4 + (32*i - 124)*s^3*t + (-186*i - 48)*s^2*t^2 + (-32*i +
%    124)*s*t^3 + (31*i + 8)*t^4;
%F;
%A:=8*s^4 - 124*s^3*t -48*s^2*t^2 + 124*s*t^3 + 8*t^4;
%B:=31*s^4 + 32*s^3*t - 186*s^2*t^2 - 32*s*t^3 + 31*t^4;
%F-A-i*B;

\begin{lstlisting}
P<s,t,u,v>:=ProjectiveSpace(Rationals(),3);
A:=8*s^4 - 124*s^3*t -48*s^2*t^2 + 124*s*t^3 + 8*t^4-8*u^4;
B:=31*s^4 + 32*s^3*t - 186*s^2*t^2 - 32*s*t^3 + 31*t^4-v^4;
S:=Scheme(P,[A,B]);
IsLocallySolvable(S,2);
\end{lstlisting}

{\bf Case 3:} $8u^4+v^4i=i^2(2-i)^2(5+4i)(s+ti)^4$. Equating the real and imaginary parts on both sides of this equation gives
\begin{equation}\label{4100-3}
\begin{cases}
- 31s^4 - 32s^3t +186s^2t^2 + 32st^3 - 31t^4-8u^4=0,\\
8s^4 - 124s^3t - 48s^2t^2 + 124st^3 + 8t^4-v^4=0.
\end{cases}
\end{equation}
The scheme defined by system \eqref{4100-3} is locally insoluble at $2$.

% _<x>:=PolynomialRing(Rationals());
%k<i>:=NumberField(x^2+1);
%K<s,t>:=PolynomialRing(k,2);
%e:=2;
%F:=i^e*(2-i)^2*(5+4*i)*(s+t*i)^4;
%F;
%L<s,t>:=PolynomialRing(Rationals(),2);
%_<i>:=PolynomialRing(L);
%F:=(8*i - 31)*s^4 + (-124*i - 32)*s^3*t + (-48*i + 186)*s^2*t^2 + (124*i +
%    32)*s*t^3 + (8*i - 31)*t^4;
%F;
%A:=- 31*s^4 - 32*s^3*t +186*s^2*t^2 + 32*s*t^3 - 31*t^4;
%B:=8*s^4 - 124*s^3*t - 48*s^2*t^2 + 124*s*t^3 + 8*t^4;
%F-A-i*B;

\begin{lstlisting}
P<s,t,u,v>:=ProjectiveSpace(Rationals(),3);
A:=- 31*s^4 - 32*s^3*t +186*s^2*t^2 + 32*s*t^3 - 31*t^4-8*u^4;
B:=8*s^4 - 124*s^3*t - 48*s^2*t^2 + 124*s*t^3 + 8*t^4-v^4;
S:=Scheme(P,[A,B]);
IsLocallySolvable(S,2);
\end{lstlisting}

{\bf Case 4:} $8u^4+v^4i=i^3(2-i)^2(5+4i)(s+ti)^4$. Equating the real and imaginary parts on both sides of this equation gives
\begin{equation}\label{4100-4}
\begin{cases}
- 8s^4 + 124s^3t +48s^2t^2 - 124st^3 - 8t^4-8u^4=0,\\
-31s^4 - 32s^3t + 186s^2t^2 + 32st^3 - 31t^4-v^4=0.
\end{cases}
\end{equation}
The scheme defined by system \eqref{4100-4} is locally insoluble at $5$.

% _<x>:=PolynomialRing(Rationals());
%k<i>:=NumberField(x^2+1);
%K<s,t>:=PolynomialRing(k,2);
%e:=3;
%F:=i^e*(2-i)^2*(5+4*i)*(s+t*i)^4;
%F;
%L<s,t>:=PolynomialRing(Rationals(),2);
%_<i>:=PolynomialRing(L);
%F:=(-31*i - 8)*s^4 + (-32*i + 124)*s^3*t + (186*i + 48)*s^2*t^2 + (32*i -
%    124)*s*t^3 + (-31*i - 8)*t^4;
%F;
%A:=- 8*s^4 + 124*s^3*t +48*s^2*t^2 - 124*s*t^3 - 8*t^4;
%B:=-31*s^4 - 32*s^3*t + 186*s^2*t^2 + 32*s*t^3 - 31*t^4;
%F-A-i*B;

\begin{lstlisting}
P<s,t,u,v>:=ProjectiveSpace(Rationals(),3);
A:=- 8*s^4 + 124*s^3*t +48*s^2*t^2 - 124*s*t^3 - 8*t^4-8*u^4;
B:=-31*s^4 - 32*s^3*t + 186*s^2*t^2 + 32*s*t^3 - 31*t^4-v^4;
S:=Scheme(P,[A,B]);
IsLocallySolvable(S,5);
\end{lstlisting}

{\bf Case 5:} $8u^4+v^4i=(2-i)^2(5-4i)(s+ti)^4$. Equating the real and imaginary parts on both sides of this equation gives
\begin{equation}\label{4100-5}
\begin{cases}
- s^4 + 128s^3t +6s^2t^2 - 128st^3 - t^4-8u^4=0,\\
-32s^4 - 4s^3t + 192s^2t^2 + 4st^3 - 32t^4-v^4=0.
\end{cases}
\end{equation}
The scheme defined by system \eqref{4100-5} is locally insoluble at $2$.

% _<x>:=PolynomialRing(Rationals());
%k<i>:=NumberField(x^2+1);
%K<s,t>:=PolynomialRing(k,2);
%e:=0;
%F:=i^e*(2-i)^2*(5-4*i)*(s+t*i)^4;
%F;
%L<s,t>:=PolynomialRing(Rationals(),2);
%_<i>:=PolynomialRing(L);
%F:=(-32*i - 1)*s^4 + (-4*i + 128)*s^3*t + (192*i + 6)*s^2*t^2 + (4*i - 128)*s*t^3 +
%    (-32*i - 1)*t^4;
%F;
%A:=- s^4 + 128*s^3*t +6*s^2*t^2 - 128*s*t^3 - t^4;
%B:=-32*s^4 - 4*s^3*t + 192*s^2*t^2 + 4*s*t^3 - 32*t^4;
%F-A-i*B;

\begin{lstlisting}
P<s,t,u,v>:=ProjectiveSpace(Rationals(),3);
A:=- s^4 + 128*s^3*t +6*s^2*t^2 - 128*s*t^3 - t^4-8*u^4;
B:=-32*s^4 - 4*s^3*t + 192*s^2*t^2 + 4*s*t^3 - 32*t^4-v^4;
S:=Scheme(P,[A,B]);
IsLocallySolvable(S,2);
\end{lstlisting}

{\bf Case 6:} $8u^4+v^4i=i(2-i)^2(5-4i)(s+ti)^4$. Equating the real and imaginary parts on both sides of this equation gives
\begin{equation}\label{4100-6}
\begin{cases}
32s^4 + 4s^3t -192s^2t^2 - 4st^3 + 32t^4-8u^4=0,\\
-s^4 + 128s^3t + 6s^2t^2 - 128st^3 - t^4-v^4=0.
\end{cases}
\end{equation}
The scheme defined by system \eqref{4100-6} is locally insoluble at $2$.

% _<x>:=PolynomialRing(Rationals());
%k<i>:=NumberField(x^2+1);
%K<s,t>:=PolynomialRing(k,2);
%e:=1;
%F:=i^e*(2-i)^2*(5-4*i)*(s+t*i)^4;
%F;
%L<s,t>:=PolynomialRing(Rationals(),2);
%_<i>:=PolynomialRing(L);
%F:=(-i + 32)*s^4 + (128*i + 4)*s^3*t + (6*i - 192)*s^2*t^2 + (-128*i - 4)*s*t^3 +
%    (-i + 32)*t^4;
%F;
%A:=32*s^4 + 4*s^3*t -192*s^2*t^2 - 4*s*t^3 + 32*t^4;
%B:=-s^4 + 128*s^3*t + 6*s^2*t^2 - 128*s*t^3 - t^4;
%F-A-i*B;

\begin{lstlisting}
P<s,t,u,v>:=ProjectiveSpace(Rationals(),3);
A:=32*s^4 + 4*s^3*t -192*s^2*t^2 - 4*s*t^3 + 32*t^4-8*u^4;
B:=-s^4 + 128*s^3*t + 6*s^2*t^2 - 128*s*t^3 - t^4-v^4;
S:=Scheme(P,[A,B]);
IsLocallySolvable(S,2);
\end{lstlisting}

{\bf Case 7:} $8u^4+v^4i=i^2(2-i)^2(5-4i)(s+ti)^4$. Equating the real and imaginary parts on both sides of this equation gives
\begin{equation}\label{4100-7}
\begin{cases}
s^4 - 128s^3t - 6s^2t^2 + 128st^3 + t^4 - 8u^4=0,\\
32s^4 +4s^3t - 192s^2t^2 -4st^3 + 32t^4 - v^4=0.
\end{cases}
\end{equation}
The scheme defined by system \eqref{4100-7} is locally insoluble at $2$.

% _<x>:=PolynomialRing(Rationals());
%k<i>:=NumberField(x^2+1);
%K<s,t>:=PolynomialRing(k,2);
%e:=2;
%F:=i^e*(2-i)^2*(5-4*i)*(s+t*i)^4;
%F;
%L<s,t>:=PolynomialRing(Rationals(),2);
%_<i>:=PolynomialRing(L);
%F:=(32*i + 1)*s^4 + (4*i - 128)*s^3*t + (-192*i - 6)*s^2*t^2 + (-4*i + 128)*s*t^3 +
%    (32*i + 1)*t^4;
%F;
%A:=s^4 - 128*s^3*t - 6*s^2*t^2 + 128*s*t^3 + t^4;
%B:=32*s^4 + 4*s^3*t - 192*s^2*t^2 - 4*s*t^3 + 32*t^4;
%F-A-i*B;

\begin{lstlisting}
P<s,t,u,v>:=ProjectiveSpace(Rationals(),3);
A:=s^4 - 128*s^3*t - 6*s^2*t^2 + 128*s*t^3 + t^4-8*u^4;
B:=32*s^4 + 4*s^3*t - 192*s^2*t^2 - 4*s*t^3 + 32*t^4-v^4;
S:=Scheme(P,[A,B]);
IsLocallySolvable(S,2);
\end{lstlisting}

{\bf Case 8:} $8u^4+v^4i=i^3(2-i)^2(5-4i)(s+ti)^4$. Equating the real and imaginary parts on both sides of this equation gives
\begin{equation}\label{4100-8}
\begin{cases}
- 32s^4 - 4s^3t + 192s^2t^2 + 4st^3 - 32t^4-8u^4 =0,\\
s^4 - 128s^3t - 6s^2t^2 + 128st^3 + t^4-v^4=0.
\end{cases}
\end{equation}
The scheme defined by system \eqref{4100-8} is locally insoluble at $3$.

% _<x>:=PolynomialRing(Rationals());
%k<i>:=NumberField(x^2+1);
%K<s,t>:=PolynomialRing(k,2);
%e:=3;
%F:=i^e*(2-i)^2*(5-4*i)*(s+t*i)^4;
%F;
%L<s,t>:=PolynomialRing(Rationals(),2);
%_<i>:=PolynomialRing(L);
%F:=(i - 32)*s^4 + (-128*i - 4)*s^3*t + (-6*i + 192)*s^2*t^2 + (128*i + 4)*s*t^3 +
%    (i - 32)*t^4;
%F;
%A:=- 32*s^4 - 4*s^3*t + 192*s^2*t^2 + 4*s*t^3 - 32*t^4;
%B:=s^4 - 128*s^3*t - 6*s^2*t^2 + 128*s*t^3 + t^4;
%F-A-i*B;

\begin{lstlisting}
P<s,t,u,v>:=ProjectiveSpace(Rationals(),3);
A:=- 32*s^4 - 4*s^3*t + 192*s^2*t^2 + 4*s*t^3 - 32*t^4-8*u^4;
B:=s^4 - 128*s^3*t - 6*s^2*t^2 + 128*s*t^3 + t^4-v^4;
S:=Scheme(P,[A,B]);
IsLocallySolvable(S,3);
\end{lstlisting}

\subsection{The case of $n=4165$.} %\[n=4165 \quad (I)\]
According to Table \ref{tb:factorisation}, in the case $n=4165$ it remains to
show that equation 
\begin{equation}\label{4165.1}
2401 u^8 + 4 v^8=85  w^4
\end{equation}
has no integer solutions satisfying \eqref{eq:cond}, which in this case reduces to  \begin{equation}\label{E22}\gcd(7u,2v)=\gcd(7u,85w)=\gcd(2v,85w)=1,\end{equation} 
and to show that equation 
\begin{equation}\label{4165.2}
9604 u^8 + v^8=85  w^4
\end{equation}
has no integer solutions satisfying \eqref{eq:cond}, which in this case reduces to \[\gcd(98u,v)=\gcd(98u,85w)=\gcd(v,85w)=1.\]

\subsubsection{The case of \eqref{4165.1}.}
We first consider \eqref{4165.1}, which we can write \eqref{4165.1} as 
\begin{equation}\label{4165-1} (49u^4+2v^4i)(49u^4-2iv^4)=(2+i)(2-i)(4+i)(4-i)w^4.
\end{equation}
By \eqref{E22} we have that $5\nmid uv$, hence $u^4\equiv v^4\equiv 1$ (mod $5$). Therefore, \[\begin{cases}49u^4+2v^4i&\equiv -1+2i\pmod{5}\\
&\equiv 0\pmod{2+i}\\
&\not\equiv 0\pmod{2-i}.
\end{cases}.\]
Let $d\in \Z[i]$ such that $d|\gcd(49u^4+2v^4i,49u^4-2v^4i)$. Then $d|2\cdot 49u^4$,$d|4v^4$, $d|85w^4$. Combining this with \eqref{E22} shows that  $d|1$. This implies that there exist integers $s,t$ such that 
\[49u^4+2v^4i=i^{\epsilon}(2+i)(4\pm i)(s+ti)^4,\]
with $\epsilon\in \{0,1,2,3\}$.

{\bf Case 1:} $49u^4+2v^4i=(2+i)(4+i)(s+ti)^4$.  Equating the real and imaginary parts on both sides of this equation gives
\begin{equation}\label{4165.1-1}
\begin{cases}
7s^4 - 24s^3t - 42s^2t^2 + 24st^3 + 7t^4 - 49u^4=0,\\
6s^4 + 28s^3t - 36s^2t^2 - 28st^3 + 6t^4 - 2v^4=0.
\end{cases}
\end{equation}
The scheme defined by system \eqref{4165.1-1} is locally insoluble at $2$.

%\begin{lstlisting}
%    _<x>:=PolynomialRing(Rationals());
%k<i>:=NumberField(x^2+1);
%K<s,t>:=PolynomialRing(k,2);
%e:=0;
%F:=i^e*(2+i)*(4+i)*(s+t*i)^4;
%F;
%L<s,t>:=PolynomialRing(Rationals(),2);
%_<i>:=PolynomialRing(L);
%F:=(6*i + 7)*s^4 + (28*i - 24)*s^3*t + (-36*i - 42)*s^2*t^2 + (-28*i + 24)*s*t^3 +
%    (6*i + 7)*t^4;
%F;
%A:=7*s^4 - 24*s^3*t -42*s^2*t^2 + 24*s*t^3 + 7*t^4;
%B:=(6*s^4 + 28*s^3*t - 36*s^2*t^2 - 28*s*t^3 + 6*t^4);
%F-A-i*B;

\begin{lstlisting}
P<s,t,u,v>:=ProjectiveSpace(Rationals(),3);
A:=7*s^4 - 24*s^3*t -42*s^2*t^2 + 24*s*t^3 + 7*t^4-49*u^4;
B:=6*s^4 + 28*s^3*t - 36*s^2*t^2 - 28*s*t^3 + 6*t^4-2*v^4;
S:=Scheme(P,[A,B]);
IsLocallySolvable(S,2);
\end{lstlisting}

{\bf Case 2:} $49u^4+2v^4i=i(2+i)(4+i)(s+ti)^4$. Equating the real and imaginary parts on both sides of this equation gives
\begin{equation}\label{4165.1-2}
\begin{cases}
-6s^4 - 28s^3t + 36s^2t^2 + 28st^3 - 6t^4 - 49u^4=0,\\
7s^4 - 24s^3t - 42s^2t^2 + 24st^3 + 7t^4 - 2v^4=0.
\end{cases}
\end{equation}
The scheme defined by system \eqref{4165.1-2} is locally insoluble at $2$.

%\begin{lstlisting}
%    _<x>:=PolynomialRing(Rationals());
%k<i>:=NumberField(x^2+1);
%K<s,t>:=PolynomialRing(k,2);
%e:=1;
%F:=i^e*(2+i)*(4+i)*(s+t*i)^4;
%F;
%L<s,t>:=PolynomialRing(Rationals(),2);
%_<i>:=PolynomialRing(L);
%F:=(7*i - 6)*s^4 + (-24*i - 28)*s^3*t + (-42*i + 36)*s^2*t^2 + (24*i + 28)*s*t^3 +
%    (7*i - 6)*t^4;
%F;
%A:=- 6*s^4 - 28*s^3*t +36*s^2*t^2 + 28*s*t^3 - 6*t^4;
%B:=7*s^4 - 24*s^3*t - 42*s^2*t^2 + 24*s*t^3 + 7*t^4;
%F-A-i*B;
\begin{lstlisting}
P<s,t,u,v>:=ProjectiveSpace(Rationals(),3);
A:=- 6*s^4 - 28*s^3*t +36*s^2*t^2 + 28*s*t^3 - 6*t^4-49*u^4;
B:=7*s^4 - 24*s^3*t - 42*s^2*t^2 + 24*s*t^3 + 7*t^4-2*v^4;
S:=Scheme(P,[A,B]);
IsLocallySolvable(S,2);
\end{lstlisting}

{\bf Case 3:} $49u^4+2v^4i=i^2(2+i)(4+i)(s+ti)^4$. Equating the real and imaginary parts on both sides of this equation gives
\begin{equation}\label{4165.1-3}
\begin{cases}
-7s^4 + 24s^3t + 42s^2t^2 - 24st^3 - 7t^4 - 49u^4=0,\\
-6s^4 - 28s^3t + 36s^2t^2 + 28st^3 - 6t^4 - 2v^4=0.
\end{cases}
\end{equation}
The scheme defined by system \eqref{4165.1-3} is locally insoluble at $5$.

%\begin{lstlisting}
%    _<x>:=PolynomialRing(Rationals());
%k<i>:=NumberField(x^2+1);
%K<s,t>:=PolynomialRing(k,2);
%e:=2;
%F:=i^e*(2+i)*(4+i)*(s+t*i)^4;
%F;
%L<s,t>:=PolynomialRing(Rationals(),2);
%_<i>:=PolynomialRing(L);
%F:=(-6*i - 7)*s^4 + (-28*i + 24)*s^3*t + (36*i + 42)*s^2*t^2 + (28*i - 24)*s*t^3 +
%    (-6*i - 7)*t^4;
%F;
%A:=- 7*s^4 + 24*s^3*t + 42*s^2*t^2 - 24*s*t^3 - 7*t^4;
%B:=-6*s^4 - 28*s^3*t + 36*s^2*t^2 + 28*s*t^3 - 6*t^4;
%F-A-i*B;
\begin{lstlisting}
P<s,t,u,v>:=ProjectiveSpace(Rationals(),3);
A:=- 7*s^4 + 24*s^3*t + 42*s^2*t^2 - 24*s*t^3 - 7*t^4-49*u^4;
B:=-6*s^4 - 28*s^3*t + 36*s^2*t^2 + 28*s*t^3 - 6*t^4-2*v^4;
S:=Scheme(P,[A,B]);
IsLocallySolvable(S,5);
\end{lstlisting}

{\bf Case 4:} $49u^4+2v^4i=i^3(2+i)(4+i)(s+ti)^4$. Equating the real and imaginary parts on both sides of this equation gives
\begin{equation}\label{4165.1-4}
\begin{cases}
6s^4 + 28s^3t - 36s^2t^2 - 28st^3 + 6t^4 - 49u^4=0,\\
-7s^4 + 24s^3t + 42s^2t^2 - 24st^3 - 7t^4 - 2v^4=0.
\end{cases}
\end{equation}
The scheme defined by system \eqref{4165.1-4} is locally insoluble at $2$.

%\begin{lstlisting}
%    _<x>:=PolynomialRing(Rationals());
%k<i>:=NumberField(x^2+1);
%K<s,t>:=PolynomialRing(k,2);
%e:=3;
%F:=i^e*(2+i)*(4+i)*(s+t*i)^4;
%F;
%L<s,t>:=PolynomialRing(Rationals(),2);
%_<i>:=PolynomialRing(L);
%F:=(-7*i + 6)*s^4 + (24*i + 28)*s^3*t + (42*i - 36)*s^2*t^2 + (-24*i - 28)*s*t^3 +
%    (-7*i + 6)*t^4;
%F;
%A:= 6*s^4 + 28*s^3*t -36*s^2*t^2 - 28*s*t^3 + 6*t^4;
%B:=-7*s^4 + 24*s^3*t + 42*s^2*t^2 - 24*s*t^3 - 7*t^4;
%F-A-i*B;
\begin{lstlisting}
P<s,t,u,v>:=ProjectiveSpace(Rationals(),3);
A:= 6*s^4 + 28*s^3*t -36*s^2*t^2 - 28*s*t^3 + 6*t^4-49*u^4;
B:=-7*s^4 + 24*s^3*t + 42*s^2*t^2 - 24*s*t^3 - 7*t^4-2*v^4;
S:=Scheme(P,[A,B]);
IsLocallySolvable(S,2);
\end{lstlisting}

{\bf Case 5:} $49u^4+2v^4i=(2+i)(4-i)(s+ti)^4$. Equating the real and imaginary parts on both sides of this equation gives
\begin{equation}\label{4165.1-5}
\begin{cases}
9s^4 - 8s^3t - 54s^2t^2 + 8st^3 + 9t^4 - 49u^4=0,\\
2s^4 + 36s^3t - 12s^2t^2 - 36st^3 + 2t^4 - 2v^4=0.
\end{cases}
\end{equation}
The scheme defined by system \eqref{4165.1-5} is locally insoluble at $2$.

%\begin{lstlisting}
%    _<x>:=PolynomialRing(Rationals());
%k<i>:=NumberField(x^2+1);
%K<s,t>:=PolynomialRing(k,2);
%e:=0;
%F:=i^e*(2+i)*(4-i)*(s+t*i)^4;
%F;
%L<s,t>:=PolynomialRing(Rationals(),2);
%_<i>:=PolynomialRing(L);
%F:=(2*i + 9)*s^4 + (36*i - 8)*s^3*t + (-12*i - 54)*s^2*t^2 + (-36*i + 8)*s*t^3 +
%    (2*i + 9)*t^4;
%F;
%A:=9*s^4 - 8*s^3*t -54*s^2*t^2 + 8*s*t^3 + 9*t^4;
%B:=2*s^4 + 36*s^3*t - 12*s^2*t^2 - 36*s*t^3 + 2*t^4;
%F-A-i*B;
\begin{lstlisting}
P<s,t,u,v>:=ProjectiveSpace(Rationals(),3);
A:=9*s^4 - 8*s^3*t -54*s^2*t^2 + 8*s*t^3 + 9*t^4-49*u^4;
B:=2*s^4 + 36*s^3*t - 12*s^2*t^2 - 36*s*t^3 + 2*t^4-2*v^4;
S:=Scheme(P,[A,B]);
IsLocallySolvable(S,2);
\end{lstlisting}

{\bf Case 6:}  $49u^4+2v^4i=i(2+i)(4-i)(s+ti)^4$. Equating the real and imaginary parts on both sides of this equation gives
\begin{equation}\label{4165.1-6}
\begin{cases}
-2s^4 - 36s^3t + 12s^2t^2 + 36st^3 - 2t^4 - 49u^4=0,\\
9s^4 - 8s^3t - 54s^2t^2 + 8st^3 + 9t^4 - 2v^4=0.
\end{cases}
\end{equation}
The scheme defined by system \eqref{4165.1-6} is locally insoluble at $2$.

%\begin{lstlisting}
%    _<x>:=PolynomialRing(Rationals());
%k<i>:=NumberField(x^2+1);
%K<s,t>:=PolynomialRing(k,2);
%e:=1;
%F:=i^e*(2+i)*(4-i)*(s+t*i)^4;
%F;
%L<s,t>:=PolynomialRing(Rationals(),2);
%_<i>:=PolynomialRing(L);
%F:=(9*i - 2)*s^4 + (-8*i - 36)*s^3*t + (-54*i + 12)*s^2*t^2 + (8*i + 36)*s*t^3 +
%    (9*i - 2)*t^4;
%F;
%A:=- 2*s^4 - 36*s^3*t +12*s^2*t^2 + 36*s*t^3 - 2*t^4;
%B:=9*s^4 - 8*s^3*t - 54*s^2*t^2 + 8*s*t^3 + 9*t^4;
%F-A-i*B;
\begin{lstlisting}
P<s,t,u,v>:=ProjectiveSpace(Rationals(),3);
A:=- 2*s^4 - 36*s^3*t +12*s^2*t^2 + 36*s*t^3 - 2*t^4-49*u^4;
B:=9*s^4 - 8*s^3*t - 54*s^2*t^2 + 8*s*t^3 + 9*t^4-2*v^4;
S:=Scheme(P,[A,B]);
IsLocallySolvable(S,2);
\end{lstlisting}

{\bf Case 7:}  $49u^4+2v^4i=i^2(2+i)(4-i)(s+ti)^4$. Equating the real and imaginary parts on both sides of this equation gives
\begin{equation}\label{4165.1-7}
\begin{cases}
-9s^4 + 8s^3t + 54s^2t^2 - 8st^3 - 9t^4 - 49u^4=0,\\
-2s^4 - 36s^3t + 12s^2t^2 + 36st^3 - 2t^4 - 2v^4=0.
\end{cases}
\end{equation}
The scheme defined by system \eqref{4165.1-7} is locally insoluble at $2$.

%\begin{lstlisting}
%    _<x>:=PolynomialRing(Rationals());
%k<i>:=NumberField(x^2+1);
%K<s,t>:=PolynomialRing(k,2);
%e:=2;
%F:=i^e*(2+i)*(4-i)*(s+t*i)^4;
%F;
%L<s,t>:=PolynomialRing(Rationals(),2);
%_<i>:=PolynomialRing(L);
%F:=(-2*i - 9)*s^4 + (-36*i + 8)*s^3*t + (12*i + 54)*s^2*t^2 + (36*i - 8)*s*t^3 +
%    (-2*i - 9)*t^4;
%F;
%A:=- 9*s^4 + 8*s^3*t +54*s^2*t^2 - 8*s*t^3 - 9*t^4;
%B:=-2*s^4 - 36*s^3*t + 12*s^2*t^2 + 36*s*t^3 - 2*t^4;
%F-A-i*B;
\begin{lstlisting}
P<s,t,u,v>:=ProjectiveSpace(Rationals(),3);
A:=- 9*s^4 + 8*s^3*t +54*s^2*t^2 - 8*s*t^3 - 9*t^4-49*u^4;
B:=-2*s^4 - 36*s^3*t + 12*s^2*t^2 + 36*s*t^3 - 2*t^4-2*v^4;
S:=Scheme(P,[A,B]);
IsLocallySolvable(S,2);
\end{lstlisting}

{\bf Case 8:} $49u^4+2v^4i=i^3(2+i)(4-i)(s+ti)^4$. Equating the real and imaginary parts on both sides of this equation gives
\begin{equation}\label{4165.1-8}
\begin{cases}
2s^4 + 36s^3t - 12s^2t^2 - 36st^3 + 2t^4 - 49u^4=0,\\
-9s^4 + 8s^3t + 54s^2t^2 - 8st^3 - 9t^4 - 2v^4=0.
\end{cases}
\end{equation}
The scheme defined by system \eqref{4165.1-8} is locally insoluble at $2$.

%\begin{lstlisting}
%    _<x>:=PolynomialRing(Rationals());
%k<i>:=NumberField(x^2+1);
%K<s,t>:=PolynomialRing(k,2);
%e:=3;
%F:=i^e*(2+i)*(4-i)*(s+t*i)^4;
%F;
%L<s,t>:=PolynomialRing(Rationals(),2);
%_<i>:=PolynomialRing(L);
%F:=(-9*i + 2)*s^4 + (8*i + 36)*s^3*t + (54*i - 12)*s^2*t^2 + (-8*i - 36)*s*t^3 +
%    (-9*i + 2)*t^4;
%F;
%A:=2*s^4 + 36*s^3*t -12*s^2*t^2 - 36*s*t^3 + 2*t^4;
%B:=-9*s^4 + 8*s^3*t + 54*s^2*t^2 - 8*s*t^3 - 9*t^4;
%F-A-i*B;
\begin{lstlisting}
P<s,t,u,v>:=ProjectiveSpace(Rationals(),3);
A:=2*s^4 + 36*s^3*t -12*s^2*t^2 - 36*s*t^3 + 2*t^4-49*u^4;
B:=-9*s^4 + 8*s^3*t + 54*s^2*t^2 - 8*s*t^3 - 9*t^4-2*v^4;
S:=Scheme(P,[A,B]);
IsLocallySolvable(S,2);
\end{lstlisting}

\subsubsection{\bf{The case of \eqref{4165.2}.}}% \[n=4165 \quad (II).\]
We now consider \eqref{4165.2}, which we can write \eqref{4165.2} as
\begin{equation} (98u^4+v^4i)(98u^4-v^4i)=(2+i)(2-i)(4+i)(4-i)w^4.
\end{equation}
Since $5\nmid u,v$, we have 
\[\begin{split}98u^4+v^4i&\equiv -2+i\pmod{5}\\
&\equiv 0\pmod{2-i}\\
&\not\equiv 0\pmod{2+i}.\end{split}\]
Hence, $2-i|98u^4+v^4i$. This implies that there exist integers $s,t$ such that 
\[98u^4+v^4i=i^{\epsilon}(2-i)(4\pm i)(s+ti)^4,\]
with $\epsilon \in \{0,1,2,3\}$.

{\bf Case 1:} $98u^4+v^4i=(2-i)(4+i)(s+ti)^4$. Equating the real and imaginary parts on both sides of this equation gives
\begin{equation}\label{4165.2-1}
\begin{cases}
9s^4 + 8s^3t - 54s^2t^2 - 8st^3 + 9t^4 - 98u^4=0,\\
-2s^4 + 36s^3t + 12s^2t^2 - 36st^3 - 2t^4 - v^4=0.
\end{cases}
\end{equation}
The scheme defined by system \eqref{4165.2-1} is locally insoluble at $2$.

%\begin{lstlisting}
%    _<x>:=PolynomialRing(Rationals());
%k<i>:=NumberField(x^2+1);
%K<s,t>:=PolynomialRing(k,2);
%e:=0;
%F:=i^e*(2-i)*(4+i)*(s+t*i)^4;
%F;
%L<s,t>:=PolynomialRing(Rationals(),2);
%_<i>:=PolynomialRing(L);
%F:=(-2*i + 9)*s^4 + (36*i + 8)*s^3*t + (12*i - 54)*s^2*t^2 + (-36*i - 8)*s*t^3 +
%    (-2*i + 9)*t^4;
%F;
%A:=9*s^4 + 8*s^3*t - 54*s^2*t^2 - 8*s*t^3 + 9*t^4;
%B:=-2*s^4 + 36*s^3*t + 12*s^2*t^2 - 36*s*t^3 - 2*t^4;
%F-A-i*B;
\begin{lstlisting}
P<s,t,u,v>:=ProjectiveSpace(Rationals(),3);
A:=9*s^4 + 8*s^3*t - 54*s^2*t^2 - 8*s*t^3 + 9*t^4-98*u^4;
B:=-2*s^4 + 36*s^3*t + 12*s^2*t^2 - 36*s*t^3 - 2*t^4-v^4;
S:=Scheme(P,[A,B]);
IsLocallySolvable(S,2);
\end{lstlisting}

{\bf Case 2:} $98u^4+v^4i=i(2-i)(4+i)(s+ti)^4$. Equating the real and imaginary parts on both sides of this equation gives
\begin{equation}\label{4165.2-2}
\begin{cases}
2s^4 - 36s^3t - 12s^2t^2 + 36st^3 + 2t^4 - 98u^4=0,\\
9s^4 + 8s^3t - 54s^2t^2 - 8st^3 + 9t^4 - v^4=0.
\end{cases}
\end{equation}
The scheme defined by system \eqref{4165.2-2} is locally insoluble at $2$.

%\begin{lstlisting}
%    _<x>:=PolynomialRing(Rationals());
%k<i>:=NumberField(x^2+1);
%K<s,t>:=PolynomialRing(k,2);
%e:=1;
%F:=i^e*(2-i)*(4+i)*(s+t*i)^4;
%F;
%L<s,t>:=PolynomialRing(Rationals(),2);
%_<i>:=PolynomialRing(L);
%F:=(9*i + 2)*s^4 + (8*i - 36)*s^3*t + (-54*i - 12)*s^2*t^2 + (-8*i + 36)*s*t^3 +
%    (9*i + 2)*t^4;
%F;
%A:=2*s^4 - 36*s^3*t -12*s^2*t^2 + 36*s*t^3 + 2*t^4;
%B:=9*s^4 + 8*s^3*t - 54*s^2*t^2 - 8*s*t^3 + 9*t^4;
%F-A-i*B;

\begin{lstlisting}
P<s,t,u,v>:=ProjectiveSpace(Rationals(),3);
A:=2*s^4 - 36*s^3*t -12*s^2*t^2 + 36*s*t^3 + 2*t^4-98*u^4;
B:=9*s^4 + 8*s^3*t - 54*s^2*t^2 - 8*s*t^3 + 9*t^4-v^4;
S:=Scheme(P,[A,B]);
IsLocallySolvable(S,2);
\end{lstlisting}

{\bf Case 3:} $98u^4+v^4i=i^2(2-i)(4+i)(s+ti)^4$. Equating the real and imaginary parts on both sides of this equation gives
\begin{equation}\label{4165.2-3}
\begin{cases}
-9s^4 - 8s^3t + 54s^2t^2 + 8st^3 - 9t^4 - 98u^4=0,\\
2s^4 - 36s^3t - 12s^2t^2 + 36st^3 + 2t^4 - v^4=0.
\end{cases}
\end{equation}
The scheme defined by system \eqref{4165.2-3} is locally insoluble at $2$.

%\begin{lstlisting}
%    _<x>:=PolynomialRing(Rationals());
%k<i>:=NumberField(x^2+1);
%K<s,t>:=PolynomialRing(k,2);
%e:=2;
%F:=i^e*(2-i)*(4+i)*(s+t*i)^4;
%F;
%L<s,t>:=PolynomialRing(Rationals(),2);
%_<i>:=PolynomialRing(L);
%F:=(2*i - 9)*s^4 + (-36*i - 8)*s^3*t + (-12*i + 54)*s^2*t^2 + (36*i + 8)*s*t^3 +
%    (2*i - 9)*t^4;
%F;
%A:=- 9*s^4 - 8*s^3*t + 54*s^2*t^2 + 8*s*t^3 - 9*t^4;
%B:=2*s^4 - 36*s^3*t - 12*s^2*t^2 + 36*s*t^3 + 2*t^4;
%F-A-i*B;
\begin{lstlisting}
P<s,t,u,v>:=ProjectiveSpace(Rationals(),3);
A:=- 9*s^4 - 8*s^3*t + 54*s^2*t^2 + 8*s*t^3 - 9*t^4-98*u^4;
B:=2*s^4 - 36*s^3*t - 12*s^2*t^2 + 36*s*t^3 + 2*t^4-v^4;
S:=Scheme(P,[A,B]);
IsLocallySolvable(S,2);
\end{lstlisting}

{\bf Case 4:} $98u^4+v^4i=i^3(2-i)(4+i)(s+ti)^4$. Equating the real and imaginary parts on both sides of this equation gives
\begin{equation}\label{4165.2-4}
\begin{cases}
-2s^4 + 36s^3t + 12s^2t^2 - 36st^3 - 2t^4 - 98u^4=0,\\
-9s^4 - 8s^3t + 54s^2t^2 + 8st^3 - 9t^4 - v^4=0.
\end{cases}
\end{equation}
The scheme defined by system \eqref{4165.2-4} is locally insoluble at $2$.

%\begin{lstlisting}
%    _<x>:=PolynomialRing(Rationals());
%k<i>:=NumberField(x^2+1);
%K<s,t>:=PolynomialRing(k,2);
%e:=3;
%F:=i^e*(2-i)*(4+i)*(s+t*i)^4;
%F;
%L<s,t>:=PolynomialRing(Rationals(),2);
%_<i>:=PolynomialRing(L);
%F:=(-9*i - 2)*s^4 + (-8*i + 36)*s^3*t + (54*i + 12)*s^2*t^2 + (8*i - 36)*s*t^3 +
%    (-9*i - 2)*t^4;
%F;
%A:=- 2*s^4 + 36*s^3*t + 12*s^2*t^2 - 36*s*t^3 - 2*t^4;
%B:=-9*s^4 - 8*s^3*t + 54*s^2*t^2 + 8*s*t^3 - 9*t^4;
%F-A-i*B;
\begin{lstlisting}
P<s,t,u,v>:=ProjectiveSpace(Rationals(),3);
A:=- 2*s^4 + 36*s^3*t + 12*s^2*t^2 - 36*s*t^3 - 2*t^4-98*u^4;
B:=-9*s^4 - 8*s^3*t + 54*s^2*t^2 + 8*s*t^3 - 9*t^4-v^4;
S:=Scheme(P,[A,B]);
IsLocallySolvable(S,2);
\end{lstlisting}

{\bf Case 5:} $98u^4+v^4i=(2-i)(4-i)(s+ti)^4$. Equating the real and imaginary parts on both sides of this equation gives
\begin{equation}\label{4165.2-5}
\begin{cases}
7s^4 + 24s^3t - 42s^2t^2 - 24st^3 + 7t^4 - 98u^4=0,\\
-6s^4 + 28s^3t + 36s^2t^2 - 28st^3 - 6t^4 - v^4=0.
\end{cases}
\end{equation}
The scheme defined by system \eqref{4165.2-5} is locally insoluble at $2$.

%\begin{lstlisting}
%    _<x>:=PolynomialRing(Rationals());
%k<i>:=NumberField(x^2+1);
%K<s,t>:=PolynomialRing(k,2);
%e:=0;
%F:=i^e*(2-i)*(4-i)*(s+t*i)^4;
%F;
%L<s,t>:=PolynomialRing(Rationals(),2);
%_<i>:=PolynomialRing(L);
%F:=(-6*i + 7)*s^4 + (28*i + 24)*s^3*t + (36*i - 42)*s^2*t^2 + (-28*i - 24)*s*t^3 +
%    (-6*i + 7)*t^4;
%F;
%A:=7*s^4 + 24*s^3*t - 42*s^2*t^2 - 24*s*t^3 + 7*t^4;
%B:=-6*s^4 + 28*s^3*t + 36*s^2*t^2 - 28*s*t^3 - 6*t^4;
%F-A-i*B;
\begin{lstlisting}
P<s,t,u,v>:=ProjectiveSpace(Rationals(),3);
A:=7*s^4 + 24*s^3*t - 42*s^2*t^2 - 24*s*t^3 + 7*t^4-98*u^4;
B:=-6*s^4 + 28*s^3*t + 36*s^2*t^2 - 28*s*t^3 - 6*t^4-v^4;
S:=Scheme(P,[A,B]);
IsLocallySolvable(S,2);
\end{lstlisting}

{\bf Case 6:}  $98u^4+v^4i=i(2-i)(4-i)(s+ti)^4$. Equating the real and imaginary parts on both sides of this equation gives
\begin{equation}\label{4165.2-6}
\begin{cases}
6s^4 - 28s^3t - 36s^2t^2 + 28st^3 + 6t^4 - 98u^4=0,\\
7s^4 + 24s^3t - 42s^2t^2 - 24st^3 + 7t^4 - v^4=0.
\end{cases}
\end{equation}
The scheme defined by system \eqref{4165.2-6} is locally insoluble at $2$.

%\begin{lstlisting}
%    _<x>:=PolynomialRing(Rationals());
%k<i>:=NumberField(x^2+1);
%K<s,t>:=PolynomialRing(k,2);
%e:=1;
%F:=i^e*(2-i)*(4-i)*(s+t*i)^4;
%F;
%L<s,t>:=PolynomialRing(Rationals(),2);
%_<i>:=PolynomialRing(L);
%F:=(7*i + 6)*s^4 + (24*i - 28)*s^3*t + (-42*i - 36)*s^2*t^2 + (-24*i + 28)*s*t^3 +
%    (7*i + 6)*t^4;
%F;
%A:=6*s^4 - 28*s^3*t -36*s^2*t^2 + 28*s*t^3 + 6*t^4;
%B:=7*s^4 + 24*s^3*t - 42*s^2*t^2 - 24*s*t^3 + 7*t^4;
%F-A-i*B;
\begin{lstlisting}
P<s,t,u,v>:=ProjectiveSpace(Rationals(),3);
A:=6*s^4 - 28*s^3*t - 36*s^2*t^2 + 28*s*t^3 + 6*t^4-98*u^4;
B:=7*s^4 + 24*s^3*t - 42*s^2*t^2 - 24*s*t^3 + 7*t^4-v^4;
S:=Scheme(P,[A,B]);
IsLocallySolvable(S,2);
\end{lstlisting}

{\bf Case 7:}  $98u^4+v^4i=i^2(2-i)(4-i)(s+ti)^4$. Equating the real and imaginary parts on both sides of this equation gives
\begin{equation}\label{4165.2-7}
\begin{cases}
-7s^4 - 24s^3t + 42s^2t^2 + 24st^3 - 7t^4 - 98u^4=0,\\
6s^4 - 28s^3t - 36s^2t^2 + 28st^3 + 6t^4 - v^4=0.
\end{cases}
\end{equation}
The scheme defined by system \eqref{4165.2-7} is locally insoluble at $2$.

%\begin{lstlisting}
%    _<x>:=PolynomialRing(Rationals());
%k<i>:=NumberField(x^2+1);
%K<s,t>:=PolynomialRing(k,2);
%e:=2;
%F:=i^e*(2-i)*(4-i)*(s+t*i)^4;
%F;
%L<s,t>:=PolynomialRing(Rationals(),2);
%_<i>:=PolynomialRing(L);
%F:=(6*i - 7)*s^4 + (-28*i - 24)*s^3*t + (-36*i + 42)*s^2*t^2 + (28*i + 24)*s*t^3 +
%    (6*i - 7)*t^4;
%F;
%A:=- 7*s^4 - 24*s^3*t +42*s^2*t^2 + 24*s*t^3 - 7*t^4;
%B:=6*s^4 - 28*s^3*t - 36*s^2*t^2 + 28*s*t^3 + 6*t^4;
%F-A-i*B;
\begin{lstlisting}
P<s,t,u,v>:=ProjectiveSpace(Rationals(),3);
A:=- 7*s^4 - 24*s^3*t +42*s^2*t^2 + 24*s*t^3 - 7*t^4-98*u^4;
B:=6*s^4 - 28*s^3*t - 36*s^2*t^2 + 28*s*t^3 + 6*t^4-v^4;
S:=Scheme(P,[A,B]);
IsLocallySolvable(S,2);
\end{lstlisting}

{\bf Case 8:} $98u^4+v^4i=i^3(2-i)(4-i)(s+ti)^4$. Equating the real and imaginary parts on both sides of this equation gives
\begin{equation}\label{4165.2-8}
\begin{cases}
-6s^4 + 28s^3t + 36s^2t^2 - 28st^3 - 6t^4 - 98u^4=0,\\
-7s^4 - 24s^3t + 42s^2t^2 + 24st^3 - 7t^4 - v^4=0.
\end{cases}
\end{equation}
The scheme defined by system \eqref{4165.2-8} is locally insoluble at $5$.

%\begin{lstlisting}
%    _<x>:=PolynomialRing(Rationals());
%k<i>:=NumberField(x^2+1);
%K<s,t>:=PolynomialRing(k,2);
%e:=3;
%F:=i^e*(2-i)*(4-i)*(s+t*i)^4;
%F;
%L<s,t>:=PolynomialRing(Rationals(),2);
%_<i>:=PolynomialRing(L);
%F:=(-7*i - 6)*s^4 + (-24*i + 28)*s^3*t + (42*i + 36)*s^2*t^2 + (24*i - 28)*s*t^3 +
%    (-7*i - 6)*t^4;
%F;
%A:=- 6*s^4 + 28*s^3*t +36*s^2*t^2 - 28*s*t^3 - 6*t^4;
%B:=-7*s^4 - 24*s^3*t + 42*s^2*t^2 + 24*s*t^3 - 7*t^4;
%F-A-i*B;
\begin{lstlisting}
P<s,t,u,v>:=ProjectiveSpace(Rationals(),3);
A:=- 6*s^4 + 28*s^3*t +36*s^2*t^2 - 28*s*t^3 - 6*t^4-98*u^4;
B:=-7*s^4 - 24*s^3*t + 42*s^2*t^2 + 24*s*t^3 - 7*t^4-v^4;
S:=Scheme(P,[A,B]);
IsLocallySolvable(S,5);
\end{lstlisting}

\subsection{The case of $n=4357$.}% \[n=4357\] 
According to Table \ref{tb:factorisation}, in the case $n=4357$ it remains to
show that equation 
\begin{equation}\label{4357}
4 u^8 + v^8= 4357  w^4
\end{equation}
has no integer solutions satisfying \eqref{eq:cond}, which in this case reduces to \[\gcd(2u,v)=\gcd(u,4357w)=\gcd(v,4357w)=1.\]
Write \eqref{4357} as
\begin{equation} (2u^4+v^4i)(2u^4-v^4i)=(66+i)(66-i)w^4.
\end{equation}
This implies that there exist integers $s,t$ such that 
\[2u^4+v^4i=i^{\epsilon}(66\pm i)(s+ti)^4,\]
with $\epsilon \in \{0,1,2,3\}$.

{\bf Case 1:} $2u^4+v^4i=(66+i)(s+ti)^4$. Equating the real and imaginary parts on both sides of this equation gives
\begin{equation}\label{4357-1}
\begin{cases}
66s^4 - 4s^3t - 396s^2t^2 + 4st^3 + 66t^4 - 2u^4=0,\\
s^4 + 264s^3t - 6s^2t^2 - 264st^3 + t^4 - v^4=0.
\end{cases}
\end{equation}
The scheme defined by system \eqref{4357-1} is locally insoluble at $5$.

%\begin{lstlisting}
%    _<x>:=PolynomialRing(Rationals());
%k<i>:=NumberField(x^2+1);
%K<s,t>:=PolynomialRing(k,2);
%e:=0;
%F:=i^e*(66+i)*(s+t*i)^4;
%F;
%L<s,t>:=PolynomialRing(Rationals(),2);
%_<i>:=PolynomialRing(L);
%F:=(i + 66)*s^4 + (264*i - 4)*s^3*t + (-6*i - 396)*s^2*t^2 + (-264*i + 4)*s*t^3 +
%    (i + 66)*t^4;
%F;
%A:=66*s^4 - 4*s^3*t - 396*s^2*t^2 + 4*s*t^3 + 66*t^4;
%B:=s^4 + 264*s^3*t - 6*s^2*t^2 - 264*s*t^3 + t^4;
%F-A-i*B;
\begin{lstlisting}
P<s,t,u,v>:=ProjectiveSpace(Rationals(),3);
A:=66*s^4 - 4*s^3*t - 396*s^2*t^2 + 4*s*t^3 + 66*t^4-2*u^4;
B:=s^4 + 264*s^3*t - 6*s^2*t^2 - 264*s*t^3 + t^4-v^4;
S:=Scheme(P,[A,B]);
IsLocallySolvable(S,5);
\end{lstlisting}

{\bf Case 2:} $2u^4+v^4i=i(66+i)(s+ti)^4$. Equating the real and imaginary parts on both sides of this equation gives
\begin{equation}\label{4357-2}
\begin{cases}
-s^4 - 264s^3t + 6s^2t^2 + 264st^3 - t^4 - 2u^4=0,\\
66s^4 - 4s^3t - 396s^2t^2 + 4st^3 + 66t^4 - v^4=0.
\end{cases}
\end{equation}
The scheme defined by system \eqref{4357-2} is locally insoluble at $2$.

%\begin{lstlisting}
%    _<x>:=PolynomialRing(Rationals());
%k<i>:=NumberField(x^2+1);
%K<s,t>:=PolynomialRing(k,2);
%e:=1;
%F:=i^e*(66+i)*(s+t*i)^4;
%F;
%L<s,t>:=PolynomialRing(Rationals(),2);
%_<i>:=PolynomialRing(L);
%F:=(66*i - 1)*s^4 + (-4*i - 264)*s^3*t + (-396*i + 6)*s^2*t^2 + (4*i + 264)*s*t^3 +
%    (66*i - 1)*t^4;
%F;
%A:= - s^4 - 264*s^3*t + 6*s^2*t^2 + 264*s*t^3 - t^4;
%B:=66*s^4 - 4*s^3*t - 396*s^2*t^2 + 4*s*t^3 + 66*t^4;
%F-A-i*B;
\begin{lstlisting}
P<s,t,u,v>:=ProjectiveSpace(Rationals(),3);
A:= - s^4 - 264*s^3*t + 6*s^2*t^2 + 264*s*t^3 - t^4-2*u^4;
B:=66*s^4 - 4*s^3*t - 396*s^2*t^2 + 4*s*t^3 + 66*t^4-v^4;
S:=Scheme(P,[A,B]);
IsLocallySolvable(S,2);
\end{lstlisting}

{\bf Case 3:} $2u^4+v^4i=i^2(66+i)(s+ti)^4$. Equating the real and imaginary parts on both sides of this equation gives
\begin{equation}\label{4357-3}
\begin{cases}
-66s^4 + 4s^3t + 396s^2t^2 - 4st^3 - 66t^4 - 2u^4=0,\\
-s^4 - 264s^3t + 6s^2t^2 + 264st^3 - t^4 - v^4=0.
\end{cases}
\end{equation}
The scheme defined by system \eqref{4357-3} is locally insoluble at $2$.

%\begin{lstlisting}
%    _<x>:=PolynomialRing(Rationals());
%k<i>:=NumberField(x^2+1);
%K<s,t>:=PolynomialRing(k,2);
%e:=2;
%F:=i^e*(66+i)*(s+t*i)^4;
%F;
%L<s,t>:=PolynomialRing(Rationals(),2);
%_<i>:=PolynomialRing(L);
%F:=(-i - 66)*s^4 + (-264*i + 4)*s^3*t + (6*i + 396)*s^2*t^2 + (264*i - 4)*s*t^3 +
%    (-i - 66)*t^4;
%F;
%A:=- 66*s^4 + 4*s^3*t +396*s^2*t^2 - 4*s*t^3 - 66*t^4;
%B:=-s^4 - 264*s^3*t + 6*s^2*t^2 + 264*s*t^3 - t^4;
%F-A-i*B;
\begin{lstlisting}
P<s,t,u,v>:=ProjectiveSpace(Rationals(),3);
A:=- 66*s^4 + 4*s^3*t +396*s^2*t^2 - 4*s*t^3 - 66*t^4-2*u^4;
B:=-s^4 - 264*s^3*t + 6*s^2*t^2 + 264*s*t^3 - t^4-v^4;
S:=Scheme(P,[A,B]);
IsLocallySolvable(S,2);
\end{lstlisting}

{\bf Case 4:} $2u^4+v^4i=i^3(66+i)(s+ti)^4$. Equating the real and imaginary parts on both sides of this equation gives
\begin{equation}\label{4357-4}
\begin{cases}
s^4 + 264s^3t - 6s^2t^2 - 264st^3 + t^4 - 2u^4=0,\\
-66s^4 + 4s^3t + 396s^2t^2 - 4st^3 - 66t^4 - v^4=0.
\end{cases}
\end{equation}
The scheme defined by system \eqref{4357-4} is locally insoluble at $2$.
%\begin{lstlisting}
%    _<x>:=PolynomialRing(Rationals());
%k<i>:=NumberField(x^2+1);
%K<s,t>:=PolynomialRing(k,2);
%e:=3;
%F:=i^e*(66+i)*(s+t*i)^4;
%F;
%L<s,t>:=PolynomialRing(Rationals(),2);
%_<i>:=PolynomialRing(L);
%F:=(-66*i + 1)*s^4 + (4*i + 264)*s^3*t + (396*i - 6)*s^2*t^2 + (-4*i - 264)*s*t^3 +
%    (-66*i + 1)*t^4;
%F;
%A:=s^4 + 264*s^3*t -6*s^2*t^2 - 264*s*t^3 + t^4;
%B:=-66*s^4 + 4*s^3*t + 396*s^2*t^2 - 4*s*t^3 - 66*t^4;
%F-A-i*B;
\begin{lstlisting}
P<s,t,u,v>:=ProjectiveSpace(Rationals(),3);
A:=s^4 + 264*s^3*t -6*s^2*t^2 - 264*s*t^3 + t^4-2*u^4;
B:=-66*s^4 + 4*s^3*t + 396*s^2*t^2 - 4*s*t^3 - 66*t^4-v^4;
S:=Scheme(P,[A,B]);
IsLocallySolvable(S,2);
\end{lstlisting}

{\bf Case 5:} $2u^4+v^4i=(66-i)(s+ti)^4$. Equating the real and imaginary parts on both sides of this equation gives
\begin{equation}\label{4357-5}
\begin{cases}
66s^4 + 4s^3t - 396s^2t^2 - 4st^3 + 66t^4 - 2u^4=0,\\
-s^4 + 264s^3t + 6s^2t^2 - 264st^3 - t^4 - v^4=0.
\end{cases}
\end{equation}
The scheme defined by system \eqref{4357-5} is locally insoluble at $2$.

%\begin{lstlisting}
%    _<x>:=PolynomialRing(Rationals());
%k<i>:=NumberField(x^2+1);
%K<s,t>:=PolynomialRing(k,2);
%e:=0;
%F:=i^e*(66-i)*(s+t*i)^4;
%F;
%L<s,t>:=PolynomialRing(Rationals(),2);
%_<i>:=PolynomialRing(L);
%F:=(-i + 66)*s^4 + (264*i + 4)*s^3*t + (6*i - 396)*s^2*t^2 + (-264*i - 4)*s*t^3 +
%    (-i + 66)*t^4;
%F;
%A:=66*s^4 + 4*s^3*t -396*s^2*t^2 - 4*s*t^3 + 66*t^4;
%B:=-s^4 + 264*s^3*t + 6*s^2*t^2 - 264*s*t^3 - t^4;
%F-A-i*B;
\begin{lstlisting}
P<s,t,u,v>:=ProjectiveSpace(Rationals(),3);
A:=66*s^4 + 4*s^3*t -396*s^2*t^2 - 4*s*t^3 + 66*t^4-2*u^4;
B:=-s^4 + 264*s^3*t + 6*s^2*t^2 - 264*s*t^3 - t^4-v^4;
S:=Scheme(P,[A,B]);
IsLocallySolvable(S,2);
\end{lstlisting}
{\bf Case 6:} $2u^4+v^4i=i(66-i)(s+ti)^4$. Equating the real and imaginary parts on both sides of this equation gives

\begin{equation}\label{4357-6}
\begin{cases}
s^4 - 264s^3t - 6s^2t^2 + 264st^3 + t^4 - 2u^4=0,\\
66s^4 + 4s^3t - 396s^2t^2 - 4st^3 + 66t^4 - v^4=0.
\end{cases}
\end{equation}
The scheme defined by system \eqref{4357-6} is locally insoluble at $2$.

%\begin{lstlisting}
%    _<x>:=PolynomialRing(Rationals());
%k<i>:=NumberField(x^2+1);
%K<s,t>:=PolynomialRing(k,2);
%e:=1;
%F:=i^e*(66-i)*(s+t*i)^4;
%F;
%L<s,t>:=PolynomialRing(Rationals(),2);
%_<i>:=PolynomialRing(L);
%F:=(66*i + 1)*s^4 + (4*i - 264)*s^3*t + (-396*i - 6)*s^2*t^2 + (-4*i + 264)*s*t^3 +
%    (66*i + 1)*t^4;
%F;
%A:=s^4 - 264*s^3*t - 6*s^2*t^2 + 264*s*t^3 + t^4;
%B:=66*s^4 + 4*s^3*t - 396*s^2*t^2 - 4*s*t^3 + 66*t^4;
%F-A-i*B;
\begin{lstlisting}
P<s,t,u,v>:=ProjectiveSpace(Rationals(),3);
A:=s^4 - 264*s^3*t - 6*s^2*t^2 + 264*s*t^3 + t^4-2*u^4;
B:=66*s^4 + 4*s^3*t - 396*s^2*t^2 - 4*s*t^3 + 66*t^4-v^4;
S:=Scheme(P,[A,B]);
IsLocallySolvable(S,2);
\end{lstlisting}

{\bf Case 7:} $2u^4+v^4i=i^2(66+i)(s+ti)^4$. Equating the real and imaginary parts on both sides of this equation gives
\begin{equation}\label{4357-7}
\begin{cases}
-66s^4 - 4s^3t + 396s^2t^2 + 4st^3 - 66t^4 - 2u^4=0,\\
s^4 - 264s^3t - 6s^2t^2 + 264st^3 + t^4 - v^4=0.
\end{cases}
\end{equation}
The scheme defined by system \eqref{4357-7} is locally insoluble at $2$.

%\begin{lstlisting}
%    _<x>:=PolynomialRing(Rationals());
%k<i>:=NumberField(x^2+1);
%K<s,t>:=PolynomialRing(k,2);
%e:=2;
%F:=i^e*(66-i)*(s+t*i)^4;
%F;
%L<s,t>:=PolynomialRing(Rationals(),2);
%_<i>:=PolynomialRing(L);
%F:=(i - 66)*s^4 + (-264*i - 4)*s^3*t + (-6*i + 396)*s^2*t^2 + (264*i + 4)*s*t^3 +
%    (i - 66)*t^4;
%F;
%A:=- 66*s^4 - 4*s^3*t + 396*s^2*t^2 + 4*s*t^3 - 66*t^4;
%B:=s^4 - 264*s^3*t - 6*s^2*t^2 + 264*s*t^3 + t^4;
%F-A-i*B;
\begin{lstlisting}
P<s,t,u,v>:=ProjectiveSpace(Rationals(),3);
A:=- 66*s^4 - 4*s^3*t + 396*s^2*t^2 + 4*s*t^3 - 66*t^4-2*u^4;
B:=s^4 - 264*s^3*t - 6*s^2*t^2 + 264*s*t^3 + t^4-v^4;
S:=Scheme(P,[A,B]);
IsLocallySolvable(S,2);
\end{lstlisting}

{\bf Case 8:} $2u^4+v^4i=i^3(66-i)(s+ti)^4$. Equating the real and imaginary parts on both sides of this equation gives
\begin{equation}\label{4357-8}
\begin{cases}
-s^4 + 264s^3t + 6s^2t^2 - 264st^3 - t^4 - 2u^4=0,\\
-66s^4 - 4s^3t + 396s^2t^2 + 4st^3 - 66t^4 - v^4=0.
\end{cases}
\end{equation}
The scheme defined by system \eqref{4357-8} is locally insoluble at $2$.
%\begin{lstlisting}
%    _<x>:=PolynomialRing(Rationals());
%k<i>:=NumberField(x^2+1);
%K<s,t>:=PolynomialRing(k,2);
%e:=3;
%F:=i^e*(66-i)*(s+t*i)^4;
%F;
%L<s,t>:=PolynomialRing(Rationals(),2);
%_<i>:=PolynomialRing(L);
%F:=(-66*i - 1)*s^4 + (-4*i + 264)*s^3*t + (396*i + 6)*s^2*t^2 + (4*i - 264)*s*t^3 +
%    (-66*i - 1)*t^4;
%F;
%A:=- s^4 + 264*s^3*t + 6*s^2*t^2 - 264*s*t^3 - t^4;
%B:=-66*s^4 - 4*s^3*t + 396*s^2*t^2 + 4*s*t^3 - 66*t^4;
%F-A-i*B;
\begin{lstlisting}
P<s,t,u,v>:=ProjectiveSpace(Rationals(),3);
A:=- s^4 + 264*s^3*t + 6*s^2*t^2 - 264*s*t^3 - t^4-2*u^4;
B:=-66*s^4 - 4*s^3*t + 396*s^2*t^2 + 4*s*t^3 - 66*t^4-v^4;
S:=Scheme(P,[A,B]);
IsLocallySolvable(S,2);
\end{lstlisting}

\subsection{The case of $n=4901$.} %\[n=4901\]
According to Table \ref{tb:factorisation}, in the case $n=4901$ it remains to
show that equation 
\begin{equation}\label{4901}
4u^8+v^8=4901w^4
\end{equation}
has no integer solutions satisfying \eqref{eq:cond}, which in this case reduces to \[\gcd(2u,v)=\gcd(2u,4901w)=\gcd(v,4901w)=1.\]
Write \eqref{4901} as
\begin{equation} (2u^4+v^4i)(2u^4-v^4i)=(3+2i)^2(3-2i)^2(5+2i)(5-2i)w^4.
\end{equation}
This implies that there exist integers $s,t$ such that 
\[2u^4+v^4i=i^{\epsilon}(3\pm 2i)^2(5\pm 2i)(s+ti)^4,\]
with $\epsilon \in \{0,1,2,3\}$.

{\bf Case 1:} $2u^4+v^4i=(3+ 2i)^2(5+ 2i)(s+ti)^4$. Equating the real and imaginary parts on both sides of this equation gives
\begin{equation}\label{4901-1}
\begin{cases}
s^4 - 280s^3t - 6s^2t^2 + 280st^3 + t^4 - 2u^4=0,\\
70s^4 + 4s^3t - 420s^2t^2 - 4st^3 + 70t^4 - v^4=0.
\end{cases}
\end{equation}
The scheme defined by system \eqref{4901-1} is locally insoluble at $2$.

%\begin{lstlisting}
%    _<x>:=PolynomialRing(Rationals());
%k<i>:=NumberField(x^2+1);
%K<s,t>:=PolynomialRing(k,2);
%e:=0;
%F:=i^e*(3+2*i)^2*(5+2*i)*(s+t*i)^4;
%F;
%L<s,t>:=PolynomialRing(Rationals(),2);
%_<i>:=PolynomialRing(L);
%F:=(70*i + 1)*s^4 + (4*i - 280)*s^3*t + (-420*i - 6)*s^2*t^2 + (-4*i + 280)*s*t^3 +
%    (70*i + 1)*t^4;
%F;
%A:=s^4 - 280*s^3*t - 6*s^2*t^2 + 280*s*t^3 + t^4;
%B:=70*s^4 + 4*s^3*t - 420*s^2*t^2 - 4*s*t^3 + 70*t^4;
%F-A-i*B;

\begin{lstlisting}
P<s,t,u,v>:=ProjectiveSpace(Rationals(),3);
A:=s^4 - 280*s^3*t - 6*s^2*t^2 + 280*s*t^3 + t^4-2*u^4;
B:=70*s^4 + 4*s^3*t - 420*s^2*t^2 - 4*s*t^3 + 70*t^4-v^4;
S:=Scheme(P,[A,B]);
IsLocallySolvable(S,2);
\end{lstlisting}

{\bf Case 2:} $2u^4+v^4i=i(3+ 2i)^2(5+ 2i)(s+ti)^4$. Equating the real and imaginary parts on both sides of this equation gives
\begin{equation}\label{4901-2}
\begin{cases}
-70s^4 - 4s^3t + 420s^2t^2 + 4st^3 - 70t^4 - 2u^4=0,\\
s^4 - 280s^3t - 6s^2t^2 + 280st^3 + t^4 - v^4=0.
\end{cases}
\end{equation}
The scheme defined by system \eqref{4901-2} is locally insoluble at $2$.

%\begin{lstlisting}
%    _<x>:=PolynomialRing(Rationals());
%k<i>:=NumberField(x^2+1);
%K<s,t>:=PolynomialRing(k,2);
%e:=1;
%F:=i^e*(3+2*i)^2*(5+2*i)*(s+t*i)^4;
%F;
%L<s,t>:=PolynomialRing(Rationals(),2);
%_<i>:=PolynomialRing(L);
%F:=(i - 70)*s^4 + (-280*i - 4)*s^3*t + (-6*i + 420)*s^2*t^2 + (280*i + 4)*s*t^3 +
%    (i - 70)*t^4;
%F;
%A:=- 70*s^4 - 4*s^3*t +420*s^2*t^2 + 4*s*t^3 - 70*t^4;
%B:=s^4 - 280*s^3*t - 6*s^2*t^2 + 280*s*t^3 + t^4;
%F-A-i*B;

\begin{lstlisting}
P<s,t,u,v>:=ProjectiveSpace(Rationals(),3);
A:=- 70*s^4 - 4*s^3*t +420*s^2*t^2 + 4*s*t^3 - 70*t^4-2*u^4;
B:=s^4 - 280*s^3*t - 6*s^2*t^2 + 280*s*t^3 + t^4-v^4;
S:=Scheme(P,[A,B]);
IsLocallySolvable(S,2);
\end{lstlisting}

{\bf Case 3:} $2u^4+v^4i=i^2(3+ 2i)^2(5+ 2i)(s+ti)^4$.  Equating the real and imaginary parts on both sides of this equation gives
\begin{equation}\label{4901-3}
\begin{cases}
-s^4 + 280s^3t + 6s^2t^2 - 280st^3 - t^4 - 2u^4=0,\\
-70s^4 - 4s^3t + 420s^2t^2 + 4st^3 - 70t^4 - v^4=0.
\end{cases}
\end{equation}
The scheme defined by system \eqref{4901-3} is locally insoluble at $2$.

%\begin{lstlisting}
%    _<x>:=PolynomialRing(Rationals());
%k<i>:=NumberField(x^2+1);
%K<s,t>:=PolynomialRing(k,2);
%e:=2;
%F:=i^e*(3+2*i)^2*(5+2*i)*(s+t*i)^4;
%F;
%L<s,t>:=PolynomialRing(Rationals(),2);
%_<i>:=PolynomialRing(L);
%F:=(-70*i - 1)*s^4 + (-4*i + 280)*s^3*t + (420*i + 6)*s^2*t^2 + (4*i - 280)*s*t^3 +
%    (-70*i - 1)*t^4;
%F;
%A:=- s^4 + 280*s^3*t + 6*s^2*t^2 - 280*s*t^3 - t^4;
%B:=-70*s^4 - 4*s^3*t + 420*s^2*t^2 + 4*s*t^3 - 70*t^4;
%F-A-i*B;

\begin{lstlisting}
P<s,t,u,v>:=ProjectiveSpace(Rationals(),3);
A:=- s^4 + 280*s^3*t + 6*s^2*t^2 - 280*s*t^3 - t^4-2*u^4;
B:=-70*s^4 - 4*s^3*t + 420*s^2*t^2 + 4*s*t^3 - 70*t^4-v^4;
S:=Scheme(P,[A,B]);
IsLocallySolvable(S,2);
\end{lstlisting}

{\bf Case 4:} $2u^4+v^4i=i^3(3+ 2i)^2(5+ 2i)(s+ti)^4$.  Equating the real and imaginary parts on both sides of this equation gives
\begin{equation}\label{4901-4}
\begin{cases}
70s^4 + 4s^3t - 420s^2t^2 - 4st^3 + 70t^4 - 2u^4=0,\\
-s^4 + 280s^3t + 6s^2t^2 - 280st^3 - t^4 - v^4=0.
\end{cases}
\end{equation}
The scheme defined by system \eqref{4901-4} is locally insoluble at $2$.

%\begin{lstlisting}
%_<x>:=PolynomialRing(Rationals());
%k<i>:=NumberField(x^2+1);
%K<s,t>:=PolynomialRing(k,2);
%e:=3;
%F:=i^e*(3+2*i)^2*(5+2*i)*(s+t*i)^4;
%F;
%L<s,t>:=PolynomialRing(Rationals(),2);
%_<i>:=PolynomialRing(L);
%F:=(-i + 70)*s^4 + (280*i + 4)*s^3*t + (6*i - 420)*s^2*t^2 + (-280*i - 4)*s*t^3 +
%    (-i + 70)*t^4;
%F;
%A:=70*s^4 + 4*s^3*t -420*s^2*t^2 - 4*s*t^3 + 70*t^4;
%B:=-s^4 + 280*s^3*t + 6*s^2*t^2 - 280*s*t^3 - t^4;
%F-A-i*B;
\begin{lstlisting}
P<s,t,u,v>:=ProjectiveSpace(Rationals(),3);
A:=70*s^4 + 4*s^3*t -420*s^2*t^2 - 4*s*t^3 + 70*t^4-2*u^4;
B:=-s^4 + 280*s^3*t + 6*s^2*t^2 - 280*s*t^3 - t^4-v^4;
S:=Scheme(P,[A,B]);
IsLocallySolvable(S,2);
\end{lstlisting}

{\bf Case 5:} $2u^4+v^4i=(3+ 2i)^2(5-2i)(s+ti)^4$.  Equating the real and imaginary parts on both sides of this equation gives
\begin{equation}\label{4901-5}
\begin{cases}
49s^4 - 200s^3t - 294s^2t^2 + 200st^3 + 49t^4 - 2u^4=0,\\
50s^4 + 196s^3t - 300s^2t^2 - 196st^3 + 50t^4 - v^4=0.
\end{cases}
\end{equation}
The scheme defined by system \eqref{4901-5} is locally insoluble at $2$.

%\begin{lstlisting}
%_<x>:=PolynomialRing(Rationals());
%k<i>:=NumberField(x^2+1);
%K<s,t>:=PolynomialRing(k,2);
%e:=0;
%F:=i^e*(3+2*i)^2*(5-2*i)*(s+t*i)^4;
%F;
%L<s,t>:=PolynomialRing(Rationals(),2);
%_<i>:=PolynomialRing(L);
%F:=(50*i + 49)*s^4 + (196*i - 200)*s^3*t + (-300*i - 294)*s^2*t^2 + (-196*i +
%    200)*s*t^3 + (50*i + 49)*t^4;
%F;
%A:=49*s^4 - 200*s^3*t - 294*s^2*t^2 + 200*s*t^3 + 49*t^4;
%B:=50*s^4 + 196*s^3*t - 300*s^2*t^2 - 196*s*t^3 + 50*t^4;
%F-A-i*B;
\begin{lstlisting}
P<s,t,u,v>:=ProjectiveSpace(Rationals(),3);
A:=49*s^4 - 200*s^3*t - 294*s^2*t^2 + 200*s*t^3 + 49*t^4-2*u^4;
B:=50*s^4 + 196*s^3*t - 300*s^2*t^2 - 196*s*t^3 + 50*t^4-v^4;
S:=Scheme(P,[A,B]);
IsLocallySolvable(S,2);
\end{lstlisting}

{\bf Case 6:} $2u^4+v^4i=i(3+ 2i)^2(5-2i)(s+ti)^4$.  Equating the real and imaginary parts on both sides of this equation gives
\begin{equation}\label{4901-6}
\begin{cases}
-50s^4 - 196s^3t + 300s^2t^2 + 196st^3 - 50t^4 - 2u^4=0,\\
49s^4 - 200s^3t - 294s^2t^2 + 200st^3 + 49t^4 - v^4=0.
\end{cases}
\end{equation}
The scheme defined by system \eqref{4901-6} is locally insoluble at $2$.

%\begin{lstlisting}
%_<x>:=PolynomialRing(Rationals());
%k<i>:=NumberField(x^2+1);
%K<s,t>:=PolynomialRing(k,2);
%e:=1;
%F:=i^e*(3+2*i)^2*(5-2*i)*(s+t*i)^4;
%F;
%L<s,t>:=PolynomialRing(Rationals(),2);
%_<i>:=PolynomialRing(L);
%F:=(49*i - 50)*s^4 + (-200*i - 196)*s^3*t + (-294*i + 300)*s^2*t^2 + (200*i +
%    196)*s*t^3 + (49*i - 50)*t^4;
%F;
%A:=- 50*s^4 - 196*s^3*t + 300*s^2*t^2 + 196*s*t^3 - 50*t^4;
%B:=49*s^4 - 200*s^3*t - 294*s^2*t^2 + 200*s*t^3 + 49*t^4;
%F-A-i*B;
\begin{lstlisting}
P<s,t,u,v>:=ProjectiveSpace(Rationals(),3);
A:=- 50*s^4 - 196*s^3*t + 300*s^2*t^2 + 196*s*t^3 - 50*t^4-2*u^4;
B:=49*s^4 - 200*s^3*t - 294*s^2*t^2 + 200*s*t^3 + 49*t^4-v^4;
S:=Scheme(P,[A,B]);
IsLocallySolvable(S,2);
\end{lstlisting}

{\bf Case 7:} $2u^4+v^4i=i^2(3+ 2i)^2(5-2i)(s+ti)^4$.  Equating the real and imaginary parts on both sides of this equation gives
\begin{equation}\label{4901-7}
\begin{cases}
-49s^4 + 200s^3t + 294s^2t^2 - 200st^3 - 49t^4 - 2u^4=0,\\
-50s^4 - 196s^3t + 300s^2t^2 + 196st^3 - 50t^4 - v^4=0.
\end{cases}
\end{equation}
The scheme defined by system \eqref{4901-7} is locally insoluble at $2$.

%\begin{lstlisting}
%_<x>:=PolynomialRing(Rationals());
%k<i>:=NumberField(x^2+1);
%K<s,t>:=PolynomialRing(k,2);
%e:=2;
%F:=i^e*(3+2*i)^2*(5-2*i)*(s+t*i)^4;
%F;
%L<s,t>:=PolynomialRing(Rationals(),2);
%_<i>:=PolynomialRing(L);
%F:=(-50*i - 49)*s^4 + (-196*i + 200)*s^3*t + (300*i + 294)*s^2*t^2 + (196*i -
%    200)*s*t^3 + (-50*i - 49)*t^4;
%F;
%A:=- 49*s^4 + 200*s^3*t + 294*s^2*t^2 - 200*s*t^3 - 49*t^4;
%B:=-50*s^4 - 196*s^3*t + 300*s^2*t^2 + 196*s*t^3 - 50*t^4;
%F-A-i*B;
\begin{lstlisting}
P<s,t,u,v>:=ProjectiveSpace(Rationals(),3);
A:=- 49*s^4 + 200*s^3*t + 294*s^2*t^2 - 200*s*t^3 - 49*t^4-2*u^4;
B:=-50*s^4 - 196*s^3*t + 300*s^2*t^2 + 196*s*t^3 - 50*t^4-v^4;
S:=Scheme(P,[A,B]);
IsLocallySolvable(S,2);
\end{lstlisting}

{\bf Case 8:} $2u^4+v^4i=i^3(3+ 2i)^2(5-2i)(s+ti)^4$.  Equating the real and imaginary parts on both sides of this equation gives
\begin{equation}\label{4901-8}
\begin{cases}
50s^4 + 196s^3t - 300s^2t^2 - 196st^3 + 50t^4 - 2u^4=0,\\
-49s^4 + 200s^3t + 294s^2t^2 - 200st^3 - 49t^4 - v^4=0.
\end{cases}
\end{equation}
The scheme defined by system \eqref{4901-8} is locally insoluble at $2$.

%\begin{lstlisting}
%_<x>:=PolynomialRing(Rationals());
%k<i>:=NumberField(x^2+1);
%K<s,t>:=PolynomialRing(k,2);
%e:=3;
%F:=i^e*(3+2*i)^2*(5-2*i)*(s+t*i)^4;
%F;
%L<s,t>:=PolynomialRing(Rationals(),2);
%_<i>:=PolynomialRing(L);
%F:=(-49*i + 50)*s^4 + (200*i + 196)*s^3*t + (294*i - 300)*s^2*t^2 + (-200*i -
%    196)*s*t^3 + (-49*i + 50)*t^4;
%F;
%A:=50*s^4 + 196*s^3*t - 300*s^2*t^2 - 196*s*t^3 + 50*t^4;
%B:=-49*s^4 + 200*s^3*t + 294*s^2*t^2 - 200*s*t^3 - 49*t^4;
%F-A-i*B;

\begin{lstlisting}
P<s,t,u,v>:=ProjectiveSpace(Rationals(),3);
A:=50*s^4 + 196*s^3*t - 300*s^2*t^2 - 196*s*t^3 + 50*t^4-2*u^4;
B:=-49*s^4 + 200*s^3*t + 294*s^2*t^2 - 200*s*t^3 - 49*t^4-v^4;
S:=Scheme(P,[A,B]);
IsLocallySolvable(S,2);
\end{lstlisting}

{\bf Case 9:} $2u^4+v^4i=(3- 2i)^2(5+2i)(s+ti)^4$.  Equating the real and imaginary parts on both sides of this equation gives
\begin{equation}\label{4901-9}
\begin{cases}
49s^4 + 200s^3t - 294s^2t^2 - 200st^3 + 49t^4 - 2u^4=0,\\
-50s^4 + 196s^3t + 300s^2t^2 - 196st^3 - 50t^4 - v^4=0.
\end{cases}
\end{equation}
The scheme defined by system \eqref{4901-9} is locally insoluble at $2$.

%\begin{lstlisting}
%_<x>:=PolynomialRing(Rationals());
%k<i>:=NumberField(x^2+1);
%K<s,t>:=PolynomialRing(k,2);
%e:=0;
%F:=i^e*(3-2*i)^2*(5+2*i)*(s+t*i)^4;
%F;
%L<s,t>:=PolynomialRing(Rationals(),2);
%_<i>:=PolynomialRing(L);
%F:=(-50*i + 49)*s^4 + (196*i + 200)*s^3*t + (300*i - 294)*s^2*t^2 + (-196*i -
%    200)*s*t^3 + (-50*i + 49)*t^4;
%F;
%A:= 49*s^4 + 200*s^3*t  - 294*s^2*t^2 - 200*s*t^3 + 49*t^4;
%B:=-50*s^4 + 196*s^3*t + 300*s^2*t^2 - 196*s*t^3 - 50*t^4;
%F-A-i*B;
\begin{lstlisting}
P<s,t,u,v>:=ProjectiveSpace(Rationals(),3);
A:= 49*s^4 + 200*s^3*t  - 294*s^2*t^2 - 200*s*t^3 + 49*t^4-2*u^4;
B:=-50*s^4 + 196*s^3*t + 300*s^2*t^2 - 196*s*t^3 - 50*t^4-v^4;
S:=Scheme(P,[A,B]);
IsLocallySolvable(S,2);
\end{lstlisting}

{\bf Case 10:} $2u^4+v^4i=i(3- 2i)^2(5+2i)(s+ti)^4$.  Equating the real and imaginary parts on both sides of this equation gives
\begin{equation}\label{4901-10}
\begin{cases}
50s^4 - 196s^3t - 300s^2t^2 + 196st^3 + 50t^4 - 2u^4=0,\\
49s^4 + 200s^3t - 294s^2t^2 - 200st^3 + 49t^4 - v^4=0.
\end{cases}
\end{equation}
The scheme defined by system \eqref{4901-10} is locally insoluble at $5$.

%\begin{lstlisting}
%_<x>:=PolynomialRing(Rationals());
%k<i>:=NumberField(x^2+1);
%K<s,t>:=PolynomialRing(k,2);
%e:=1;
%F:=i^e*(3-2*i)^2*(5+2*i)*(s+t*i)^4;
%F;
%L<s,t>:=PolynomialRing(Rationals(),2);
%_<i>:=PolynomialRing(L);
%F:=(49*i + 50)*s^4 + (200*i - 196)*s^3*t + (-294*i - 300)*s^2*t^2 + (-200*i +
%    196)*s*t^3 + (49*i + 50)*t^4;
%F;
%A:=50*s^4 - 196*s^3*t - 300*s^2*t^2 + 196*s*t^3 + 50*t^4;
%B:=49*s^4 + 200*s^3*t - 294*s^2*t^2 - 200*s*t^3 + 49*t^4;
%F-A-i*B;
\begin{lstlisting}
P<s,t,u,v>:=ProjectiveSpace(Rationals(),3);
A:=50*s^4 - 196*s^3*t - 300*s^2*t^2 + 196*s*t^3 + 50*t^4-2*u^4;
B:=49*s^4 + 200*s^3*t - 294*s^2*t^2 - 200*s*t^3 + 49*t^4-v^4;
S:=Scheme(P,[A,B]);
IsLocallySolvable(S,5);
\end{lstlisting}

{\bf Case 11:} $2u^4+v^4i=i^2(3- 2i)^2(5+2i)(s+ti)^4$.  Equating the real and imaginary parts on both sides of this equation gives
\begin{equation}\label{4901-11}
\begin{cases}
-49s^4 - 200s^3t + 294s^2t^2 + 200st^3 - 49t^4 - 2u^4=0,\\
50s^4 - 196s^3t - 300s^2t^2 + 196st^3 + 50t^4 - v^4=0.
\end{cases}
\end{equation}
The scheme defined by system \eqref{4901-11} is locally insoluble at $2$.

%\begin{lstlisting}
%_<x>:=PolynomialRing(Rationals());
%k<i>:=NumberField(x^2+1);
%K<s,t>:=PolynomialRing(k,2);
%e:=2;
%F:=i^e*(3-2*i)^2*(5+2*i)*(s+t*i)^4;
%F;
%L<s,t>:=PolynomialRing(Rationals(),2);
%_<i>:=PolynomialRing(L);
%F:=(50*i - 49)*s^4 + (-196*i - 200)*s^3*t + (-300*i + 294)*s^2*t^2 + (196*i +
%    200)*s*t^3 + (50*i - 49)*t^4;
%F;
%A:=- 49*s^4 - 200*s^3*t +294*s^2*t^2 + 200*s*t^3 - 49*t^4;
%B:=50*s^4 - 196*s^3*t - 300*s^2*t^2 + 196*s*t^3 + 50*t^4;
%F-A-i*B;
\begin{lstlisting}
P<s,t,u,v>:=ProjectiveSpace(Rationals(),3);
A:=- 49*s^4 - 200*s^3*t +294*s^2*t^2 + 200*s*t^3 - 49*t^4-2*u^4;
B:=50*s^4 - 196*s^3*t - 300*s^2*t^2 + 196*s*t^3 + 50*t^4-v^4;
S:=Scheme(P,[A,B]);
IsLocallySolvable(S,2);
\end{lstlisting}

{\bf Case 12:} $2u^4+v^4i=i^3(3- 2i)^2(5+2i)(s+ti)^4$.  Equating the real and imaginary parts on both sides of this equation gives
\begin{equation}\label{4901-12}
\begin{cases}
-50s^4 + 196s^3t + 300s^2t^2 - 196st^3 - 50t^4 - 2u^4=0,\\
-49s^4 - 200s^3t + 294s^2t^2 + 200st^3 - 49t^4 - v^4=0.
\end{cases}
\end{equation}
The scheme defined by system \eqref{4901-12} is locally insoluble at $2$.

%\begin{lstlisting}
%_<x>:=PolynomialRing(Rationals());
%k<i>:=NumberField(x^2+1);
%K<s,t>:=PolynomialRing(k,2);
%e:=3;
%F:=i^e*(3-2*i)^2*(5+2*i)*(s+t*i)^4;
%F;
%L<s,t>:=PolynomialRing(Rationals(),2);
%_<i>:=PolynomialRing(L);
%F:=(-49*i - 50)*s^4 + (-200*i + 196)*s^3*t + (294*i + 300)*s^2*t^2 + (200*i -
%    196)*s*t^3 + (-49*i - 50)*t^4;
%F;
%A:= - 50*s^4 + 196*s^3*t + 300*s^2*t^2 - 196*s*t^3 - 50*t^4;
%B:=-49*s^4 - 200*s^3*t + 294*s^2*t^2 + 200*s*t^3 - 49*t^4;
%F-A-i*B;
\begin{lstlisting}
P<s,t,u,v>:=ProjectiveSpace(Rationals(),3);
A:= - 50*s^4 + 196*s^3*t + 300*s^2*t^2 - 196*s*t^3 - 50*t^4-2*u^4;
B:=-49*s^4 - 200*s^3*t + 294*s^2*t^2 + 200*s*t^3 - 49*t^4-v^4;
S:=Scheme(P,[A,B]);
IsLocallySolvable(S,2);
\end{lstlisting}

{\bf Case 13:} $2u^4+v^4i=(3- 2i)^2(5-2i)(s+ti)^4$.  Equating the real and imaginary parts on both sides of this equation gives
\begin{equation}\label{4901-13}
\begin{cases}
s^4 + 280s^3t - 6s^2t^2 - 280st^3 + t^4 - 2u^4=0,\\
-70s^4 + 4s^3t + 420s^2t^2 - 4st^3 - 70t^4 - v^4=0.
\end{cases}
\end{equation}
The scheme defined by system \eqref{4901-13} is locally insoluble at $2$.

%\begin{lstlisting}
%_<x>:=PolynomialRing(Rationals());
%k<i>:=NumberField(x^2+1);
%K<s,t>:=PolynomialRing(k,2);
%e:=0;
%F:=i^e*(3-2*i)^2*(5-2*i)*(s+t*i)^4;
%F;
%L<s,t>:=PolynomialRing(Rationals(),2);
%_<i>:=PolynomialRing(L);
%F:=(-70*i + 1)*s^4 + (4*i + 280)*s^3*t + (420*i - 6)*s^2*t^2 + (-4*i - 280)*s*t^3 +
%    (-70*i + 1)*t^4;
%F;
%A:=  s^4 + 280*s^3*t - 6*s^2*t^2 - 280*s*t^3 + t^4;
%B:=-70*s^4 + 4*s^3*t + 420*s^2*t^2 - 4*s*t^3 - 70*t^4;
%F-A-i*B;
\begin{lstlisting}
P<s,t,u,v>:=ProjectiveSpace(Rationals(),3);
A:=  s^4 + 280*s^3*t - 6*s^2*t^2 - 280*s*t^3 + t^4-2*u^4;
B:=-70*s^4 + 4*s^3*t + 420*s^2*t^2 - 4*s*t^3 - 70*t^4-v^4;
S:=Scheme(P,[A,B]);
IsLocallySolvable(S,2);
\end{lstlisting}

{\bf Case 14:} $2u^4+v^4i=i(3- 2i)^2(5-2i)(s+ti)^4$.  Equating the real and imaginary parts on both sides of this equation gives
\begin{equation}\label{4901-14}
\begin{cases}
70s^4 - 4s^3t - 420s^2t^2 + 4st^3 + 70t^4 - 2u^4=0,\\
s^4 + 280s^3t - 6s^2t^2 - 280st^3 + t^4 - v^4=0.
\end{cases}
\end{equation}
The scheme defined by system \eqref{4901-14} is locally insoluble at $2$.

%\begin{lstlisting}
%_<x>:=PolynomialRing(Rationals());
%k<i>:=NumberField(x^2+1);
%K<s,t>:=PolynomialRing(k,2);
%e:=1;
%F:=i^e*(3-2*i)^2*(5-2*i)*(s+t*i)^4;
%F;
%L<s,t>:=PolynomialRing(Rationals(),2);
%_<i>:=PolynomialRing(L);
%F:=(i + 70)*s^4 + (280*i - 4)*s^3*t + (-6*i - 420)*s^2*t^2 + (-280*i + 4)*s*t^3 +
%    (i + 70)*t^4;
%F;
%A:=70*s^4 - 4*s^3*t -  420*s^2*t^2 + 4*s*t^3 + 70*t^4;
%B:=s^4 + 280*s^3*t - 6*s^2*t^2 - 280*s*t^3 + t^4;
%F-A-i*B;

\begin{lstlisting}
P<s,t,u,v>:=ProjectiveSpace(Rationals(),3);
A:=70*s^4 - 4*s^3*t -  420*s^2*t^2 + 4*s*t^3 + 70*t^4-2*u^4;
B:=s^4 + 280*s^3*t - 6*s^2*t^2 - 280*s*t^3 + t^4-v^4;
S:=Scheme(P,[A,B]);
IsLocallySolvable(S,2);
\end{lstlisting}

{\bf Case 15:} $2u^4+v^4i=i^2(3- 2i)^2(5-2i)(s+ti)^4$.  Equating the real and imaginary parts on both sides of this equation gives
\begin{equation}\label{4901-15}
\begin{cases}
-s^4 - 280s^3t + 6s^2t^2 + 280st^3 - t^4 - 2u^4=0,\\
70s^4 - 4s^3t - 420s^2t^2 + 4st^3 + 70t^4 - v^4=0.
\end{cases}
\end{equation}
The scheme defined by system \eqref{4901-15} is locally insoluble at $2$.

%\begin{lstlisting}
%_<x>:=PolynomialRing(Rationals());
%k<i>:=NumberField(x^2+1);
%K<s,t>:=PolynomialRing(k,2);
%e:=2;
%F:=i^e*(3-2*i)^2*(5-2*i)*(s+t*i)^4;
%F;
%L<s,t>:=PolynomialRing(Rationals(),2);
%_<i>:=PolynomialRing(L);
%F:=(70*i - 1)*s^4 + (-4*i - 280)*s^3*t + (-420*i + 6)*s^2*t^2 + (4*i + 280)*s*t^3 +
%    (70*i - 1)*t^4;
%F;
%A:=- s^4 - 280*s^3*t +6*s^2*t^2 + 280*s*t^3 - t^4;
%B:=70*s^4 - 4*s^3*t - 420*s^2*t^2 + 4*s*t^3 + 70*t^4;
%F-A-i*B;
\begin{lstlisting}
P<s,t,u,v>:=ProjectiveSpace(Rationals(),3);
A:=- s^4 - 280*s^3*t +6*s^2*t^2 + 280*s*t^3 - t^4-2*u^4;
B:=70*s^4 - 4*s^3*t - 420*s^2*t^2 + 4*s*t^3 + 70*t^4-v^4;
S:=Scheme(P,[A,B]);
IsLocallySolvable(S,2);
\end{lstlisting}

{\bf Case 16:} $2u^4+v^4i=i^3(3- 2i)^2(5-2i)(s+ti)^4$.  Equating the real and imaginary parts on both sides of this equation gives
\begin{equation}\label{4901-16}
\begin{cases}
-70s^4 + 4s^3t + 420s^2t^2 - 4st^3 - 70t^4 - 2u^4=0,\\
-s^4 - 280s^3t + 6s^2t^2 + 280st^3 - t^4 - v^4=0.
\end{cases}
\end{equation}
The scheme defined by system \eqref{4901-16} is locally insoluble at $2$.

%\begin{lstlisting}
%_<x>:=PolynomialRing(Rationals());
%k<i>:=NumberField(x^2+1);
%K<s,t>:=PolynomialRing(k,2);
%e:=3;
%F:=i^e*(3-2*i)^2*(5-2*i)*(s+t*i)^4;
%F;
%L<s,t>:=PolynomialRing(Rationals(),2);
%_<i>:=PolynomialRing(L);
%F:=(-i - 70)*s^4 + (-280*i + 4)*s^3*t + (6*i + 420)*s^2*t^2 + (280*i - 4)*s*t^3 +
%    (-i - 70)*t^4;
%F;
%A:=- 70*s^4 + 4*s^3*t + 420*s^2*t^2 - 4*s*t^3 - 70*t^4;
%B:=-s^4 - 280*s^3*t + 6*s^2*t^2 + 280*s*t^3 - t^4;
%F-A-i*B;
\begin{lstlisting}
P<s,t,u,v>:=ProjectiveSpace(Rationals(),3);
A:=- 70*s^4 + 4*s^3*t + 420*s^2*t^2 - 4*s*t^3 - 70*t^4-2*u^4;
B:=-s^4 - 280*s^3*t + 6*s^2*t^2 + 280*s*t^3 - t^4-v^4;
S:=Scheme(P,[A,B]);
IsLocallySolvable(S,2);
\end{lstlisting}

\subsection{The case of $n=4961$.} %\[n=4961\]
According to Table \ref{tb:factorisation}, in the case $n=4961$ it remains to
show that equation 
\begin{equation}\label{4961}
14641u^8+v^8=82 w^4
\end{equation}
has no integer solutions satisfying \eqref{eq:cond}, which in this case reduces to \[\gcd(121u,v)=\gcd(121u,82w)=\gcd(v,82w)=1.\]
Write \eqref{4961} as
\begin{equation} (121u^4+v^4i)(121u^4-v^4i)=(1+i)(1-i)(5+4i)(5-4i)w^4.
\end{equation}
This implies that there exist integers $s,t$ such that
\[121u^4+v^4i=i^{\epsilon}(1+i)(5\pm 4i)(s+ti)^4,\]
with $\epsilon \in \{0,1,2,3\}$.

{\bf Case 1:} $121u^4+v^4i=(1+i)(5+ 4i)(s+ti)^4$. Equating the real and imaginary parts on both sides of this equation gives
\begin{equation}\label{4961-1}
\begin{cases}
s^4 - 36s^3t - 6s^2t^2 + 36st^3 + t^4 - 121u^4=0,\\
9s^4 + 4s^3t - 54s^2t^2 - 4st^3 + 9t^4 - v^4=0.
\end{cases}
\end{equation}
The scheme defined by system \eqref{4961-1} is locally insoluble at $2$.

%\begin{lstlisting}
%_<x>:=PolynomialRing(Rationals());
%k<i>:=NumberField(x^2+1);
%K<s,t>:=PolynomialRing(k,2);
%e:=0;
%F:=i^e*(1+i)*(5+4*i)*(s+t*i)^4;
%F;
%L<s,t>:=PolynomialRing(Rationals(),2);
%_<i>:=PolynomialRing(L);
%F:=(9*i + 1)*s^4 + (4*i - 36)*s^3*t + (-54*i - 6)*s^2*t^2 + (-4*i + 36)*s*t^3 +
%    (9*i + 1)*t^4;
%F;
%A:=s^4 - 36*s^3*t - 6*s^2*t^2+ 36*s*t^3 + t^4;
%B:=9*s^4 + 4*s^3*t - 54*s^2*t^2 - 4*s*t^3 + 9*t^4;
%F-A-i*B;
\begin{lstlisting}
P<s,t,u,v>:=ProjectiveSpace(Rationals(),3);
A:=s^4 - 36*s^3*t - 6*s^2*t^2+ 36*s*t^3 + t^4-121*u^4;
B:=9*s^4 + 4*s^3*t - 54*s^2*t^2 - 4*s*t^3 + 9*t^4-v^4;
S:=Scheme(P,[A,B]);
IsLocallySolvable(S,2);
\end{lstlisting}

{\bf Case 2:} $121u^4+v^4i=i(1+i)(5+ 4i)(s+ti)^4$.  Equating the real and imaginary parts on both sides of this equation gives
\begin{equation}\label{4961-2}
\begin{cases}
-9s^4 - 4s^3t + 54s^2t^2 + 4st^3 - 9t^4 - 121u^4=0,\\
s^4 - 36s^3t - 6s^2t^2 + 36st^3 + t^4 - v^4=0.
\end{cases}
\end{equation}
The scheme defined by system \eqref{4961-2} is locally insoluble at $2$.

%\begin{lstlisting}
%_<x>:=PolynomialRing(Rationals());
%k<i>:=NumberField(x^2+1);
%K<s,t>:=PolynomialRing(k,2);
%e:=1;
%F:=i^e*(1+i)*(5+4*i)*(s+t*i)^4;
%F;
%L<s,t>:=PolynomialRing(Rationals(),2);
%_<i>:=PolynomialRing(L);
%F:=(i - 9)*s^4 + (-36*i - 4)*s^3*t + (-6*i + 54)*s^2*t^2 + (36*i + 4)*s*t^3 + (i -
%    9)*t^4;
%F;
%A:=- 9*s^4 - 4*s^3*t + 54*s^2*t^2 +4*s*t^3 - 9*t^4;
%B:=s^4 - 36*s^3*t - 6*s^2*t^2 + 36*s*t^3 + t^4;
%F-A-i*B;
\begin{lstlisting}
P<s,t,u,v>:=ProjectiveSpace(Rationals(),3);
A:=- 9*s^4 - 4*s^3*t + 54*s^2*t^2 +4*s*t^3 - 9*t^4-121*u^4;
B:=s^4 - 36*s^3*t - 6*s^2*t^2 + 36*s*t^3 + t^4-v^4;
S:=Scheme(P,[A,B]);
IsLocallySolvable(S,2);
\end{lstlisting}

{\bf Case 3:} $121u^4+v^4i=i^2(1+i)(5+ 4i)(s+ti)^4$.  Equating the real and imaginary parts on both sides of this equation gives
\begin{equation}\label{4961-3}
\begin{cases}
-s^4 + 36s^3t + 6s^2t^2 - 36st^3 - t^4 - 121u^4=0,\\
-9s^4 - 4s^3t + 54s^2t^2 + 4st^3 - 9t^4 - v^4=0.
\end{cases}
\end{equation}
The scheme defined by system \eqref{4961-3} is locally insoluble at $2$.

%\begin{lstlisting}
%_<x>:=PolynomialRing(Rationals());
%k<i>:=NumberField(x^2+1);
%K<s,t>:=PolynomialRing(k,2);
%e:=2;
%F:=i^e*(1+i)*(5+4*i)*(s+t*i)^4;
%F;
%L<s,t>:=PolynomialRing(Rationals(),2);
%_<i>:=PolynomialRing(L);
%F:=(-9*i - 1)*s^4 + (-4*i + 36)*s^3*t + (54*i + 6)*s^2*t^2 + (4*i - 36)*s*t^3 +
%    (-9*i - 1)*t^4;
%F;
%A:=- s^4 + 36*s^3*t + 6*s^2*t^2  - 36*s*t^3 - t^4;
%B:=-9*s^4 - 4*s^3*t + 54*s^2*t^2 + 4*s*t^3 - 9*t^4;
%F-A-i*B;
\begin{lstlisting}
P<s,t,u,v>:=ProjectiveSpace(Rationals(),3);
A:=- s^4 + 36*s^3*t + 6*s^2*t^2  - 36*s*t^3 - t^4-121*u^4;
B:=-9*s^4 - 4*s^3*t + 54*s^2*t^2 + 4*s*t^3 - 9*t^4-v^4;
S:=Scheme(P,[A,B]);
IsLocallySolvable(S,2);
\end{lstlisting}

{\bf Case 4:} $121u^4+v^4i=i^3(1+i)(5+ 4i)(s+ti)^4$.  Equating the real and imaginary parts on both sides of this equation gives
\begin{equation}\label{4961-4}
\begin{cases}
9s^4 + 4s^3t - 54s^2t^2 - 4st^3 + 9t^4 - 121u^4=0, \\
-s^4 + 36s^3t + 6s^2t^2 - 36st^3 - t^4 - v^4=0.
%-s^4 + 36*s^3*t + 6*s^2*t^2 - 36*s*t^3 - t^4 - 121*u^4=0,\\
%-9*s^4 - 4*s^3*t + 54*s^2*t^2 + 4*s*t^3 - 9*t^4 - v^4=0.
\end{cases}
\end{equation}
The scheme defined by system \eqref{4961-4} is locally insoluble at $2$.

%\begin{lstlisting}
%_<x>:=PolynomialRing(Rationals());
%k<i>:=NumberField(x^2+1);
%K<s,t>:=PolynomialRing(k,2);
%e:=2;
%F:=i^e*(1+i)*(5+4*i)*(s+t*i)^4;
%F;
%L<s,t>:=PolynomialRing(Rationals(),2);
%_<i>:=PolynomialRing(L);
%F:=(-9*i - 1)*s^4 + (-4*i + 36)*s^3*t + (54*i + 6)*s^2*t^2 + (4*i - 36)*s*t^3 +
%    (-9*i - 1)*t^4;
%F;
%A:=- s^4 + 36*s^3*t + 6*s^2*t^2  - 36*s*t^3 - t^4;
%B:=-9*s^4 - 4*s^3*t + 54*s^2*t^2 + 4*s*t^3 - 9*t^4;
%F-A-i*B;
\begin{lstlisting}
P<s,t,u,v>:=ProjectiveSpace(Rationals(),3);
A:=9*s^4 + 4*s^3*t - 54*s^2*t^2 - 4*s*t^3 + 9*t^4 - 121*u^4;
B:=-s^4 + 36*s^3*t + 6*s^2*t^2 - 36*s*t^3 - t^4 - v^4;
S:=Scheme(P,[A,B]);
IsLocallySolvable(S,2);
\end{lstlisting}

{\bf Case 5:} $121u^4+v^4i=(1+i)(5- 4i)(s+ti)^4$.  Equating the real and imaginary parts on both sides of this equation gives
\begin{equation}\label{4961-5}
\begin{cases}
9s^4 - 4s^3t - 54s^2t^2 + 4st^3 + 9t^4 - 121u^4=0,\\
s^4 + 36s^3t - 6s^2t^2 - 36st^3 + t^4 - v^4=0.
\end{cases}
\end{equation}
The scheme defined by system \eqref{4961-5} is locally insoluble at $5$.

%\begin{lstlisting}
%_<x>:=PolynomialRing(Rationals());
%k<i>:=NumberField(x^2+1);
%K<s,t>:=PolynomialRing(k,2);
%e:=0;
%F:=i^e*(1+i)*(5-4*i)*(s+t*i)^4;
%F;
%L<s,t>:=PolynomialRing(Rationals(),2);
%_<i>:=PolynomialRing(L);
%F:=(i + 9)*s^4 + (36*i - 4)*s^3*t + (-6*i - 54)*s^2*t^2 + (-36*i + 4)*s*t^3 + (i +
%    9)*t^4;
%F;
%A:= 9*s^4 - 4*s^3*t - 54*s^2*t^2 +4*s*t^3 + 9*t^4;
%B:=s^4 + 36*s^3*t - 6*s^2*t^2 - 36*s*t^3 + t^4;
%F-A-i*B;
\begin{lstlisting}
P<s,t,u,v>:=ProjectiveSpace(Rationals(),3);
A:= 9*s^4 - 4*s^3*t - 54*s^2*t^2 +4*s*t^3 + 9*t^4-121*u^4;
B:=s^4 + 36*s^3*t - 6*s^2*t^2 - 36*s*t^3 + t^4-v^4;
S:=Scheme(P,[A,B]);
IsLocallySolvable(S,5);
\end{lstlisting}

{\bf Case 6:} $121u^4+v^4i=i(1+i)(5- 4i)(s+ti)^4$.  Equating the real and imaginary parts on both sides of this equation gives
\begin{equation}\label{4961-6}
\begin{cases}
-s^4 - 36s^3t + 6s^2t^2 + 36st^3 - t^4 - 121u^4=0,\\
9s^4 - 4s^3t - 54s^2t^2 + 4st^3 + 9t^4 - v^4=0.
\end{cases}
\end{equation}
The scheme defined by system \eqref{4961-6} is locally insoluble at $2$.

%\begin{lstlisting}
%_<x>:=PolynomialRing(Rationals());
%k<i>:=NumberField(x^2+1);
%K<s,t>:=PolynomialRing(k,2);
%e:=1;
%F:=i^e*(1+i)*(5-4*i)*(s+t*i)^4;
%F;
%L<s,t>:=PolynomialRing(Rationals(),2);
%_<i>:=PolynomialRing(L);
%F:=(9*i - 1)*s^4 + (-4*i - 36)*s^3*t + (-54*i + 6)*s^2*t^2 + (4*i + 36)*s*t^3 +
%    (9*i - 1)*t^4;
%F;
%A:=- s^4 - 36*s^3*t + 6*s^2*t^2 + 36*s*t^3 - t^4;
%B:=9*s^4 - 4*s^3*t - 54*s^2*t^2 + 4*s*t^3 + 9*t^4;
%F-A-i*B;
\begin{lstlisting}
P<s,t,u,v>:=ProjectiveSpace(Rationals(),3);
A:=- s^4 - 36*s^3*t + 6*s^2*t^2 + 36*s*t^3 - t^4-121*u^4;
B:=9*s^4 - 4*s^3*t - 54*s^2*t^2 + 4*s*t^3 + 9*t^4-v^4;
S:=Scheme(P,[A,B]);
IsLocallySolvable(S,2);
\end{lstlisting}

{\bf Case 7:} $121u^4+v^4i=i^2(1+i)(5- 4i)(s+ti)^4$.  Equating the real and imaginary parts on both sides of this equation gives
\begin{equation}\label{4961-7}
\begin{cases}
-9s^4 + 4s^3t + 54s^2t^2 - 4st^3 - 9t^4 - 121u^4=0,\\
-s^4 - 36s^3t + 6s^2t^2 + 36st^3 - t^4 - v^4=0.
\end{cases}
\end{equation}
The scheme defined by system \eqref{4961-7} is locally insoluble at $2$.

%\begin{lstlisting}
%_<x>:=PolynomialRing(Rationals());
%k<i>:=NumberField(x^2+1);
%K<s,t>:=PolynomialRing(k,2);
%e:=2;
%F:=i^e*(1+i)*(5-4*i)*(s+t*i)^4;
%F;
%L<s,t>:=PolynomialRing(Rationals(),2);
%_<i>:=PolynomialRing(L);
%F:=(-i - 9)*s^4 + (-36*i + 4)*s^3*t + (6*i + 54)*s^2*t^2 + (36*i - 4)*s*t^3 + (-i -
%    9)*t^4;
%F;
%A:= - 9*s^4 + 4*s^3*t + 54*s^2*t^2  - 4*s*t^3 - 9*t^4;
%B:=-s^4 - 36*s^3*t + 6*s^2*t^2 + 36*s*t^3 - t^4;
%F-A-i*B;
\begin{lstlisting}
P<s,t,u,v>:=ProjectiveSpace(Rationals(),3);
A:=- 9*s^4 + 4*s^3*t + 54*s^2*t^2 - 4*s*t^3 - 9*t^4-121*u^4;
B:=-s^4 - 36*s^3*t + 6*s^2*t^2 + 36*s*t^3 - t^4-v^4;
S:=Scheme(P,[A,B]);
IsLocallySolvable(S,2);
\end{lstlisting}

{\bf Case 8:} $121u^4+v^4i=i^3(1+i)(5- 4i)(s+ti)^4$.  Equating the real and imaginary parts on both sides of this equation gives
\begin{equation}\label{4961-8}
\begin{cases}
s^4 + 36s^3t - 6s^2t^2 - 36st^3 + t^4 - 121u^4=0,\\
-9s^4 + 4s^3t + 54s^2t^2 - 4st^3 - 9t^4 - v^4=0.
\end{cases}
\end{equation}
The scheme defined by system \eqref{4961-8} is locally insoluble at $2$.

%\begin{lstlisting}
%_<x>:=PolynomialRing(Rationals());
%k<i>:=NumberField(x^2+1);
%K<s,t>:=PolynomialRing(k,2);
%e:=3;
%F:=i^e*(1+i)*(5-4*i)*(s+t*i)^4;
%F;
%L<s,t>:=PolynomialRing(Rationals(),2);
%_<i>:=PolynomialRing(L);
%F:=(-9*i + 1)*s^4 + (4*i + 36)*s^3*t + (54*i - 6)*s^2*t^2 + (-4*i - 36)*s*t^3 +
%    (-9*i + 1)*t^4;
%F;
%A:=s^4 + 36*s^3*t - 6*s^2*t^2  - 36*s*t^3 + t^4;
%B:=-9*s^4 + 4*s^3*t + 54*s^2*t^2 - 4*s*t^3 - 9*t^4;
%F-A-i*B;
%
\begin{lstlisting}
P<s,t,u,v>:=ProjectiveSpace(Rationals(),3);
A:=s^4 + 36*s^3*t - 6*s^2*t^2  - 36*s*t^3 + t^4-121*u^4;
B:=-9*s^4 + 4*s^3*t + 54*s^2*t^2 - 4*s*t^3 - 9*t^4-v^4;
S:=Scheme(P,[A,B]);
IsLocallySolvable(S,2);
\end{lstlisting}

\subsection{The case of $n=5807$.}%\[n=5807.\]
According to Table \ref{tb:factorisation}, in the case $n=5807$ it remains to
show that equation
\begin{equation}\label{5807}
33721249 u^8+v^8=2  w^4
\end{equation}
has no integer solutions satisfying \eqref{eq:cond}, which in this case reduces to \[\gcd(5807u,v)=\gcd(5807u,2w)=\gcd(v,2w)=1.\] 
Write \eqref{5807} as
\begin{equation} (5807u^4+v^4i)(5807u^4-v^4i)=(1+i)(1-i)w^4.
\end{equation}
This implies that there exist integers $s,t$ such that
\[5807u^4+v^4i=i^{\epsilon}(1+i)(s+ti)^4,\]
with $\epsilon \in \{0,1,2,3\}$.

{\bf Case 1:} $5807u^4+v^4i=(1+i)(s+ti)^4$. Equating the real and imaginary parts on both sides of this equation gives
\begin{equation}\label{5807-1}
\begin{cases}
s^4 - 4s^3t - 6s^2t^2 + 4st^3 + t^4 - 5807u^4=0,\\
s^4 + 4s^3t - 6s^2t^2 - 4st^3 + t^4 - v^4=0.
\end{cases}
\end{equation}
The scheme defined by system \eqref{5807-1} is locally insoluble at $2$.

%\begin{lstlisting}
%_<x>:=PolynomialRing(Rationals());
%k<i>:=NumberField(x^2+1);
%K<s,t>:=PolynomialRing(k,2);
%e:=0;
%F:=i^e*(1+i)*(s+t*i)^4;
%F;
%L<s,t>:=PolynomialRing(Rationals(),2);
%_<i>:=PolynomialRing(L);
%F:=(i + 1)*s^4 + (4*i - 4)*s^3*t + (-6*i - 6)*s^2*t^2 + (-4*i + 4)*s*t^3 + (i +
%    1)*t^4;
%F;
%A:= s^4 - 4*s^3*t - 6*s^2*t^2 + 4*s*t^3 + t^4;
%B:=s^4 + 4*s^3*t - 6*s^2*t^2 - 4*s*t^3 + t^4;
%F-A-i*B;
\begin{lstlisting}
P<s,t,u,v>:=ProjectiveSpace(Rationals(),3);
A:= s^4 - 4*s^3*t - 6*s^2*t^2 + 4*s*t^3 + t^4-5807*u^4;
B:=s^4 + 4*s^3*t - 6*s^2*t^2 - 4*s*t^3 + t^4-v^4;
S:=Scheme(P,[A,B]);
IsLocallySolvable(S,2);
\end{lstlisting}

{\bf Case 2:} $5807u^4+v^4i=i(1+i)(s+ti)^4$. Equating the real and imaginary parts on both sides of this equation gives
\begin{equation}\label{5807-2}
\begin{cases}
-s^4 - 4s^3t + 6s^2t^2 + 4st^3 - t^4 - 5807u^4=0,\\
s^4 - 4s^3t - 6s^2t^2 + 4st^3 + t^4 - v^4=0.
\end{cases}
\end{equation}
The scheme defined by system \eqref{5807-2} is locally insoluble at $5$.

%\begin{lstlisting}
%_<x>:=PolynomialRing(Rationals());
%k<i>:=NumberField(x^2+1);
%K<s,t>:=PolynomialRing(k,2);
%e:=1;
%F:=i^e*(1+i)*(s+t*i)^4;
%F;
%L<s,t>:=PolynomialRing(Rationals(),2);
%_<i>:=PolynomialRing(L);
%F:=(i - 1)*s^4 + (-4*i - 4)*s^3*t + (-6*i + 6)*s^2*t^2 + (4*i + 4)*s*t^3 + (i -
%    1)*t^4;
%F;
%A:= - s^4 - 4*s^3*t + 6*s^2*t^2 +4*s*t^3 - t^4;
%B:=s^4 - 4*s^3*t - 6*s^2*t^2 + 4*s*t^3 + t^4;
%F-A-i*B;
\begin{lstlisting}
P<s,t,u,v>:=ProjectiveSpace(Rationals(),3);
A:= - s^4 - 4*s^3*t + 6*s^2*t^2 +4*s*t^3 - t^4-5807*u^4;
B:=s^4 - 4*s^3*t - 6*s^2*t^2 + 4*s*t^3 + t^4-v^4;
S:=Scheme(P,[A,B]);
IsLocallySolvable(S,5);
\end{lstlisting}
{\bf Case 3:} $5807u^4+v^4i=i^2(1+i)(s+ti)^4$.  Equating the real and imaginary parts on both sides of this equation gives
\begin{equation}\label{5807-3}
\begin{cases}
-s^4 + 4s^3t + 6s^2t^2 - 4st^3 - t^4 - 5807u^4=0,\\
-s^4 - 4s^3t + 6s^2t^2 + 4st^3 - t^4 - v^4=0.
\end{cases}
\end{equation}
The scheme defined by system \eqref{5807-3} is locally insoluble at $2$.

%\begin{lstlisting}
%_<x>:=PolynomialRing(Rationals());
%k<i>:=NumberField(x^2+1);
%K<s,t>:=PolynomialRing(k,2);
%e:=2;
%F:=i^e*(1+i)*(s+t*i)^4;
%F;
%L<s,t>:=PolynomialRing(Rationals(),2);
%_<i>:=PolynomialRing(L);
%F:=(-i - 1)*s^4 + (-4*i + 4)*s^3*t + (6*i + 6)*s^2*t^2 + (4*i - 4)*s*t^3 + (-i -
%    1)*t^4;
%F;
%A:=- s^4 + 4*s^3*t + 6*s^2*t^2 - 4*s*t^3 - t^4;
%B:=-s^4 - 4*s^3*t + 6*s^2*t^2 + 4*s*t^3 - t^4;
%F-A-i*B;
\begin{lstlisting}
P<s,t,u,v>:=ProjectiveSpace(Rationals(),3);
A:=- s^4 + 4*s^3*t + 6*s^2*t^2 - 4*s*t^3 - t^4-5807*u^4;
B:=-s^4 - 4*s^3*t + 6*s^2*t^2 + 4*s*t^3 - t^4-v^4;
S:=Scheme(P,[A,B]);
IsLocallySolvable(S,2);
\end{lstlisting}

{\bf Case 4:} $5807u^4+v^4i=i^3(1+i)(s+ti)^4$.  Equating the real and imaginary parts on both sides of this equation gives
\begin{equation}\label{5807-4}
\begin{cases}
s^4 + 4s^3t - 6s^2t^2 - 4st^3 + t^4 - 5807u^4=0,\\
-s^4 + 4s^3t + 6s^2t^2 - 4st^3 - t^4 - v^4=0.
\end{cases}
\end{equation}
The scheme defined by system \eqref{5807-4} is locally insoluble at $2$.

%\begin{lstlisting}
%_<x>:=PolynomialRing(Rationals());
%k<i>:=NumberField(x^2+1);
%K<s,t>:=PolynomialRing(k,2);
%e:=3;
%F:=i^e*(1+i)*(s+t*i)^4;
%F;
%L<s,t>:=PolynomialRing(Rationals(),2);
%_<i>:=PolynomialRing(L);
%F:=(-i + 1)*s^4 + (4*i + 4)*s^3*t + (6*i - 6)*s^2*t^2 + (-4*i - 4)*s*t^3 + (-i +
%    1)*t^4;
%F;
%A:= s^4 + 4*s^3*t - 6*s^2*t^2 - 4*s*t^3 + t^4;
%B:=-s^4 + 4*s^3*t + 6*s^2*t^2 - 4*s*t^3 - t^4;
%F-A-i*B;
\begin{lstlisting}
P<s,t,u,v>:=ProjectiveSpace(Rationals(),3);
A:= s^4 + 4*s^3*t - 6*s^2*t^2 - 4*s*t^3 + t^4-5807*u^4;
B:=-s^4 + 4*s^3*t + 6*s^2*t^2 - 4*s*t^3 - t^4-v^4;
S:=Scheme(P,[A,B]);
IsLocallySolvable(S,2);
\end{lstlisting}

\subsection{The case of $n=5911$.}%\[n=5911\] 
According to Table \ref{tb:factorisation}, in the case $n=5911$ it remains to
show that equation
\begin{equation}\label{5911.1}
66049 u^8 + 2116 v^8 =w^4
\end{equation}
has no integer solutions satisfying \eqref{eq:cond}, which in this case reduces to \[\gcd(257u,46v)=\gcd(257u,w)=\gcd(46v,w)=1.\]  
Write \eqref{5911.1} as
\begin{equation} (257u^4+46v^4i)(257u^4-46v^4i)=w^4.
\end{equation}
This implies that there exist integers $s,t$ such that 
\[257u^4+46v^4i=i^{\epsilon}(s+ti)^4,\]
with $\epsilon \in \{0,1,2,3\}$.

{\bf Case 1:} $257u^4+46v^4i=(s+ti)^4$. Equating the real and imaginary parts on both sides of this equation gives
\begin{equation}\label{5911.1-1}
\begin{cases}
s^4 - 6s^2t^2 + t^4 - 257u^4=0,\\
4s^3t - 4st^3 - 46v^4=0.
\end{cases}
\end{equation}
The scheme defined by system \eqref{5911.1-1} is locally insoluble at $5$.
%\begin{lstlisting}
%_<x>:=PolynomialRing(Rationals());
%k<i>:=NumberField(x^2+1);
%K<s,t>:=PolynomialRing(k,2);
%e:=0;
%F:=i^e*(s+t*i)^4;
%F;
%L<s,t>:=PolynomialRing(Rationals(),2);
%_<i>:=PolynomialRing(L);
%F:=s^4 + 4*i*s^3*t - 6*s^2*t^2 - 4*i*s*t^3 + t^4;
%F;
%A:=s^4 - 6*s^2*t^2 + t^4;
%B:=4*s^3*t - 4*s*t^3;
%F-A-i*B;
\begin{lstlisting}
P<s,t,u,v>:=ProjectiveSpace(Rationals(),3);
A:=s^4 - 6*s^2*t^2 + t^4-257*u^4;
B:=4*s^3*t - 4*s*t^3-46*v^4;
S:=Scheme(P,[A,B]);
IsLocallySolvable(S,5);
\end{lstlisting}

{\bf Case 2:} $257u^4+46v^4i=i(s+ti)^4$. Equating the real and imaginary parts on both sides of this equation gives
\begin{equation}\label{5911.1-2}
\begin{cases}
-4s^3t + 4st^3 - 257u^4=0,\\
s^4 - 6s^2t^2 + t^4 - 46v^4=0.
\end{cases}
\end{equation}
The scheme defined by system \eqref{5911.1-2} is locally insoluble at $2$.
%\begin{lstlisting}
%_<x>:=PolynomialRing(Rationals());
%k<i>:=NumberField(x^2+1);
%K<s,t>:=PolynomialRing(k,2);
%e:=1;
%F:=i^e*(s+t*i)^4;
%F;
%L<s,t>:=PolynomialRing(Rationals(),2);
%_<i>:=PolynomialRing(L);
%F:=i*s^4 - 4*s^3*t - 6*i*s^2*t^2 + 4*s*t^3 + i*t^4;
%F;
%A:=- 4*s^3*t + 4*s*t^3;
%B:=s^4 - 6*s^2*t^2 + t^4;
%F-A-i*B;
\begin{lstlisting}
P<s,t,u,v>:=ProjectiveSpace(Rationals(),3);
A:=- 4*s^3*t + 4*s*t^3-257*u^4;
B:=s^4 - 6*s^2*t^2 + t^4-46*v^4;
S:=Scheme(P,[A,B]);
IsLocallySolvable(S,2);
\end{lstlisting}

{\bf Case 3:} $257u^4+46v^4i=i^2(s+ti)^4$. Equating the real and imaginary parts on both sides of this equation gives
\begin{equation}\label{5911.1-3}
\begin{cases}
-s^4 + 6s^2t^2 - t^4 - 257u^4=0,\\
-4s^3t + 4st^3 - 46v^4=0.
\end{cases}
\end{equation}
The scheme defined by system \eqref{5911.1-3} is locally insoluble at $2$.

%\begin{lstlisting}
%_<x>:=PolynomialRing(Rationals());
%k<i>:=NumberField(x^2+1);
%K<s,t>:=PolynomialRing(k,2);
%e:=2;
%F:=i^e*(s+t*i)^4;
%F;
%L<s,t>:=PolynomialRing(Rationals(),2);
%_<i>:=PolynomialRing(L);
%F:=-s^4 - 4*i*s^3*t + 6*s^2*t^2 + 4*i*s*t^3 - t^4;
%F;
%A:=- s^4 + 6*s^2*t^2 - t^4;
%B:=(-4*s^3*t + 4*s*t^3);
%F-A-i*B;

\begin{lstlisting}
P<s,t,u,v>:=ProjectiveSpace(Rationals(),3);
A:=- s^4 + 6*s^2*t^2 - t^4-257*u^4;
B:=-4*s^3*t + 4*s*t^3-46*v^4;
S:=Scheme(P,[A,B]);
IsLocallySolvable(S,2);
\end{lstlisting}

{\bf Case 4:} $257u^4+46v^4i=i^3(s+ti)^4$. Equating the real and imaginary parts on both sides of this equation gives
\begin{equation}\label{5911.1-4}
\begin{cases}
4s^3t - 4st^3 - 257u^4=0,\\
-s^4 + 6s^2t^2 - t^4 - 46v^4=0.
\end{cases}
\end{equation}
The scheme defined by system \eqref{5911.1-4} is locally insoluble at $2$.

%\begin{lstlisting}
%_<x>:=PolynomialRing(Rationals());
%k<i>:=NumberField(x^2+1);
%K<s,t>:=PolynomialRing(k,2);
%e:=3;
%F:=i^e*(s+t*i)^4;
%F;
%L<s,t>:=PolynomialRing(Rationals(),2);
%_<i>:=PolynomialRing(L);
%F:=-i*s^4 + 4*s^3*t + 6*i*s^2*t^2 - 4*s*t^3 - i*t^4;
%F;
%A:= 4*s^3*t - 4*s*t^3;
%B:=(-s^4 + 6*s^2*t^2 - t^4);
%F-A-i*B;
\begin{lstlisting}
P<s,t,u,v>:=ProjectiveSpace(Rationals(),3);
A:=4*s^3*t - 4*s*t^3-257*u^4;
B:=-s^4 + 6*s^2*t^2 - t^4-46*v^4;
S:=Scheme(P,[A,B]);
IsLocallySolvable(S,2);
\end{lstlisting}

\subsection{The case of $n=6001$.}%\[n=6001\] 
According to Table \ref{tb:factorisation}, in the case $n=6001$ it remains to
show that equation
\begin{equation}\label{6001}
289 u^8+4 v^8=353  w^4
\end{equation}
has no integer solutions satisfying \eqref{eq:cond}, which in this case reduces to \[\gcd(17u,2v)=\gcd(17u,353w)=\gcd(2v,353w)=1.\] Write \eqref{6001} as
\begin{equation} (17u^4+2v^4i)(17u^4-2v^4i)=(17+8i)(17-8i)w^4.
\end{equation}
This implies that there exist integers $s,t$ such that 
\[17u^4+2v^4i=i^{\epsilon}(17\pm 8i)(s+ti)^4,\]
with $\epsilon \in \{0,1,2,3\}$.

{\bf Case 1:} $17u^4+2v^4i=(17+8i)(s+ti)^4$. Equating the real and imaginary parts on both sides of this equation gives
\begin{equation}\label{6001-1}
\begin{cases}
17s^4 - 32s^3t - 102s^2t^2 + 32st^3 + 17t^4 - 17u^4=0,\\
8s^4 + 68s^3t - 48s^2t^2 - 68st^3 + 8t^4 - 2v^4=0.
\end{cases}
\end{equation}
The scheme defined by system \eqref{6001-1} is locally insoluble at $2$.

%\begin{lstlisting}
%_<x>:=PolynomialRing(Rationals());
%k<i>:=NumberField(x^2+1);
%K<s,t>:=PolynomialRing(k,2);
%e:=0;
%F:=i^e*(17+8*i)*(s+t*i)^4;
%F;
%L<s,t>:=PolynomialRing(Rationals(),2);
%_<i>:=PolynomialRing(L);
%F:=(8*i + 17)*s^4 + (68*i - 32)*s^3*t + (-48*i - 102)*s^2*t^2 + (-68*i + 32)*s*t^3
%    + (8*i + 17)*t^4;
%F;
%A:= 17*s^4 - 32*s^3*t -102*s^2*t^2 + 32*s*t^3 + 17*t^4;
%B:=8*s^4 + 68*s^3*t - 48*s^2*t^2 - 68*s*t^3 + 8*t^4;
%F-A-i*B;
\begin{lstlisting}
P<s,t,u,v>:=ProjectiveSpace(Rationals(),3);
A:= 17*s^4 - 32*s^3*t -102*s^2*t^2 + 32*s*t^3 + 17*t^4-17*u^4;
B:=8*s^4 + 68*s^3*t - 48*s^2*t^2 - 68*s*t^3 + 8*t^4-2*v^4;
S:=Scheme(P,[A,B]);
IsLocallySolvable(S,2);
\end{lstlisting}

{\bf Case 2:} $17u^4+2v^4i=i(17+8i)(s+ti)^4$. Equating the real and imaginary parts on both sides of this equation gives
\begin{equation}\label{6001-2}
\begin{cases}
-8s^4 - 68s^3t + 48s^2t^2 + 68st^3 - 8t^4 - 17u^4=0,\\
17s^4 - 32s^3t - 102s^2t^2 + 32st^3 + 17t^4 - 2v^4=0.
\end{cases}
\end{equation}
The scheme defined by system \eqref{6001-2} is locally insoluble at $2$.

%\begin{lstlisting}
%_<x>:=PolynomialRing(Rationals());
%k<i>:=NumberField(x^2+1);
%K<s,t>:=PolynomialRing(k,2);
%e:=1;
%F:=i^e*(17+8*i)*(s+t*i)^4;
%F;
%L<s,t>:=PolynomialRing(Rationals(),2);
%_<i>:=PolynomialRing(L);
%F:=(17*i - 8)*s^4 + (-32*i - 68)*s^3*t + (-102*i + 48)*s^2*t^2 + (32*i + 68)*s*t^3
%    + (17*i - 8)*t^4;
%F;
%A:= - 8*s^4 - 68*s^3*t +48*s^2*t^2 + 68*s*t^3 - 8*t^4;
%B:=17*s^4 - 32*s^3*t - 102*s^2*t^2 + 32*s*t^3 + 17*t^4;
%F-A-i*B;
\begin{lstlisting}
P<s,t,u,v>:=ProjectiveSpace(Rationals(),3);
A:= - 8*s^4 - 68*s^3*t +48*s^2*t^2 + 68*s*t^3 - 8*t^4-17*u^4;
B:=17*s^4 - 32*s^3*t - 102*s^2*t^2 + 32*s*t^3 + 17*t^4-2*v^4;
S:=Scheme(P,[A,B]);
IsLocallySolvable(S,2);
\end{lstlisting}

{\bf Case 3:} $17u^4+2v^4i=i^2(17+8i)(s+ti)^4$. Equating the real and imaginary parts on both sides of this equation gives
\begin{equation}\label{6001-3}
\begin{cases}
-17s^4 + 32s^3t + 102s^2t^2 - 32st^3 - 17t^4 - 17u^4=0,\\
-8s^4 - 68s^3t + 48s^2t^2 + 68st^3 - 8t^4 - 2v^4=0.
\end{cases}
\end{equation}
The scheme defined by system \eqref{6001-3} is locally insoluble at $2$.

%\begin{lstlisting}
%_<x>:=PolynomialRing(Rationals());
%k<i>:=NumberField(x^2+1);
%K<s,t>:=PolynomialRing(k,2);
%e:=2;
%F:=i^e*(17+8*i)*(s+t*i)^4;
%F;
%L<s,t>:=PolynomialRing(Rationals(),2);
%_<i>:=PolynomialRing(L);
%F:=(-8*i - 17)*s^4 + (-68*i + 32)*s^3*t + (48*i + 102)*s^2*t^2 + (68*i - 32)*s*t^3
%    + (-8*i - 17)*t^4;
%F;
%A:= - 17*s^4 + 32*s^3*t +102*s^2*t^2 - 32*s*t^3 - 17*t^4;
%B:=-8*s^4 - 68*s^3*t + 48*s^2*t^2 + 68*s*t^3 - 8*t^4;
%F-A-i*B;
\begin{lstlisting}
P<s,t,u,v>:=ProjectiveSpace(Rationals(),3);
A:= - 17*s^4 + 32*s^3*t +102*s^2*t^2 - 32*s*t^3 - 17*t^4-17*u^4;
B:=-8*s^4 - 68*s^3*t + 48*s^2*t^2 + 68*s*t^3 - 8*t^4-2*v^4;
S:=Scheme(P,[A,B]);
IsLocallySolvable(S,2);

\end{lstlisting}

{\bf Case 4:} $17u^4+2v^4i=i^3(17+8i)(s+ti)^4$. Equating the real and imaginary parts on both sides of this equation gives
\begin{equation}\label{6001-4}
\begin{cases}
8s^4 + 68s^3t - 48s^2t^2 - 68st^3 + 8t^4 - 17u^4=0,\\
-17s^4 + 32s^3t + 102s^2t^2 - 32st^3 - 17t^4 - 2v^4=0.
\end{cases}
\end{equation}
The scheme defined by system \eqref{6001-4} is locally insoluble at $2$.

%\begin{lstlisting}
%_<x>:=PolynomialRing(Rationals());
%k<i>:=NumberField(x^2+1);
%K<s,t>:=PolynomialRing(k,2);
%e:=3;
%F:=i^e*(17+8*i)*(s+t*i)^4;
%F;
%L<s,t>:=PolynomialRing(Rationals(),2);
%_<i>:=PolynomialRing(L);
%F:=(-17*i + 8)*s^4 + (32*i + 68)*s^3*t + (102*i - 48)*s^2*t^2 + (-32*i - 68)*s*t^3
%    + (-17*i + 8)*t^4;
%F;
%A:=8*s^4 + 68*s^3*t -  48*s^2*t^2 - 68*s*t^3 + 8*t^4;
%B:=-17*s^4 + 32*s^3*t + 102*s^2*t^2 - 32*s*t^3 - 17*t^4;
%F-A-i*B;
\begin{lstlisting}
P<s,t,u,v>:=ProjectiveSpace(Rationals(),3);
A:=8*s^4 + 68*s^3*t -  48*s^2*t^2 - 68*s*t^3 + 8*t^4-17*u^4;
B:=-17*s^4 + 32*s^3*t + 102*s^2*t^2 - 32*s*t^3 - 17*t^4-2*v^4;
S:=Scheme(P,[A,B]);
IsLocallySolvable(S,2);
\end{lstlisting}

{\bf Case 5:} $17u^4+2v^4i=(17-8i)(s+ti)^4$. Equating the real and imaginary parts on both sides of this equation gives
\begin{equation}\label{6001-5}
\begin{cases}
17s^4 + 32s^3t - 102s^2t^2 - 32st^3 + 17t^4 - 17u^4=0,\\
-8s^4 + 68s^3t + 48s^2t^2 - 68st^3 - 8t^4 - 2v^4=0.
\end{cases}
\end{equation}
The scheme defined by system \eqref{6001-5} is locally insoluble at $2$.

%\begin{lstlisting}
%_<x>:=PolynomialRing(Rationals());
%k<i>:=NumberField(x^2+1);
%K<s,t>:=PolynomialRing(k,2);
%e:=0;
%F:=i^e*(17-8*i)*(s+t*i)^4;
%F;
%L<s,t>:=PolynomialRing(Rationals(),2);
%_<i>:=PolynomialRing(L);
%F:=(-8*i + 17)*s^4 + (68*i + 32)*s^3*t + (48*i - 102)*s^2*t^2 + (-68*i - 32)*s*t^3
%    + (-8*i + 17)*t^4;
%F;
%A:=17*s^4 + 32*s^3*t -  102*s^2*t^2 - 32*s*t^3 + 17*t^4;
%B:=-8*s^4 + 68*s^3*t + 48*s^2*t^2 - 68*s*t^3 - 8*t^4;
%F-A-i*B;
\begin{lstlisting}
P<s,t,u,v>:=ProjectiveSpace(Rationals(),3);
A:=17*s^4 + 32*s^3*t -  102*s^2*t^2 - 32*s*t^3 + 17*t^4-17*u^4;
B:=-8*s^4 + 68*s^3*t + 48*s^2*t^2 - 68*s*t^3 - 8*t^4-2*v^4;
S:=Scheme(P,[A,B]);
IsLocallySolvable(S,2);
\end{lstlisting}

{\bf Case 6:} $17u^4+2v^4i=i(17-8i)(s+ti)^4$. Equating the real and imaginary parts on both sides of this equation gives

\begin{equation}\label{6001-6}
\begin{cases}
8s^4 - 68s^3t - 48s^2t^2 + 68st^3 + 8t^4 - 17u^4=0,\\
17s^4 + 32s^3t - 102s^2t^2 - 32st^3 + 17t^4 - 2v^4=0.
\end{cases}
\end{equation}
The scheme defined by system \eqref{6001-6} is locally insoluble at $2$.

%\begin{lstlisting}
%_<x>:=PolynomialRing(Rationals());
%k<i>:=NumberField(x^2+1);
%K<s,t>:=PolynomialRing(k,2);
%e:=1;
%F:=i^e*(17-8*i)*(s+t*i)^4;
%F;
%L<s,t>:=PolynomialRing(Rationals(),2);
%_<i>:=PolynomialRing(L);
%F:=(17*i + 8)*s^4 + (32*i - 68)*s^3*t + (-102*i - 48)*s^2*t^2 + (-32*i + 68)*s*t^3
%    + (17*i + 8)*t^4;
%F;
%A:=8*s^4 - 68*s^3*t -  48*s^2*t^2 + 68*s*t^3 + 8*t^4;
%B:=17*s^4 + 32*s^3*t - 102*s^2*t^2 - 32*s*t^3 + 17*t^4;
%F-A-i*B;
\begin{lstlisting}
P<s,t,u,v>:=ProjectiveSpace(Rationals(),3);
A:=8*s^4 - 68*s^3*t -  48*s^2*t^2 + 68*s*t^3 + 8*t^4-17*u^4;
B:=17*s^4 + 32*s^3*t - 102*s^2*t^2 - 32*s*t^3 + 17*t^4-2*v^4;
S:=Scheme(P,[A,B]);
IsLocallySolvable(S,2);
\end{lstlisting}

{\bf Case 7:} $17u^4+2v^4i=i^2(17-8i)(s+ti)^4$. Equating the real and imaginary parts on both sides of this equation gives
\begin{equation}\label{6001-7}
\begin{cases}
-17s^4 - 32s^3t + 102s^2t^2 + 32st^3 - 17t^4 - 17u^4=0,\\
8s^4 - 68s^3t - 48s^2t^2 + 68st^3 + 8t^4 - 2v^4=0.
\end{cases}
\end{equation}
The scheme defined by system \eqref{6001-7} is locally insoluble at $2$.

%\begin{lstlisting}
%_<x>:=PolynomialRing(Rationals());
%k<i>:=NumberField(x^2+1);
%K<s,t>:=PolynomialRing(k,2);
%e:=2;
%F:=i^e*(17-8*i)*(s+t*i)^4;
%F;
%L<s,t>:=PolynomialRing(Rationals(),2);
%_<i>:=PolynomialRing(L);
%F:=(8*i - 17)*s^4 + (-68*i - 32)*s^3*t + (-48*i + 102)*s^2*t^2 + (68*i + 32)*s*t^3
%    + (8*i - 17)*t^4;
%F;
%A:= - 17*s^4 - 32*s^3*t +102*s^2*t^2 + 32*s*t^3 - 17*t^4;
%B:=8*s^4 - 68*s^3*t - 48*s^2*t^2 + 68*s*t^3 + 8*t^4;
%F-A-i*B;
\begin{lstlisting}
P<s,t,u,v>:=ProjectiveSpace(Rationals(),3);
A:= - 17*s^4 - 32*s^3*t +102*s^2*t^2 + 32*s*t^3 - 17*t^4-17*u^4;
B:=8*s^4 - 68*s^3*t - 48*s^2*t^2 + 68*s*t^3 + 8*t^4-2*v^4;
S:=Scheme(P,[A,B]);
IsLocallySolvable(S,2);
\end{lstlisting}

{\bf Case 8:} $17u^4+2v^4i=i^3(17-8i)(s+ti)^4$. Equating the real and imaginary parts on both sides of this equation gives
\begin{equation}\label{6001-8}
\begin{cases}
-8s^4 + 68s^3t + 48s^2t^2 - 68st^3 - 8t^4 - 17u^4=0,\\
-17s^4 - 32s^3t + 102s^2t^2 + 32st^3 - 17t^4 - 2v^4=0.
\end{cases}
\end{equation}
The scheme defined by system \eqref{6001-8} is locally insoluble at $2$.

%\begin{lstlisting}
%_<x>:=PolynomialRing(Rationals());
%k<i>:=NumberField(x^2+1);
%K<s,t>:=PolynomialRing(k,2);
%e:=3;
%F:=i^e*(17-8*i)*(s+t*i)^4;
%F;
%L<s,t>:=PolynomialRing(Rationals(),2);
%_<i>:=PolynomialRing(L);
%F:=(-17*i - 8)*s^4 + (-32*i + 68)*s^3*t + (102*i + 48)*s^2*t^2 + (32*i - 68)*s*t^3
%    + (-17*i - 8)*t^4;
%F;
%A:= - 8*s^4 + 68*s^3*t +48*s^2*t^2 - 68*s*t^3 - 8*t^4;
%B:=-17*s^4 - 32*s^3*t + 102*s^2*t^2 + 32*s*t^3 - 17*t^4;
%F-A-i*B;
\begin{lstlisting}
P<s,t,u,v>:=ProjectiveSpace(Rationals(),3);
A:= - 8*s^4 + 68*s^3*t +48*s^2*t^2 - 68*s*t^3 - 8*t^4-17*u^4;
B:=-17*s^4 - 32*s^3*t + 102*s^2*t^2 + 32*s*t^3 - 17*t^4-2*v^4;
S:=Scheme(P,[A,B]);
IsLocallySolvable(S,2);

\end{lstlisting}

\subsection{The case of $n=6440$.}%\[n=6440\]
According to Table \ref{tb:factorisation}, in the case $n=6440$ it remains to
show that equation
\begin{equation}\label{6440}
25921 u^8 + 25 v^8= w^4
\end{equation}
has no integer solutions satisfying \eqref{eq:cond}, which in this case reduces to 
\begin{equation}\label{cond:6440}
\gcd(161u,5v)=\gcd(161u,w)=\gcd(5v,w)=1.
\end{equation}
Modulo 16 analysis shows that $2|v$, therefore by \eqref{cond:6440} we must have $2\nmid uw$. We can write \eqref{6440} as
\begin{equation} (161u^4+5v^4i)(161u^4-5v^4i)=w^4.
\end{equation}
This implies that there exist integers $s,t$ such that 
\[161u^4+5v^4i=i^{\epsilon}(s+ti)^4,\]
with $\epsilon \in \{0,1,2,3\}$.

{\bf Case 1:} $161u^4+5v^4i=(s+ti)^4$. Equating the real and imaginary parts on both sides of this equation gives
\begin{equation}\label{6440-1}
\begin{cases}
s^4 - 6s^2t^2 + t^4 - 161u^4=0,\\
4s^3t - 4st^3 - 5v^4=0.
\end{cases}
\end{equation}
The scheme defined by system \eqref{6440-1} is locally insoluble at $17$.

%\begin{lstlisting}
%_<x>:=PolynomialRing(Rationals());
%k<i>:=NumberField(x^2+1);
%K<s,t>:=PolynomialRing(k,2);
%e:=0;
%F:=i^e*(s+t*i)^4;
%F;
%L<s,t>:=PolynomialRing(Rationals(),2);
%_<i>:=PolynomialRing(L);
%F:=s^4 + 4*i*s^3*t - 6*s^2*t^2 - 4*i*s*t^3 + t^4;
%F;
%A:=s^4 - 6*s^2*t^2 + t^4;
%B:=4*s^3*t - 4*s*t^3;
%F-A-i*B;

\begin{lstlisting}
P<s,t,u,v>:=ProjectiveSpace(Rationals(),3);
A:=s^4 - 6*s^2*t^2 + t^4-161*u^4;
B:=4*s^3*t - 4*s*t^3-5*v^4;
S:=Scheme(P,[A,B]);
IsLocallySolvable(S,17);
\end{lstlisting}

{\bf Case 2:} $161u^4+5v^4i=i(s+ti)^4$. Equating the real and imaginary parts on both sides of this equation gives
\begin{equation}\label{6440-2}
\begin{cases}
-4s^3t + 4st^3 - 161u^4=0,\\
s^4 - 6s^2t^2 + t^4 - 5v^4=0.
\end{cases}
\end{equation}
The scheme defined by system \eqref{6440-2} is locally insoluble at $2$.

%\begin{lstlisting}
%_<x>:=PolynomialRing(Rationals());
%k<i>:=NumberField(x^2+1);
%K<s,t>:=PolynomialRing(k,2);
%e:=1;
%F:=i^e*(s+t*i)^4;
%F;
%L<s,t>:=PolynomialRing(Rationals(),2);
%_<i>:=PolynomialRing(L);
%F:=i*s^4 - 4*s^3*t - 6*i*s^2*t^2 + 4*s*t^3 + i*t^4;
%F;
%A:=- 4*s^3*t + 4*s*t^3;
%B:=s^4 - 6*s^2*t^2 + t^4;
%F-A-i*B;
\begin{lstlisting}
P<s,t,u,v>:=ProjectiveSpace(Rationals(),3);
A:=- 4*s^3*t + 4*s*t^3-161*u^4;
B:=s^4 - 6*s^2*t^2 + t^4-5*v^4;
S:=Scheme(P,[A,B]);
IsLocallySolvable(S,2);
\end{lstlisting}

{\bf Case 3:} $161u^4+5v^4i=i^2(s+ti)^4$. Equating the real and imaginary parts on both sides of this equation gives
\begin{equation}\label{6440-3}
\begin{cases}
-s^4 + 6s^2t^2 - t^4 -161u^4=0,\\
-4s^3t + 4st^3 - 5v^4=0.
\end{cases}
\end{equation}
The scheme defined by system \eqref{6440-3} is locally insoluble at $2$.

%\begin{lstlisting}
%_<x>:=PolynomialRing(Rationals());
%k<i>:=NumberField(x^2+1);
%K<s,t>:=PolynomialRing(k,2);
%e:=2;
%F:=i^e*(s+t*i)^4;
%F;
%L<s,t>:=PolynomialRing(Rationals(),2);
%_<i>:=PolynomialRing(L);
%F:=-s^4 - 4*i*s^3*t + 6*s^2*t^2 + 4*i*s*t^3 - t^4;
%F;
%A:=- s^4 + 6*s^2*t^2 - t^4;
%B:=(-4*s^3*t + 4*s*t^3);
%F-A-i*B;

\begin{lstlisting}
P<s,t,u,v>:=ProjectiveSpace(Rationals(),3);
A:=- s^4 + 6*s^2*t^2 - t^4-161*u^4;
B:=-4*s^3*t + 4*s*t^3-5*v^4;
S:=Scheme(P,[A,B]);
IsLocallySolvable(S,2);
\end{lstlisting}

{\bf Case 4:} $161u^4+5v^4i=i^3(s+ti)^4$. Equating the real and imaginary parts on both sides of this equation gives
\begin{equation}\label{6440-4}
\begin{cases}
4s^3t - 4st^3 -161u^4=0,\\
-s^4 + 6s^2t^2 - t^4 - 5v^4=0.
\end{cases}
\end{equation}
The scheme defined by system \eqref{6440-4} is locally insoluble at $5$.

%\begin{lstlisting}
%_<x>:=PolynomialRing(Rationals());
%k<i>:=NumberField(x^2+1);
%K<s,t>:=PolynomialRing(k,2);
%e:=3;
%F:=i^e*(s+t*i)^4;
%F;
%L<s,t>:=PolynomialRing(Rationals(),2);
%_<i>:=PolynomialRing(L);
%F:=-i*s^4 + 4*s^3*t + 6*i*s^2*t^2 - 4*s*t^3 - i*t^4;
%F;
%A:= 4*s^3*t - 4*s*t^3;
%B:=(-s^4 + 6*s^2*t^2 - t^4);
%F-A-i*B;

\begin{lstlisting}
P<s,t,u,v>:=ProjectiveSpace(Rationals(),3);
A:=4*s^3*t - 4*s*t^3-161*u^4;
B:=-s^4 + 6*s^2*t^2 - t^4-5*v^4;
S:=Scheme(P,[A,B]);
IsLocallySolvable(S,5);
\end{lstlisting}

\subsection{The case of $n=6625$.}%\[n=6625\quad (I)\]
According to Table \ref{tb:factorisation}, in the case $n=6625$ it remains to
show that equation
\begin{equation}\label{6625}
4 u^8 + v^8 = 6625  w^4
\end{equation}
has no integer solutions satisfying \eqref{eq:cond}, which in this case reduces to \[\gcd(2u,v)=\gcd(2u,6625w)=\gcd(v,6625w)=1,\]
and to show that equation
\begin{equation}\label{6625.2}
15625 u^8 + v^8= 106  w^4
\end{equation}
has no integer solutions satisfying \eqref{eq:cond}, which in this case reduces to \[\gcd(125u,v)=\gcd(125u,106w)=\gcd(v,106w)=1.\]

 \subsubsection{\bf{The case of \eqref{6625}.}}
 We first consider \eqref{6625}, which we can write as
\begin{equation} (2u^4+v^4i)(2u^4-v^4i)=(2+i)^3(2-i)^3(7+2i)(7-2i)w^4.
\end{equation}
Then,
\[\begin{split}
    2u^4+v^4i&\equiv 2+i\pmod{5}\\
    &\equiv 0\pmod{2+i}\\
    &\not\equiv 0\pmod{2-i}.
\end{split}\]
This implies that there exist integers $s,t$ such that 
\[2u^4+v^4i=i^{\epsilon}(2+i)^3(7\pm 2i)(s+ti)^4,\]
with $\epsilon \in \{0,1,2,3\}$.

{\bf Case 1:} $2u^4+v^4i=(2+i)^3(7+2i)(s+ti)^4$. Equating the real and imaginary parts on both sides of this equation gives
\begin{equation}\label{6625-1}
\begin{cases}
-8s^4 - 324s^3t + 48s^2t^2 + 324st^3 - 8t^4 - 2u^4=0,\\
81s^4 - 32s^3t - 486s^2t^2 + 32st^3 + 81t^4 - v^4=0.
\end{cases}
\end{equation}
The scheme defined by system \eqref{6625-1} is locally insoluble at $2$.

%\begin{lstlisting}
%_<x>:=PolynomialRing(Rationals());
%k<i>:=NumberField(x^2+1);
%K<s,t>:=PolynomialRing(k,2);
%e:=0;
%F:=i^e*(2+i)^3*(7+2*i)*(s+t*i)^4;
%F;
%L<s,t>:=PolynomialRing(Rationals(),2);
%_<i>:=PolynomialRing(L);
%F:=(81*i - 8)*s^4 + (-32*i - 324)*s^3*t + (-486*i + 48)*s^2*t^2 + (32*i +
%    324)*s*t^3 + (81*i - 8)*t^4;
%F;
%A:=- 8*s^4 - 324*s^3*t +48*s^2*t^2 + 324*s*t^3 - 8*t^4;
%B:=81*s^4 - 32*s^3*t - 486*s^2*t^2 + 32*s*t^3 + 81*t^4;
%F-A-i*B;
\begin{lstlisting}
P<s,t,u,v>:=ProjectiveSpace(Rationals(),3);
A:=- 8*s^4 - 324*s^3*t +48*s^2*t^2 + 324*s*t^3 - 8*t^4-2*u^4;
B:=81*s^4 - 32*s^3*t - 486*s^2*t^2 + 32*s*t^3 + 81*t^4-v^4;
S:=Scheme(P,[A,B]);
IsLocallySolvable(S,2);
\end{lstlisting}

{\bf Case 2:} $2u^4+v^4i=i(2+i)^3(7+2i)(s+ti)^4$. Equating the real and imaginary parts on both sides of this equation gives
\begin{equation}\label{6625-2}
\begin{cases}
-81s^4 + 32s^3t + 486s^2t^2 - 32st^3 - 81t^4 - 2u^4=0,\\
-8s^4 - 324s^3t + 48s^2t^2 + 324st^3 - 8t^4 - v^4=0.
\end{cases}
\end{equation}
The scheme defined by system \eqref{6625-2} is locally insoluble at $2$.

%\begin{lstlisting}
%_<x>:=PolynomialRing(Rationals());
%k<i>:=NumberField(x^2+1);
%K<s,t>:=PolynomialRing(k,2);
%e:=1;
%F:=i^e*(2+i)^3*(7+2*i)*(s+t*i)^4;
%F;
%L<s,t>:=PolynomialRing(Rationals(),2);
%_<i>:=PolynomialRing(L);
%F:=(-8*i - 81)*s^4 + (-324*i + 32)*s^3*t + (48*i + 486)*s^2*t^2 + (324*i -
%    32)*s*t^3 + (-8*i - 81)*t^4;
%F;
%A:=- 81*s^4 + 32*s^3*t +486*s^2*t^2 - 32*s*t^3 - 81*t^4;
%B:=-8*s^4 - 324*s^3*t + 48*s^2*t^2 + 324*s*t^3 - 8*t^4;
%F-A-i*B;

\begin{lstlisting}
P<s,t,u,v>:=ProjectiveSpace(Rationals(),3);
A:=- 81*s^4 + 32*s^3*t +486*s^2*t^2 - 32*s*t^3 - 81*t^4-2*u^4;
B:=-8*s^4 - 324*s^3*t + 48*s^2*t^2 + 324*s*t^3 - 8*t^4-v^4;
S:=Scheme(P,[A,B]);
IsLocallySolvable(S,2);
\end{lstlisting}
{\bf Case 3:} $2u^4+v^4i=i^2(2+i)^3(7+2i)(s+ti)^4$. Equating the real and imaginary parts on both sides of this equation gives
\begin{equation}\label{6625-3}
\begin{cases}
8s^4 + 324s^3t - 48s^2t^2 - 324st^3 + 8t^4 - 2u^4=0,\\
-81s^4 + 32s^3t + 486s^2t^2 - 32st^3 - 81t^4 - v^4=0.
\end{cases}
\end{equation}
The scheme defined by system \eqref{6625-3} is locally insoluble at $2$.

%\begin{lstlisting}
%_<x>:=PolynomialRing(Rationals());
%k<i>:=NumberField(x^2+1);
%K<s,t>:=PolynomialRing(k,2);
%e:=2;
%F:=i^e*(2+i)^3*(7+2*i)*(s+t*i)^4;
%F;
%L<s,t>:=PolynomialRing(Rationals(),2);
%_<i>:=PolynomialRing(L);
%F:=(-81*i + 8)*s^4 + (32*i + 324)*s^3*t + (486*i - 48)*s^2*t^2 + (-32*i -
%    324)*s*t^3 + (-81*i + 8)*t^4;
%F;
%A:=8*s^4 + 324*s^3*t -  48*s^2*t^2 - 324*s*t^3 + 8*t^4;
%B:=-81*s^4 + 32*s^3*t + 486*s^2*t^2 - 32*s*t^3 - 81*t^4;
%F-A-i*B;

\begin{lstlisting}
P<s,t,u,v>:=ProjectiveSpace(Rationals(),3);
A:=8*s^4 + 324*s^3*t -  48*s^2*t^2 - 324*s*t^3 + 8*t^4-2*u^4;
B:=-81*s^4 + 32*s^3*t + 486*s^2*t^2 - 32*s*t^3 - 81*t^4-v^4;
S:=Scheme(P,[A,B]);
IsLocallySolvable(S,2);
\end{lstlisting}

{\bf Case 4:} $2u^4+v^4i=i^3(2+i)^3(7+2i)(s+ti)^4$. Equating the real and imaginary parts on both sides of this equation gives
\begin{equation}\label{6625-4}
\begin{cases}
81s^4 - 32s^3t - 486s^2t^2 + 32st^3 + 81t^4 - 2u^4=0,\\
8s^4 + 324s^3t - 48s^2t^2 - 324st^3 + 8t^4 - v^4=0.
\end{cases}
\end{equation}
The scheme defined by system \eqref{6625-4} is locally insoluble at $2$.

%\begin{lstlisting}
%_<x>:=PolynomialRing(Rationals());
%k<i>:=NumberField(x^2+1);
%K<s,t>:=PolynomialRing(k,2);
%e:=3;
%F:=i^e*(2+i)^3*(7+2*i)*(s+t*i)^4;
%F;
%L<s,t>:=PolynomialRing(Rationals(),2);
%_<i>:=PolynomialRing(L);
%F:=(8*i + 81)*s^4 + (324*i - 32)*s^3*t + (-48*i - 486)*s^2*t^2 + (-324*i +
%    32)*s*t^3 + (8*i + 81)*t^4;
%F;
%A:=81*s^4 - 32*s^3*t -486*s^2*t^2 + 32*s*t^3 + 81*t^4;
%B:=8*s^4 + 324*s^3*t - 48*s^2*t^2 - 324*s*t^3 + 8*t^4;
%F-A-i*B;

\begin{lstlisting}
P<s,t,u,v>:=ProjectiveSpace(Rationals(),3);
A:=81*s^4 - 32*s^3*t -486*s^2*t^2 + 32*s*t^3 + 81*t^4-2*u^4;
B:=8*s^4 + 324*s^3*t - 48*s^2*t^2 - 324*s*t^3 + 8*t^4-v^4;
S:=Scheme(P,[A,B]);
IsLocallySolvable(S,2);
\end{lstlisting}

{\bf Case 5:} $2u^4+v^4i=(2+i)^3(7-2i)(s+ti)^4$. Equating the real and imaginary parts on both sides of this equation gives
\begin{equation}\label{6625-5}
\begin{cases}
36s^4 - 292s^3t - 216s^2t^2 + 292st^3 + 36t^4 - 2u^4=0,\\
73s^4 + 144s^3t - 438s^2t^2 - 144st^3 + 73t^4 - v^4=0.
\end{cases}
\end{equation}
The scheme defined by system \eqref{6625-5} is locally insoluble at $2$.

%\begin{lstlisting}
%_<x>:=PolynomialRing(Rationals());
%k<i>:=NumberField(x^2+1);
%K<s,t>:=PolynomialRing(k,2);
%e:=0;
%F:=i^e*(2+i)^3*(7-2*i)*(s+t*i)^4;
%F;
%L<s,t>:=PolynomialRing(Rationals(),2);
%_<i>:=PolynomialRing(L);
%F:=(73*i + 36)*s^4 + (144*i - 292)*s^3*t + (-438*i - 216)*s^2*t^2 + (-144*i +
%    292)*s*t^3 + (73*i + 36)*t^4;
%F;
%A:=36*s^4 - 292*s^3*t -216*s^2*t^2 + 292*s*t^3 + 36*t^4;
%B:=73*s^4 + 144*s^3*t - 438*s^2*t^2 - 144*s*t^3 + 73*t^4;
%F-A-i*B;
\begin{lstlisting}
P<s,t,u,v>:=ProjectiveSpace(Rationals(),3);
A:=36*s^4 - 292*s^3*t -216*s^2*t^2 + 292*s*t^3 + 36*t^4-2*u^4;
B:=73*s^4 + 144*s^3*t - 438*s^2*t^2 - 144*s*t^3 + 73*t^4-v^4;
S:=Scheme(P,[A,B]);
IsLocallySolvable(S,2);
\end{lstlisting}

{\bf Case 6:} $2u^4+v^4i=i(2+i)^3(7-2i)(s+ti)^4$. Equating the real and imaginary parts on both sides of this equation gives
\begin{equation}\label{6625-6}
\begin{cases}
-73s^4 - 144s^3t + 438s^2t^2 + 144st^3 - 73t^4 - 2u^4=0,\\
36s^4 - 292s^3t - 216s^2t^2 + 292st^3 + 36t^4 - v^4=0.
\end{cases}
\end{equation}
The scheme defined by system \eqref{6625-6} is locally insoluble at $2$.

%\begin{lstlisting}
%_<x>:=PolynomialRing(Rationals());
%k<i>:=NumberField(x^2+1);
%K<s,t>:=PolynomialRing(k,2);
%e:=1;
%F:=i^e*(2+i)^3*(7-2*i)*(s+t*i)^4;
%F;
%L<s,t>:=PolynomialRing(Rationals(),2);
%_<i>:=PolynomialRing(L);
%F:=(36*i - 73)*s^4 + (-292*i - 144)*s^3*t + (-216*i + 438)*s^2*t^2 + (292*i +
%    144)*s*t^3 + (36*i - 73)*t^4;
%F;
%A:=- 73*s^4 - 144*s^3*t +438*s^2*t^2 + 144*s*t^3 - 73*t^4;
%B:=36*s^4 - 292*s^3*t - 216*s^2*t^2 + 292*s*t^3 + 36*t^4;
%F-A-i*B;

\begin{lstlisting}
P<s,t,u,v>:=ProjectiveSpace(Rationals(),3);
A:=- 73*s^4 - 144*s^3*t +438*s^2*t^2 + 144*s*t^3 - 73*t^4-2*u^4;
B:=36*s^4 - 292*s^3*t - 216*s^2*t^2 + 292*s*t^3 + 36*t^4-v^4;
S:=Scheme(P,[A,B]);
IsLocallySolvable(S,2);
\end{lstlisting}
{\bf Case 7:} $2u^4+v^4i=i^2(2+i)^3(7-2i)(s+ti)^4$. Equating the real and imaginary parts on both sides of this equation gives
\begin{equation}\label{6625-7}
\begin{cases}
-36s^4 + 292s^3t + 216s^2t^2 - 292st^3 - 36t^4 - 2u^4=0,\\
-73s^4 - 144s^3t + 438s^2t^2 + 144st^3 - 73t^4 - v^4=0.
\end{cases}
\end{equation}
The scheme defined by system \eqref{6625-7} is locally insoluble at $2$.

%\begin{lstlisting}
%_<x>:=PolynomialRing(Rationals());
%k<i>:=NumberField(x^2+1);
%K<s,t>:=PolynomialRing(k,2);
%e:=2;
%F:=i^e*(2+i)^3*(7-2*i)*(s+t*i)^4;
%F;
%L<s,t>:=PolynomialRing(Rationals(),2);
%_<i>:=PolynomialRing(L);
%F:=(-73*i - 36)*s^4 + (-144*i + 292)*s^3*t + (438*i + 216)*s^2*t^2 + (144*i -
%    292)*s*t^3 + (-73*i - 36)*t^4;
%F;
%A:=- 36*s^4 + 292*s^3*t+ 216*s^2*t^2 - 292*s*t^3 - 36*t^4;
%B:=-73*s^4 - 144*s^3*t + 438*s^2*t^2 + 144*s*t^3 - 73*t^4;
%F-A-i*B;

\begin{lstlisting}
P<s,t,u,v>:=ProjectiveSpace(Rationals(),3);
A:=- 36*s^4 + 292*s^3*t+ 216*s^2*t^2 - 292*s*t^3 - 36*t^4-2*u^4;
B:=-73*s^4 - 144*s^3*t + 438*s^2*t^2 + 144*s*t^3 - 73*t^4-v^4;
S:=Scheme(P,[A,B]);
IsLocallySolvable(S,2);
\end{lstlisting}

{\bf Case 8:} $2u^4+v^4i=i^3(2+i)^3(7-2i)(s+ti)^4$. Equating the real and imaginary parts on both sides of this equation gives
\begin{equation}\label{6625-8}
\begin{cases}
73s^4 + 144s^3t - 438s^2t^2 - 144st^3 + 73t^4 - 2u^4=0,\\
-36s^4 + 292s^3t + 216s^2t^2 - 292st^3 - 36t^4 - v^4=0.
\end{cases}
\end{equation}
The scheme defined by system \eqref{6625-8} is locally insoluble at $2$.

%\begin{lstlisting}
%_<x>:=PolynomialRing(Rationals());
%k<i>:=NumberField(x^2+1);
%K<s,t>:=PolynomialRing(k,2);
%e:=3;
%F:=i^e*(2+i)^3*(7-2*i)*(s+t*i)^4;
%F;
%L<s,t>:=PolynomialRing(Rationals(),2);
%_<i>:=PolynomialRing(L);
%F:=(-36*i + 73)*s^4 + (292*i + 144)*s^3*t + (216*i - 438)*s^2*t^2 + (-292*i -
%    144)*s*t^3 + (-36*i + 73)*t^4;
%F;
%A:=73*s^4 + 144*s^3*t - 438*s^2*t^2 - 144*s*t^3 + 73*t^4;
%B:=-36*s^4 + 292*s^3*t + 216*s^2*t^2 - 292*s*t^3 - 36*t^4;
%F-A-i*B;

\begin{lstlisting}
P<s,t,u,v>:=ProjectiveSpace(Rationals(),3);
A:=73*s^4 + 144*s^3*t - 438*s^2*t^2 - 144*s*t^3 + 73*t^4-2*u^4;
B:=-36*s^4 + 292*s^3*t + 216*s^2*t^2 - 292*s*t^3 - 36*t^4-v^4;
S:=Scheme(P,[A,B]);
IsLocallySolvable(S,2);
\end{lstlisting}

\subsubsection{\bf{The case of \eqref{6625.2}.}}%\[n=6625\quad (II)\]
We now consider \eqref{6625.2} which we can write \eqref{6625.2} as
\begin{equation} (125u^4+v^4i)(125u^4-v^4i)=(1+i)(1-i)(7+2i)(7-2i)w^4.
\end{equation}
This implies that there exist integers $s,t$ such that 
\[125u^4+v^4i=i^{\epsilon}(1+i)(7\pm 2i)(s+ti)^4,\]
with $\epsilon \in \{0,1,2,3\}$.

{\bf Case 1:} $125u^4+v^4i=(1+i)(7+2i)(s+ti)^4$. Equating the real and imaginary parts on both sides of this equation gives
\begin{equation}\label{6625.2-1}
\begin{cases}
5s^4 - 36s^3t - 30s^2t^2 + 36st^3 + 5t^4 - 125u^4=0,\\
9s^4 + 20s^3t - 54s^2t^2 - 20st^3 + 9t^4 - v^4=0.
\end{cases}
\end{equation}
The scheme defined by system \eqref{6625.2-1} is locally insoluble at $2$.

%\begin{lstlisting}
%_<x>:=PolynomialRing(Rationals());
%k<i>:=NumberField(x^2+1);
%K<s,t>:=PolynomialRing(k,2);
%e:=0;
%F:=i^e*(1+i)*(7+2*i)*(s+t*i)^4;
%F;
%L<s,t>:=PolynomialRing(Rationals(),2);
%_<i>:=PolynomialRing(L);
%F:=(9*i + 5)*s^4 + (20*i - 36)*s^3*t + (-54*i - 30)*s^2*t^2 + (-20*i + 36)*s*t^3 +
%    (9*i + 5)*t^4;
%F;
%A:=5*s^4 - 36*s^3*t -30*s^2*t^2 + 36*s*t^3 + 5*t^4;
%B:=9*s^4 + 20*s^3*t - 54*s^2*t^2 - 20*s*t^3 + 9*t^4;
%F-A-i*B;
\begin{lstlisting}
P<s,t,u,v>:=ProjectiveSpace(Rationals(),3);
A:=5*s^4 - 36*s^3*t -30*s^2*t^2 + 36*s*t^3 + 5*t^4-125*u^4;
B:=9*s^4 + 20*s^3*t - 54*s^2*t^2 - 20*s*t^3 + 9*t^4-v^4;
S:=Scheme(P,[A,B]);
IsLocallySolvable(S,2);
\end{lstlisting}

{\bf Case 2:} $125u^4+v^4i=i(1+i)(7+2i)(s+ti)^4$.  Equating the real and imaginary parts on both sides of this equation gives
\begin{equation}\label{6625.2-2}
\begin{cases}
-9s^4 - 20s^3t + 54s^2t^2 + 20st^3 - 9t^4 - 125u^4=0,\\
5s^4 - 36s^3t - 30s^2t^2 + 36st^3 + 5t^4 - v^4=0.
\end{cases}
\end{equation}
The scheme defined by system \eqref{6625.2-2} is locally insoluble at $2$.

%\begin{lstlisting}
%_<x>:=PolynomialRing(Rationals());
%k<i>:=NumberField(x^2+1);
%K<s,t>:=PolynomialRing(k,2);
%e:=1;
%F:=i^e*(1+i)*(7+2*i)*(s+t*i)^4;
%F;
%L<s,t>:=PolynomialRing(Rationals(),2);
%_<i>:=PolynomialRing(L);
%F:=(5*i - 9)*s^4 + (-36*i - 20)*s^3*t + (-30*i + 54)*s^2*t^2 + (36*i + 20)*s*t^3 +
%    (5*i - 9)*t^4;
%F;
%A:=- 9*s^4 - 20*s^3*t +54*s^2*t^2 + 20*s*t^3 - 9*t^4;
%B:=5*s^4 - 36*s^3*t - 30*s^2*t^2 + 36*s*t^3 + 5*t^4;
%F-A-i*B;
\begin{lstlisting}
P<s,t,u,v>:=ProjectiveSpace(Rationals(),3);
A:=- 9*s^4 - 20*s^3*t +54*s^2*t^2 + 20*s*t^3 - 9*t^4-125*u^4;
B:=5*s^4 - 36*s^3*t - 30*s^2*t^2 + 36*s*t^3 + 5*t^4-v^4;
S:=Scheme(P,[A,B]);
IsLocallySolvable(S,2);
\end{lstlisting}

{\bf Case 3:} $125u^4+v^4i=i^2(1+i)(7+2i)(s+ti)^4$.  Equating the real and imaginary parts on both sides of this equation gives
\begin{equation}\label{6625.2-3}
\begin{cases}
-5s^4 + 36s^3t + 30s^2t^2 - 36st^3 - 5t^4 - 125u^4=0,\\
-9s^4 - 20s^3t + 54s^2t^2 + 20st^3 - 9t^4 - v^4=0.
\end{cases}
\end{equation}
The scheme defined by system \eqref{6625.2-3} is locally insoluble at $2$.

%\begin{lstlisting}
%_<x>:=PolynomialRing(Rationals());
%k<i>:=NumberField(x^2+1);
%K<s,t>:=PolynomialRing(k,2);
%e:=2;
%F:=i^e*(1+i)*(7+2*i)*(s+t*i)^4;
%F;
%L<s,t>:=PolynomialRing(Rationals(),2);
%_<i>:=PolynomialRing(L);
%F:=(-9*i - 5)*s^4 + (-20*i + 36)*s^3*t + (54*i + 30)*s^2*t^2 + (20*i - 36)*s*t^3 +
%    (-9*i - 5)*t^4;
%F;
%A:=- 5*s^4 + 36*s^3*t +30*s^2*t^2 - 36*s*t^3 - 5*t^4;
%B:=-9*s^4 - 20*s^3*t + 54*s^2*t^2 + 20*s*t^3 - 9*t^4;
%F-A-i*B;
\begin{lstlisting}
P<s,t,u,v>:=ProjectiveSpace(Rationals(),3);
A:=- 5*s^4 + 36*s^3*t +30*s^2*t^2 - 36*s*t^3 - 5*t^4-125*u^4;
B:=-9*s^4 - 20*s^3*t + 54*s^2*t^2 + 20*s*t^3 - 9*t^4-v^4;
S:=Scheme(P,[A,B]);
IsLocallySolvable(S,2);
\end{lstlisting}

{\bf Case 4:} $125u^4+v^4i=i^3(1+i)(7+2i)(s+ti)^4$.  Equating the real and imaginary parts on both sides of this equation gives
\begin{equation}\label{6625.2-4}
\begin{cases}
9s^4 + 20s^3t - 54s^2t^2 - 20st^3 + 9t^4 - 125u^4=0,\\
-5s^4 + 36s^3t + 30s^2t^2 - 36st^3 - 5t^4 - v^4=0.
\end{cases}
\end{equation}
The scheme defined by system \eqref{6625.2-4} is locally insoluble at $2$.

%\begin{lstlisting}
%_<x>:=PolynomialRing(Rationals());
%k<i>:=NumberField(x^2+1);
%K<s,t>:=PolynomialRing(k,2);
%e:=3;
%F:=i^e*(1+i)*(7+2*i)*(s+t*i)^4;
%F;
%L<s,t>:=PolynomialRing(Rationals(),2);
%_<i>:=PolynomialRing(L);
%F:=(-5*i + 9)*s^4 + (36*i + 20)*s^3*t + (30*i - 54)*s^2*t^2 + (-36*i - 20)*s*t^3 +
%    (-5*i + 9)*t^4;
%F;
%A:= 9*s^4 + 20*s^3*t -54*s^2*t^2 - 20*s*t^3 + 9*t^4;
%B:=-5*s^4 + 36*s^3*t + 30*s^2*t^2 - 36*s*t^3 - 5*t^4;
%F-A-i*B;
\begin{lstlisting}
P<s,t,u,v>:=ProjectiveSpace(Rationals(),3);
A:= 9*s^4 + 20*s^3*t -54*s^2*t^2 - 20*s*t^3 + 9*t^4-125*u^4;
B:=-5*s^4 + 36*s^3*t + 30*s^2*t^2 - 36*s*t^3 - 5*t^4-v^4;
S:=Scheme(P,[A,B]);
IsLocallySolvable(S,2);
\end{lstlisting}

{\bf Case 5:} $125u^4+v^4i=(1+i)(7-2i)(s+ti)^4$. Equating the real and imaginary parts on both sides of this equation gives
\begin{equation}\label{6625.2-5}
\begin{cases}
9s^4 - 20s^3t - 54s^2t^2 + 20st^3 + 9t^4 - 125u^4=0,\\
5s^4 + 36s^3t - 30s^2t^2 - 36st^3 + 5t^4 - v^4=0.
\end{cases}
\end{equation}
The scheme defined by system \eqref{6625.2-5} is locally insoluble at $2$.

%\begin{lstlisting}
%_<x>:=PolynomialRing(Rationals());
%k<i>:=NumberField(x^2+1);
%K<s,t>:=PolynomialRing(k,2);
%e:=0;
%F:=i^e*(1+i)*(7-2*i)*(s+t*i)^4;
%F;
%L<s,t>:=PolynomialRing(Rationals(),2);
%_<i>:=PolynomialRing(L);
%F:=(5*i + 9)*s^4 + (36*i - 20)*s^3*t + (-30*i - 54)*s^2*t^2 + (-36*i + 20)*s*t^3 +
%    (5*i + 9)*t^4;
%F;
%A:=  9*s^4 - 20*s^3*t -54*s^2*t^2 + 20*s*t^3 + 9*t^4;
%B:=5*s^4 + 36*s^3*t - 30*s^2*t^2 - 36*s*t^3 + 5*t^4;
%F-A-i*B;
\begin{lstlisting}
P<s,t,u,v>:=ProjectiveSpace(Rationals(),3);
A:=  9*s^4 - 20*s^3*t -54*s^2*t^2 + 20*s*t^3 + 9*t^4-125*u^4;
B:=5*s^4 + 36*s^3*t - 30*s^2*t^2 - 36*s*t^3 + 5*t^4-v^4;
S:=Scheme(P,[A,B]);
IsLocallySolvable(S,2);
\end{lstlisting}

{\bf Case 6:} $125u^4+v^4i=i(1+i)(7-2i)(s+ti)^4$. Equating the real and imaginary parts on both sides of this equation gives
\begin{equation}\label{6625.2-6}
\begin{cases}
-5s^4 - 36s^3t + 30s^2t^2 + 36st^3 - 5t^4 - 125u^4=0,\\
9s^4 - 20s^3t - 54s^2t^2 + 20st^3 + 9t^4 - v^4=0.
\end{cases}
\end{equation}
The scheme defined by system \eqref{6625.2-6} is locally insoluble at $2$.

%\begin{lstlisting}
%_<x>:=PolynomialRing(Rationals());
%k<i>:=NumberField(x^2+1);
%K<s,t>:=PolynomialRing(k,2);
%e:=1;
%F:=i^e*(1+i)*(7-2*i)*(s+t*i)^4;
%F;
%L<s,t>:=PolynomialRing(Rationals(),2);
%_<i>:=PolynomialRing(L);
%F:=(9*i - 5)*s^4 + (-20*i - 36)*s^3*t + (-54*i + 30)*s^2*t^2 + (20*i + 36)*s*t^3 +
%    (9*i - 5)*t^4;
%F;
%A:=- 5*s^4 - 36*s^3*t +  30*s^2*t^2 + 36*s*t^3 - 5*t^4;
%B:=9*s^4 - 20*s^3*t - 54*s^2*t^2 + 20*s*t^3 + 9*t^4;
%F-A-i*B;
\begin{lstlisting}
P<s,t,u,v>:=ProjectiveSpace(Rationals(),3);
A:=- 5*s^4 - 36*s^3*t +  30*s^2*t^2 + 36*s*t^3 - 5*t^4-125*u^4;
B:=9*s^4 - 20*s^3*t - 54*s^2*t^2 + 20*s*t^3 + 9*t^4-v^4;
S:=Scheme(P,[A,B]);
IsLocallySolvable(S,2);
\end{lstlisting}

{\bf Case 7:} $125u^4+v^4i=i^2(1+i)(7-2i)(s+ti)^4$. Equating the real and imaginary parts on both sides of this equation gives
\begin{equation}\label{6625.2-7}
\begin{cases}
-9s^4 + 20s^3t + 54s^2t^2 - 20st^3 - 9t^4 - 125u^4=0,\\
-5s^4 - 36s^3t + 30s^2t^2 + 36st^3 - 5t^4 - v^4=0.
\end{cases}
\end{equation}
The scheme defined by system \eqref{6625.2-7} is locally insoluble at $2$.

%\begin{lstlisting}
%_<x>:=PolynomialRing(Rationals());
%k<i>:=NumberField(x^2+1);
%K<s,t>:=PolynomialRing(k,2);
%e:=2;
%F:=i^e*(1+i)*(7-2*i)*(s+t*i)^4;
%F;
%L<s,t>:=PolynomialRing(Rationals(),2);
%_<i>:=PolynomialRing(L);
%F:=(-5*i - 9)*s^4 + (-36*i + 20)*s^3*t + (30*i + 54)*s^2*t^2 + (36*i - 20)*s*t^3 +
%    (-5*i - 9)*t^4;
%F;
%A:=- 9*s^4 + 20*s^3*t +54*s^2*t^2 - 20*s*t^3 - 9*t^4;
%B:=-5*s^4 - 36*s^3*t + 30*s^2*t^2 + 36*s*t^3 - 5*t^4;
%F-A-i*B;
\begin{lstlisting}
P<s,t,u,v>:=ProjectiveSpace(Rationals(),3);
A:=- 9*s^4 + 20*s^3*t +54*s^2*t^2 - 20*s*t^3 - 9*t^4-125*u^4;
B:=-5*s^4 - 36*s^3*t + 30*s^2*t^2 + 36*s*t^3 - 5*t^4-v^4;
S:=Scheme(P,[A,B]);
IsLocallySolvable(S,2);
\end{lstlisting}

{\bf Case 8:} $125u^4+v^4i=i^3(1+i)(7-2i)(s+ti)^4$. Equating the real and imaginary parts on both sides of this equation gives
\begin{equation}\label{6625.2-8}
\begin{cases}
5s^4 + 36s^3t - 30s^2t^2 - 36st^3 + 5t^4 - 125u^4=0,\\
-9s^4 + 20s^3t + 54s^2t^2 - 20st^3 - 9t^4 - v^4=0.
\end{cases}
\end{equation}
The scheme defined by system \eqref{6625.2-8} is locally insoluble at $2$.

%\begin{lstlisting}
%_<x>:=PolynomialRing(Rationals());
%k<i>:=NumberField(x^2+1);
%K<s,t>:=PolynomialRing(k,2);
%e:=3;
%F:=i^e*(1+i)*(7-2*i)*(s+t*i)^4;
%F;
%L<s,t>:=PolynomialRing(Rationals(),2);
%_<i>:=PolynomialRing(L);
%F:=(-9*i + 5)*s^4 + (20*i + 36)*s^3*t + (54*i - 30)*s^2*t^2 + (-20*i - 36)*s*t^3 +
%    (-9*i + 5)*t^4;
%F;
%A:=5*s^4 + 36*s^3*t - 30*s^2*t^2 - 36*s*t^3 + 5*t^4;
%B:=-9*s^4 + 20*s^3*t + 54*s^2*t^2 - 20*s*t^3 - 9*t^4;
%F-A-i*B;
\begin{lstlisting}
P<s,t,u,v>:=ProjectiveSpace(Rationals(),3);
A:=73*s^4 + 144*s^3*t - 438*s^2*t^2 - 144*s*t^3 + 73*t^4-2*u^4;
B:=-36*s^4 + 292*s^3*t + 216*s^2*t^2 - 292*s*t^3 - 36*t^4-v^4;
S:=Scheme(P,[A,B]);
IsLocallySolvable(S,2);
\end{lstlisting}

\subsection{The case of $n=6984$.} 
 According to Table \ref{tb:factorisation}, in the case $n=6984$ it remains to
show that equation
\begin{equation}\label{6984}
81 u^8 + v^8=97  w^4
\end{equation}
has no integer solutions satisfying \eqref{eq:cond}, which in this case reduces to \[\gcd(3u,v)=\gcd(3u,97w)=\gcd(v,97w)=1.\]
Write \eqref{6984} as 
\begin{equation} (9u^4+v^4i)(9u^4-v^4i)=(9+4i)(9-4i)w^4.
\end{equation}
This implies that there exist integers $s,t$ such that 
\[9u^4+v^4i=i^{\epsilon}(9\pm 4i)(s+ti)^4,\]
with $\epsilon \in \{0,1,2,3\}$.

{\bf Case 1:} $9u^4+v^4i=(9+4i)(s+ti)^4$. Equating the real and imaginary parts on both sides of this equation gives
\begin{equation}\label{6984-1}
\begin{cases}
9s^4 - 16s^3t - 54s^2t^2 + 16st^3 + 9t^4 - 9u^4=0,\\
4s^4 + 36s^3t - 24s^2t^2 - 36st^3 + 4t^4 - v^4=0.
\end{cases}
\end{equation}
The scheme defined by system \eqref{6984-1} is locally insoluble at $2$.
%
%\begin{lstlisting}
%_<x>:=PolynomialRing(Rationals());
%k<i>:=NumberField(x^2+1);
%K<s,t>:=PolynomialRing(k,2);
%e:=0;
%F:=i^e*(9+4*i)*(s+t*i)^4;
%F;
%L<s,t>:=PolynomialRing(Rationals(),2);
%_<i>:=PolynomialRing(L);
%F:=(4*i + 9)*s^4 + (36*i - 16)*s^3*t + (-24*i - 54)*s^2*t^2 + (-36*i + 16)*s*t^3 +
%    (4*i + 9)*t^4;
%F;
%A:=9*s^4 - 16*s^3*t -54*s^2*t^2 + 16*s*t^3 + 9*t^4;
%B:=4*s^4 + 36*s^3*t - 24*s^2*t^2 - 36*s*t^3 + 4*t^4;
%F-A-i*B;
\begin{lstlisting}
P<s,t,u,v>:=ProjectiveSpace(Rationals(),3);
A:=9*s^4 - 16*s^3*t -54*s^2*t^2 + 16*s*t^3 + 9*t^4-9*u^4;
B:=4*s^4 + 36*s^3*t - 24*s^2*t^2 - 36*s*t^3 + 4*t^4-v^4;
S:=Scheme(P,[A,B]);
IsLocallySolvable(S,2);
\end{lstlisting}

{\bf Case 2:} $9u^4+v^4i=i(9+4i)(s+ti)^4$. Equating the real and imaginary parts on both sides of this equation gives
\begin{equation}\label{6984-2}
\begin{cases}
-4s^4 - 36s^3t + 24s^2t^2 + 36st^3 - 4t^4 - 9u^4=0,\\
9s^4 - 16s^3t - 54s^2t^2 + 16st^3 + 9t^4 - v^4=0.
\end{cases}
\end{equation}
The scheme defined by system \eqref{6984-2} is locally insoluble at $2$.

%\begin{lstlisting}
%_<x>:=PolynomialRing(Rationals());
%k<i>:=NumberField(x^2+1);
%K<s,t>:=PolynomialRing(k,2);
%e:=1;
%F:=i^e*(9+4*i)*(s+t*i)^4;
%F;
%L<s,t>:=PolynomialRing(Rationals(),2);
%_<i>:=PolynomialRing(L);
%F:=(9*i - 4)*s^4 + (-16*i - 36)*s^3*t + (-54*i + 24)*s^2*t^2 + (16*i + 36)*s*t^3 +
%    (9*i - 4)*t^4;
%F;
%A:=- 4*s^4 - 36*s^3*t +  24*s^2*t^2 + 36*s*t^3 - 4*t^4;
%B:=9*s^4 - 16*s^3*t - 54*s^2*t^2 + 16*s*t^3 + 9*t^4;
%F-A-i*B;
\begin{lstlisting}
P<s,t,u,v>:=ProjectiveSpace(Rationals(),3);
A:=- 4*s^4 - 36*s^3*t +  24*s^2*t^2 + 36*s*t^3 - 4*t^4-9*u^4;
B:=9*s^4 - 16*s^3*t - 54*s^2*t^2 + 16*s*t^3 + 9*t^4-v^4;
S:=Scheme(P,[A,B]);
IsLocallySolvable(S,2);
\end{lstlisting}

{\bf Case 3:} $9u^4+v^4i=i^2(9+4i)(s+ti)^4$. Equating the real and imaginary parts on both sides of this equation gives
\begin{equation}\label{6984-3}
\begin{cases}
-9s^4 + 16s^3t + 54s^2t^2 - 16st^3 - 9t^4 - 9u^4=0,\\
-4s^4 - 36s^3t + 24s^2t^2 + 36st^3 - 4t^4 - v^4=0.
\end{cases}
\end{equation}
The scheme defined by system \eqref{6984-3} is locally insoluble at $2$.

%\begin{lstlisting}
%_<x>:=PolynomialRing(Rationals());
%k<i>:=NumberField(x^2+1);
%K<s,t>:=PolynomialRing(k,2);
%e:=2;
%F:=i^e*(9+4*i)*(s+t*i)^4;
%F;
%L<s,t>:=PolynomialRing(Rationals(),2);
%_<i>:=PolynomialRing(L);
%F:=(-4*i - 9)*s^4 + (-36*i + 16)*s^3*t + (24*i + 54)*s^2*t^2 + (36*i - 16)*s*t^3 +
%    (-4*i - 9)*t^4;
%F;
%A:=- 9*s^4 + 16*s^3*t +  54*s^2*t^2 - 16*s*t^3 - 9*t^4;
%B:=-4*s^4 - 36*s^3*t + 24*s^2*t^2 + 36*s*t^3 - 4*t^4;
%F-A-i*B;

\begin{lstlisting}
P<s,t,u,v>:=ProjectiveSpace(Rationals(),3);
A:=- 9*s^4 + 16*s^3*t +  54*s^2*t^2 - 16*s*t^3 - 9*t^4-9*u^4;
B:=-4*s^4 - 36*s^3*t + 24*s^2*t^2 + 36*s*t^3 - 4*t^4-v^4;
S:=Scheme(P,[A,B]);
IsLocallySolvable(S,2);
\end{lstlisting}

{\bf Case 4:}  $9u^4+v^4i=i^3(9+4i)(s+ti)^4$. Equating the real and imaginary parts on both sides of this equation gives
\begin{equation}\label{6984-4}
\begin{cases}
4s^4 + 36s^3t - 24s^2t^2 - 36st^3 + 4t^4 - 9u^4=0,\\
-9s^4 + 16s^3t + 54s^2t^2 - 16st^3 - 9t^4 - v^4=0.
\end{cases}
\end{equation}
The scheme defined by system \eqref{6984-4} is locally insoluble at $2$.
%
%\begin{lstlisting}
%_<x>:=PolynomialRing(Rationals());
%k<i>:=NumberField(x^2+1);
%K<s,t>:=PolynomialRing(k,2);
%e:=3;
%F:=i^e*(9+4*i)*(s+t*i)^4;
%F;
%L<s,t>:=PolynomialRing(Rationals(),2);
%_<i>:=PolynomialRing(L);
%F:=(-9*i + 4)*s^4 + (16*i + 36)*s^3*t + (54*i - 24)*s^2*t^2 + (-16*i - 36)*s*t^3 +
%    (-9*i + 4)*t^4;
%F;
%A:=4*s^4 + 36*s^3*t -24*s^2*t^2 - 36*s*t^3 + 4*t^4;
%B:=-9*s^4 + 16*s^3*t + 54*s^2*t^2 - 16*s*t^3 - 9*t^4;
%F-A-i*B;

\begin{lstlisting}
P<s,t,u,v>:=ProjectiveSpace(Rationals(),3);
A:=4*s^4 + 36*s^3*t -24*s^2*t^2 - 36*s*t^3 + 4*t^4-9*u^4;
B:=-9*s^4 + 16*s^3*t + 54*s^2*t^2 - 16*s*t^3 - 9*t^4-v^4;
S:=Scheme(P,[A,B]);
IsLocallySolvable(S,2);
\end{lstlisting}

\subsection{The case of $n=8001$.}%\[n=8001.\] 
 According to Table \ref{tb:factorisation}, in the case $n=8001$ it remains to
show that equation
\begin{equation}\label{8001}
64516 u^8 + 3969 v^8= w^4
\end{equation}
has no integer solutions satisfying \eqref{eq:cond}, which in this case reduces to  \[\gcd(254u,63v)=\gcd(254u,w)=\gcd(63v,w)=1.\]
Write \eqref{8001} as 
\begin{equation} (254u^4+63v^4i)(254u^4-63v^4i)=w^4.
\end{equation}
This implies that there exist integers $s,t$ such that 
\[254u^4+63v^4i=i^{\epsilon}(s+ti)^4,\]
with $\epsilon \in \{0,1,2,3\}$.

{\bf Case 1:} $254u^4+63v^4i=(s+ti)^4$. Equating the real and imaginary parts on both sides of this equation gives
\begin{equation}\label{8001-1}
\begin{cases}
s^4 - 6s^2t^2 + t^4 - 254u^4=0,\\
4s^3t - 4st^3 - 63v^4=0.
\end{cases}
\end{equation}
The scheme defined by system \eqref{8001-1} is locally insoluble at $2$.

%\begin{lstlisting}
%_<x>:=PolynomialRing(Rationals());
%k<i>:=NumberField(x^2+1);
%K<s,t>:=PolynomialRing(k,2);
%e:=0;
%F:=i^e*(s+t*i)^4;
%F;
%L<s,t>:=PolynomialRing(Rationals(),2);
%_<i>:=PolynomialRing(L);
%F:=s^4 + 4*i*s^3*t - 6*s^2*t^2 - 4*i*s*t^3 + t^4;
%F;
%A:=s^4 - 6*s^2*t^2 + t^4;
%B:=4*s^3*t - 4*s*t^3;
%F-A-i*B;

\begin{lstlisting}
P<s,t,u,v>:=ProjectiveSpace(Rationals(),3);
A:=s^4 - 6*s^2*t^2 + t^4-254*u^4;
B:=4*s^3*t - 4*s*t^3-63*v^4;
S:=Scheme(P,[A,B]);
IsLocallySolvable(S,2);
\end{lstlisting}

{\bf Case 2:} $254u^4+63v^4i=i(s+ti)^4$. Equating the real and imaginary parts on both sides of this equation gives
\begin{equation}\label{8001-2}
\begin{cases}
-4s^3t + 4st^3 - 254u^4=0,\\
s^4 - 6s^2t^2 + t^4 - 63v^4=0.
\end{cases}
\end{equation}
The scheme defined by system \eqref{8001-2} is locally insoluble at $2$.

%\begin{lstlisting}
%_<x>:=PolynomialRing(Rationals());
%k<i>:=NumberField(x^2+1);
%K<s,t>:=PolynomialRing(k,2);
%e:=1;
%F:=i^e*(s+t*i)^4;
%F;
%L<s,t>:=PolynomialRing(Rationals(),2);
%_<i>:=PolynomialRing(L);
%F:=i*s^4 - 4*s^3*t - 6*i*s^2*t^2 + 4*s*t^3 + i*t^4;
%F;
%A:=- 4*s^3*t + 4*s*t^3;
%B:=s^4 - 6*s^2*t^2 + t^4;
%F-A-i*B;

\begin{lstlisting}
P<s,t,u,v>:=ProjectiveSpace(Rationals(),3);
A:=- 4*s^3*t + 4*s*t^3-254*u^4;
B:=s^4 - 6*s^2*t^2 + t^4-63*v^4;
S:=Scheme(P,[A,B]);
IsLocallySolvable(S,2);
\end{lstlisting}

{\bf Case 3:} $254u^4+63v^4i=i^2(s+ti)^4$. Equating the real and imaginary parts on both sides of this equation gives
\begin{equation}\label{8001-3}
\begin{cases}
-s^4 + 6s^2t^2 - t^4 -254u^4=0,\\
-4s^3t + 4st^3 - 63v^4=0.
\end{cases}
\end{equation}
The scheme defined by system \eqref{8001-3} is locally insoluble at $2$.

%\begin{lstlisting}
%_<x>:=PolynomialRing(Rationals());
%k<i>:=NumberField(x^2+1);
%K<s,t>:=PolynomialRing(k,2);
%e:=2;
%F:=i^e*(s+t*i)^4;
%F;
%L<s,t>:=PolynomialRing(Rationals(),2);
%_<i>:=PolynomialRing(L);
%F:=-s^4 - 4*i*s^3*t + 6*s^2*t^2 + 4*i*s*t^3 - t^4;
%F;
%A:=- s^4 + 6*s^2*t^2 - t^4;
%B:=(-4*s^3*t + 4*s*t^3);
%F-A-i*B;
\begin{lstlisting}
P<s,t,u,v>:=ProjectiveSpace(Rationals(),3);
A:=- s^4 + 6*s^2*t^2 - t^4-254*u^4;
B:=-4*s^3*t + 4*s*t^3-63*v^4;
S:=Scheme(P,[A,B]);
IsLocallySolvable(S,2);
\end{lstlisting}

{\bf Case 4:} $254u^4+63v^4i=i^3(s+ti)^4$. Equating the real and imaginary parts on both sides of this equation gives
\begin{equation}\label{8001-4}
\begin{cases}
4s^3t - 4st^3 -254u^4=0,\\
-s^4 + 6s^2t^2 - t^4 - 63v^4=0.
\end{cases}
\end{equation}
The scheme defined by system \eqref{8001-4} is locally insoluble at $3$.

%\begin{lstlisting}
%_<x>:=PolynomialRing(Rationals());
%k<i>:=NumberField(x^2+1);
%K<s,t>:=PolynomialRing(k,2);
%e:=3;
%F:=i^e*(s+t*i)^4;
%F;
%L<s,t>:=PolynomialRing(Rationals(),2);
%_<i>:=PolynomialRing(L);
%F:=-i*s^4 + 4*s^3*t + 6*i*s^2*t^2 - 4*s*t^3 - i*t^4;
%F;
%A:= 4*s^3*t - 4*s*t^3;
%B:=(-s^4 + 6*s^2*t^2 - t^4);
%F-A-i*B;

\begin{lstlisting}
P<s,t,u,v>:=ProjectiveSpace(Rationals(),3);
A:=4*s^3*t - 4*s*t^3-254*u^4;
B:=-s^4 + 6*s^2*t^2 - t^4-63*v^4;
S:=Scheme(P,[A,B]);
IsLocallySolvable(S,3);
\end{lstlisting}

\subsection{The case of $n=8005$.}%\[n=8005.\] 
 According to Table \ref{tb:factorisation}, in the case $n=8005$ it remains to
show that equation
\begin{equation}\label{8005}
100 u^8 + v^8= 1601  w^4
\end{equation}
has no integer solutions satisfying \eqref{eq:cond}, which in this case reduces to \[\gcd(10u,v)=\gcd(10u,1601w)=\gcd(v,1601w)=1.\]
Write \eqref{8005} as 
\begin{equation} (10u^4+v^4i)(10u^4-v^4i)=(40+i)(40-i)w^4.
\end{equation}
This implies that there exist integers $s,t$ such that 
\[10u^4+v^4i=i^{\epsilon}(40\pm i)(s+ti)^4,\]
with $\epsilon \in \{0,1,2,3\}$.

{\bf Case 1:} $10u^4+v^4i=(40+i)(s+ti)^4$. Equating the real and imaginary parts on both sides of this equation gives
\begin{equation}\label{8005-1}
\begin{cases}
40s^4 - 4s^3t - 240s^2t^2 + 4st^3 + 40t^4 - 10u^4=0,\\
s^4 + 160s^3t - 6s^2t^2 - 160st^3 + t^4 - v^4=0.
\end{cases}
\end{equation}
The scheme defined by system \eqref{8005-1} is locally insoluble at $2$.

%\begin{lstlisting}
%_<x>:=PolynomialRing(Rationals());
%k<i>:=NumberField(x^2+1);
%K<s,t>:=PolynomialRing(k,2);
%e:=0;
%F:=i^e*(40+i)*(s+t*i)^4;
%F;
%L<s,t>:=PolynomialRing(Rationals(),2);
%_<i>:=PolynomialRing(L);
%F:=(i + 40)*s^4 + (160*i - 4)*s^3*t + (-6*i - 240)*s^2*t^2 + (-160*i + 4)*s*t^3 +
%    (i + 40)*t^4;
%F;
%A:=40*s^4 - 4*s^3*t -240*s^2*t^2 + 4*s*t^3 + 40*t^4;
%B:=s^4 + 160*s^3*t - 6*s^2*t^2 - 160*s*t^3 + t^4;
%F-A-i*B;

\begin{lstlisting}
P<s,t,u,v>:=ProjectiveSpace(Rationals(),3);
A:=40*s^4 - 4*s^3*t -240*s^2*t^2 + 4*s*t^3 + 40*t^4-10*u^4;
B:=s^4 + 160*s^3*t - 6*s^2*t^2 - 160*s*t^3 + t^4-v^4;
S:=Scheme(P,[A,B]);
IsLocallySolvable(S,2);
\end{lstlisting}

{\bf Case 2:} $10u^4+v^4i=i(40+i)(s+ti)^4$. Equating the real and imaginary parts on both sides of this equation gives
\begin{equation}\label{8005-2}
\begin{cases}
-s^4 - 160s^3t + 6s^2t^2 + 160st^3 - t^4 - 10u^4=0,\\
40s^4 - 4s^3t - 240s^2t^2 + 4st^3 + 40t^4 - v^4=0.
\end{cases}
\end{equation}
The scheme defined by system \eqref{8005-2} is locally insoluble at $2$.

%\begin{lstlisting}
%_<x>:=PolynomialRing(Rationals());
%k<i>:=NumberField(x^2+1);
%K<s,t>:=PolynomialRing(k,2);
%e:=1;
%F:=i^e*(40+i)*(s+t*i)^4;
%F;
%L<s,t>:=PolynomialRing(Rationals(),2);
%_<i>:=PolynomialRing(L);
%F:=(40*i - 1)*s^4 + (-4*i - 160)*s^3*t + (-240*i + 6)*s^2*t^2 + (4*i + 160)*s*t^3 +
%    (40*i - 1)*t^4;
%F;
%A:=- s^4 - 160*s^3*t + 6*s^2*t^2 + 160*s*t^3 - t^4;
%B:=40*s^4 - 4*s^3*t - 240*s^2*t^2 + 4*s*t^3 + 40*t^4;
%F-A-i*B;

\begin{lstlisting}
P<s,t,u,v>:=ProjectiveSpace(Rationals(),3);
A:=- s^4 - 160*s^3*t + 6*s^2*t^2 + 160*s*t^3 - t^4-10*u^4;
B:=40*s^4 - 4*s^3*t - 240*s^2*t^2 + 4*s*t^3 + 40*t^4-v^4;
S:=Scheme(P,[A,B]);
IsLocallySolvable(S,2);
\end{lstlisting}

{\bf Case 3:} $10u^4+v^4i=i^2(40+i)(s+ti)^4$. Equating the real and imaginary parts on both sides of this equation gives
\begin{equation}\label{8005-3}
\begin{cases}
-40s^4 + 4s^3t + 240s^2t^2 - 4st^3 - 40t^4 - 10u^4=0,\\
-s^4 - 160s^3t + 6s^2t^2 + 160st^3 - t^4 - v^4=0.
\end{cases}
\end{equation}
The scheme defined by system \eqref{8005-3} is locally insoluble at $2$.

%\begin{lstlisting}
%_<x>:=PolynomialRing(Rationals());
%k<i>:=NumberField(x^2+1);
%K<s,t>:=PolynomialRing(k,2);
%e:=2;
%F:=i^e*(40+i)*(s+t*i)^4;
%F;
%L<s,t>:=PolynomialRing(Rationals(),2);
%_<i>:=PolynomialRing(L);
%F:=(-i - 40)*s^4 + (-160*i + 4)*s^3*t + (6*i + 240)*s^2*t^2 + (160*i - 4)*s*t^3 +
%    (-i - 40)*t^4;
%F;
%A:=- 40*s^4 + 4*s^3*t + 240*s^2*t^2 - 4*s*t^3 - 40*t^4;
%B:=-s^4 - 160*s^3*t + 6*s^2*t^2 + 160*s*t^3 - t^4;
%F-A-i*B;
\begin{lstlisting}
P<s,t,u,v>:=ProjectiveSpace(Rationals(),3);
A:=- 40*s^4 + 4*s^3*t + 240*s^2*t^2 - 4*s*t^3 - 40*t^4-10*u^4;
B:=-s^4 - 160*s^3*t + 6*s^2*t^2 + 160*s*t^3 - t^4-v^4;
S:=Scheme(P,[A,B]);
IsLocallySolvable(S,2);
\end{lstlisting}

{\bf Case 4:} $10u^4+v^4i=i^3(40+i)(s+ti)^4$. Equating the real and imaginary parts on both sides of this equation gives
\begin{equation}\label{8005-4}
\begin{cases}
s^4 + 160s^3t - 6s^2t^2 - 160st^3 + t^4 - 10u^4=0,\\
-40s^4 + 4s^3t + 240s^2t^2 - 4st^3 - 40t^4 - v^4=0.
\end{cases}
\end{equation}
The scheme defined by system \eqref{8005-4} is locally insoluble at $2$.

%\begin{lstlisting}
%_<x>:=PolynomialRing(Rationals());
%k<i>:=NumberField(x^2+1);
%K<s,t>:=PolynomialRing(k,2);
%e:=3;
%F:=i^e*(40+i)*(s+t*i)^4;
%F;
%L<s,t>:=PolynomialRing(Rationals(),2);
%_<i>:=PolynomialRing(L);
%F:=(-40*i + 1)*s^4 + (4*i + 160)*s^3*t + (240*i - 6)*s^2*t^2 + (-4*i - 160)*s*t^3 +
%    (-40*i + 1)*t^4;
%F;
%A:=s^4 + 160*s^3*t - 6*s^2*t^2 - 160*s*t^3 + t^4;
%B:=-40*s^4 + 4*s^3*t + 240*s^2*t^2 - 4*s*t^3 - 40*t^4;
%F-A-i*B;
\begin{lstlisting}
P<s,t,u,v>:=ProjectiveSpace(Rationals(),3);
A:=s^4 + 160*s^3*t - 6*s^2*t^2 - 160*s*t^3 + t^4-10*u^4;
B:=-40*s^4 + 4*s^3*t + 240*s^2*t^2 - 4*s*t^3 - 40*t^4-v^4;
S:=Scheme(P,[A,B]);
IsLocallySolvable(S,2);
\end{lstlisting}

{\bf Case 5:} $10u^4+v^4i=(40-i)(s+ti)^4$. Equating the real and imaginary parts on both sides of this equation gives
\begin{equation}\label{8005-5}
\begin{cases}
40s^4 + 4s^3t - 240s^2t^2 - 4st^3 + 40t^4 - 10u^4=0,\\
-s^4 + 160s^3t + 6s^2t^2 - 160st^3 - t^4 - v^4=0.
\end{cases}
\end{equation}
The scheme defined by system \eqref{8005-5} is locally insoluble at $2$.

%\begin{lstlisting}
%_<x>:=PolynomialRing(Rationals());
%k<i>:=NumberField(x^2+1);
%K<s,t>:=PolynomialRing(k,2);
%e:=0;
%F:=i^e*(40-i)*(s+t*i)^4;
%F;
%L<s,t>:=PolynomialRing(Rationals(),2);
%_<i>:=PolynomialRing(L);
%F:=(-i + 40)*s^4 + (160*i + 4)*s^3*t + (6*i - 240)*s^2*t^2 + (-160*i - 4)*s*t^3 +
%    (-i + 40)*t^4;
%F;
%A:=40*s^4 + 4*s^3*t - 240*s^2*t^2 - 4*s*t^3 + 40*t^4;
%B:=-s^4 + 160*s^3*t + 6*s^2*t^2 - 160*s*t^3 - t^4;
%F-A-i*B;
\begin{lstlisting}
P<s,t,u,v>:=ProjectiveSpace(Rationals(),3);
A:=40*s^4 + 4*s^3*t - 240*s^2*t^2 - 4*s*t^3 + 40*t^4-10*u^4;
B:=-s^4 + 160*s^3*t + 6*s^2*t^2 - 160*s*t^3 - t^4-v^4;
S:=Scheme(P,[A,B]);
IsLocallySolvable(S,2);
\end{lstlisting}

{\bf Case 6:} $10u^4+v^4i=i(40-i)(s+ti)^4$. Equating the real and imaginary parts on both sides of this equation gives
\begin{equation}\label{8005-6}
\begin{cases}
s^4 - 160s^3t - 6s^2t^2 + 160st^3 + t^4 - 10u^4=0,\\
40s^4 + 4s^3t - 240s^2t^2 - 4st^3 + 40t^4 - v^4=0.
\end{cases}
\end{equation}
The scheme defined by system \eqref{8005-6} is locally insoluble at $2$.

%\begin{lstlisting}
%_<x>:=PolynomialRing(Rationals());
%k<i>:=NumberField(x^2+1);
%K<s,t>:=PolynomialRing(k,2);
%e:=1;
%F:=i^e*(40-i)*(s+t*i)^4;
%F;
%L<s,t>:=PolynomialRing(Rationals(),2);
%_<i>:=PolynomialRing(L);
%F:=(40*i + 1)*s^4 + (4*i - 160)*s^3*t + (-240*i - 6)*s^2*t^2 + (-4*i + 160)*s*t^3 +
%    (40*i + 1)*t^4;
%F;
%A:=s^4 - 160*s^3*t - 6*s^2*t^2 + 160*s*t^3 + t^4;
%B:=40*s^4 + 4*s^3*t - 240*s^2*t^2 - 4*s*t^3 + 40*t^4;
%F-A-i*B;
\begin{lstlisting}
P<s,t,u,v>:=ProjectiveSpace(Rationals(),3);
A:=s^4 - 160*s^3*t - 6*s^2*t^2 + 160*s*t^3 + t^4-10*u^4;
B:=40*s^4 + 4*s^3*t - 240*s^2*t^2 - 4*s*t^3 + 40*t^4-v^4;
S:=Scheme(P,[A,B]);
IsLocallySolvable(S,2);
\end{lstlisting}

{\bf Case 7:} $10u^4+v^4i=i^2(40-i)(s+ti)^4$. Equating the real and imaginary parts on both sides of this equation gives
\begin{equation}\label{8005-7}
\begin{cases}
-40s^4 - 4s^3t + 240s^2t^2 + 4st^3 - 40t^4 - 10u^4=0,\\
s^4 - 160s^3t - 6s^2t^2 + 160st^3 + t^4 - v^4=0.
\end{cases}
\end{equation}
The scheme defined by system \eqref{8005-7} is locally insoluble at $2$.

%\begin{lstlisting}
%_<x>:=PolynomialRing(Rationals());
%k<i>:=NumberField(x^2+1);
%K<s,t>:=PolynomialRing(k,2);
%e:=2;
%F:=i^e*(40-i)*(s+t*i)^4;
%F;
%L<s,t>:=PolynomialRing(Rationals(),2);
%_<i>:=PolynomialRing(L);
%F:=(i - 40)*s^4 + (-160*i - 4)*s^3*t + (-6*i + 240)*s^2*t^2 + (160*i + 4)*s*t^3 +
%    (i - 40)*t^4;
%F;
%A:=- 40*s^4 - 4*s^3*t +  240*s^2*t^2 + 4*s*t^3 - 40*t^4;
%B:=s^4 - 160*s^3*t - 6*s^2*t^2 + 160*s*t^3 + t^4;
%F-A-i*B;
\begin{lstlisting}
P<s,t,u,v>:=ProjectiveSpace(Rationals(),3);
A:=- 40*s^4 - 4*s^3*t +  240*s^2*t^2 + 4*s*t^3 - 40*t^4-10*u^4;
B:=s^4 - 160*s^3*t - 6*s^2*t^2 + 160*s*t^3 + t^4-v^4;
S:=Scheme(P,[A,B]);
IsLocallySolvable(S,2);
\end{lstlisting}

{\bf Case 8:} $10u^4+v^4i=i^3(40-i)(s+ti)^4$. Equating the real and imaginary parts on both sides of this equation gives
\begin{equation}\label{8005-8}
\begin{cases}
-s^4 + 160s^3t + 6s^2t^2 - 160st^3 - t^4 - 10u^4=0,\\
-40s^4 - 4s^3t + 240s^2t^2 + 4st^3 - 40t^4 - v^4=0.
\end{cases}
\end{equation}
The scheme defined by system \eqref{8005-8} is locally insoluble at $2$.

%\begin{lstlisting}
%_<x>:=PolynomialRing(Rationals());
%k<i>:=NumberField(x^2+1);
%K<s,t>:=PolynomialRing(k,2);
%e:=3;
%F:=i^e*(40-i)*(s+t*i)^4;
%F;
%L<s,t>:=PolynomialRing(Rationals(),2);
%_<i>:=PolynomialRing(L);
%F:=(-40*i - 1)*s^4 + (-4*i + 160)*s^3*t + (240*i + 6)*s^2*t^2 + (4*i - 160)*s*t^3 +
%    (-40*i - 1)*t^4;
%F;
%A:= - s^4 + 160*s^3*t +6*s^2*t^2 - 160*s*t^3 - t^4;
%B:=-40*s^4 - 4*s^3*t + 240*s^2*t^2 + 4*s*t^3 - 40*t^4;
%F-A-i*B;
\begin{lstlisting}
P<s,t,u,v>:=ProjectiveSpace(Rationals(),3);
A:= - s^4 + 160*s^3*t +6*s^2*t^2 - 160*s*t^3 - t^4-10*u^4;
B:=-40*s^4 - 4*s^3*t + 240*s^2*t^2 + 4*s*t^3 - 40*t^4-v^4;
S:=Scheme(P,[A,B]);
IsLocallySolvable(S,2);
\end{lstlisting}

\subsection{The case of $n=8075$.}%\[n=8075.\] 
 According to Table \ref{tb:factorisation}, in the case $n=8705$ it remains to
show that equation
\begin{equation}\label{8075}
361 u^8 + 4 v^8=425  w^4
\end{equation}
has no integer solutions satisfying \eqref{eq:cond}, which in this case reduces to \[\gcd(19u,2v)=\gcd(19u,85w)=\gcd(2v,85w)=1.\]
Write \eqref{8075} as 
\begin{equation} (19u^4+2v^4i)(19u^4-2v^4i)=(2+i)^2(2-i)^2(4+i)(4-i)w^4.
\end{equation}
Then, 
\[\begin{split}
    19u^4+2v^4i&\equiv -1+2i\pmod{5}\\
    &\equiv 2+i\pmod{5}\\
    &\not\equiv 2-i\pmod{5}.
\end{split}\]
This implies that there exist integers $s,t$ such that 
\[19u^4+2v^4i=i^{\epsilon}(2+i)^2(4\pm i)(s+ti)^4,\]
with $\epsilon \in \{0,1,2,3\}$.

{\bf Case 1:} $19u^4+2v^4i=(2+i)^2(4+i)(s+ti)^4$. Equating the real and imaginary parts on both sides of this equation gives
\begin{equation}\label{8075-1}
\begin{cases}
8s^4 - 76s^3t - 48s^2t^2 + 76st^3 + 8t^4 - 19u^4=0,\\
19s^4 + 32s^3t - 114s^2t^2 - 32st^3 + 19t^4 - 2v^4=0.
\end{cases}
\end{equation}
The scheme defined by system \eqref{8075-1} is locally insoluble at $2$.

%\begin{lstlisting}
%_<x>:=PolynomialRing(Rationals());
%k<i>:=NumberField(x^2+1);
%K<s,t>:=PolynomialRing(k,2);
%e:=0;
%F:=i^e*(2+i)^2*(4+i)*(s+t*i)^4;
%F;
%L<s,t>:=PolynomialRing(Rationals(),2);
%_<i>:=PolynomialRing(L);
%F:=(19*i + 8)*s^4 + (32*i - 76)*s^3*t + (-114*i - 48)*s^2*t^2 + (-32*i + 76)*s*t^3
%    + (19*i + 8)*t^4;
%F;
%A:=8*s^4 - 76*s^3*t - 48*s^2*t^2 + 76*s*t^3 + 8*t^4;
%B:=19*s^4 + 32*s^3*t - 114*s^2*t^2 - 32*s*t^3 + 19*t^4;
%F-A-i*B;

\begin{lstlisting}
P<s,t,u,v>:=ProjectiveSpace(Rationals(),3);
A:=8*s^4 - 76*s^3*t - 48*s^2*t^2 + 76*s*t^3 + 8*t^4-19*u^4;
B:=19*s^4 + 32*s^3*t - 114*s^2*t^2 - 32*s*t^3 + 19*t^4-2*v^4;
S:=Scheme(P,[A,B]);
IsLocallySolvable(S,2);
\end{lstlisting}

{\bf Case 2:} $19u^4+2v^4i=i(2+i)^2(4+i)(s+ti)^4$.  Equating the real and imaginary parts on both sides of this equation gives
\begin{equation}\label{8075-2}
\begin{cases}
-19s^4 - 32s^3t + 114s^2t^2 + 32st^3 - 19t^4 - 19u^4=0,\\
8s^4 - 76s^3t - 48s^2t^2 + 76st^3 + 8t^4 - 2v^4=0.
\end{cases}
\end{equation}
The scheme defined by system \eqref{8075-2} is locally insoluble at $2$.

%\begin{lstlisting}
%_<x>:=PolynomialRing(Rationals());
%k<i>:=NumberField(x^2+1);
%K<s,t>:=PolynomialRing(k,2);
%e:=1;
%F:=i^e*(2+i)^2*(4+i)*(s+t*i)^4;
%F;
%L<s,t>:=PolynomialRing(Rationals(),2);
%_<i>:=PolynomialRing(L);
%F:=(8*i - 19)*s^4 + (-76*i - 32)*s^3*t + (-48*i + 114)*s^2*t^2 + (76*i + 32)*s*t^3
%    + (8*i - 19)*t^4;
%F;
%A:=- 19*s^4 - 32*s^3*t +114*s^2*t^2 + 32*s*t^3 - 19*t^4;
%B:=8*s^4 - 76*s^3*t - 48*s^2*t^2 + 76*s*t^3 + 8*t^4;
%F-A-i*B;

\begin{lstlisting}
P<s,t,u,v>:=ProjectiveSpace(Rationals(),3);
A:=- 19*s^4 - 32*s^3*t +114*s^2*t^2 + 32*s*t^3 - 19*t^4-19*u^4;
B:=8*s^4 - 76*s^3*t - 48*s^2*t^2 + 76*s*t^3 + 8*t^4-2*v^4;
S:=Scheme(P,[A,B]);
IsLocallySolvable(S,2);
\end{lstlisting}

{\bf Case 3:} $19u^4+2v^4i=i^2(2+i)^2(4+i)(s+ti)^4$.  Equating the real and imaginary parts on both sides of this equation gives
\begin{equation}\label{8075-3}
\begin{cases}
-8s^4 + 76s^3t + 48s^2t^2 - 76st^3 - 8t^4 - 19u^4=0,\\
-19s^4 - 32s^3t + 114s^2t^2 + 32st^3 - 19t^4 - 2v^4=0.
\end{cases}
\end{equation}
The scheme defined by system \eqref{8075-3} is locally insoluble at $2$.

%\begin{lstlisting}
%_<x>:=PolynomialRing(Rationals());
%k<i>:=NumberField(x^2+1);
%K<s,t>:=PolynomialRing(k,2);
%e:=2;
%F:=i^e*(2+i)^2*(4+i)*(s+t*i)^4;
%F;
%L<s,t>:=PolynomialRing(Rationals(),2);
%_<i>:=PolynomialRing(L);
%F:=(-19*i - 8)*s^4 + (-32*i + 76)*s^3*t + (114*i + 48)*s^2*t^2 + (32*i - 76)*s*t^3
%    + (-19*i - 8)*t^4;
%F;
%A:=- 8*s^4 + 76*s^3*t +48*s^2*t^2 - 76*s*t^3 - 8*t^4;
%B:=-19*s^4 - 32*s^3*t + 114*s^2*t^2 + 32*s*t^3 - 19*t^4;
%F-A-i*B;
\begin{lstlisting}
P<s,t,u,v>:=ProjectiveSpace(Rationals(),3);
A:=- 8*s^4 + 76*s^3*t +48*s^2*t^2 - 76*s*t^3 - 8*t^4-19*u^4;
B:=-19*s^4 - 32*s^3*t + 114*s^2*t^2 + 32*s*t^3 - 19*t^4-2*v^4;
S:=Scheme(P,[A,B]);
IsLocallySolvable(S,2);
\end{lstlisting}

{\bf Case 4:} $19u^4+2v^4i=i^3(2+i)^2(4+i)(s+ti)^4$.  Equating the real and imaginary parts on both sides of this equation gives
\begin{equation}\label{8075-4}
\begin{cases}
19s^4 + 32s^3t - 114s^2t^2 - 32st^3 + 19t^4 - 19u^4=0,\\
-8s^4 + 76s^3t + 48s^2t^2 - 76st^3 - 8t^4 - 2v^4=0.
\end{cases}
\end{equation}
The scheme defined by system \eqref{8075-4} is locally insoluble at $2$.

%\begin{lstlisting}
%_<x>:=PolynomialRing(Rationals());
%k<i>:=NumberField(x^2+1);
%K<s,t>:=PolynomialRing(k,2);
%e:=3;
%F:=i^e*(2+i)^2*(4+i)*(s+t*i)^4;
%F;
%L<s,t>:=PolynomialRing(Rationals(),2);
%_<i>:=PolynomialRing(L);
%F:=(-8*i + 19)*s^4 + (76*i + 32)*s^3*t + (48*i - 114)*s^2*t^2 + (-76*i - 32)*s*t^3
%    + (-8*i + 19)*t^4;
%F;
%A:=19*s^4 + 32*s^3*t -114*s^2*t^2 - 32*s*t^3 + 19*t^4;
%B:=-8*s^4 + 76*s^3*t + 48*s^2*t^2 - 76*s*t^3 - 8*t^4;
%F-A-i*B;
\begin{lstlisting}
P<s,t,u,v>:=ProjectiveSpace(Rationals(),3);
A:=19*s^4 + 32*s^3*t -114*s^2*t^2 - 32*s*t^3 + 19*t^4-19*u^4;
B:=-8*s^4 + 76*s^3*t + 48*s^2*t^2 - 76*s*t^3 - 8*t^4-2*v^4;
S:=Scheme(P,[A,B]);
IsLocallySolvable(S,2);
\end{lstlisting}

{\bf Case 5:} $19u^4+2v^4i=(2+i)^2(4-i)(s+ti)^4$. Equating the real and imaginary parts on both sides of this equation gives
\begin{equation}\label{8075-5}
\begin{cases}
16s^4 - 52s^3t - 96s^2t^2 + 52st^3 + 16t^4 - 19u^4=0,\\
13s^4 + 64s^3t - 78s^2t^2 - 64st^3 + 13t^4 - 2v^4=0.
\end{cases}
\end{equation}
The scheme defined by system \eqref{8075-5} is locally insoluble at $2$.

%\begin{lstlisting}
%_<x>:=PolynomialRing(Rationals());
%k<i>:=NumberField(x^2+1);
%K<s,t>:=PolynomialRing(k,2);
%e:=0;
%F:=i^e*(2+i)^2*(4-i)*(s+t*i)^4;
%F;
%L<s,t>:=PolynomialRing(Rationals(),2);
%_<i>:=PolynomialRing(L);
%F:=(13*i + 16)*s^4 + (64*i - 52)*s^3*t + (-78*i - 96)*s^2*t^2 + (-64*i + 52)*s*t^3
%    + (13*i + 16)*t^4;
%F;
%A:=16*s^4 - 52*s^3*t -96*s^2*t^2 + 52*s*t^3 + 16*t^4;
%B:=13*s^4 + 64*s^3*t - 78*s^2*t^2 - 64*s*t^3 + 13*t^4;
%F-A-i*B;

\begin{lstlisting}
P<s,t,u,v>:=ProjectiveSpace(Rationals(),3);
A:=16*s^4 - 52*s^3*t -96*s^2*t^2 + 52*s*t^3 + 16*t^4-19*u^4;
B:=13*s^4 + 64*s^3*t - 78*s^2*t^2 - 64*s*t^3 + 13*t^4-2*v^4;
S:=Scheme(P,[A,B]);
IsLocallySolvable(S,2);
\end{lstlisting}

{\bf Case 6:} $19u^4+2v^4i=i(2+i)^2(4-i)(s+ti)^4$.  Equating the real and imaginary parts on both sides of this equation gives
\begin{equation}\label{8075-6}
\begin{cases}
-13s^4 - 64s^3t + 78s^2t^2 + 64st^3 - 13t^4 - 19u^4=0,\\
16s^4 - 52s^3t - 96s^2t^2 + 52st^3 + 16t^4 - 2v^4=0.
\end{cases}
\end{equation}
The scheme defined by system \eqref{8075-6} is locally insoluble at $5$.

%\begin{lstlisting}
%_<x>:=PolynomialRing(Rationals());
%k<i>:=NumberField(x^2+1);
%K<s,t>:=PolynomialRing(k,2);
%e:=1;
%F:=i^e*(2+i)^2*(4-i)*(s+t*i)^4;
%F;
%L<s,t>:=PolynomialRing(Rationals(),2);
%_<i>:=PolynomialRing(L);
%F:=(16*i - 13)*s^4 + (-52*i - 64)*s^3*t + (-96*i + 78)*s^2*t^2 + (52*i + 64)*s*t^3
%    + (16*i - 13)*t^4;
%F;
%A:= - 13*s^4 - 64*s^3*t +78*s^2*t^2 + 64*s*t^3 - 13*t^4;
%B:=16*s^4 - 52*s^3*t - 96*s^2*t^2 + 52*s*t^3 + 16*t^4;
%F-A-i*B;
\begin{lstlisting}
P<s,t,u,v>:=ProjectiveSpace(Rationals(),3);
A:= - 13*s^4 - 64*s^3*t +78*s^2*t^2 + 64*s*t^3 - 13*t^4-19*u^4;
B:=16*s^4 - 52*s^3*t - 96*s^2*t^2 + 52*s*t^3 + 16*t^4-2*v^4;
S:=Scheme(P,[A,B]);
IsLocallySolvable(S,5);
\end{lstlisting}

{\bf Case 7:} $19u^4+2v^4i=i^2(2+i)^2(4-i)(s+ti)^4$.  Equating the real and imaginary parts on both sides of this equation gives
\begin{equation}\label{8075-7}
\begin{cases}
-16s^4 + 52s^3t + 96s^2t^2 - 52st^3 - 16t^4 - 19u^4=0,\\
-13s^4 - 64s^3t + 78s^2t^2 + 64st^3 - 13t^4 - 2v^4=0.
\end{cases}
\end{equation}
The scheme defined by system \eqref{8075-7} is locally insoluble at $2$.

%\begin{lstlisting}
%_<x>:=PolynomialRing(Rationals());
%k<i>:=NumberField(x^2+1);
%K<s,t>:=PolynomialRing(k,2);
%e:=2;
%F:=i^e*(2+i)^2*(4-i)*(s+t*i)^4;
%F;
%L<s,t>:=PolynomialRing(Rationals(),2);
%_<i>:=PolynomialRing(L);
%F:=(-13*i - 16)*s^4 + (-64*i + 52)*s^3*t + (78*i + 96)*s^2*t^2 + (64*i - 52)*s*t^3
%    + (-13*i - 16)*t^4;
%F;
%A:=- 16*s^4 + 52*s^3*t +96*s^2*t^2 - 52*s*t^3 - 16*t^4;
%B:=-13*s^4 - 64*s^3*t + 78*s^2*t^2 + 64*s*t^3 - 13*t^4;
%F-A-i*B;

\begin{lstlisting}
P<s,t,u,v>:=ProjectiveSpace(Rationals(),3);
A:=- 16*s^4 + 52*s^3*t +96*s^2*t^2 - 52*s*t^3 - 16*t^4-19*u^4;
B:=-13*s^4 - 64*s^3*t + 78*s^2*t^2 + 64*s*t^3 - 13*t^4-2*v^4;
S:=Scheme(P,[A,B]);
IsLocallySolvable(S,2);
\end{lstlisting}

{\bf Case 8:} $19u^4+2v^4i=i^3(2+i)^2(4-i)(s+ti)^4$.  Equating the real and imaginary parts on both sides of this equation gives
\begin{equation}\label{8075-8}
\begin{cases}
13s^4 + 64s^3t - 78s^2t^2 - 64st^3 + 13t^4 - 19u^4=0,\\
-16s^4 + 52s^3t + 96s^2t^2 - 52st^3 - 16t^4 - 2v^4=0.
\end{cases}
\end{equation}
The scheme defined by system \eqref{8075-8} is locally insoluble at $2$.

%\begin{lstlisting}
%_<x>:=PolynomialRing(Rationals());
%k<i>:=NumberField(x^2+1);
%K<s,t>:=PolynomialRing(k,2);
%e:=3;
%F:=i^e*(2+i)^2*(4-i)*(s+t*i)^4;
%F;
%L<s,t>:=PolynomialRing(Rationals(),2);
%_<i>:=PolynomialRing(L);
%F:=(-16*i + 13)*s^4 + (52*i + 64)*s^3*t + (96*i - 78)*s^2*t^2 + (-52*i - 64)*s*t^3
%    + (-16*i + 13)*t^4;
%F;
%A:=13*s^4 + 64*s^3*t -78*s^2*t^2 - 64*s*t^3 + 13*t^4;
%B:=-16*s^4 + 52*s^3*t + 96*s^2*t^2 - 52*s*t^3 - 16*t^4;
%F-A-i*B;

\begin{lstlisting}
P<s,t,u,v>:=ProjectiveSpace(Rationals(),3);
A:=13*s^4 + 64*s^3*t -78*s^2*t^2 - 64*s*t^3 + 13*t^4-19*u^4;
B:=-16*s^4 + 52*s^3*t + 96*s^2*t^2 - 52*s*t^3 - 16*t^4-2*v^4;
S:=Scheme(P,[A,B]);
IsLocallySolvable(S,2);
\end{lstlisting}

\subsection{The case of $n=8235$.}%\[n=8235\]
According to Table \ref{tb:factorisation}, in the case $n=8235$ it remains to
show that equation
\begin{equation}\label{8235}
2916u^8+25v^8=61 w^4
\end{equation}
has no integer solutions satisfying \eqref{eq:cond}, which in this case reduces to  \[\gcd(54u,5v)=\gcd(54u,61w)=\gcd(5v,61w)=1.\]
Write \eqref{8235} as 
\begin{equation} (54u^4+5v^4i)(54u^4-5v^4i)=(6+5i)(6-5i)w^4.
\end{equation}
This implies that there exist integers $s,t$ such that 
\[54u^4+5v^4i=i^{\epsilon}(6\pm 5i)(s+ti)^4,\]
with $\epsilon \in \{0,1,2,3\}$.

{\bf Case 1:} $54u^4+5v^4i=(6+5i)(s+ti)^4$. Equating the real and imaginary parts on both sides of this equation gives
\begin{equation}\label{8235-1}
\begin{cases}
6s^4 - 20s^3t - 36s^2t^2 + 20st^3 + 6t^4 - 54u^4=0,\\
5s^4 + 24s^3t - 30s^2t^2 - 24st^3 + 5t^4 - 5v^4=0.
\end{cases}
\end{equation}
The scheme defined by system \eqref{8235-1} is locally insoluble at $5$.

%\begin{lstlisting}
%_<x>:=PolynomialRing(Rationals());
%k<i>:=NumberField(x^2+1);
%K<s,t>:=PolynomialRing(k,2);
%e:=0;
%F:=i^e*(6+5*i)*(s+t*i)^4;
%F;
%L<s,t>:=PolynomialRing(Rationals(),2);
%_<i>:=PolynomialRing(L);
%F:=(5*i + 6)*s^4 + (24*i - 20)*s^3*t + (-30*i - 36)*s^2*t^2 + (-24*i + 20)*s*t^3 +
%    (5*i + 6)*t^4;
%F;
%A:=6*s^4 - 20*s^3*t - 36*s^2*t^2 + 20*s*t^3 + 6*t^4;
%B:=5*s^4 + 24*s^3*t - 30*s^2*t^2 - 24*s*t^3 + 5*t^4;
%F-A-i*B;

\begin{lstlisting}
P<s,t,u,v>:=ProjectiveSpace(Rationals(),3);
A:=6*s^4 - 20*s^3*t - 36*s^2*t^2 + 20*s*t^3 + 6*t^4-54*u^4;
B:=5*s^4 + 24*s^3*t - 30*s^2*t^2 - 24*s*t^3 + 5*t^4-5*v^4;
S:=Scheme(P,[A,B]);
IsLocallySolvable(S,5);
\end{lstlisting}

{\bf Case 2:} $54u^4+5v^4i=i(6+5i)(s+ti)^4$. Equating the real and imaginary parts on both sides of this equation gives
\begin{equation}\label{8235-2}
\begin{cases}
-5s^4 - 24s^3t + 30s^2t^2 + 24st^3 - 5t^4 - 54u^4=0,\\
6s^4 - 20s^3t - 36s^2t^2 + 20st^3 + 6t^4 - 5v^4=0.
\end{cases}
\end{equation}
The scheme defined by system \eqref{8235-2} is locally insoluble at $5$.
%
%\begin{lstlisting}
%_<x>:=PolynomialRing(Rationals());
%k<i>:=NumberField(x^2+1);
%K<s,t>:=PolynomialRing(k,2);
%e:=1;
%F:=i^e*(6+5*i)*(s+t*i)^4;
%F;
%L<s,t>:=PolynomialRing(Rationals(),2);
%_<i>:=PolynomialRing(L);
%F:=(6*i - 5)*s^4 + (-20*i - 24)*s^3*t + (-36*i + 30)*s^2*t^2 + (20*i + 24)*s*t^3 +
%    (6*i - 5)*t^4;
%F;
%A:=- 5*s^4 - 24*s^3*t +30*s^2*t^2 + 24*s*t^3 - 5*t^4;
%B:=6*s^4 - 20*s^3*t - 36*s^2*t^2 + 20*s*t^3 + 6*t^4;
%F-A-i*B;

\begin{lstlisting}
P<s,t,u,v>:=ProjectiveSpace(Rationals(),3);
A:=- 5*s^4 - 24*s^3*t +30*s^2*t^2 + 24*s*t^3 - 5*t^4-54*u^4;
B:=6*s^4 - 20*s^3*t - 36*s^2*t^2 + 20*s*t^3 + 6*t^4-5*v^4;
S:=Scheme(P,[A,B]);
IsLocallySolvable(S,5);
\end{lstlisting}

{\bf Case 3:} $54u^4+5v^4i=i^2(6+5i)(s+ti)^4$. Equating the real and imaginary parts on both sides of this equation gives
\begin{equation}\label{8235-3}
\begin{cases}
-6s^4 + 20s^3t + 36s^2t^2 - 20st^3 - 6t^4 - 54u^4=0,\\
-5s^4 - 24s^3t + 30s^2t^2 + 24st^3 - 5t^4 - 5v^4=0.
\end{cases}
\end{equation}
The scheme defined by system \eqref{8235-3} is locally insoluble at $2$.

%\begin{lstlisting}
%_<x>:=PolynomialRing(Rationals());
%k<i>:=NumberField(x^2+1);
%K<s,t>:=PolynomialRing(k,2);
%e:=2;
%F:=i^e*(6+5*i)*(s+t*i)^4;
%F;
%L<s,t>:=PolynomialRing(Rationals(),2);
%_<i>:=PolynomialRing(L);
%F:=(-5*i - 6)*s^4 + (-24*i + 20)*s^3*t + (30*i + 36)*s^2*t^2 + (24*i - 20)*s*t^3 +
%    (-5*i - 6)*t^4;
%F;
%A:=- 6*s^4 + 20*s^3*t +36*s^2*t^2 - 20*s*t^3 - 6*t^4;
%B:=-5*s^4 - 24*s^3*t + 30*s^2*t^2 + 24*s*t^3 - 5*t^4;
%F-A-i*B;

\begin{lstlisting}
P<s,t,u,v>:=ProjectiveSpace(Rationals(),3);
A:=- 6*s^4 + 20*s^3*t +36*s^2*t^2 - 20*s*t^3 - 6*t^4-54*u^4;
B:=-5*s^4 - 24*s^3*t + 30*s^2*t^2 + 24*s*t^3 - 5*t^4-5*v^4;
S:=Scheme(P,[A,B]);
IsLocallySolvable(S,2);
\end{lstlisting}

{\bf Case 4:} $54u^4+5v^4i=i^3(6+5i)(s+ti)^4$. Equating the real and imaginary parts on both sides of this equation gives
\begin{equation}\label{8235-4}
\begin{cases}
5s^4 + 24s^3t - 30s^2t^2 - 24st^3 + 5t^4 - 54u^4=0,\\
-6s^4 + 20s^3t + 36s^2t^2 - 20st^3 - 6t^4 - 5v^4=0.
\end{cases}
\end{equation}
The scheme defined by system \eqref{8235-4} is locally insoluble at $2$.

%\begin{lstlisting}
%_<x>:=PolynomialRing(Rationals());
%k<i>:=NumberField(x^2+1);
%K<s,t>:=PolynomialRing(k,2);
%e:=3;
%F:=i^e*(6+5*i)*(s+t*i)^4;
%F;
%L<s,t>:=PolynomialRing(Rationals(),2);
%_<i>:=PolynomialRing(L);
%F:=(-6*i + 5)*s^4 + (20*i + 24)*s^3*t + (36*i - 30)*s^2*t^2 + (-20*i - 24)*s*t^3 +
%    (-6*i + 5)*t^4;
%F;
%A:=5*s^4 + 24*s^3*t -30*s^2*t^2 - 24*s*t^3 + 5*t^4;
%B:=-6*s^4 + 20*s^3*t + 36*s^2*t^2 - 20*s*t^3 - 6*t^4;
%F-A-i*B;

\begin{lstlisting}
P<s,t,u,v>:=ProjectiveSpace(Rationals(),3);
A:=5*s^4 + 24*s^3*t -30*s^2*t^2 - 24*s*t^3 + 5*t^4-54*u^4;
B:=-6*s^4 + 20*s^3*t + 36*s^2*t^2 - 20*s*t^3 - 6*t^4-5*v^4;
S:=Scheme(P,[A,B]);
IsLocallySolvable(S,2);
\end{lstlisting}

{\bf Case 5:} $54u^4+5v^4i=(6-5i)(s+ti)^4$. Equating the real and imaginary parts on both sides of this equation gives
\begin{equation}\label{8235-5}
\begin{cases}
6s^4 + 20s^3t - 36s^2t^2 - 20st^3 + 6t^4 - 54u^4=0,\\
-5s^4 + 24s^3t + 30s^2t^2 - 24st^3 - 5t^4 - 5v^4=0.
\end{cases}
\end{equation}
The scheme defined by system \eqref{8235-5} is locally insoluble at $2$.

%\begin{lstlisting}
%_<x>:=PolynomialRing(Rationals());
%k<i>:=NumberField(x^2+1);
%K<s,t>:=PolynomialRing(k,2);
%e:=0;
%F:=i^e*(6-5*i)*(s+t*i)^4;
%F;
%L<s,t>:=PolynomialRing(Rationals(),2);
%_<i>:=PolynomialRing(L);
%F:=(-5*i + 6)*s^4 + (24*i + 20)*s^3*t + (30*i - 36)*s^2*t^2 + (-24*i - 20)*s*t^3 +
%    (-5*i + 6)*t^4;
%F;
%A:=6*s^4 + 20*s^3*t - 36*s^2*t^2 - 20*s*t^3 + 6*t^4;
%B:=-5*s^4 + 24*s^3*t + 30*s^2*t^2 - 24*s*t^3 - 5*t^4;
%F-A-i*B;

\begin{lstlisting}
P<s,t,u,v>:=ProjectiveSpace(Rationals(),3);
A:=6*s^4 + 20*s^3*t - 36*s^2*t^2 - 20*s*t^3 + 6*t^4-54*u^4;
B:=-5*s^4 + 24*s^3*t + 30*s^2*t^2 - 24*s*t^3 - 5*t^4-5*v^4;
S:=Scheme(P,[A,B]);
IsLocallySolvable(S,2);
\end{lstlisting}

{\bf Case 6:} $54u^4+5v^4i=i(6-5i)(s+ti)^4$. Equating the real and imaginary parts on both sides of this equation gives
\begin{equation}\label{8235-6}
\begin{cases}
5s^4 - 24s^3t - 30s^2t^2 + 24st^3 + 5t^4 - 54u^4=0, \\
6s^4 + 20s^3t - 36s^2t^2 - 20st^3 + 6t^4 - 5v^4=0.
\end{cases}
\end{equation}
The scheme defined by system \eqref{8235-6} is locally insoluble at $2$.

%\begin{lstlisting}
%_<x>:=PolynomialRing(Rationals());
%k<i>:=NumberField(x^2+1);
%K<s,t>:=PolynomialRing(k,2);
%e:=1;
%F:=i^e*(6-5*i)*(s+t*i)^4;
%F;
%L<s,t>:=PolynomialRing(Rationals(),2);
%_<i>:=PolynomialRing(L);
%F:=(6*i + 5)*s^4 + (20*i - 24)*s^3*t + (-36*i - 30)*s^2*t^2 + (-20*i + 24)*s*t^3 +
%    (6*i + 5)*t^4;
%F;
%A:=5*s^4 - 24*s^3*t -30*s^2*t^2 + 24*s*t^3 + 5*t^4;
%B:=6*s^4 + 20*s^3*t - 36*s^2*t^2 - 20*s*t^3 + 6*t^4;
%F-A-i*B;

\begin{lstlisting}
P<s,t,u,v>:=ProjectiveSpace(Rationals(),3);
A:=5*s^4 - 24*s^3*t -30*s^2*t^2 + 24*s*t^3 + 5*t^4-54*u^4;
B:=6*s^4 + 20*s^3*t - 36*s^2*t^2 - 20*s*t^3 + 6*t^4-5*v^4;
S:=Scheme(P,[A,B]);
IsLocallySolvable(S,2);
\end{lstlisting}

{\bf Case 7:} $54u^4+5v^4i=i^2(6-5i)(s+ti)^4$. Equating the real and imaginary parts on both sides of this equation gives
\begin{equation}\label{8235-7}
\begin{cases}
-6s^4 - 20s^3t + 36s^2t^2 + 20st^3 - 6t^4 - 54u^4=0, \\
5s^4 - 24s^3t - 30s^2t^2 + 24st^3 + 5t^4 - 5v^4=0.
\end{cases}
\end{equation}
The scheme defined by system \eqref{8235-7} is locally insoluble at $2$.

%\begin{lstlisting}
%_<x>:=PolynomialRing(Rationals());
%k<i>:=NumberField(x^2+1);
%K<s,t>:=PolynomialRing(k,2);
%e:=2;
%F:=i^e*(6-5*i)*(s+t*i)^4;
%F;
%L<s,t>:=PolynomialRing(Rationals(),2);
%_<i>:=PolynomialRing(L);
%F:=(5*i - 6)*s^4 + (-24*i - 20)*s^3*t + (-30*i + 36)*s^2*t^2 + (24*i + 20)*s*t^3 +
%    (5*i - 6)*t^4;
%F;
%A:=- 6*s^4 - 20*s^3*t +36*s^2*t^2 + 20*s*t^3 - 6*t^4;
%B:=5*s^4 - 24*s^3*t - 30*s^2*t^2 + 24*s*t^3 + 5*t^4;
%F-A-i*B;

\begin{lstlisting}
P<s,t,u,v>:=ProjectiveSpace(Rationals(),3);
A:=- 6*s^4 - 20*s^3*t +36*s^2*t^2 + 20*s*t^3 - 6*t^4-54*u^4;
B:=5*s^4 - 24*s^3*t - 30*s^2*t^2 + 24*s*t^3 + 5*t^4-5*v^4;
S:=Scheme(P,[A,B]);
IsLocallySolvable(S,2);
\end{lstlisting}

{\bf Case 8:} $54u^4+5v^4i=i^3(6-5i)(s+ti)^4$. Equating the real and imaginary parts on both sides of this equation gives
\begin{equation}\label{8235-8}
\begin{cases}
-5s^4 + 24s^3t + 30s^2t^2 - 24st^3 - 5t^4 - 54u^4=0,\\
-6s^4 - 20s^3t + 36s^2t^2 + 20st^3 - 6t^4 - 5v^4=0.
\end{cases}
\end{equation}
The scheme defined by system \eqref{8235-8} is locally insoluble at $2$.

%\begin{lstlisting}
%_<x>:=PolynomialRing(Rationals());
%k<i>:=NumberField(x^2+1);
%K<s,t>:=PolynomialRing(k,2);
%e:=3;
%F:=i^e*(6-5*i)*(s+t*i)^4;
%F;
%L<s,t>:=PolynomialRing(Rationals(),2);
%_<i>:=PolynomialRing(L);
%F:=(-6*i - 5)*s^4 + (-20*i + 24)*s^3*t + (36*i + 30)*s^2*t^2 + (20*i - 24)*s*t^3 +
%    (-6*i - 5)*t^4;
%F;
%A:=- 5*s^4 + 24*s^3*t + 30*s^2*t^2 - 24*s*t^3 - 5*t^4;
%B:=-6*s^4 - 20*s^3*t + 36*s^2*t^2 + 20*s*t^3 - 6*t^4;
%F-A-i*B;

\begin{lstlisting}
P<s,t,u,v>:=ProjectiveSpace(Rationals(),3);
A:=- 5*s^4 + 24*s^3*t + 30*s^2*t^2 - 24*s*t^3 - 5*t^4-54*u^4;
B:=-6*s^4 - 20*s^3*t + 36*s^2*t^2 + 20*s*t^3 - 6*t^4-5*v^4;
S:=Scheme(P,[A,B]);
IsLocallySolvable(S,2);
\end{lstlisting}

\subsection{The case of $n=8705$.}%\[n=8705\]

 According to Table \ref{tb:factorisation}, in the case $n=8705$ it remains to
show that equation
\begin{equation}\label{8705}
4u^8+v^8=8705 w^4
\end{equation}
has no integer solutions satisfying \eqref{eq:cond}, which in this case reduces to  \[\gcd(2u,v)=\gcd(2u,8705w)=\gcd(v,8705w)=1.\]
Write \eqref{8705} as 
\begin{equation} (2u^4+v^4i)(2u^4-v^4i)=(2+i)(2-i)(30+29i)(30-29i)w^4.
\end{equation}
Then, \[\begin{split}
    2u^4+v^4i&\equiv 2+i\pmod{5}\\
    &\equiv 0\pmod{2+i}\\
    &\not\equiv 0\pmod{2-i}.
\end{split}\]
This implies that there exist integers $s,t$ such that 
\[2u^4+v^4i=i^{\epsilon}(2+i)(30\pm 29i)(s+ti)^4,\]
with $\epsilon \in \{0,1,2,3\}$.

{\bf Case 1:} $2u^4+v^4i=(2+i)(30+29i)(s+ti)^4$. Equating the real and imaginary parts on both sides of this equation gives
\begin{equation}\label{8705-1}
\begin{cases}
31s^4 - 352s^3t - 186s^2t^2 + 352st^3 + 31t^4 - 2u^4=0,\\
88s^4 + 124s^3t - 528s^2t^2 - 124st^3 + 88t^4 - v^4=0.
\end{cases}
\end{equation}
The scheme defined by system \eqref{8705-1} is locally insoluble at $2$.

%\begin{lstlisting}
%_<x>:=PolynomialRing(Rationals());
%k<i>:=NumberField(x^2+1);
%K<s,t>:=PolynomialRing(k,2);
%e:=0;
%F:=i^e*(2+i)*(30+29*i)*(s+t*i)^4;
%F;
%L<s,t>:=PolynomialRing(Rationals(),2);
%_<i>:=PolynomialRing(L);
%F:=(88*i + 31)*s^4 + (124*i - 352)*s^3*t + (-528*i - 186)*s^2*t^2 + (-124*i +
%    352)*s*t^3 + (88*i + 31)*t^4;
%F;
%A:=31*s^4 - 352*s^3*t -186*s^2*t^2 + 352*s*t^3 + 31*t^4;
%B:=88*s^4 + 124*s^3*t - 528*s^2*t^2 - 124*s*t^3 + 88*t^4;
%F-A-i*B;
\begin{lstlisting}
P<s,t,u,v>:=ProjectiveSpace(Rationals(),3);
A:=31*s^4 - 352*s^3*t -186*s^2*t^2 + 352*s*t^3 + 31*t^4-2*u^4;
B:=88*s^4 + 124*s^3*t - 528*s^2*t^2 - 124*s*t^3 + 88*t^4-v^4;
S:=Scheme(P,[A,B]);
IsLocallySolvable(S,2);
\end{lstlisting}

{\bf Case 2:} $2u^4+v^4i=i(2+i)(30+29i)(s+ti)^4$. Equating the real and imaginary parts on both sides of this equation gives
\begin{equation}\label{8705-2}
\begin{cases}
-88s^4 - 124s^3t + 528s^2t^2 + 124st^3 - 88t^4 - 2u^4=0,\\
31s^4 - 352s^3t - 186s^2t^2 + 352st^3 + 31t^4 - v^4=0.
\end{cases}
\end{equation}
The scheme defined by system \eqref{8705-2} is locally insoluble at $2$.

%\begin{lstlisting}
%_<x>:=PolynomialRing(Rationals());
%k<i>:=NumberField(x^2+1);
%K<s,t>:=PolynomialRing(k,2);
%e:=1;
%F:=i^e*(2+i)*(30+29*i)*(s+t*i)^4;
%F;
%L<s,t>:=PolynomialRing(Rationals(),2);
%_<i>:=PolynomialRing(L);
%F:=(31*i - 88)*s^4 + (-352*i - 124)*s^3*t + (-186*i + 528)*s^2*t^2 + (352*i +
%    124)*s*t^3 + (31*i - 88)*t^4;
%F;
%A:=- 88*s^4 - 124*s^3*t + 528*s^2*t^2 + 124*s*t^3 - 88*t^4;
%B:=31*s^4 - 352*s^3*t - 186*s^2*t^2 + 352*s*t^3 + 31*t^4;
%F-A-i*B;

\begin{lstlisting}
P<s,t,u,v>:=ProjectiveSpace(Rationals(),3);
A:=- 88*s^4 - 124*s^3*t + 528*s^2*t^2 + 124*s*t^3 - 88*t^4-2*u^4;
B:=31*s^4 - 352*s^3*t - 186*s^2*t^2 + 352*s*t^3 + 31*t^4-v^4;
S:=Scheme(P,[A,B]);
IsLocallySolvable(S,2);
\end{lstlisting}

{\bf Case 3:} $2u^4+v^4i=i^2(2+i)(30+29i)(s+ti)^4$. Equating the real and imaginary parts on both sides of this equation gives
\begin{equation}\label{8705-3}
\begin{cases}
-31s^4 + 352s^3t + 186s^2t^2 - 352st^3 - 31t^4 - 2u^4=0,\\
-88s^4 - 124s^3t + 528s^2t^2 + 124st^3 - 88t^4 - v^4=0.
\end{cases}
\end{equation}
The scheme defined by system \eqref{8705-3} is locally insoluble at $2$.

%\begin{lstlisting}
%_<x>:=PolynomialRing(Rationals());
%k<i>:=NumberField(x^2+1);
%K<s,t>:=PolynomialRing(k,2);
%e:=2;
%F:=i^e*(2+i)*(30+29*i)*(s+t*i)^4;
%F;
%L<s,t>:=PolynomialRing(Rationals(),2);
%_<i>:=PolynomialRing(L);
%F:=(-88*i - 31)*s^4 + (-124*i + 352)*s^3*t + (528*i + 186)*s^2*t^2 + (124*i -
%    352)*s*t^3 + (-88*i - 31)*t^4;
%F;
%A:=- 31*s^4 + 352*s^3*t + 186*s^2*t^2 - 352*s*t^3 - 31*t^4;
%B:=-88*s^4 - 124*s^3*t + 528*s^2*t^2 + 124*s*t^3 - 88*t^4;
%F-A-i*B;
\begin{lstlisting}
P<s,t,u,v>:=ProjectiveSpace(Rationals(),3);
A:=- 31*s^4 + 352*s^3*t + 186*s^2*t^2 - 352*s*t^3 - 31*t^4-2*u^4;
B:=-88*s^4 - 124*s^3*t + 528*s^2*t^2 + 124*s*t^3 - 88*t^4-v^4;
S:=Scheme(P,[A,B]);
IsLocallySolvable(S,2);
\end{lstlisting}

{\bf Case 4:} $2u^4+v^4i=i^3(2+i)(30+29i)(s+ti)^4$. Equating the real and imaginary parts on both sides of this equation gives
\begin{equation}\label{8705-4}
\begin{cases}
88s^4 + 124s^3t - 528s^2t^2 - 124st^3 + 88t^4 - 2u^4=0,\\
-31s^4 + 352s^3t + 186s^2t^2 - 352st^3 - 31t^4 - v^4=0.
\end{cases}
\end{equation}
The scheme defined by system \eqref{8705-4} is locally insoluble at $2$.

%\begin{lstlisting}
%_<x>:=PolynomialRing(Rationals());
%k<i>:=NumberField(x^2+1);
%K<s,t>:=PolynomialRing(k,2);
%e:=3;
%F:=i^e*(2+i)*(30+29*i)*(s+t*i)^4;
%F;
%L<s,t>:=PolynomialRing(Rationals(),2);
%_<i>:=PolynomialRing(L);
%F:=(-31*i + 88)*s^4 + (352*i + 124)*s^3*t + (186*i - 528)*s^2*t^2 + (-352*i -
%    124)*s*t^3 + (-31*i + 88)*t^4;
%F;
%A:=88*s^4 + 124*s^3*t- 528*s^2*t^2 - 124*s*t^3 + 88*t^4;
%B:=-31*s^4 + 352*s^3*t + 186*s^2*t^2 - 352*s*t^3 - 31*t^4;
%F-A-i*B;
\begin{lstlisting}
P<s,t,u,v>:=ProjectiveSpace(Rationals(),3);
A:=88*s^4 + 124*s^3*t- 528*s^2*t^2 - 124*s*t^3 + 88*t^4-2*u^4;
B:=-31*s^4 + 352*s^3*t + 186*s^2*t^2 - 352*s*t^3 - 31*t^4-v^4;
S:=Scheme(P,[A,B]);
IsLocallySolvable(S,2);
\end{lstlisting}

{\bf Case 5:} $2u^4+v^4i=(2+i)(30-29i)(s+ti)^4$. Equating the real and imaginary parts on both sides of this equation gives
\begin{equation}\label{8705-5}
\begin{cases}
89s^4 + 112s^3t - 534s^2t^2 - 112st^3 + 89t^4 - 2u^4=0,\\
-28s^4 + 356s^3t + 168s^2t^2 - 356st^3 - 28t^4 - v^4=0.
\end{cases}
\end{equation}
The scheme defined by system \eqref{8705-5} is locally insoluble at $2$.

%\begin{lstlisting}
%_<x>:=PolynomialRing(Rationals());
%k<i>:=NumberField(x^2+1);
%K<s,t>:=PolynomialRing(k,2);
%e:=0;
%F:=i^e*(2+i)*(30-29*i)*(s+t*i)^4;
%F;
%L<s,t>:=PolynomialRing(Rationals(),2);
%_<i>:=PolynomialRing(L);
%F:=(-28*i + 89)*s^4 + (356*i + 112)*s^3*t + (168*i - 534)*s^2*t^2 + (-356*i -
%    112)*s*t^3 + (-28*i + 89)*t^4;
%F;
%A:=89*s^4 + 112*s^3*t- 534*s^2*t^2 - 112*s*t^3 + 89*t^4;
%B:=-28*s^4 + 356*s^3*t + 168*s^2*t^2 - 356*s*t^3 - 28*t^4;
%F-A-i*B;
\begin{lstlisting}
P<s,t,u,v>:=ProjectiveSpace(Rationals(),3);
A:=89*s^4 + 112*s^3*t- 534*s^2*t^2 - 112*s*t^3 + 89*t^4-2*u^4;
B:=-28*s^4 + 356*s^3*t + 168*s^2*t^2 - 356*s*t^3 - 28*t^4-v^4;
S:=Scheme(P,[A,B]);
IsLocallySolvable(S,2);
\end{lstlisting}

{\bf Case 6:} $2u^4+v^4i=i(2+i)(30-29i)(s+ti)^4$. Equating the real and imaginary parts on both sides of this equation gives
\begin{equation}\label{8705-6}
\begin{cases}
28s^4 - 356s^3t - 168s^2t^2 + 356st^3 + 28t^4 - 2u^4=0,\\
89s^4 + 112s^3t - 534s^2t^2 - 112st^3 + 89t^4 - v^4=0.
\end{cases}
\end{equation}
The scheme defined by system \eqref{8705-6} is locally insoluble at $2$.

%\begin{lstlisting}
%_<x>:=PolynomialRing(Rationals());
%k<i>:=NumberField(x^2+1);
%K<s,t>:=PolynomialRing(k,2);
%e:=1;
%F:=i^e*(2+i)*(30-29*i)*(s+t*i)^4;
%F;
%L<s,t>:=PolynomialRing(Rationals(),2);
%_<i>:=PolynomialRing(L);
%F:=(89*i + 28)*s^4 + (112*i - 356)*s^3*t + (-534*i - 168)*s^2*t^2 + (-112*i +
%    356)*s*t^3 + (89*i + 28)*t^4;
%F;
%A:= 28*s^4 - 356*s^3*t - 168*s^2*t^2 + 356*s*t^3 + 28*t^4;
%B:=89*s^4 + 112*s^3*t - 534*s^2*t^2 - 112*s*t^3 + 89*t^4;
%F-A-i*B;
\begin{lstlisting}
P<s,t,u,v>:=ProjectiveSpace(Rationals(),3);
A:= 28*s^4 - 356*s^3*t - 168*s^2*t^2 + 356*s*t^3 + 28*t^4-2*u^4;
B:=89*s^4 + 112*s^3*t - 534*s^2*t^2 - 112*s*t^3 + 89*t^4-v^4;
S:=Scheme(P,[A,B]);
IsLocallySolvable(S,2);
\end{lstlisting}

{\bf Case 7:} $2u^4+v^4i=i^2(2+i)(30-29i)(s+ti)^4$. Equating the real and imaginary parts on both sides of this equation gives
\begin{equation}\label{8705-7}
\begin{cases}
-89s^4 - 112s^3t + 534s^2t^2 + 112st^3 - 89t^4 - 2u^4=0,\\
28s^4 - 356s^3t - 168s^2t^2 + 356st^3 + 28t^4 - v^4=0.
\end{cases}
\end{equation}
The scheme defined by system \eqref{8705-7} is locally insoluble at $2$.

%\begin{lstlisting}
%_<x>:=PolynomialRing(Rationals());
%k<i>:=NumberField(x^2+1);
%K<s,t>:=PolynomialRing(k,2);
%e:=2;
%F:=i^e*(2+i)*(30-29*i)*(s+t*i)^4;
%F;
%L<s,t>:=PolynomialRing(Rationals(),2);
%_<i>:=PolynomialRing(L);
%F:=(28*i - 89)*s^4 + (-356*i - 112)*s^3*t + (-168*i + 534)*s^2*t^2 + (356*i +
%    112)*s*t^3 + (28*i - 89)*t^4;
%F;
%A:=- 89*s^4 - 112*s^3*t +534*s^2*t^2 + 112*s*t^3 - 89*t^4;
%B:=28*s^4 - 356*s^3*t - 168*s^2*t^2 + 356*s*t^3 + 28*t^4;
%F-A-i*B;
\begin{lstlisting}
P<s,t,u,v>:=ProjectiveSpace(Rationals(),3);
A:=- 89*s^4 - 112*s^3*t +534*s^2*t^2 + 112*s*t^3 - 89*t^4-2*u^4;
B:=28*s^4 - 356*s^3*t - 168*s^2*t^2 + 356*s*t^3 + 28*t^4-v^4;
S:=Scheme(P,[A,B]);
IsLocallySolvable(S,2);
\end{lstlisting}

{\bf Case 8:} $2u^4+v^4i=i^3(2+i)(30-29i)(s+ti)^4$. Equating the real and imaginary parts on both sides of this equation gives
\begin{equation}\label{8705-8}
\begin{cases}
-28s^4 + 356s^3t + 168s^2t^2 - 356st^3 - 28t^4 - 2u^4=0,\\
-89s^4 - 112s^3t + 534s^2t^2 + 112st^3 - 89t^4 - v^4=0.
\end{cases}
\end{equation}
The scheme defined by system \eqref{8705-8} is locally insoluble at $2$.

%\begin{lstlisting}
%_<x>:=PolynomialRing(Rationals());
%k<i>:=NumberField(x^2+1);
%K<s,t>:=PolynomialRing(k,2);
%e:=3;
%F:=i^e*(2+i)*(30-29*i)*(s+t*i)^4;
%F;
%L<s,t>:=PolynomialRing(Rationals(),2);
%_<i>:=PolynomialRing(L);
%F:=(-89*i - 28)*s^4 + (-112*i + 356)*s^3*t + (534*i + 168)*s^2*t^2 + (112*i -
%    356)*s*t^3 + (-89*i - 28)*t^4;
%F;
%A:=- 28*s^4 + 356*s^3*t + 168*s^2*t^2 - 356*s*t^3 - 28*t^4;
%B:=-89*s^4 - 112*s^3*t + 534*s^2*t^2 + 112*s*t^3 - 89*t^4;
%F-A-i*B;
\begin{lstlisting}
P<s,t,u,v>:=ProjectiveSpace(Rationals(),3);
A:=- 28*s^4 + 356*s^3*t + 168*s^2*t^2 - 356*s*t^3 - 28*t^4-2*u^4;
B:=-89*s^4 - 112*s^3*t + 534*s^2*t^2 + 112*s*t^3 - 89*t^4-v^4;
S:=Scheme(P,[A,B]);
IsLocallySolvable(S,2);
\end{lstlisting}

\subsection{The case of $n=8931$.}%\[n=8931\]
 According to Table \ref{tb:factorisation}, in the case $n=8931$ it remains to
show that equation
\begin{equation}\label{8931}
36u^8+v^8=2977 w^4
\end{equation}
has no integer solutions satisfying \eqref{eq:cond}, which in this case reduces to  \[\gcd(6u,v)=\gcd(6u,2977w)=\gcd(v,2977w)=1.\]
Write \eqref{8931} as 
\begin{equation} (6u^4+v^4i)(6u^4-v^4i)=(3+2i)(3-2i)(15+2i)(15-2i)w^4.
\end{equation}
This implies that there exist integers $s,t$ such that 
\[6u^4+v^4i=i^{\epsilon}(3\pm 2i)(15\pm 2i)(s+ti)^4,\]
with $\epsilon \in \{0,1,2,3\}$.

{\bf Case 1:} $6u^4+v^4i=(3+2i)(15+2i)(s+ti)^4$. Equating the real and imaginary parts on both sides of this equation gives
\begin{equation}\label{8931-1}
\begin{cases}
41*s^4 - 144*s^3*t - 246*s^2*t^2 + 144*s*t^3 + 41*t^4 - 6*u^4=0,\\
36*s^4 + 164*s^3*t - 216*s^2*t^2 - 164*s*t^3 + 36*t^4 - v^4=0.
\end{cases}
\end{equation}
The scheme defined by system \eqref{8931-1} is locally insoluble at $2$.

%\begin{lstlisting}
%_<x>:=PolynomialRing(Rationals());
%k<i>:=NumberField(x^2+1);
%K<s,t>:=PolynomialRing(k,2);
%e:=0;
%F:=i^e*(3+2*i)*(15+2*i)*(s+t*i)^4;
%F;
%L<s,t>:=PolynomialRing(Rationals(),2);
%_<i>:=PolynomialRing(L);
%F:=(36*i + 41)*s^4 + (164*i - 144)*s^3*t + (-216*i - 246)*s^2*t^2 + (-164*i +
%    144)*s*t^3 + (36*i + 41)*t^4;
%F;
%A:=41*s^4 - 144*s^3*t - 246*s^2*t^2 + 144*s*t^3 + 41*t^4;
%B:=36*s^4 + 164*s^3*t - 216*s^2*t^2 - 164*s*t^3 + 36*t^4;
%F-A-i*B;

\begin{lstlisting}
P<s,t,u,v>:=ProjectiveSpace(Rationals(),3);
A:=41*s^4 - 144*s^3*t - 246*s^2*t^2 + 144*s*t^3 + 41*t^4-6*u^4;
B:=36*s^4 + 164*s^3*t - 216*s^2*t^2 - 164*s*t^3 + 36*t^4-v^4;
S:=Scheme(P,[A,B]);
IsLocallySolvable(S,2);
\end{lstlisting}

{\bf Case 2:} $6u^4+v^4i=i(3+2i)(15+2i)(s+ti)^4$. Equating the real and imaginary parts on both sides of this equation gives
\begin{equation}\label{8931-2}
\begin{cases}
-36*s^4 - 164*s^3*t + 216*s^2*t^2 + 164*s*t^3 - 36*t^4 - 6*u^4=0,\\
41*s^4 - 144*s^3*t - 246*s^2*t^2 + 144*s*t^3 + 41*t^4 - v^4=0.
\end{cases}
\end{equation}
The scheme defined by system \eqref{8931-2} is locally insoluble at $2$.

%\begin{lstlisting}
%_<x>:=PolynomialRing(Rationals());
%k<i>:=NumberField(x^2+1);
%K<s,t>:=PolynomialRing(k,2);
%e:=1;
%F:=i^e*(3+2*i)*(15+2*i)*(s+t*i)^4;
%F;
%L<s,t>:=PolynomialRing(Rationals(),2);
%_<i>:=PolynomialRing(L);
%F:=(41*i - 36)*s^4 + (-144*i - 164)*s^3*t + (-246*i + 216)*s^2*t^2 + (144*i +
%    164)*s*t^3 + (41*i - 36)*t^4;
%F;
%A:=- 36*s^4 - 164*s^3*t +216*s^2*t^2 + 164*s*t^3 - 36*t^4;
%B:=41*s^4 - 144*s^3*t - 246*s^2*t^2 + 144*s*t^3 + 41*t^4;
%F-A-i*B;
\begin{lstlisting}
P<s,t,u,v>:=ProjectiveSpace(Rationals(),3);
A:=- 36*s^4 - 164*s^3*t +216*s^2*t^2 + 164*s*t^3 - 36*t^4-6*u^4;
B:=41*s^4 - 144*s^3*t - 246*s^2*t^2 + 144*s*t^3 + 41*t^4-v^4;
S:=Scheme(P,[A,B]);
IsLocallySolvable(S,2);
\end{lstlisting}
{\bf Case 3:} $6u^4+v^4i=i^2(3+2i)(15+2i)(s+ti)^4$.  Equating the real and imaginary parts on both sides of this equation gives
\begin{equation}\label{8931-3}
\begin{cases}
-41*s^4 + 144*s^3*t + 246*s^2*t^2 - 144*s*t^3 - 41*t^4 - 6*u^4=0,\\
-36*s^4 - 164*s^3*t + 216*s^2*t^2 + 164*s*t^3 - 36*t^4 - v^4=0.
\end{cases}
\end{equation}
The scheme defined by system \eqref{8931-3} is locally insoluble at $2$.

%\begin{lstlisting}
%_<x>:=PolynomialRing(Rationals());
%k<i>:=NumberField(x^2+1);
%K<s,t>:=PolynomialRing(k,2);
%e:=2;
%F:=i^e*(3+2*i)*(15+2*i)*(s+t*i)^4;
%F;
%L<s,t>:=PolynomialRing(Rationals(),2);
%_<i>:=PolynomialRing(L);
%F:=(-36*i - 41)*s^4 + (-164*i + 144)*s^3*t + (216*i + 246)*s^2*t^2 + (164*i -
%    144)*s*t^3 + (-36*i - 41)*t^4;
%F;
%A:= - 41*s^4 + 144*s^3*t+ 246*s^2*t^2 - 144*s*t^3 - 41*t^4;
%B:=-36*s^4 - 164*s^3*t + 216*s^2*t^2 + 164*s*t^3 - 36*t^4;
%F-A-i*B;
\begin{lstlisting}
P<s,t,u,v>:=ProjectiveSpace(Rationals(),3);
A:= - 41*s^4 + 144*s^3*t+ 246*s^2*t^2 - 144*s*t^3 - 41*t^4-6*u^4;
B:=-36*s^4 - 164*s^3*t + 216*s^2*t^2 + 164*s*t^3 - 36*t^4-v^4;
S:=Scheme(P,[A,B]);
IsLocallySolvable(S,2);
\end{lstlisting}

{\bf Case 4:} $6u^4+v^4i=i^3(3+2i)(15+2i)(s+ti)^4$.  Equating the real and imaginary parts on both sides of this equation gives
\begin{equation}\label{8931-4}
\begin{cases}
-41*s^4 + 144*s^3*t + 246*s^2*t^2 - 144*s*t^3 - 41*t^4 - 6*u^4=0,\\
-36*s^4 - 164*s^3*t + 216*s^2*t^2 + 164*s*t^3 - 36*t^4 - v^4=0.
\end{cases}
\end{equation}
The scheme defined by system \eqref{8931-4} is locally insoluble at $2$.
%
%\begin{lstlisting}
%_<x>:=PolynomialRing(Rationals());
%k<i>:=NumberField(x^2+1);
%K<s,t>:=PolynomialRing(k,2);
%e:=3;
%F:=i^e*(3+2*i)*(15+2*i)*(s+t*i)^4;
%F;
%L<s,t>:=PolynomialRing(Rationals(),2);
%_<i>:=PolynomialRing(L);
%F:=(-36*i - 41)*s^4 + (-164*i + 144)*s^3*t + (216*i + 246)*s^2*t^2 + (164*i -
%    144)*s*t^3 + (-36*i - 41)*t^4;
%F;
%A:= - 41*s^4 + 144*s^3*t+ 246*s^2*t^2 - 144*s*t^3 - 41*t^4;
%B:=-36*s^4 - 164*s^3*t + 216*s^2*t^2 + 164*s*t^3 - 36*t^4;
%F-A-i*B;
\begin{lstlisting}
P<s,t,u,v>:=ProjectiveSpace(Rationals(),3);
A:= - 41*s^4 + 144*s^3*t+ 246*s^2*t^2 - 144*s*t^3 - 41*t^4-6*u^4;
B:=-36*s^4 - 164*s^3*t + 216*s^2*t^2 + 164*s*t^3 - 36*t^4-v^4;
S:=Scheme(P,[A,B]);
IsLocallySolvable(S,2);
\end{lstlisting}

{\bf Case 5:} $6u^4+v^4i=(3+2i)(15+2i)(s+ti)^4$. Equating the real and imaginary parts on both sides of this equation gives
\begin{equation}\label{8931-5}
\begin{cases}
49*s^4 - 96*s^3*t - 294*s^2*t^2 + 96*s*t^3 + 49*t^4 - 6*u^4=0,\\
24*s^4 + 196*s^3*t - 144*s^2*t^2 - 196*s*t^3 + 24*t^4 - v^4=0.
\end{cases}
\end{equation}
The scheme defined by system \eqref{8931-5} is locally insoluble at $2$.

%\begin{lstlisting}
%_<x>:=PolynomialRing(Rationals());
%k<i>:=NumberField(x^2+1);
%K<s,t>:=PolynomialRing(k,2);
%e:=0;
%F:=i^e*(3+2*i)*(15-2*i)*(s+t*i)^4;
%F;
%L<s,t>:=PolynomialRing(Rationals(),2);
%_<i>:=PolynomialRing(L);
%F:=(24*i + 49)*s^4 + (196*i - 96)*s^3*t + (-144*i - 294)*s^2*t^2 + (-196*i +
%    96)*s*t^3 + (24*i + 49)*t^4;
%F;
%A:= 49*s^4 - 96*s^3*t -294*s^2*t^2 + 96*s*t^3 + 49*t^4;
%B:=24*s^4 + 196*s^3*t - 144*s^2*t^2 - 196*s*t^3 + 24*t^4;
%F-A-i*B;
\begin{lstlisting}
P<s,t,u,v>:=ProjectiveSpace(Rationals(),3);
A:= 49*s^4 - 96*s^3*t -294*s^2*t^2 + 96*s*t^3 + 49*t^4-6*u^4;
B:=24*s^4 + 196*s^3*t - 144*s^2*t^2 - 196*s*t^3 + 24*t^4-v^4;
S:=Scheme(P,[A,B]);
IsLocallySolvable(S,2);
\end{lstlisting}

{\bf Case 6:} $6u^4+v^4i=i(3+2i)(15+2i)(s+ti)^4$. Equating the real and imaginary parts on both sides of this equation gives
\begin{equation}\label{8931-6}
\begin{cases}
-24*s^4 - 196*s^3*t + 144*s^2*t^2 + 196*s*t^3 - 24*t^4 - 6*u^4=0,\\
49*s^4 - 96*s^3*t - 294*s^2*t^2 + 96*s*t^3 + 49*t^4 - v^4=0.
\end{cases}
\end{equation}
The scheme defined by system \eqref{8931-6} is locally insoluble at $2$.

%\begin{lstlisting}
%_<x>:=PolynomialRing(Rationals());
%k<i>:=NumberField(x^2+1);
%K<s,t>:=PolynomialRing(k,2);
%e:=1;
%F:=i^e*(3+2*i)*(15-2*i)*(s+t*i)^4;
%F;
%L<s,t>:=PolynomialRing(Rationals(),2);
%_<i>:=PolynomialRing(L);
%F:=(49*i - 24)*s^4 + (-96*i - 196)*s^3*t + (-294*i + 144)*s^2*t^2 + (96*i +
%    196)*s*t^3 + (49*i - 24)*t^4;
%F;
%A:=- 24*s^4 - 196*s^3*t +144*s^2*t^2 + 196*s*t^3 - 24*t^4;
%B:=49*s^4 - 96*s^3*t - 294*s^2*t^2 + 96*s*t^3 + 49*t^4;
%F-A-i*B;
\begin{lstlisting}
P<s,t,u,v>:=ProjectiveSpace(Rationals(),3);
A:=- 24*s^4 - 196*s^3*t +144*s^2*t^2 + 196*s*t^3 - 24*t^4-6*u^4;
B:=49*s^4 - 96*s^3*t - 294*s^2*t^2 + 96*s*t^3 + 49*t^4-v^4;
S:=Scheme(P,[A,B]);
IsLocallySolvable(S,2);
\end{lstlisting}

{\bf Case 7:} $6u^4+v^4i=i^2(3+2i)(15+2i)(s+ti)^4$.  Equating the real and imaginary parts on both sides of this equation gives
\begin{equation}\label{8931-7}
\begin{cases}
-49*s^4 + 96*s^3*t + 294*s^2*t^2 - 96*s*t^3 - 49*t^4 - 6*u^4=0,\\
-24*s^4 - 196*s^3*t + 144*s^2*t^2 + 196*s*t^3 - 24*t^4 - v^4=0.
\end{cases}
\end{equation}
The scheme defined by system \eqref{8931-7} is locally insoluble at $2$.

%\begin{lstlisting}
%_<x>:=PolynomialRing(Rationals());
%k<i>:=NumberField(x^2+1);
%K<s,t>:=PolynomialRing(k,2);
%e:=2;
%F:=i^e*(3+2*i)*(15-2*i)*(s+t*i)^4;
%F;
%L<s,t>:=PolynomialRing(Rationals(),2);
%_<i>:=PolynomialRing(L);
%F:=(-24*i - 49)*s^4 + (-196*i + 96)*s^3*t + (144*i + 294)*s^2*t^2 + (196*i -
%    96)*s*t^3 + (-24*i - 49)*t^4;
%F;
%A:=- 49*s^4 + 96*s^3*t +294*s^2*t^2 - 96*s*t^3 - 49*t^4;
%B:=-24*s^4 - 196*s^3*t + 144*s^2*t^2 + 196*s*t^3 - 24*t^4;
%F-A-i*B;
\begin{lstlisting}
P<s,t,u,v>:=ProjectiveSpace(Rationals(),3);
A:=- 49*s^4 + 96*s^3*t +294*s^2*t^2 - 96*s*t^3 - 49*t^4-6*u^4;
B:=-24*s^4 - 196*s^3*t + 144*s^2*t^2 + 196*s*t^3 - 24*t^4-v^4;
S:=Scheme(P,[A,B]);
IsLocallySolvable(S,2);
\end{lstlisting}

{\bf Case 8:} $6u^4+v^4i=i^3(3+2i)(15+2i)(s+ti)^4$.  Equating the real and imaginary parts on both sides of this equation gives
\begin{equation}\label{8931-8}
\begin{cases}
24*s^4 + 196*s^3*t - 144*s^2*t^2 - 196*s*t^3 + 24*t^4 - 6*u^4=0,\\
-49*s^4 + 96*s^3*t + 294*s^2*t^2 - 96*s*t^3 - 49*t^4 - v^4=0.
\end{cases}
\end{equation}
The scheme defined by system \eqref{8931-8} is locally insoluble at $2$.

%\begin{lstlisting}
%_<x>:=PolynomialRing(Rationals());
%k<i>:=NumberField(x^2+1);
%K<s,t>:=PolynomialRing(k,2);
%e:=3;
%F:=i^e*(3+2*i)*(15-2*i)*(s+t*i)^4;
%F;
%L<s,t>:=PolynomialRing(Rationals(),2);
%_<i>:=PolynomialRing(L);
%F:=(-49*i + 24)*s^4 + (96*i + 196)*s^3*t + (294*i - 144)*s^2*t^2 + (-96*i -
%    196)*s*t^3 + (-49*i + 24)*t^4;
%F;
%A:= 24*s^4 + 196*s^3*t -144*s^2*t^2 - 196*s*t^3 + 24*t^4;
%B:=-49*s^4 + 96*s^3*t + 294*s^2*t^2 - 96*s*t^3 - 49*t^4;
%F-A-i*B;
\begin{lstlisting}
P<s,t,u,v>:=ProjectiveSpace(Rationals(),3);
A:= 24*s^4 + 196*s^3*t -144*s^2*t^2 - 196*s*t^3 + 24*t^4-6*u^4;
B:=-49*s^4 + 96*s^3*t + 294*s^2*t^2 - 96*s*t^3 - 49*t^4-v^4;
S:=Scheme(P,[A,B]);
IsLocallySolvable(S,2);
\end{lstlisting}

{\bf Case 9:} $6u^4+v^4i=(3-2i)(15+2i)(s+ti)^4$. Equating the real and imaginary parts on both sides of this equation gives
\begin{equation}\label{8931-9}
\begin{cases}
49*s^4 + 96*s^3*t - 294*s^2*t^2 - 96*s*t^3 + 49*t^4 - 6*u^4=0,\\
-24*s^4 + 196*s^3*t + 144*s^2*t^2 - 196*s*t^3 - 24*t^4 - v^4=0.
\end{cases}
\end{equation}
The scheme defined by system \eqref{8931-9} is locally insoluble at $2$.

%\begin{lstlisting}
%_<x>:=PolynomialRing(Rationals());
%k<i>:=NumberField(x^2+1);
%K<s,t>:=PolynomialRing(k,2);
%e:=0;
%F:=i^e*(3-2*i)*(15+2*i)*(s+t*i)^4;
%F;
%L<s,t>:=PolynomialRing(Rationals(),2);
%_<i>:=PolynomialRing(L);
%F:=(-24*i + 49)*s^4 + (196*i + 96)*s^3*t + (144*i - 294)*s^2*t^2 + (-196*i -
%    96)*s*t^3 + (-24*i + 49)*t^4;
%F;
%A:=49*s^4 + 96*s^3*t -294*s^2*t^2 - 96*s*t^3 + 49*t^4;
%B:=-24*s^4 + 196*s^3*t + 144*s^2*t^2 - 196*s*t^3 - 24*t^4;
%F-A-i*B;
\begin{lstlisting}
P<s,t,u,v>:=ProjectiveSpace(Rationals(),3);
A:=49*s^4 + 96*s^3*t -294*s^2*t^2 - 96*s*t^3 + 49*t^4-6*u^4;
B:=-24*s^4 + 196*s^3*t + 144*s^2*t^2 - 196*s*t^3 - 24*t^4-v^4;
S:=Scheme(P,[A,B]);
IsLocallySolvable(S,2);
\end{lstlisting}

{\bf Case 10:} $6u^4+v^4i=i(3-2i)(15+2i)(s+ti)^4$. Equating the real and imaginary parts on both sides of this equation gives
\begin{equation}\label{8931-10}
\begin{cases}
24*s^4 - 196*s^3*t - 144*s^2*t^2 + 196*s*t^3 + 24*t^4 - 6*u^4=0,\\
49*s^4 + 96*s^3*t - 294*s^2*t^2 - 96*s*t^3 + 49*t^4 - v^4=0.
\end{cases}
\end{equation}
The scheme defined by system \eqref{8931-10} is locally insoluble at $2$.

%\begin{lstlisting}
%_<x>:=PolynomialRing(Rationals());
%k<i>:=NumberField(x^2+1);
%K<s,t>:=PolynomialRing(k,2);
%e:=1;
%F:=i^e*(3-2*i)*(15+2*i)*(s+t*i)^4;
%F;
%L<s,t>:=PolynomialRing(Rationals(),2);
%_<i>:=PolynomialRing(L);
%F:=(49*i + 24)*s^4 + (96*i - 196)*s^3*t + (-294*i - 144)*s^2*t^2 + (-96*i +
%    196)*s*t^3 + (49*i + 24)*t^4;
%F;
%A:=24*s^4 - 196*s^3*t -144*s^2*t^2 + 196*s*t^3 + 24*t^4;
%B:=49*s^4 + 96*s^3*t - 294*s^2*t^2 - 96*s*t^3 + 49*t^4;
%F-A-i*B;
\begin{lstlisting}
P<s,t,u,v>:=ProjectiveSpace(Rationals(),3);
A:=24*s^4 - 196*s^3*t -144*s^2*t^2 + 196*s*t^3 + 24*t^4-6*u^4;
B:=49*s^4 + 96*s^3*t - 294*s^2*t^2 - 96*s*t^3 + 49*t^4-v^4;
S:=Scheme(P,[A,B]);
IsLocallySolvable(S,2);
\end{lstlisting}
{\bf Case 11:} $6u^4+v^4i=i^2(3-2i)(15+2i)(s+ti)^4$.  Equating the real and imaginary parts on both sides of this equation gives

\begin{equation}\label{8931-11}
\begin{cases}
-49*s^4 - 96*s^3*t + 294*s^2*t^2 + 96*s*t^3 - 49*t^4 - 6*u^4=0,\\
24*s^4 - 196*s^3*t - 144*s^2*t^2 + 196*s*t^3 + 24*t^4 - v^4=0.
\end{cases}
\end{equation}
The scheme defined by system \eqref{8931-11} is locally insoluble at $2$.

%\begin{lstlisting}
%_<x>:=PolynomialRing(Rationals());
%k<i>:=NumberField(x^2+1);
%K<s,t>:=PolynomialRing(k,2);
%e:=2;
%F:=i^e*(3-2*i)*(15+2*i)*(s+t*i)^4;
%F;
%L<s,t>:=PolynomialRing(Rationals(),2);
%_<i>:=PolynomialRing(L);
%F:=(24*i - 49)*s^4 + (-196*i - 96)*s^3*t + (-144*i + 294)*s^2*t^2 + (196*i +
%    96)*s*t^3 + (24*i - 49)*t^4;
%F;
%A:=- 49*s^4 - 96*s^3*t + 294*s^2*t^2 + 96*s*t^3 - 49*t^4;
%B:=24*s^4 - 196*s^3*t - 144*s^2*t^2 + 196*s*t^3 + 24*t^4;
%F-A-i*B;
\begin{lstlisting}
P<s,t,u,v>:=ProjectiveSpace(Rationals(),3);
A:=- 49*s^4 - 96*s^3*t + 294*s^2*t^2 + 96*s*t^3 - 49*t^4-6*u^4;
B:=24*s^4 - 196*s^3*t - 144*s^2*t^2 + 196*s*t^3 + 24*t^4-v^4;
S:=Scheme(P,[A,B]);
IsLocallySolvable(S,2);
\end{lstlisting}

{\bf Case 12:} $6u^4+v^4i=i^3(3-2i)(15+2i)(s+ti)^4$.  Equating the real and imaginary parts on both sides of this equation gives

\begin{equation}\label{8931-12}
\begin{cases}
-24*s^4 + 196*s^3*t + 144*s^2*t^2 - 196*s*t^3 - 24*t^4 - 6*u^4=0,\\
-49*s^4 - 96*s^3*t + 294*s^2*t^2 + 96*s*t^3 - 49*t^4 - v^4=0.
\end{cases}
\end{equation}
The scheme defined by system \eqref{8931-12} is locally insoluble at $2$.

%\begin{lstlisting}
%_<x>:=PolynomialRing(Rationals());
%k<i>:=NumberField(x^2+1);
%K<s,t>:=PolynomialRing(k,2);
%e:=3;
%F:=i^e*(3-2*i)*(15+2*i)*(s+t*i)^4;
%F;
%L<s,t>:=PolynomialRing(Rationals(),2);
%_<i>:=PolynomialRing(L);
%F:=(-49*i - 24)*s^4 + (-96*i + 196)*s^3*t + (294*i + 144)*s^2*t^2 + (96*i -
%    196)*s*t^3 + (-49*i - 24)*t^4;
%F;
%A:=- 24*s^4 + 196*s^3*t + 144*s^2*t^2 - 196*s*t^3 - 24*t^4;
%B:=-49*s^4 - 96*s^3*t + 294*s^2*t^2 + 96*s*t^3 - 49*t^4;
%F-A-i*B;
\begin{lstlisting}
P<s,t,u,v>:=ProjectiveSpace(Rationals(),3);
A:=- 24*s^4 + 196*s^3*t + 144*s^2*t^2 - 196*s*t^3 - 24*t^4-6*u^4;
B:=-49*s^4 - 96*s^3*t + 294*s^2*t^2 + 96*s*t^3 - 49*t^4-v^4;
S:=Scheme(P,[A,B]);
IsLocallySolvable(S,2);
\end{lstlisting}
{\bf Case 13:} $6u^4+v^4i=(3-2i)(15-2i)(s+ti)^4$. Equating the real and imaginary parts on both sides of this equation gives

\begin{equation}\label{8931-13}
\begin{cases}
41*s^4 + 144*s^3*t - 246*s^2*t^2 - 144*s*t^3 + 41*t^4 - 6*u^4=0,\\
-36*s^4 + 164*s^3*t + 216*s^2*t^2 - 164*s*t^3 - 36*t^4 - v^4=0.
\end{cases}
\end{equation}
The scheme defined by system \eqref{8931-13} is locally insoluble at $2$.

%\begin{lstlisting}
%_<x>:=PolynomialRing(Rationals());
%k<i>:=NumberField(x^2+1);
%K<s,t>:=PolynomialRing(k,2);
%e:=0;
%F:=i^e*(3-2*i)*(15-2*i)*(s+t*i)^4;
%F;
%L<s,t>:=PolynomialRing(Rationals(),2);
%_<i>:=PolynomialRing(L);
%F:=(-36*i + 41)*s^4 + (164*i + 144)*s^3*t + (216*i - 246)*s^2*t^2 + (-164*i -
%    144)*s*t^3 + (-36*i + 41)*t^4;
%F;
%A:= 41*s^4 + 144*s^3*t- 246*s^2*t^2 - 144*s*t^3 + 41*t^4;
%B:=-36*s^4 + 164*s^3*t + 216*s^2*t^2 - 164*s*t^3 - 36*t^4;
%F-A-i*B;
\begin{lstlisting}
P<s,t,u,v>:=ProjectiveSpace(Rationals(),3);
A:= 41*s^4 + 144*s^3*t- 246*s^2*t^2 - 144*s*t^3 + 41*t^4-6*u^4;
B:=-36*s^4 + 164*s^3*t + 216*s^2*t^2 - 164*s*t^3 - 36*t^4-v^4;
S:=Scheme(P,[A,B]);
IsLocallySolvable(S,2);
\end{lstlisting}

{\bf Case 14:} $6u^4+v^4i=i(3-2i)(15-2i)(s+ti)^4$. Equating the real and imaginary parts on both sides of this equation gives

\begin{equation}\label{8931-14}
\begin{cases}
36*s^4 - 164*s^3*t - 216*s^2*t^2 + 164*s*t^3 + 36*t^4 - 6*u^4=0,\\
41*s^4 + 144*s^3*t - 246*s^2*t^2 - 144*s*t^3 + 41*t^4 - v^4=0.
\end{cases}
\end{equation}
The scheme defined by system \eqref{8931-14} is locally insoluble at $2$.

%\begin{lstlisting}
%_<x>:=PolynomialRing(Rationals());
%k<i>:=NumberField(x^2+1);
%K<s,t>:=PolynomialRing(k,2);
%e:=1;
%F:=i^e*(3-2*i)*(15-2*i)*(s+t*i)^4;
%F;
%L<s,t>:=PolynomialRing(Rationals(),2);
%_<i>:=PolynomialRing(L);
%F:=(41*i + 36)*s^4 + (144*i - 164)*s^3*t + (-246*i - 216)*s^2*t^2 + (-144*i +
%    164)*s*t^3 + (41*i + 36)*t^4;
%F;
%A:=36*s^4 - 164*s^3*t -216*s^2*t^2 + 164*s*t^3 + 36*t^4;
%B:=41*s^4 + 144*s^3*t - 246*s^2*t^2 - 144*s*t^3 + 41*t^4;
%F-A-i*B;
\begin{lstlisting}
P<s,t,u,v>:=ProjectiveSpace(Rationals(),3);
A:=36*s^4 - 164*s^3*t -216*s^2*t^2 + 164*s*t^3 + 36*t^4-6*u^4;
B:=41*s^4 + 144*s^3*t - 246*s^2*t^2 - 144*s*t^3 + 41*t^4-v^4;
S:=Scheme(P,[A,B]);
IsLocallySolvable(S,2);
\end{lstlisting}
{\bf Case 15:} $6u^4+v^4i=i^2(3-2i)(15-2i)(s+ti)^4$.  Equating the real and imaginary parts on both sides of this equation gives

\begin{equation}\label{8931-15}
\begin{cases}
-41*s^4 - 144*s^3*t + 246*s^2*t^2 + 144*s*t^3 - 41*t^4 - 6*u^4=0,\\
36*s^4 - 164*s^3*t - 216*s^2*t^2 + 164*s*t^3 + 36*t^4 - v^4=0.
\end{cases}
\end{equation}
The scheme defined by system \eqref{8931-15} is locally insoluble at $2$.

%\begin{lstlisting}
%_<x>:=PolynomialRing(Rationals());
%k<i>:=NumberField(x^2+1);
%K<s,t>:=PolynomialRing(k,2);
%e:=2;
%F:=i^e*(3-2*i)*(15-2*i)*(s+t*i)^4;
%F;
%L<s,t>:=PolynomialRing(Rationals(),2);
%_<i>:=PolynomialRing(L);
%F:=(36*i - 41)*s^4 + (-164*i - 144)*s^3*t + (-216*i + 246)*s^2*t^2 + (164*i +
%    144)*s*t^3 + (36*i - 41)*t^4;
%F;
%A:=- 41*s^4 - 144*s^3*t + 246*s^2*t^2 + 144*s*t^3 - 41*t^4;
%B:=36*s^4 - 164*s^3*t - 216*s^2*t^2 + 164*s*t^3 + 36*t^4;
%F-A-i*B;
\begin{lstlisting}
P<s,t,u,v>:=ProjectiveSpace(Rationals(),3);
A:=- 41*s^4 - 144*s^3*t + 246*s^2*t^2 + 144*s*t^3 - 41*t^4-6*u^4;
B:=36*s^4 - 164*s^3*t - 216*s^2*t^2 + 164*s*t^3 + 36*t^4-v^4;
S:=Scheme(P,[A,B]);
IsLocallySolvable(S,2);
\end{lstlisting}

{\bf Case 16:} $6u^4+v^4i=i^3(3-2i)(15-2i)(s+ti)^4$.  Equating the real and imaginary parts on both sides of this equation gives

\begin{equation}\label{8931-16}
\begin{cases}
-36*s^4 + 164*s^3*t + 216*s^2*t^2 - 164*s*t^3 - 36*t^4 - 6*u^4=0,\\
-41*s^4 - 144*s^3*t + 246*s^2*t^2 + 144*s*t^3 - 41*t^4 - v^4=0.
\end{cases}
\end{equation}
The scheme defined by system \eqref{8931-16} is locally insoluble at $2$.

%\begin{lstlisting}
%_<x>:=PolynomialRing(Rationals());
%k<i>:=NumberField(x^2+1);
%K<s,t>:=PolynomialRing(k,2);
%e:=3;
%F:=i^e*(3-2*i)*(15-2*i)*(s+t*i)^4;
%F;
%L<s,t>:=PolynomialRing(Rationals(),2);
%_<i>:=PolynomialRing(L);
%F:=(-41*i - 36)*s^4 + (-144*i + 164)*s^3*t + (246*i + 216)*s^2*t^2 + (144*i -
%    164)*s*t^3 + (-41*i - 36)*t^4;
%F;
%A:=- 36*s^4 + 164*s^3*t + 216*s^2*t^2 - 164*s*t^3 - 36*t^4;
%B:=-41*s^4 - 144*s^3*t + 246*s^2*t^2 + 144*s*t^3 - 41*t^4;
%F-A-i*B;
\begin{lstlisting}
P<s,t,u,v>:=ProjectiveSpace(Rationals(),3);
A:=- 36*s^4 + 164*s^3*t + 216*s^2*t^2 - 164*s*t^3 - 36*t^4-6*u^4;
B:=-41*s^4 - 144*s^3*t + 246*s^2*t^2 + 144*s*t^3 - 41*t^4-v^4;
S:=Scheme(P,[A,B]);
IsLocallySolvable(S,2);
\end{lstlisting}

\subsection{The case of $n=9015$.}%\[n=9015\]

 According to Table \ref{tb:factorisation}, in the case $n=9015$ it remains to
show that equation
\begin{equation}\label{9015}
36u^8+25v^8=601 w^4
\end{equation}
has no integer solutions satisfying \eqref{eq:cond}, which in this case reduces to \[\gcd(6u,5v)=\gcd(6u,601w)=\gcd(5v,601w)=1.\]
Write \eqref{9015} as 
\begin{equation} (6u^4+5v^4i)(6u^4-5v^4i)=(24+ 5i)(24- 5i)w^4.
\end{equation}
This implies that there exist integers $s,t$ such that 
\[6u^4+5v^4i=i^{\epsilon}(24\pm 5i)(s+ti)^4,\]
with $\epsilon \in \{0,1,2,3\}$.

{\bf Case 1:} $6u^4+5v^4i=(24+5i)(s+ti)^4$. Equating the real and imaginary parts on both sides of this equation gives
\begin{equation}\label{9015-1}
\begin{cases}
24s^4 - 20s^3t - 144s^2t^2 + 20st^3 + 24t^4 - 6u^4=0,\\
5s^4 + 96s^3t - 30s^2t^2 - 96st^3 + 5t^4 - 5v^4=0.
\end{cases}
\end{equation}
The scheme defined by system \eqref{9015-1} is locally insoluble at $2$.

%\begin{lstlisting}
%_<x>:=PolynomialRing(Rationals());
%k<i>:=NumberField(x^2+1);
%K<s,t>:=PolynomialRing(k,2);
%e:=0;
%F:=i^e*(24+5*i)*(s+t*i)^4;
%F;
%L<s,t>:=PolynomialRing(Rationals(),2);
%_<i>:=PolynomialRing(L);
%F:=(5*i + 24)*s^4 + (96*i - 20)*s^3*t + (-30*i - 144)*s^2*t^2 + (-96*i + 20)*s*t^3
%    + (5*i + 24)*t^4;
%F;
%A:=24*s^4 - 20*s^3*t -144*s^2*t^2 + 20*s*t^3 + 24*t^4;
%B:=5*s^4 + 96*s^3*t - 30*s^2*t^2 - 96*s*t^3 + 5*t^4;
%F-A-i*B;
\begin{lstlisting}
P<s,t,u,v>:=ProjectiveSpace(Rationals(),3);
A:=24*s^4 - 20*s^3*t -144*s^2*t^2 + 20*s*t^3 + 24*t^4-6*u^4;
B:=5*s^4 + 96*s^3*t - 30*s^2*t^2 - 96*s*t^3 + 5*t^4-5*v^4;
S:=Scheme(P,[A,B]);
IsLocallySolvable(S,2);
\end{lstlisting}

{\bf Case 2:} $6u^4+5v^4i=i(24+5i)(s+ti)^4$. Equating the real and imaginary parts on both sides of this equation gives
\begin{equation}\label{9015-2}
\begin{cases}
-5s^4 - 96s^3t + 30s^2t^2 + 96st^3 - 5t^4 - 6u^4=0,\\
24s^4 - 20s^3t - 144s^2t^2 + 20st^3 + 24t^4 - 5v^4=0.
\end{cases}
\end{equation}
The scheme defined by system \eqref{9015-2} is locally insoluble at $2$.

%\begin{lstlisting}
%_<x>:=PolynomialRing(Rationals());
%k<i>:=NumberField(x^2+1);
%K<s,t>:=PolynomialRing(k,2);
%e:=1;
%F:=i^e*(24+5*i)*(s+t*i)^4;
%F;
%L<s,t>:=PolynomialRing(Rationals(),2);
%_<i>:=PolynomialRing(L);
%F:=(24*i - 5)*s^4 + (-20*i - 96)*s^3*t + (-144*i + 30)*s^2*t^2 + (20*i + 96)*s*t^3
%    + (24*i - 5)*t^4;
%F;
%A:=- 5*s^4 - 96*s^3*t + 30*s^2*t^2 + 96*s*t^3 - 5*t^4;
%B:=24*s^4 - 20*s^3*t - 144*s^2*t^2 + 20*s*t^3 + 24*t^4;
%F-A-i*B;

\begin{lstlisting}
P<s,t,u,v>:=ProjectiveSpace(Rationals(),3);
A:=- 5*s^4 - 96*s^3*t + 30*s^2*t^2 + 96*s*t^3 - 5*t^4-6*u^4;
B:=24*s^4 - 20*s^3*t - 144*s^2*t^2 + 20*s*t^3 + 24*t^4-5*v^4;
S:=Scheme(P,[A,B]);
IsLocallySolvable(S,2);
\end{lstlisting}

{\bf Case 3:} $6u^4+5v^4i=i^2(24+5i)(s+ti)^4$. Equating the real and imaginary parts on both sides of this equation gives
\begin{equation}\label{9015-3}
\begin{cases}
-24s^4 + 20s^3t + 144s^2t^2 - 20st^3 - 24t^4 - 6u^4=0,\\
-5s^4 - 96s^3t + 30s^2t^2 + 96st^3 - 5t^4 - 5v^4=0.
\end{cases}
\end{equation}
The scheme defined by system \eqref{9015-3} is locally insoluble at $2$.

%\begin{lstlisting}
%_<x>:=PolynomialRing(Rationals());
%k<i>:=NumberField(x^2+1);
%K<s,t>:=PolynomialRing(k,2);
%e:=2;
%F:=i^e*(24+5*i)*(s+t*i)^4;
%F;
%L<s,t>:=PolynomialRing(Rationals(),2);
%_<i>:=PolynomialRing(L);
%F:=(-5*i - 24)*s^4 + (-96*i + 20)*s^3*t + (30*i + 144)*s^2*t^2 + (96*i - 20)*s*t^3
%    + (-5*i - 24)*t^4;
%F;
%A:=- 24*s^4 + 20*s^3*t +144*s^2*t^2 - 20*s*t^3 - 24*t^4;
%B:=-5*s^4 - 96*s^3*t + 30*s^2*t^2 + 96*s*t^3 - 5*t^4;
%F-A-i*B;
\begin{lstlisting}
P<s,t,u,v>:=ProjectiveSpace(Rationals(),3);
A:=- 24*s^4 + 20*s^3*t +144*s^2*t^2 - 20*s*t^3 - 24*t^4-6*u^4;
B:=-5*s^4 - 96*s^3*t + 30*s^2*t^2 + 96*s*t^3 - 5*t^4-5*v^4;
S:=Scheme(P,[A,B]);
IsLocallySolvable(S,2);
\end{lstlisting}

{\bf Case 4:} $6u^4+5v^4i=i^3(24+5i)(s+ti)^4$. Equating the real and imaginary parts on both sides of this equation gives
\begin{equation}\label{9015-4}
\begin{cases}
5s^4 + 96s^3t - 30s^2t^2 - 96st^3 + 5t^4 - 6u^4=0,\\
-24s^4 + 20s^3t + 144s^2t^2 - 20st^3 - 24t^4 - 5v^4=0.
\end{cases}
\end{equation}
The scheme defined by system \eqref{9015-4} is locally insoluble at $2$.

%\begin{lstlisting}
%_<x>:=PolynomialRing(Rationals());
%k<i>:=NumberField(x^2+1);
%K<s,t>:=PolynomialRing(k,2);
%e:=3;
%F:=i^e*(24+5*i)*(s+t*i)^4;
%F;
%L<s,t>:=PolynomialRing(Rationals(),2);
%_<i>:=PolynomialRing(L);
%F:=(-24*i + 5)*s^4 + (20*i + 96)*s^3*t + (144*i - 30)*s^2*t^2 + (-20*i - 96)*s*t^3
%    + (-24*i + 5)*t^4;
%F;
%A:= 5*s^4 + 96*s^3*t -30*s^2*t^2 - 96*s*t^3 + 5*t^4;
%B:=-24*s^4 + 20*s^3*t + 144*s^2*t^2 - 20*s*t^3 - 24*t^4;
%F-A-i*B;
\begin{lstlisting}
P<s,t,u,v>:=ProjectiveSpace(Rationals(),3);
A:= 5*s^4 + 96*s^3*t -30*s^2*t^2 - 96*s*t^3 + 5*t^4-6*u^4;
B:=-24*s^4 + 20*s^3*t + 144*s^2*t^2 - 20*s*t^3 - 24*t^4-5*v^4;
S:=Scheme(P,[A,B]);
IsLocallySolvable(S,2);
\end{lstlisting}

{\bf Case 5:} $6u^4+5v^4i=(24-5i)(s+ti)^4$. Equating the real and imaginary parts on both sides of this equation gives
\begin{equation}\label{9015-5}
\begin{cases}
24s^4 + 20s^3t - 144s^2t^2 - 20st^3 + 24t^4 - 6u^4=0,\\
-5s^4 + 96s^3t + 30s^2t^2 - 96st^3 - 5t^4 - 5v^4=0.
\end{cases}
\end{equation}
The scheme defined by system \eqref{9015-5} is locally insoluble at $2$.

%\begin{lstlisting}
%_<x>:=PolynomialRing(Rationals());
%k<i>:=NumberField(x^2+1);
%K<s,t>:=PolynomialRing(k,2);
%e:=0;
%F:=i^e*(24+5*i)*(s+t*i)^4;
%F;
%L<s,t>:=PolynomialRing(Rationals(),2);
%_<i>:=PolynomialRing(L);
%F:=(5*i + 24)*s^4 + (96*i - 20)*s^3*t + (-30*i - 144)*s^2*t^2 + (-96*i + 20)*s*t^3
%    + (5*i + 24)*t^4;
%F;
%A:=24*s^4 - 20*s^3*t -144*s^2*t^2 + 20*s*t^3 + 24*t^4;
%B:=5*s^4 + 96*s^3*t - 30*s^2*t^2 - 96*s*t^3 + 5*t^4;
%F-A-i*B;
\begin{lstlisting}
P<s,t,u,v>:=ProjectiveSpace(Rationals(),3);
A:=24*s^4 + 20*s^3*t - 144*s^2*t^2 - 20*s*t^3 + 24*t^4-6*u^4;
B:=-5*s^4 + 96*s^3*t + 30*s^2*t^2 - 96*s*t^3 - 5*t^4-5*v^4;
S:=Scheme(P,[A,B]);
IsLocallySolvable(S,2);
\end{lstlisting}

{\bf Case 6:} $6u^4+5v^4i=i(24-5i)(s+ti)^4$. Equating the real and imaginary parts on both sides of this equation gives
\begin{equation}\label{9015-6}
\begin{cases}
5s^4 - 96s^3t - 30s^2t^2 + 96st^3 + 5t^4 - 6u^4=0,\\
24s^4 + 20s^3t - 144s^2t^2 - 20st^3 + 24t^4 - 5v^4=0
\end{cases}
\end{equation}
The scheme defined by system \eqref{9015-6} is locally insoluble at $2$.

%\begin{lstlisting}
%_<x>:=PolynomialRing(Rationals());
%k<i>:=NumberField(x^2+1);
%K<s,t>:=PolynomialRing(k,2);
%e:=1;
%F:=i^e*(24-5*i)*(s+t*i)^4;
%F;
%L<s,t>:=PolynomialRing(Rationals(),2);
%_<i>:=PolynomialRing(L);
%F:=(24*i + 5)*s^4 + (20*i - 96)*s^3*t + (-144*i - 30)*s^2*t^2 + (-20*i + 96)*s*t^3
%    + (24*i + 5)*t^4;
%F;
%A:=5*s^4 - 96*s^3*t -30*s^2*t^2 + 96*s*t^3 + 5*t^4;
%B:=24*s^4 + 20*s^3*t - 144*s^2*t^2 - 20*s*t^3 + 24*t^4;
%F-A-i*B;
\begin{lstlisting}
P<s,t,u,v>:=ProjectiveSpace(Rationals(),3);
A:=5*s^4 - 96*s^3*t -30*s^2*t^2 + 96*s*t^3 + 5*t^4-6*u^4;
B:=24*s^4 + 20*s^3*t - 144*s^2*t^2 - 20*s*t^3 + 24*t^4-5*v^4;
S:=Scheme(P,[A,B]);
IsLocallySolvable(S,2);
\end{lstlisting}

{\bf Case 7:} $6u^4+5v^4i=i^2(24-5i)(s+ti)^4$. Equating the real and imaginary parts on both sides of this equation gives
\begin{equation}\label{9015-7}
\begin{cases}
-24s^4 - 20s^3t + 144s^2t^2 + 20st^3 - 24t^4 - 6u^4=0,\\
5s^4 - 96s^3t - 30s^2t^2 + 96st^3 + 5t^4 - 5v^4=0.
\end{cases}
\end{equation}
The scheme defined by system \eqref{9015-7} is locally insoluble at $2$.

%\begin{lstlisting}
%_<x>:=PolynomialRing(Rationals());
%k<i>:=NumberField(x^2+1);
%K<s,t>:=PolynomialRing(k,2);
%e:=2;
%F:=i^e*(24-5*i)*(s+t*i)^4;
%F;
%L<s,t>:=PolynomialRing(Rationals(),2);
%_<i>:=PolynomialRing(L);
%F:=(5*i - 24)*s^4 + (-96*i - 20)*s^3*t + (-30*i + 144)*s^2*t^2 + (96*i + 20)*s*t^3
%    + (5*i - 24)*t^4;
%F;
%A:=- 24*s^4 - 20*s^3*t +144*s^2*t^2 + 20*s*t^3 - 24*t^4;
%B:=5*s^4 - 96*s^3*t - 30*s^2*t^2 + 96*s*t^3 + 5*t^4;
%F-A-i*B;
\begin{lstlisting}
P<s,t,u,v>:=ProjectiveSpace(Rationals(),3);
A:=- 24*s^4 - 20*s^3*t +144*s^2*t^2 + 20*s*t^3 - 24*t^4-6*u^4;
B:=5*s^4 - 96*s^3*t - 30*s^2*t^2 + 96*s*t^3 + 5*t^4-5*v^4;
S:=Scheme(P,[A,B]);
IsLocallySolvable(S,2);
\end{lstlisting}

{\bf Case 8:} $6u^4+5v^4i=i^3(24-5i)(s+ti)^4$. Equating the real and imaginary parts on both sides of this equation gives
\begin{equation}\label{9015-8}
\begin{cases}
-5s^4 + 96s^3t + 30s^2t^2 - 96st^3 - 5t^4 - 6u^4=0,\\
-24s^4 - 20s^3t + 144s^2t^2 + 20st^3 - 24t^4 - 5v^4=0.
\end{cases}
\end{equation}
The scheme defined by system \eqref{9015-8} is locally insoluble at $2$.

%\begin{lstlisting}
%_<x>:=PolynomialRing(Rationals());
%k<i>:=NumberField(x^2+1);
%K<s,t>:=PolynomialRing(k,2);
%e:=3;
%F:=i^e*(24-5*i)*(s+t*i)^4;
%F;
%L<s,t>:=PolynomialRing(Rationals(),2);
%_<i>:=PolynomialRing(L);
%F:=(-24*i - 5)*s^4 + (-20*i + 96)*s^3*t + (144*i + 30)*s^2*t^2 + (20*i - 96)*s*t^3
%    + (-24*i - 5)*t^4;
%F;
%A:=- 5*s^4 + 96*s^3*t + 30*s^2*t^2 - 96*s*t^3 - 5*t^4;
%B:=-24*s^4 - 20*s^3*t + 144*s^2*t^2 + 20*s*t^3 - 24*t^4;
%F-A-i*B;
\begin{lstlisting}
P<s,t,u,v>:=ProjectiveSpace(Rationals(),3);
A:=- 5*s^4 + 96*s^3*t + 30*s^2*t^2 - 96*s*t^3 - 5*t^4-6*u^4;
B:=-24*s^4 - 20*s^3*t + 144*s^2*t^2 + 20*s*t^3 - 24*t^4-5*v^4;
S:=Scheme(P,[A,B]);
IsLocallySolvable(S,2);
\end{lstlisting}

\bibliographystyle{plain}
\bibliography{references}

\end{document}